\documentclass[12pt,centertags,oneside,reqno]{amsart} 	%\documentclass[reqno,preprint,12]{article}
\usepackage{amsmath,amssymb}
\usepackage{mathrsfs}\usepackage{bm}\usepackage{graphicx}
\allowdisplaybreaks

 \usepackage{fancyhdr}
 \usepackage{datetime2}
\usepackage[hidelinks]{hyperref}

\makeatletter
\@addtoreset{theorem}{section}
\@addtoreset{lemma}{section}
\@addtoreset{proposition}{section}
\@addtoreset{corollary}{section}

\@addtoreset{definition}{section}
\@addtoreset{remark}{section}

\usepackage{mathrsfs}\usepackage{bm}

\usepackage{amsthm,amsmath,amsfonts,amssymb}
\usepackage{a4wide}
\usepackage[mathscr]{eucal}
\usepackage{typearea}
\usepackage{charter}
\usepackage[a4paper,width=16.2cm,top=3cm,bottom=3cm]{geometry}
\usepackage{algorithm}
\usepackage{subfig}

\RequirePackage[numbers]{natbib}
\theoremstyle{plain}

\newtheorem{theorem}{Theorem}[section]
\newtheorem{proposition}{Proposition}[section]
\newtheorem{lemma}{Lemma}[section]
\newtheorem{corollary}{Corollary}[section]
\theoremstyle{remark}
\newtheorem{remark}{Remark}[section]
\newtheorem{definition}{Definition}[section]

\newtheorem{example}{Example}[section]

\newcommand\DN{\newcommand}
\DN\lref[1]{Lemma~\ref{#1}}
 \DN\tref[1]{Theorem~\ref{#1}}
 \DN\pref[1]{Proposition~\ref{#1}}
 \DN\sref[1]{Section~\ref{#1}}
 \DN\ssref[1]{Subsection~\ref{#1}}
 \DN\dref[1]{Definition~\ref{#1}}
 \DN\rref[1]{Remark~\ref{#1}} 
 \DN\coref[1]{Corollary~\ref{#1}}
 \DN\eref[1]{Example~\ref{#1}}

 \numberwithin{equation}{section} 
\newcounter{Const} 
\numberwithin{Const}{section}

\DN\Ct{\refstepcounter{Const}c_{\theConst}}\DN\cref[1]{c_{\small \ref{#1}}}
\DN\As[1]{$ ($\textbf{#1}$)$}\DN\Ass[1]{$ \{ $\textbf{#1}$\}$}
\DN\AS[1]{$ ($\textbf{#1}$)$}\DN\ASS[1]{$ \{ $\textbf{#1}$\}$}
\DN\uL[1]{\underline{#1}}\DN\oL[1]{\overline{#1}}

\DN\PF{\begin{proof}} 
	\DN\PFEND{\end{proof}} %
\DN\elaw{\mathop{\sim}\limits_{\mathrm{law}}}

\DN\eac{\stackrel{\mathrm{ac}}{\sim}}\DN\esg{\stackrel{\mathrm{sg}}{\sim}}

\DN\bs{\bigskip} \DN\ms{\medskip}\DN\ssp{\smallskip}
\DN\PD[2]{\frac{\partial#1}{\partial#2}}
\DN\half{\frac{1}{2}} \DN\halfbeta{\frac{\beta}{2}}
\DN\map[3]{#1\!:\!#2\!\to\!#3}\DN\ot{ \otimes } \DN\ts{ \times }
\DN\limi[1]{\lim_{#1\to\infty}} 	\DN\limz[1]{\lim_{#1\to0}}

\DN\limsupi[1]{\limsup_{#1\to\infty}} 
\DN\limsupz[1]{\limsup_{#1\to0}}
\DN\liminfi[1]{\liminf_{#1\to\infty}} \DN\liminfz[1]{\liminf_{#1\to 0}}
\DN\sumii[1]{\sum_{#1=1}^{\infty}}\DN\sumi[1]{\sum_{#1=0}^{\infty}}

\DN\sumijN{\sum _{i < j }^{\nN }}

\DN\sumiN{ \sum_{i \in \N } }
\DN\sumjN{ \sum_{j \in \N } }
\DN\sumkN{ \sum_{k \in \N } }

\DN\supnN{\sup _{\n \in \N }}
\DN\supN{\sup_{\nN \in \N } }
\DN\supx{\sup_{ x \in \Rtwo } }
\DN\supR{\sup_{\rR \in \N } }
\DN\supQ{\sup_{\qQ \in \N } }
\DN\infN{\inf_{\nN \in \N } }

\DN\limsupN{\limsupi{\nN }} 
\DN\limsupR{\limsupi{\rR }}
\DN\liminfR{\liminfi{\rR }}
\DN\liminfN{\liminf{\nN }} 

\DN\esup{\mathrm{ess.}\sup}

\DN\R{\mathbb{R}}\DN\N{\mathbb{N}}\DN\Q{\mathbb{Q}}\DN\Z{\mathbb{Z}}
\DN\two{2}\DN\Rtwo{\R ^2}\DN\RRtwo{\mathcal{R}^2 }
\DN\Rd{\R ^d}
\DN\tT{T}
\DN\uU{U}
\DN\RyyU{\rR , \yy , \uU }
\DN\RyyT{\rR , \yy , \tT }

\DN\nN{N}\DN\pP{P}\DN\qQ{Q}\DN\rR{R}\DN\sS{ S }

\DN\Nn{\nN _ {\n }}
\DN\psiN{\psi ^{\nN }}
\DN\psiNn{\psi ^{\Nn }}
\DN\e{\epsilon}\DN\ve{\varepsilon}

\DN\sumij{\sum_{\langle i , j \rangle}}

\DN\Lvert{\lvert } \DN\Rvert{\rvert }

\DN\xione{x^i}
\DN\iGamma{i_1! \cdots i_d!}
\DN\xiGamma{ \frac{ \x ^{\mathbf{i}} }{\iGamma }}

\DN\zetat{ t \wedge \zeta }
\DN\taukXR{ t \wedge \tau _k (\mathbb{X}_{\rR }^i )}
\DN\taukX{ t \wedge \tau _k (\mathbb{X}^i )}
\DN\XtaukXR{\mathbb{X}_{\rR }^i (\taukXR )}
\DN\XtaukX{\mathbb{X}^i (\taukX )}

\DN\oLSRxk{\oLSR ^k }
\DN\oLSRxC{(\oLSRc )^{ k }} 
\DN\oLSRxCl{(\oLSRc )^{ l }} 
\DN\oLSRxCz{(\oLSRc )^{\nN - 1 }} 
\DN\oLS{\oL{\sS }}
\DN\SSaa{\sSS_{\aaa }}
\DN\SSaam{\sSS_{\aaa }^{[m]}}
\DN\Saam{\mathbf{S}_{\aaa }^{[m]}}

\DN\oLSS{\oL{\sSS }}

\DN\SSRkone{\sSS _{\rR , k } ^{[1]}} 
\DN\oLSSRkone{\oLSS _{\rR , k } ^{[1]}} 

\DN\oLSSRkNone{\oLSS _{\rR , k } ^{\None }} 

	\DN\oLSSRklNone{\oLSSRkl ^{\None }} 
	\DN\oLSSRkl{\mathsf{M}_{ \Rkl }}
	\DN\oLSSRklone{\mathsf{M}_{\Rkl }^{[1]}} 

	\DN\oLSSRyykl{\mathsf{M} _{ \Ryykl }}
	\DN\oLSSRyyklNone{\oLSSRyykl ^{\None }} 
	\DN\oLSSRyyklone{\oLSSRyykl ^{[1]}}

\DN\SRSS{ \SR \times \sSS }
\DN\SRSSR{ \SR \times \SSR } \DN\oLSRSSR{ \oLSR \times \oLSSR }
\DN\sSsS{\oLSRSSR  }

\DN\SSRek{\oLSS _{\rRe }^k}
\DN\TT{\mathsf{T}}
\DN\UU{\mathsf{U}}
\DN\VV{\mathsf{V}}
\DN\rRe{\rR , \e }
\DN\Rekl{_{\rRe } ^{k,l}}

	\DN\piR{\pi _{\SR }} 
	\DN\piRc{\pi _{\SRc }} 
	\DN\pioLR{\pi _{\oLSR }}
	\DN\pioLRc{\pi _{\oLSR ^c }}
	\DN\piA{\pi _{A}}

\DN\lz{\ell } 
\DN\lzz{\lz _{0}}
\DN\Rone{\rR , 1 }
\DN\Rk{\rR , k }
\DN\Rkl{\rR , k , l }
\DN\Rkll{\Rkl + 1 }
\DN\Ryykl{\rR , \yy , k , l }
\DN\RyyklNn{\rR , \yyNn , k , l }
\DN\RyykNn{\rR , \yyNn , k } 
\DN\Ryykll{\rR , \yy , k , l + 1 }
\DN\limil{\limi{l} }

\DN\FFn{\FF _{\mmm }}
\DN\Nd{\nN ^{ \frac{2 - d}{d }}}

\DN\aaaadmu{ \mathrm{div} \aaaa + \dmu \aaaa }

\DN\aR{\aaa , \rR }
\DN\aRs{\aR , \sss }
\DN\aRy{\aR , \yy }
\DN\aRyNn{\rR ,\yyNn }

\DN\aRNnm{_{\rR } ^{\Nn , [m]}}
\DN\RNnm{_{\rR } ^{\Nn , [m]}}

\DN\EDaRNnm{( \E \aRNnm , \oLd \aRNnm )}
\DN\EDaRNnmSTAR{( \E \aRNnm , \oLd _{\Rstar  }^{\Nn , [m]} )}

\DN\uLEDaRNnm{( \E \aRNnm , \uLd \aRNnm )}
\DN\uLEDRNnm{( \E \RNnm , \uLd \RNnm )}

\DN\uLEDRNnmSTAR{( \ER ^{\Nn , [m]}  , \uLd _{\Rstar  }^{\Nn , [m]} )}
\DN\oLEDRNnmSTAR{( \ER ^{\Nn , [m]}  , \oLd _{\Rstar  }^{\Nn , [m]} )}

\DN\uLEDastarm{(\E _{\aaa }^{[m]} , \dom _{\aaa , \mathrm{lwr} \star }^{[m]} )}

\DN\uLEDstarm{(\Estarm , \dom _{\lwrstar }^{[m]} )}
\DN\oLEDstarm{(\Estarm , \dom _{\uprstar }^{[m]} )}

\DN\uLEDRstarm{(\ERstarm  , \uL{\mathscr{D}}_{\Rstar }^{[m]} )}
\DN\oLEDRstarm{(\ERstarm  , \oL{\mathscr{D}}_{\Rstar }^{[m]} )}

\DN\uLEDaRstarm{( \EaRstarm , \uLDaRstarm )}
\DN\uLEDaRstarmSTAR{( \E _{\star , \Rstar }^{[m]} , \uLd _{\star , \Rstar }^{[m]} )}

\DN\oLEDaRstarm{( \EaRstarm , \oLDaRstarm )}
\DN\oLEDaRstarmSTAR{( \E _{\star , \Rstar }^{[m]} , \oLd _{\star , \Rstar }^{[m]} )}

\DN\EaRstarm{\E _{\aR }^{[m]} } 
\DN\uLDaRstarm{\uLd _{\aR \star }^{[m]}}
\DN\oLDaRstarm{\oLd _{\aR \star }^{[m]}}

\DN\0{\m _{\rR }^{\None }}
\DN\1{ \sum_{\mathbf{i} \in \mathbb{I}( \lz )} }
\DN\III{\mathbb{I}(\lz ) }
\DN\IIIl{\mathbb{I}( l ) }

\DN\2{ \nabla \Psi (\xsi ) }
			\DN\nablaPhi{\nabla \Phi }
\DN\nablaPhiN{\nabla \Phi ^{\nN }}
\DN\nablaPhiNn{\nabla \Phi ^{\Nn }}

\DN\sumyiRc{\sum_{ \yiSRc }}
 \DN\sumyiRec{\sum_{\yiSRec }} 
\DN\sumyiRcN{\sumyiRc ^{\nN }}

\DN\3{\sumyiRc ^{ \Nn } }
\DN\4{\pioLR (\sss ) + \pioLRc (\yy ) }

\DN\9{\sum_{ \yiSRc } ^{\infty} }

\DN\sumyiRecN{\sumyiRec ^{\infty}}
\DN\sumyiRecNn{\sumyiRec ^{\Nn }}

\DN\SIXON{C^1(\oLSR \cap \ON )}
\DN\SIXONn{C^1(\oLSR \cap \ONn )}
\DN\6{C^1 (\oLSR )}
\DN\SIX{C^1 (\oLSR )}

\DN\SIXSIX{C^1 (\oLSRse )}
\DN\7{C^1 (\oLSQ )}
\DN\SEVEN{C^1 ( \oLSQse )}
\DN\8{C^1 ( \oLSSRe )}

\DN\IQ{\infty , \oLSR }

	\DN\xixyi{\x , \yi }

\DN\bbb{\mathfrak{b}} \DN\bbbN{\bbb ^{\nN }}
\DN\bbbQ{\bbb _{\qQ }}
\DN\bbbQe{\bbb _{\qQ }^{\epsilon }}

\DN\Ry{\rR , \yy }
\DN\RRy{\rR + 1 , \yy }\DN\Rs{\rR , \sss }

\DN\RylN{\hat{r} _{\Ry }^{\lz , \n } } 

\DN\rrr{\mathscr{R} }
\DN\Ryl{\rrr _{\Ry }^{\lz } }
\DN\RRyl{\rrr _{\RRy }^{\lz } }
\DN\Ryln{\rrr _{\Ry }^{\lz , \n } }
\DN\Rylnn{\rrr _{\Ry }^{\lz , \n , \n } }
\DN\Rsl{\rrr _{\Rs }^{ \lz } }
\DN\Resl{\rrr _{\rR + \e , \sss }^{ \lz } }

\DN\vep{v } \DN\logvep{\log \vep }\DN\Xij{X_u^i-X_u^j}\DN\Xji{X_u^j-X_u^i }
\DN\Xkl{X_u^k-X_u^l}

\DN\CR{\eta } 
\DN\CRyinn{\CR _{\Ry }^{\mathbf{i} , \n , \n }}
\DN\CRyin{\CR _{\Ry }^{\mathbf{i} , \n }}
\DN\CRyiN{\CR _{\Ry }^{\mathbf{i} , \n }}
\DN\CRyii{\CR _{\Ry }^{\mathbf{i} }}
	\DN\CRi{\CR _{\Rs }^{\mathbf{i} }}
	\DN\CRei{\CR _{\rR + \e . \sss }^{\mathbf{i} }}

\DN\fFF{F}	\DN\aAA{{A}}
\DN\bBB{{B}}
\DN\kKK{{K}}
\DN\capam{\mathrm{Cap}^{[m]}}
\DN\capa{\mathrm{Cap}}
	 \DN\oOO{\mathcal{O} }

\DN\limN{\limi{\nN }} 
\DN\liminfn{\liminfi{\n }} \DN\limin{\limi{\n }} 
\DN\limiR{\limi{\rR }}

	\DN\llnn{\lim _{\mmm }}
\DN\Rl{\mathscr{Q} ^{\lz }}
\DN\Rln{\mathscr{Q} ^{\lz , \n }}
\DN\Rde{ \sS _{\rR , \e }^2 }
\DN\oLSRe{\oLS _{\rR + \epsilon }}

\DN\oLPm[1]{\oL{P}_{#1}^{[m]}}\DN\oLP[1]{\oL{P}_{#1}}
\DN\twoD{\frac{2-d}{2}}
\DN\Dtwo{\frac{d-2}{2 \cref{;12}}}

\DN\vQR{\varphi _{\qQ , \rR }}
\DN\vq{\varphi_{\rr }}%
	\DN\vh{\widetilde{\varphi }_{h}}
\DN\hp{h_{\pp }}

\DN\sumijm{\sum_{ i , j > m}}

\DN\taue{\tau _{\epsilon } }
\DN\tauRet{t\wedge \tauRe }
\DN\tauRetkl{t\wedge \tauRe }
\DN\taueRT{\tau _{\e }} %		\DN\taueRT{\tau _{\e , \rR , T }}

\DN\tauRe{\taueRT } %	\DN\tauRe{\taueRT ^{ i , j }} 

\DN\tauRz{\tau _{ 0 }} 	%	\DN\tauRz{\tau _{ 0 }^{ i , j } }
\DN\tauet{t\wedge \taue }

\DN\tauRzt{t\wedge \tauRz }
\DN\tauRn[1]{\tau_{ #1 }}

\DN\vse{\varsigma _{\epsilon } }
\DN\vset{t\wedge \vse }

	\DN\idia{i \diamond }\DN\jdia{j \diamond }\DN\iD{i \diamond }
	\DN\kdia{k \diamond }
	\DN\ASN{\As{AC}, \As{SIN}, and \As{NBJ}}
	\DN\IASN{\As{IFC}, \As{AC}, \As{SIN}, and \As{NBJ}}
	\DN\iFc{$\mathbf{IFC}$}

	\DN\nR{\mathbf{n}_{\rR }}
	\DN\anR{ \half \mathbf{n}_{\rR }\aaaa }
	\DN\anN{ \half \mathbf{n}_{\nN }\aaaa }
	\DN\SRover{\oL{\sS }_{\rR }}
	\DN\SQover{\oL{\sS }_{\qQ }}
	\DN\xSR{ \x \in \SRover }
	\DN\xSQ{ \x \in \SQover }
	\DN\xix{\xi _{\x }}

	\DN\CRd{W (\Rd )}
	\DN\CRdSS{C([0,\infty);\RdSS )}
	\DN\CRdN{\CRd ^{\mathbb{N}}}
	\DN\CiSS{C([0,\infty );\sSS )}
	\DN\CRdm{\CRd ^m } 

	\DN\WRm{ \W ^m }
	\DN\WRNz{ \W _{\mathbf{0}} ^{\N } }
	\DN\WRN{ \W ^{\mathbb{N}}} %	%KEEP
	\DN\WWdm{\WRdzm \ts \WRN }
	\DN\WRdzm{ \W _{\mathbf{0}} ^m}
	\DN\WRdm{ \W ^m } %	\DN\WRdm{ \CRd ^m} %

\DN\Lmugonep{L^{p}(\muone )}\DN\Lmugoneone{L^{1}(\muone )}	\DN\Lmugone{L^{2}(\muone )}
\DN\dgin{\dlog ^{\mu }}

\DN\XRiu{\uLxX _{\rR }^{i}(u)}
\DN\XQiu{\uLxX _{\qQ }^{i}(u)}
\DN\XRju{\uLxX _{\rR }^{j}(u)}
\DN\XQju{\uLxX _{\qQ }^{j}(u)}

\DN\XRidu{\uLXXR ^{\diai }(u)}
\DN\XQidu{\uLXXQ ^{\diai }(u)}
\DN\XRjdu{\uLXXR ^{\diaj }(u)}
\DN\XQjdu{\uLXXQ ^{\diaj }(u)}
\DN\pL{\Big(} 

\DN\XXXi{ \mathscr{X}_u^i }
\DN\XXXk{( \mathscr{X}_u^k )}

	\DN\uLXXXi{\uL{\mathscr{X}} _u^i }

	\DN\oLXXXi{\oL{\mathscr{X}} _u^i }

\DN\XXXRiu{\uL{\mathscr{X}}_{\rR }^i(u) }

\DN\XXXQiu{(\XQiu ,\XQidu )}
\DN\SQXoneu{1_{\SQ }(\XRu )}
\DN\Xiu{\uLXRi (u)} 
\DN\pR{\Big)_{i \in \N }} 
 \DN\ER{\E _R}
\DN\ERstarm{\E _{\rR }^{[m]}}
\DN\Estar{\E _{\star }}
\DN\Estarm{\E ^{[m]}}

\DN\RpiD{\rR , i \diamond }
\DN\SQXu{1_{\SQ }(X_u^{\rR , i })}
\DN\SQXuj{1_{\SQ }(X_u^{\rR , j })}

\DN\OFF{(\Omega ,\mathcal{F},\{ \mathcal{F}_t \} )}
\DN\OFP{(\Omega ,\mathcal{F}, P )}

	\DN\OFPF{(\Omega ,\mathscr{F}, P ,\{ \mathscr{F}_t \} )}
	\DN\OFQFs{(\Omega ,\mathcal{F}, \QQs , \{ \mathcal{F}_t \} )}
	\DN\s{\mathbf{s}} 
	\DN\QQla{\qQ _{\la }} \DN\QQxs{\qQ _{\mathbf{x},\sss }}
	\DN\zN{\{ 0 \} \cup \N }
	\DN\QQxsM{\QQxs ^{[m]}}
	\DN\QQlaM{\QQla ^{[m]}}
	\DN\OFQFm{(\Omega ,\mathcal{F},\{ \qQ _{\ulab (\mathbf{x}) + \sss } \}, \{ \mathcal{F}_t \} )}

	\DN\maxT{\max_{0\le t \le T} } \DN\supT{\sup_{0\le t \le T} }

	\DN\QQs{\qQ _{\sss }}

\DN\K{\Ki }
\DN\KaQk{\K _{\qQ }[\mathit{a}_{ q }]}
\DN\KQ{\K _{\qQ }}

\DN\Ki{\mathsf{K}}
\DN\Ka{\Ki [\mathbf{a}]}

\DN\chiwtI{\widetilde{\chi }_{\infty}}

\DN\SmSS{\Rdm \ts \sSS }
\DN\labi{\lab ^i} \DN\labj{\lab ^j}
\DN\dkQ{\mathit{d}_{ \qqq }^{\qQ }}
\DN\dki{\mathit{d}_{ \qqq }^{ \infty }}

\DN\chiwtN{\chiwt ^{\nN }}
\DN\chiwt{\widetilde{\chi }_{\qQ }}

\DN\anest{\mathbf{a}}
\DN\ane{a}
\DN\ak{\ane _{\qqq }}
\DN\akk{\ane _{\qqq +1}}
\DN\akkk{\ane _{\qqq - 1}}
\DN\akR{\ane _{\qqq }(\rR ) }
\DN\akRR{\ane _{\qqq } (\rR + 1 )}

\DN\ar{\ane _{\rr }}
\DN\arr{\ane _{\rr - 1}}
\DN\arR{\ane _{\rr }(\rR ) }
\DN\arRR{\ane _{\rr } (\rR + 1 )}
\DN\arm{\ar ^m }

\DN\chin{\chi _{\mmm }}
\DN\chinn{\chi _{\mmm +1}}

	\DN\pp{\mathsf{p}} \DN\qqq{\mathsf{q}}\DN\rr{\mathsf{r}}

\DN\RRprCs{\sS _{\pp ,\rr }^m (\sss )}
\DN\RRprs{\RRpr (\sss )} 
 \DN\RRpr{\SOm _{\pp ,\rr }}

\DN\Han{\Ha _{\mmm }}
\DN\Hann{\Ha _{\mmm +1}}
\DN\HanC{\Han ^{\circ }}
\DN\HmnPc{\Pi _2 (\Han )}

\DN\Ha{\hH [\anest ]}
\DN\hH{\mathfrak{H}}
\DN\pq{\pp ,\rr }
\DN\pqr{\pp , \qqq , \rr } 

\DN\nn{\mathsf{n}}

\DN\Kakk{\Ki [\mathit{a}_{\qqq }^+]}
\DN\Hb{\Ha _{\pqr }}
	\DN\NNNone{\N }
	\DN\NNNtwo{\N ^2 }
	\DN\NNNthree{\N ^3 }

	\DN\NNN{\mathbf{N}}
	\DN\nnNNN{\mmm \in \NNN }
\DN\mmm{\mathsf{k}}

\DN\UV{\w ^{[m]} }
\DN\vv{\mathsf{v}}
\DN\sigmam{\sigma ^m}

 \DN\SO{\overline{\sS }}
 \DN\SOm{\SO ^m}
 \DN\RRr{\SOm _{\rr }} 

\DN\hhhl{l}		%[[[

 \DN\QQR{\qQ + 2 \le \rR } \DN\QR{\qQ + 1 \le \rR }

\DN\lH{\lab (\mathsf{H}) }
\DN\Ft{\{ \mathscr{F}_t \}} 
\DN\Fs{F_{\mathbf{s}}} 
\DN\PBr{P _{\mathrm{Br}}^{\infty}} \DN\PBrm{P _{\mathrm{Br}}^{m}}

\DN\sumpd{\sum_{p = 1}^d }%
\DN\partialp{\partial _p }

\DN\rNy{r (\nN , \yy )}
	\DN\daaa{ d-2 } \DN\daa{ d-1 } \DN\da{ d }

\DN\vertwo{\rvert ^{2}} \DN\vertd{\rvert ^{d}} \DN\vertda{\rvert ^{ \da }}

\DN\sR{\mathbf{s}}\DN\tR{\mathbf{t}} 

	\DN\Ni{\nN , i } \DN\Nj{\nN , j } 
\DN\Nm{\nN , [m] } 
\DN\Nk{\nN , k }
\DN\Nni{\Nn , i } \DN\Nnj{\Nn , j } 
\DN\Nnm{\Nn , [m] } 

\DN\Rt{\rR , t } \DN\Ru{\rR , u }

\DN\Rdn{(\Rd )^{\nN }} \DN\Rdm{(\Rd )^m} 
\DN\RdN{(\Rd )^{\N }}
\DN\xx{\mathsf{x}}\DN\yy{\mathsf{y}} \DN\zz{\mathsf{z}}
\DN\yyR{\yy _{\rR }}
\DN\IRT{I_{\rR , T }}

	\DN\xji{\x _{\mathbf{j}}^{\mathbf{i}} }

	\DN\xjX{ \x \in \mathscr{S}_{\rR }^{\lz } }

	\DN\ulab{\mathfrak{u}}\DN\lab{\mathfrak{l}}
	\DN\ulabm{\ulab ^{[m]}}

	\DN\upath{\ulab _{\mathrm{path}}}
	\DN\upathm{\upath ^{[m]}}
	\DN\upathz{\upath ^{[0]}}

	\DN\ulabmz{\ulab ^{[m,0]}}	\DN\ulabzone{\ulab ^{[1,0]}}	\DN\ulabmn{\ulab ^{[n,m]}}
	\DN\upathmn{\upath ^{[n,m]}} 	\DN\upathmz{\upath ^{[m,0]}}

\DN\labm{\lab ^{[m]}}
	\DN\labmm{\lab ^{[m-1]}} %[[[
	\DN\labN{\lab _{\nN }}\DN\labNn{\lab _{\Nn }}
	\DN\lpath{\mathfrak{l}_{\mathrm{path}}}

\DN\f{\mathsf{f}}\DN\g{\mathsf{g}} \DN\fg{\f \ot \g } \DN\h{\mathsf{h}}

\DN\fQ{\f _{\qQ }} 
\DN\fQN{\f _{\qQ , \nN }}
	
\DN\fQn{\f _{\qQ }}

\DN\bb{\mathsf{b}}
\DN\aaa{\mathsf{a}} 
\DN\aaaa{\mathfrak{a}} 

\DN\m{\mathfrak{m}}
 \DN\n{\mathfrak{n}} 
\DN\q{\mathfrak{q}} 
\DN\x{{x}}\DN\y{{y}}\DN\z{{z}}

\DN\xone{\x ^1} \DN\xN{\mathbf{\x }^{\nN }} \DN\xI{\x ^i}\DN\xd{\x ^d}

\DN\xxN{\x ^1 ,\ldots, \x ^{\nN }}
\DN\xj{\x ^j}\DN\xk{\x ^k}\DN\xm{\x ^m}

\DN\xs{ \x , \sss }\DN\ys{ \y , \sss }\DN\xy{\x - \y } \DN\xz{(\x - \z )}
\DN\xst{\xs , \mathsf{t}}
\DN\xsR{\xs _{\rR }}
\DN\xpiRs{\x , \piR (\sss ) }
\DN\si{s^i }\DN\sii{s^{i+1} }\DN\sj{s^j} \DN\yi{\y ^i}\DN\yj{\y ^j}

\DN\yiSRc{\yi \not\in \oL{\sS }_{\rR } }

\DN\yiSRec{\yi \notin \oLSRe }

\DN\yiONn{ \yi \in \ONn }
\DN\yiSRcO{\yi \not\in \oL{\sS }_{\rR } ,\, \yiONn }

\DN\siSRc{\si \not\in \oL{\sS }_{\rR } }
\DN\ON{O_{\nN }}
\DN\ONn{O_{\Nn }}
\DN\ONtwo{\ON ^2}

\DN\siSR{\si \in \SR }
\DN\siSSR{\si \in \oL{S}_{\rR }}
\DN\sioLSR{\si \in \oL{S}_{\rR }}

	\DN\xsi{ \x - \si }\DN\xsj{ \x - \sj }
\DN\xyi{ \x - \yi }\DN\xyj{ \x - \yj }
\DN\zs{ 0 , \sss }

\DN\nablax{\nabla _{\x }}
\DN\nablai{\nabla ^i }
\DN\nablaxi{\nabla _{\x ^i }}
\DN\nablasi{\nabla _{\si }}

\DN\nablaxh{\nablax \h}
\DN\nablaPsi{\nabla \Psi }

\DN\xX{X} 

\DN\oLxXti{\oLxX _t^i }
\DN\oLxXui{\oLxX _u^i }
\DN\oLxXuj{\oLxX _u^j }
\DN\oLxXzi{\oLxX _0^i }
\DN\oLxXtj{\oLxX _t^j }

\DN\uLxX{\uL{\xX }}\DN\oLxX{\oL{\xX }}
\DN\uLxXR{\uL{\xX }_{\rR }}
\DN\oLxXR{\oLxX _{\rR }}
\DN\oLxXRNi{\oLxXR ^{\Ni }}
\DN\oLxXRNk{\oLxXR ^{\Nk }}

\DN\uLXi{\uLxX ^i }
\DN\XRi{\xX _{\rR }^i} \DN\XRj{\xX _{\rR }^j}
\DN\uLXRi{\uLxX _{\rR }^i} \DN\uLXRj{\uLxX _{\rR }^j}
\DN\oLXRi{\oLxX _{\rR }^i} \DN\oLXRj{\oLxX _{\rR }^j}

\DN\XRNi{ \xX _{\rR }^{\Ni }}\DN\XRNj{ \xX _{\rR }^{\Nj }}
\DN\uLXRNi{ \uLxXR ^{\Ni }}	\DN\uLXRNj{ \uLxXR ^{\Nj }}
\DN\oLXRNi{ \oLxX _{\rR }^{\Ni }}	\DN\oLXRNj{ \oLxX _{\rR }^{\Nj }}

\DN\X{\mathbf{X}}
	\DN\XN{\X ^{\nN }}
	\DN\XNn{\X ^{\Nn }}
\DN\XR{\X _{\rR }} 
	\DN\uLX{\uL{\X }}\DN\oLX{\oL{\X }}
\DN\oLXRm{\oLXR ^{[m]}}
\DN\XRN{\XR ^{\nN }}

\DN\XRNn{\XR ^{\Nn }}
\DN\XRNni{ X _{\rR }^{\Nni }}
\DN\XRNnj{ X _{\rR }^{\Nnj }}

\DN\uLXN{\uLX ^{\nN }}
\DN\uLXR{\uLX _{\rR }} 
\DN\uLXaR{\uLX _{\aR }} 
\DN\uLxXRi{\uLxXR ^i} \DN\uLxXRj{\uLxXR ^j}
\DN\uLxXQi{\uL{\xX }_{\qQ }^i} \DN\uLxXQj{\uL{\xX }_{\qQ }^j}

\DN\uLXRN{\uLXR ^{\nN }}
\DN\uLxXRNi{ \uLxXR ^{\Ni }}\DN\uLxXRNj{ \uLxXR ^{\Nj }}

\DN\uLXRNn{\uLXR ^{\Nn }}

\DN\uLxXRNni{ \uL{\xX } _{\rR }^{\Nni }}
\DN\uLxXRNnj{ \uLxXR ^{\Nnj }}

\DN\oLXN{\oLX ^{\nN }}
\DN\oLXR{\oLX _{\rR }} 
\DN\oLxXRi{\oLxX _{\rR }^i} \DN\oLxXRj{\oLxX _{\rR }^j}
\DN\oLxXRk{\oLxX _{\rR }^k} %	\DN\oLxXRl{\oLxX _{\rR }^l}

\DN\oLXRN{\oLXR ^{\nN }}
\DN\oLXRNm{\oLXR ^{\nN ,m }}
\DN\oLXRNmm{\oLXR ^{\nN , [m] }}

\DN\oLxXRNj{ \oLxX _{\rR }^{\Nj }}

\DN\oLXRNn{\oLXR ^{\Nn }}
\DN\oLxXRNni{ \oLxX _{\rR }^{\Nni }}
\DN\oLxXRNnj{ \oLxX _{\rR }^{\Nnj }}

\DN\Xm{\X ^{[m]}}
\DN\uLXm{\uLX ^{[m]}}
\DN\oLXm{\oLX ^{[m]}}
\DN\uLXRm{\uLXR ^{[m]}}
\DN\uLXaRm{\uLXaR ^{[m]}}

\DN\XX{\mathsf{X}}
\DN\XXN{\XX ^{\nN } }
\DN\XXR{\XX _{\rR }}
\DN\XXRN{\XXR ^{\nN } }
\DN\XXRtN{\XX _{\Rt }^{\nN } }

\DN\XXRustar{ \pioLRc (\XX _u )}
\DN\XXQustar{\XX_{\qQ , u }^{*}}

\DN\uLXX{\uL{\XX }}
\DN\uLXXN{\uLXX ^{\nN } }
\DN\uLXXR{\uLXX _{\rR }}
\DN\uLXXQ{\uLXX _{\qQ }}
\DN\uLXXRN{\uLXXR ^{\nN } }
\DN\uLXXRtN{\uLXX _{\Rt }^{\nN } }

\DN\oLXX{\oL{\XX }}
\DN\oLXXN{\oLXX ^{\nN } }
\DN\oLXXR{\oLXX _{\rR }}
\DN\oLXXRN{\oLXXR ^{\nN } }
\DN\oLXXRtN{\oLXX _{\Rt }^{\nN } }

\DN\Y{\mathbf{Y}}\DN\YY{\mathsf{Y}}
\DN\EN{\mathbf{E}^{\nN } }
	\DN\B{\mathbf{B}}

	\DN\PhiN{\Phi ^{\nN }}
	\DN\PhiNn{\Phi ^{\Nn }}
	\DN\PsiN{\Psi ^{\nN }}
	\DN\PsiNn{\Psi ^{\Nn }}
\DN\Czi{C_0^{\infty} }
	\DN\CziRd{\Czi (\Rd )}
	\DN\CziRdm{\Czi ((\Rd )^m)}
	\DN\CziSR{\Czi (\SR )}
	\DN\CziSQ{\Czi (\SQ )}
	\DN\CziON{\Czi (\ON )}

\DN\lwrstar{\mathrm{lwr}\star }
\DN\uprstar{\mathrm{upr}\star }
\DN\Rstar{\rR \star }
\DN\Rb{\rR \bullet }
\DN\Rc{\rR \circ }
\DN\dblwr{\dom _{\bullet } ^{\mathrm{lwr}}}
\DN\dclwr{\dc ^{\mathrm{lwr}}}
\DN\dcupr{\dc ^{\mathrm{upr}}}

	\DN\Ps{P_{\mathbf{s}}}

	\DN\xxx{\mathbf{x}}
	\DN\sss{\mathsf{s}}

	\DN\SSsde{\sSS _{\mathrm{sde}}}
	\DN\PPPm{\Pt _{\mathbf{s}}^m }
	\DN\Pt{\widetilde{P}} 
	\DN\lBlhatm{\mathbf{B}^m,\X ^{m*}}

		\DN\XB{(\X ,\mathbf{B})}
		\DN\SSSsde{\mathbf{S}_{\mathrm{sde}}}
		\DN\Bt{\mathscr{B}_t }
		\DN\zti{ 0 \le t < \infty }\DN\zzti{ 0 < t < \infty }
		\DN\WSN{W ^{\mathbb{N}}} 			%[] \DN\WSN{C([0,\infty);\SN )}
		\DN\FtB{$\{ \mathscr{F}_t \}$-Brownian motion }
		\DN\sigmaXms{\sigma _{\XX }^m }
	\DN\SSSsdemtw{\mathbf{S}_{\mathrm{sde}}^m(t,\ww )}
	\DN\SSSsdemt{\mathbf{S}_{\mathrm{sde}}^m(t,\XX )}

	\DN\SSSsdeone{\SSsde ^{[1]}}
	\DN\bbbXms{\mathit{b} _{\XX }^m }
	\DN\Btm{\mathscr{B}_t^m }
	\DN\uPs{under $ \Ps $}

	\DN\Ehatm{\mathscr{C} ^{m}}
	\DN\Ehatmt{\Ehatm _t}

	\DN\OFpsF{(\Omega ,\mathscr{F}, \Ps , \{ \mathscr{F}_t \} )}
	\DN\Fms{F_{\mathbf{s}}^m} 

	\DN\la{\nu }
	\DN\FF{\mathsf{F}}
	\DN\FR{\mathscr{F}_{\rR }} 
	\DN\FRone{\FR ^{[1]} }

	\DN\GRone{C_0^{\infty}(\SR ) \ot \dcb }

	\DN\dc{\dom _{\circ } }%
	\DN\dom{\mathscr{D}}
\DN\uLd{\uL{\dom }} 
\DN\oLd{\oL{\dom }}
\DN\oLdc{\oLd _{\circ }}
\DN\oLdcm{\oLd _{\circ }^{[m]}}
\DN\dbm{\CziRdm \ot \db }
\DN\dcm{\CziRdm \ot \dc }

\DN\dcb{\mathscr{D} _{\circ \mathrm{b}}}
\DN\dcbR{\mathscr{D} _{\Rc \mathrm{b}}}
\DN\dcbm{\mathscr{D} _{\circ \mathrm{b}}^{[m]}}
\DN\dcbone{\dcb ^{[1]}}
\DN\dbbone{\dbb ^{[1]}}
\DN\dbbm{\dbb ^{[m]}}
\DN\dcone{\dc ^{[1]}}

		\DN\db{\mathscr{D}_{\bullet } }\DN\dbb{\mathscr{D}_{\bullet \mathrm{b}} }

\DN\dRb{\dom _{ \Rb }}
\DN\dRbmu{\dRb ^{\mu }}
\DN\dRcmu{\dom _{ \rR \circ }^{\mu }}

\DN\dRbmum{\dRb ^{[m]}}
\DN\dRcmum{\dom _{ \rR \circ }^{[m]}}

\DN\dbmu{\db ^{\mu }}
\DN\dcmu{\dc ^{\mu }}
\DN\dcnu{\dc ^{\nu }}

\DN\dbmum{\db ^{[m]}} \DN\dcmum{\dc ^{[m]}}
\DN\dcnum{\dc ^{[m]}}

\DN\dcmua{\dc ^{\mua }}
\DN\dRbmua{\dom _{\aaa , \Rb }^{\mua }}

\DN\dbmua{\db ^{\mua }}
\DN\dbmuma{\dom _{\aaa \bullet }^{[m]}} 

\DN\dmua{\dlog ^{\mua }}
\DN\dmuyy{\dlog ^{\muyy }}

	\DN\dmubu{\mathfrak{d}_{\yy }}
	\DN\dmubz{\mathfrak{d}_{\zz }}

\DN\oLdmuam{\dom _{\aaa , \mathrm{upr}}^{[m]}}
\DN\oLdcmuma{\oL{\dom  _{\aaa \circ }^{[m]}} }
\DN\dcmuma{\dom _{\aaa \circ }^{[m]}} 
	\DN\dRbmuma{\dom _{\aR \bullet }^{[m]}}
	\DN\dRcmuma{\dom _{ \aR \circ }^{[m]}}

\DN\oLdm{\dom _{\mathrm{upr}}^{[m]}}
	\DN\uLdR{\uLd _{\rR }}
	\DN\oLdR{\oLd _{\rR }}
	\DN\uLdRmum{\uLdR ^{[m]}}
	\DN\uLdRm{\uLdR ^{[m]}}
	\DN\oLdRm{\oLdR ^{[m]}}

	\DN\uLdm{\dom _{\mathrm{lwr}} ^{[m]}}
	\DN\uLdmuam{\dom _{\aaa , \mathrm{lwr}}^{[m]}} 

	\DN\uLdmuamb{\oL{\dom _{\aaa \bullet }^{[m]}} }
	\DN\EuLdmuam{( \Emuam , \uLdmuam )}

\DN\dRmu{\oLd _{\rR }^{\mu }}
	\DN\dRRmum{\uLd _{\rR +1}^{[m]}}

\DN\dOmu{\oLd ^{\mu }}
\DN\dOmum{\dom _{\mathrm{upr}} ^{[m]}}
\DN\dOmumstar{\dom _{\star , \mathrm{upr} } ^{[m]}}
\DN\dOmumloc{\dom _{\mathrm{upr}, \mathrm{loc}} ^{[m]}}

	\DN\dOmuma{\dom _{\aaa , \mathrm{upr}}^{[m]}} 
\DN\SLSm{_{\star , \mathrm{lwr} \star }^{[m]}}
	\DN\SLm{_{\mathrm{lwr}  \star }^{[m]}}
	\DN\USm{\dom _{\mathrm{upr}  \star }^{[m]}}

\DN\dOmua{\oLd ^{\mua }}
\DN\dOR{\oLd _{\rR }}
\DN\dORmu{\oLd _{\rR }^{\mu }}
\DN\dORmua{\oLd _{\rR }^{\mua }}

	\DN\dORmum{\oLd _{\rR }^{[m]}}
	\DN\dORmuma{\oLd _{\aR }^{[m]}} 
	\DN\dORmumSTAR{\oLd _{\Rstar  }^{[m]}} 

	\DN\dORRmum{\oLd _{\rR +1 }^{[m]}}

	 \DN\SSRz{\SSR ^0 }
	\DN\SSz{\sSS _{ 0 }} 
	\DN\SQSS{\SQ \ts \sSS }
	\DN\SQmSS{\SQ ^m \ts \sSS }
	\DN\SQQSS{\SQQ \ts \sSS }

\DN\Ok{\Omega_k}
\DN\Okk{\Omega_{k+1}}

\DN\diaone{1 \diamond }
\DN\diai{i \diamond }
\DN\diaj{j \diamond }

\DN\lambdaeR{\lambda _{\epsilon , R }}
\DN\lambdae{\lambda _{\epsilon , 1 }}
\DN\lambdaR{\lambda _{1, R }}

\DN\LambdaR{\Lambda _{\rR }}
\DN\LambdaRone{\LambdaR ^{[1]}}
\DN\Lambdaone{\Lambda ^{[1]}}

\DN\muxN{\muN _x } 
\DN\muRxN{\mu_{ \rR , \x }^{ \nN }}

\DN\muyy{\mu _{\yy }}
 \DN\muone{\mu ^{[1]}}
 \DN\nuone{\nu ^{[1]}}
 \DN\muaone{\mua ^{[1]}}
 \DN\muyyone{\muyy ^{[1]}}
\DN\mum{\mu ^{[m]}}
\DN\num{\nu ^{[m]}}
	\DN\muRyym{\mu _{\rR , \yy }^{[m]}}
	\DN\muRyymk{\mu _{\rR , \yy , k }^{[m]}}
\DN\mumm{\mu ^{[m+1]}}
\DN\mummm{\mu ^{[m-1]}}
\DN\muRm{\muR ^{[m]}} 

\DN\muRone{\muR ^{[1]}} 
\DN\muRNone{\muR ^{\None }}
\DN\muRNnone{\muR ^{\Nnone }}

\DN\mua{\mu _{\aaa }}\DN\mub{\mu _{\bb }} 
\DN\muam{\mua ^{[m]}}
\DN\muaNnm{\mua ^{\Nn , [m]} }

 \DN\muRyyone{\muRyy ^{[1]}}
\DN\muRcyy{\mu _{\rR , \mathrm{c}, \yy }}
\DN\muRcN{\mu _{\rR , \mathrm{c}}^N}
 \DN\muQRNone{\mu _{ \qQ , \rR }^{ \None }}

\DN\muRyyN{\muRyy ^{\nN }}\DN\muxRyyN{\muxRyy ^{N}}
\DN\muRyyNN{\mu _{\rR , \yyN } ^{\nN }}
\DN\muRyyNn{\mu _{\rR , \yyNn } ^{\Nn }}

\DN\muRyyNnr{\mu _{\rR , \yyNn , \rr }^{ \Nn }}
	\DN\muRyyNr{\mu _{\rR , \yyN , \qqq }^{ \nN }}
\DN\muRyyNnrr{\mu _{\rR , \yyNn , \rr +1 }^{ \Nn }}

\DN\muRyyNnoner{\mu _{\rR , \yyNn , \rr }^{ \Nnone }}
	\DN\muRyyNoner{\mu _{\rR , \yyN , \qqq }^{ \None }}
\DN\muxRyNN{\mu _{\rR , \yyN , \x } ^{N}}
\DN\muN{\mu^{\nN }}
\DN\muNm{\mu^{\nN , [m]}}

\DN\muNL{\widetilde{\mu }^{\nN }}
\DN\muNone{\mu ^{\None }}

\DN\mR{\m _{\rR }}
\DN\mRone{\mR ^{[1]}} 
\DN\mRN{\mR ^{\nN }} 
\DN\mRNone{\mR ^{\None }} 
\DN\mRNnone{\mR ^{\Nn ,[1]}} 

\DN\mRk{\m _{\Rk }}
\DN\mRkone{\m _{\Rk }^{[1]}} 
		\DN\mRkNone{\m _{\rR , k }^{\None }}
	\DN\mRklone{\m _{\Rkl }^{[1]}}
	\DN\mRkllone{\m _{\Rkll }^{[1]}}
	\DN\mRklNone{\m _{\Rkl }^{\None }}
	\DN\mRyyklNone{\mRyykl ^{\None }}
	\DN\mRklNnone{\m _{\Rkl }^{\Nnone }}
	\DN\mRyyklNnone{\m _{\RyyklNn }^{\Nnone }}
	\DN\mRyyNnone{\m _{\RyyNn }^{\Nnone }}
	\DN\mRyyklone{\m _{\Ryykl }^{[1]}}
	\DN\mRyykllone{\m _{\Ryykll }^{[1]}}
\DN\mRyy{\m _{\Ryy }}
\DN\mRyykl{\m _{\Ryykl }}
\DN\mRss{\m _{\rR ,\sss }}

\DN\None{\nN , [1] } \DN\Nnone{\Nn , [1] }

\DN\mRyyN{\mRyy ^{\nN }} 
\DN\mRyyNone{\mRyy ^{\None }} 
\DN\mRyyone{\mRyy ^{[1]}} 

\DN\mRyykNone{\m _{\Ryy , k }^{\None }}
\DN\mRyyRNkNnone{\m _{\rR ,\yyRcNn , k }^{\Nnone }}
\DN\mRyyRNklNnone{\m _{\rR ,\yyRcNn , k , l }^{\Nnone }}

\DN\mRyykone{\m _{\rR , \yy , k}^{[1]}}

\DN\mRe{\m _{\rR . \e }}

\DN\maRy{\m _{\aRy }}

\DN\muaRy{\mu _{\aRy }}
 \DN\muaRym{\mu _{\aRy }^{[m]}}
\DN\muaR{\mu _{\aR }}

\DN\muNn{\mu ^{\Nn }}
\DN\muNnm{\mu ^{\Nn ,[m] }}
\DN\muNnL{\widetilde{\mu }^{\Nn }}
\DN\muNnone{\mu ^{\Nnone }}
			\DN\wm{\m }
\DN\mN{\wm ^{\nN }}
\DN\mNRNone{\wm _{\rR }^{\None }}
\DN\mNRNoneG{\wm _{\GR }^{\None }}

\DN\muRyyNone{\muRyy ^{\None }}
\DN\muRyyNnone{\mu _{\rR , \yyNn }^{\Nnone }}

\DN\muRyyNNone{\mu _{\rR , \yyN } ^{\None }}
\DN\muRyyRNone{\mu _{\rR , \yyN } ^{\None }}
\DN\muRyyRN{\mu _{\rR , \yyN } ^{\nN }}

\DN\RyyN{\rR , \yyN }
\DN\yyRcN{\yyN }
\DN\yyRcNn{\yyNn }

\DN\yyN{\yy^{\nN }} \DN\yyNn{\yy^{\Nn }}

\DN\muRN{\muR ^{\nN }}
\DN\muRNm{\muR ^{\nN ,[m]}} 
\DN\muRNT{\widetilde{\muR }^{\nN }}
\DN\muR{\mu _{\rR }}
\DN\muRe{\mu _{\rRe }}
\DN\muRek{\mu _{\rRe }^k}
\DN\muRekl{\mu \Rekl }

\DN\muRex{\mu _{\rRe , x }}
	\DN\muRexkl{\mu _{\rRe , x }^{k , l }}
\DN\muRyy{\mu _{\rR , \yy }}
\DN\muRyyT{\widetilde{\mu }_{\rR , \yy }}
\DN\muxRyy{\mu _{\rR , \yy , \x }}
\DN\muRss{\mu _{\rR , \sss }}

\DN\muxRyyT{\widetilde{\mu }_{\rR , \yy , \x }}

\DN\rhoNone{\rho ^{N,1}}\DN\rhoNm{\rho ^{N,m}}
\DN\rhoone{\rho ^1 }\DN\rhom{\rho ^m }

\DN\SSR{\sSS _{\rR }}
\DN\SSRm{\SSR ^m }
\DN\SSRn{\SSR ^n }
\DN\oLSSR{\oLSS _{\rR }}
\DN\CnfRk{\mathrm{Cnf}_{\Rk }^{[m]}}
\DN\CnfRy{\mathrm{Cnf}_{\Ry }^{[m]}}
\DN\CnfR{\mathrm{Cnf}_{\rR }^{[m]}}
\DN\SSRc{\SSR ^c}

\DN\SSsi{\sSS _{\mathrm{s,i}}}
\DN\SSs{\sSS _{\mathrm{s}}}
\DN\WSsi{W (\SSsi )}
\DN\WSsiNE{W_{\mathrm{NE}} (\SSsi )}
\DN\ww{\mathsf{w}}

\DN\w{\mathbf{w}} \DN\wwm{\w ^{[m]}}\DN\wwn{\w ^{[n]}}

\DN\sSS{\mathsf{S}}

\DN\SN{(\Rd )^{\mathbb{N}}}

\DN\SQ{\sS _{\qQ }}
\DN\SQQ{\sS _{\qQ +1}}

\DN\SRm{\SR ^m }
\DN\SRmm{\SR ^{m-1} }
\DN\SRk{\SR ^k}

\DN\SSone{\sSS ^{[1]}}
\DN\SSoneyy{\SSone [\yy ]}
\DN\SSneoneyy{\sSS _{\ne }^{[1]} [\yy ]}
\DN\SSneoneaa{\sSS _{\ne }^{[1]} [\aaa ]}

\DN\oLSQse{\oL{\sS }_{\qQ , \e }(\sss ) }
\DN\oLSRse{\oL{\sS }_{\rR , \sss , \e }}
 	\DN\oLSRk{\oLSR ^k}
\DN\oLSSm{\oL{\sSS }_{ m }} \DN\oLSSn{\oL{\sSS }_{ n }}
\DN\oLSSRm{\oLSSR ^m } \DN\oLSSRn{\oLSSR ^{ n }}
\DN\oLSSRk{\oLSSR ^k } 
\DN\oLSSRe{\oLSS _{\rRe }^{[1],k} }

\DN\SRkk{\SR ^{k+1}}
\DN\Rn{ \rR }

\DN\Rh{\rR _{\mathfrak{h} }}
\DN\SR{\sS _{\rR }} 
	\DN\oLSR{\oL{\sS }_{\rR }} 
	\DN\oLSQ{\oL{\sS }_{\qQ }} 
\DN\oLSRc{\oL{\sS }_{\rR }^c } 

\DN\SRc{\sS _{\rR }^{c}} 

	\DN\SSSRxk{\mathbb{S}_{\rR }^k }
\DN\SRR{\sS _{\rR +1}}

\DN\RdSS{\Rd \ts \sSS }
\DN\RdmSS{\Rdm \ts \sSS }
\DN\RdnSS{(\Rd )^n \ts \sSS }
\DN\W{W}

\DN\varphiN{\varphi ^{ \nN }}

\DN\dlog{\mathsf{d}}
\DN\dlogyy{\dlog ^{\muyy }}

\DN\dlogRkl{\dlog _{\Rkl }}
\DN\WdlogRkl{\widetilde{\dlog } _{\Rkl }}
\DN\dlogRkll{\dlog _{\Rkll }}

\DN\dii{\dlog _{\infty , \yy }} % 	\DN\dii{\dlog _{\infty}^{\infty}}
\DN\diQi{\dlog _{\infty ,\yy }^{\qQ }}
\DN\diQQi{\dlog _{\infty ,\yy }^{\qQ +1 }}

\DN\dmu{\dlog ^{\mu }}
\DN\dmuone{\dlog ^{\mu ^1}}

\DN\dnu{\dlog ^{\nu }}
\DN\dnum{\dlog ^{\num }}
\DN\dnuN{\dlog ^{N}}
\DN\dnuone{\dlog ^{\nu ^1}}
\DN\dnun{\dnu _{\n }} 

\DN\dnub{\dlog ^{\nu }}\DN\dnubN{\dlog ^{\nN}}\DN\dnubone{\dlog ^{\nu ^1}}
		\DN\dnuc{\dlog ^{\nu }}
\DN\dnucN{\dlog ^{\nN}}

\DN\Ryy{\rR , \yy }
\DN\RyyNn{\rR , \yyNn }
\DN\RRyy{\tT , \yy }

	 \DN\dN{\dlog ^{\nN }}
	 \DN\dNn{\dlog ^{\Nn }}
	\DN\dpN{\dlog _p ^{\nN }} 
	\DN\dRyyRN {\dlog _{\rR , \RyyN }}
	\DN\dRyyRNN {\dlog _{\RyyN }^{\nN }}
	\DN\dRyyRNNn {\dlog _{\RyyNn }^{\Nn }}

\DN\dRyyNn{\dlog _{\rR , \yyRcNn }^{\Nn }}
\DN\dRRyyNn{\dlog _{\tT , \yyRcNn }^{\Nn }}

\DN\dRyyN{\dlog _{\Ryy }^{\nN }}

\DN\dbRR{\dlog _{\rR + 1 , \yy }} 
\DN\dbR{\dlog _{\Ryy }} 

\DN\dRN{\dR ^{\nN }} 
\DN\dRNn{\dR ^{\Nn }} 

\DN\dRklN{\dlog _{ \Rkl }^{\nN }} 
\DN\dRkllN{\dlog _{ \Rkll }^{\nN }} 
\DN\dRklNn{\dlog _{ \Rkl }^{\Nn }} 

\DN\dRyyklN{\dlog _{ \Ryykl }^{\nN }} 
\DN\dRyykllN{\dlog _{ \Ryykll }^{\nN }} 
\DN\dRyyklNn{\dlog _{ \RyyklNn }^{\Nn }} 

\DN\dRkN{\dlog _{ \rR , k }^{\nN }} 
\DN\dRkNn{\dlog _{ \rR , k }^{\Nn }} 

\DN\dRyy{\dlog _{\Ryy }}
\DN\dRRyy{\dlog _{\RRyy }}
\DN\dR{\dlog _{\rR }}\DN\dRR{\dlog _{\rR + 1}}

\DN\muz{\mu _{0}}
\DN\mux{\mu _{x}}
\DN\dxmu{dx \ts \mu }
\DN\muk{\mu ^{{k}}}
\DN\mukg{\mu ^{{k}}}

\DN\nuz{\nu _{0}}
\DN\nux{\nu _{x}}
\DN\dxnu{dx \ts \nu }
\DN\nuk{\nu ^{{k}}}
\DN\nukg{\nu ^{{k}}}

\DN\Lmz{L^{2}(\muz )}
\DN\Lmg{L^{2}(\mu )}
\DN\Lmgz{L^{2}(\\muz )}
\DN\Lma{L^2 (\mua )}

\DN\Lm{L^{2}(\mu )}
\DN\Lnu{L^{2}(\nu )}
\DN\Lmm{L^2(\mum )}
\DN\Lnum{L^2(\num )}
\DN\Lmuam{L^2(\muam )}

\DN\Lmone{L^2(\muone )}
\DN\muRsone{\muRss ^{[1]}}

\DN\Lmgone{L^{2}(\muone )}
\DN\Llocmgone{L^{2}_{\mathrm{loc}}(\muone )}

\DN\Lmuk{L^{2}(\muk )}
\DN\Lmug{L^{2}(\mu )}
\DN\Llocmone{\Lloctwo (\muone )}
\DN\Llocp{L_{\mathrm{loc}}^{p}}
\DN\Llocq{L_{\mathrm{loc}}^{q}}
\DN\Lloctwo{L_{\mathrm{loc}}^{2}}
\DN\Llocone{L_{\mathrm{loc}}^{1}}
\DN\Lone{L^{1}}
\DN\RdT{\Rd \ts \sSS }

\DN\E{\mathscr{E}}
\DN\Emu{\E ^{\mu }}
\DN\Enu{\E ^{\nu }}
\DN\Enum{\E ^{[m]}} 

\DN\ERmu{\Emu _{\rR }}
\DN\ERRmu{\Emu _{\rR +1}}

\DN\Emum{\E ^{[m]}} \DN\Emumm{\E ^{[m+1]}}

\DN\ERmum{\Emum _{\rR }}
\DN\ERRmum{\Emum _{\rR +1}}
\DN\ERmumm{\Emumm _{\rR }}

\DN\ERRmumm{\Emumm _{\rR +1}}

\DN\Emuam{\E _{\aaa }^{[m]}}
\DN\ERmuam{\Emuam _{\rR }}
	\DN\EuLdRmuam{( \EaRm , \uLd _{\aR }^{[m]} ) } 
	\DN\EuLdRmuamSTAR{( \E _{\star , \rR }^{[m]} , \uLd _{\star , \rR }^{[m]} ) } 
	\DN\EoLdRmuam{( \EaRm , \oLd _{\aR }^{[m]} ) } 
	\DN\EoLdRmuamSTAR{( \E _{\rR }^{[m]} , \oLd _{\star , \rR }^{[m]} ) } 

	\DN\EuLdRmum{( \ERmum , \uLdR ^{[m]})}
	\DN\EoLdRmum{( \ERmum , \oLdR ^{[m]})}
	\DN\ERyym{\E _{\Ryy }^{[m]}}
	\DN\ERyymk{\E _{\Ryy , k }^{[m]}}

		\DN\EaRyym{\E _{\aRy }^{[m] }}
		\DN\EaRm{\E _{\aR }^{[m] }}

\DN\dRyybm{\dom _{\Ryy \bullet }^{[m]}}
\DN\daRyybm{\dom _{\aRy \bullet }^{[m]}}
\DN\daRbm{\dom _{\aR \bullet }^{[m]}}

\DN\uLdaRyym{\uLd _{\aRy }^{[m]}}
\DN\uLdaRm{\uLd _{\aR }^{[m]}}

	\DN\uLdRyym{\uLd _{\Ryy }^{[m]}}
	\DN\oLdRyym{\oLd _{\Ryy }^{[m]}}

\DN\Emuz{\E ^{\muz }}
\DN\Emug{\E ^{\mu }}

\DN\DDDa{\DDD _{\aaaa }}
\DN\DDDaR{\DDD _{\aaaa , \rR }}
\DN\kappaq{\kappa _q}

\DN\DDD{\mathbb{D}}
\DN\DDDD{\mathbf{D}}
\DN\DDDone{\DDD ^{[1]}}
\DN\DDDR{\DDD _{\rR }}
\DN\DDDRone{\DDDR ^{[1]}}

\DN\DDDm{\DDD ^{[m]}}
\DN\DDDam{\DDDa ^{[m]}}
\DN\DDDRm{\DDD _{\rR } ^{[m]}}
\DN\DDDaRm{\DDDaR ^{[m]}}
\DN\DDDaRk{\DDD _{\aaaa , \rR , k }}
\DN\DDDaRmk{ \DDDaRk ^{[m]}}

	\DN\RtwoN{(\Rtwo )^{\mathbb{N}}}

	\DN\mm{m}
	\DN\EDm{( \Emum , \oL{\dcmum } ) }
	\DN\EDnum{( \Enum , \oL{\dcmum } ) }
	\DN\ED{( \Emu , \oLd _{\circ }^{\mu } ) }
	\DN\EDnu{( \Enu , \oLd _{\circ }^{\nu } ) }
\title{Infinite-dimensional stochastic differential equations  for Coulomb random point fields}

\author{Hirofumi Osada \and Shota Osada}

\begin{document}

\begin{abstract}%G[----

We study the infinite-dimensional stochastic differential equations (ISDEs) 
of infinite-particle systems associated with Coulomb random point fields. 
The stochastic dynamics described by these ISDEs are referred to as 
Coulomb interacting Brownian motions. 
In all spatial dimensions $ d \ge 2 $ and for all inverse temperatures 
$ \beta > 0 $, we construct the Coulomb interacting Brownian motions.

We prove that the ISDEs admit strong solutions and that pathwise uniqueness holds. 
The resulting labeled dynamics form an $ \RdN $-valued diffusion, possibly without 
an invariant measure, while the corresponding unlabeled process is a reversible 
diffusion with respect to the underlying Coulomb random point field.

Moreover, we identify the infinite-particle stochastic dynamics as 
the limit in path space of finite-particle systems driven by stochastic differential equations.
This identification is achieved through two approximation schemes: finite-domain
systems with reflecting boundary conditions and $ N $-particle systems.
Although the $ N $-particle approximation is more fundamental,
its justification relies crucially on the finite-domain approximation
together with the uniqueness of solutions to the ISDEs. 

Previously, only the case $ d = 2 $ and $ \beta = 2 $, known as the Ginibre interacting Brownian motion, 
was understood through random matrix theory and determinantal random point fields. 
Extending this result beyond the determinantal setting has remained a major difficulty. 

We introduce a new, conceptually clear method based on stochastic analysis of infinite-particle systems 
with long-range interactions that yields a rigorous construction of Coulomb interacting Brownian motions. 
A key ingredient is an explicit computation of the logarithmic derivatives 
of Coulomb random point fields.

%G]
%\keywords{infinite-dimensional stochastic differential equations; interacting Brownian motions; Coulomb random point fields; stochastic dynamics of infinite-particle systems; Ginibre random point fields; logarithmic derivative; Dirichlet forms
%\subclass{60K50, 82C22, 60B20,60J60}
 \end{abstract}

\maketitle

\textbf{keywords}: {Coulomb random point fields; stochastic dynamics of infinite-particle systems;
logarithmic derivatives of random point fields; infinite-dimensional stochastic differential equations
}

\bs\noindent 
\begin{flushright}
$ {}^*$ Chubu University, \texttt{osada "AT"fsc.chubu.ac.jp}\\
$ {}^\dagger$ Kagoshima University, \texttt{s-osada"AT"edu.kagoshima-u.ac.jp}
\end{flushright}

\maketitle

	\tableofcontents 
\bigskip

\bs\noindent 
\begin{flushright}
H.\,Osada,  Chubu University, \texttt{osada@fsc.chubu.ac.jp}\\
S.\,Osada, Kagoshima University, \texttt{s-osada@edu.kagoshima-u.ac.jp}
\end{flushright}

%-- Osa

%\begin{document}

	\tableofcontents 

\section{Introduction}\label{s:1}

The aim of this paper is to develop a theory of infinite-dimensional stochastic
differential equations (ISDEs) arising from long-range interacting particle systems.
We focus on the construction and characterization of strong solutions,
their pathwise uniqueness, and the relationship between labeled and unlabeled
dynamics within the framework of stochastic analysis on infinite-particle systems.
Moreover, we derive the limiting stochastic dynamics as the limit of
$ N $-particle systems described by solutions of finite-dimensional SDEs.

The Coulomb potential occupies a central position in mathematics and the natural sciences.
For $ d \geq 2 $, define $ \map{\Psi }{\Rd }{\R \cup \{ \infty \} }$ by 
\begin{align} \label{:10a} &
\Psi ( \x ) =
\begin{cases}
\dfrac{1}{ \daaa }\dfrac{1}{\vert \x \vert ^{ \daaa }} & d \ge 3 
,\\[6pt] 
 - \log \lvert \x \rvert &d = 2 
.\end{cases}
\end{align}
Then the gradient of $ \Psi $ has the simple form
\begin{align} & \notag 
 \nablaPsi ( \x ) = - \frac{ \x }{\vert \x \vert ^{ \da }}
.\end{align}
We call $ \Psi $ the $ d $-dimensional Coulomb potential.
The purpose of this paper is to construct the stochastic dynamics of
infinite-particle systems in $ \mathbb{R}^d $ interacting through $ \Psi $
at inverse temperature $ \beta > 0 $.

These dynamics are expected to satisfy an ISDE, whose prototypical form is
\begin{align}\label{:10c}&
X_t^{ i } - X_0^i = B_t^i + \frac{\beta }{2} \int_0^t 
\limiR \sum_{
\lvert X_u^{ i } - X_u^{ j } \rvert < \rR ,\ j \ne i 
} \frac{X_u^{ i } - X_u^{ j } }{\lvert X_u^{ i } - X_u^{ j } \vert ^d } 
du 
,\quad i\in\N 
,\end{align}
which we call the \textbf{Coulomb interacting Brownian motion}. 

%G[--- ---

Equilibrium states of the unlabeled dynamics of Coulomb interacting Brownian motion
are given by Coulomb random point fields (RPFs).

Coulomb RPFs are probability measures describing
unlabeled particle systems with interaction potential $ \Psi $.
They are obtained as limits of finite-particle systems whose labeled
distributions $ \widetilde{\mu }^{\nN } $ are given by the density
\begin{align}\label{:10e} 
\mN ( \mathbf{x} )
=
\frac{1}{\mathscr{Z} }
\exp \Big\{
- \beta \Big(
\sum_{i=1}^{\nN } \PhiN ( \x ^i )
+ \sumijN \PsiN ( \x ^i , \x ^j )
\Big)
\Big\}
,
\end{align}
where $ 0 < \beta < \infty $, $ \PhiN $ denotes a confining potential,
and $ \PsiN $ an interaction potential at the $ N $-particle level,
which may coincide with $ \Psi $.

%G]

%G[
We assume that these potentials converge to limiting potentials: 
\begin{align}\label{:10f}
\limi{\nN } \nablaPhiN (\x ) = \nablaPhi (\x ) , \qquad
\limi{\nN } \PsiN (\x , \y ) = \Psi (\xy )
.\end{align}
We refer to \As{A1} for the precise meaning of \eqref{:10f}. 
%G]
%G[ -- --

The $ \nN $-particle dynamics
$ \XN = ( X^{\nN , i} )_{i=1}^{\nN } $ 
are gradient dynamics associated with the Dirichlet form
$ \widetilde{\E}^N $ on $ L^2 (\widetilde{\mu }^{\nN } ) $,
which is defined by %\eqref{:10g}. 
\begin{align} &\notag %\label{:10g}&
\widetilde{\E}^N (f,g)
 =
\int_{(\Rd )^{\nN } }
\frac{1}{2} \sum_{i=1}^{\nN }
\big( \nablaxi f , \nablaxi g \big)_{\Rd }
\, \widetilde{\mu }^{\nN } ( d\mathbf{x} )
.\end{align}
Here we denote by $ ( \cdot , \cdot )_{\Rd } $ the Euclidean inner product on $ \Rd $. 

By integration by parts, we have
\begin{align}\notag
\widetilde{\E}^N (f,g)
 & =
-
\int_{(\Rd )^{\nN } }
 \frac{1}{2} \sum_{i=1}^{\nN }
\Big\{
 \Delta _{\x ^i } f
+ 
\big(
( \nablaxi \log \mN ) , \nablaxi f 
\big)_{\Rd }
\Big\} 
g \, \widetilde{\mu }^{\nN } ( d\mathbf{x} )
.\end{align}
See \ssref{s:ext} for the definition of Dirichlet forms used here.
%G]
%G[
By \eqref{:10e}, 
\begin{align*}&
 \nablaxi \log \mN (\bm{x} ) = 
 - \beta \Big( \nabla \PhiN (\x ^i ) + 
\sum_{ \substack{ j \ne i \\ 1 \le j \le \nN } }
\nablax \PsiN (\x ^i ,\xj ) \Big) 
.\end{align*} 
Consequently, $ \XN $ satisfies the following SDE in differential form:
\begin{align} \label{:10h}
dX_t^{\nN , i}
& = dB_t^{ i }
+ \frac{\beta }{2} 
\Big( - \nabla \PhiN ( X_t^{\nN , i} ) 
- 
\sum_{ \substack{ j \ne i \\ 1 \le j \le \nN } }
\nablax \PsiN ( X_t^{\nN , i} , X_t^{\nN , j} ) \Big) 
\, dt 
.\end{align}
Taking the limit $ \nN \to \infty $, we formally obtain the limiting ISDE:
\begin{align} \label{:ISDE}%
dX_t^{ i }
& = dB_t^{ i }
+ \frac{\beta }{2} 
\Big( - 
\nablaPhi ( X_t^{ i } ) 
+ 
\limi{\rR }
\sum_{\substack{ j \ne i \\ \lvert X_t^{ j } \rvert < \rR }}
\frac{ X_t^{ i } - X_t^{ j } }
{ \lvert X_t^{ i } - X_t^{ j } \rvert ^d }
\Big) 
\, dt ,
\quad  i \in \N 
.\end{align}

If, in addition, $ \mu $ is translation invariant, this equation reduces to the ISDE \eqref{:10c}.
Such multiple representations of the stochastic dynamics arise from the
long-range nature of the Coulomb interaction.

%G]

A striking feature of this problem is the exact correspondence between the ambient spatial dimension $ d $ and the exponent in the Coulomb potential. This coincidence is not accidental: it is the source of both the depth and the difficulty of the problem. On the one hand, it ensures that Coulomb random point fields (RPFs) are natural equilibrium distributions in any dimension $ d \geq 2 $; on the other hand, it makes the analysis of the associated ISDE \eqref{:ISDE} extremely delicate, due to the long-range nature and singularity of the Coulomb force.

To date, the only case where \eqref{:ISDE} has been solved is $ d=\beta=2 $, corresponding to the Ginibre RPF. In this special situation, the determinantal structure of the equilibrium measure provides powerful tools, including sharp hyperuniform estimates for particle number fluctuations. Beyond $ d=\beta=2 $, however, such determinantal methods break down, and no general construction of the dynamics has been available.

From the equilibrium viewpoint, the classical Gibbsian framework excludes Coulomb interactions because the potential is neither superstable nor integrable at infinity. The construction of Coulomb RPFs remained open for decades, until very recently it was resolved in general by Armstrong--Serfaty and Thoma \cite{a-s,thoma.24}. These works establish the existence of Coulomb RPFs for all $ d \geq 2 $ and $ \beta > 0 $. The present paper addresses the complementary problem: the construction of the corresponding stochastic dynamics, i.e., reversible diffusions with respect to these measures.

%G[ ------

Our approach is new, conceptually simple, and robust.

For the first time, it constructs a stochastic dynamics
for infinite-particle Coulomb systems in all spatial dimensions $ d \ge 2 $
and for all inverse temperatures $ \beta > 0 $ (\tref{l:12}).

Moreover, this dynamics is realized as a
\textbf{pathwise unique strong solution}
of the Coulomb ISDE \eqref{:ISDE}
with explicit coefficients
(\tref{l:13}).
This follows from the explicit representation of the logarithmic derivative of $ \mu $ 
(\tref{l:11}).

%G]

%G]

The central innovation of the method is the introduction of a new family of
analytic estimates for the logarithmic derivatives of Coulomb RPFs.
These estimates replace the determinantal and random matrix arguments that were
previously available only in the Ginibre case.

The solution $ \X $ is constructed as an $ \RdN $-valued diffusion process 
without relying on an invariant or reference measure,
while the associated unlabeled process is shown to be a reversible diffusion
with respect to the Coulomb RPF (\tref{l:1Y}).

%G]
%G[----2 

We derive the ISDE under two different thermodynamic limits.

One approach is based on finite-domain approximations, in which
finite-particle systems are connected through a two-step thermodynamic
limit (\tref{l:14}).

For each radius $ \rR $, we freeze the exterior configuration on
$ \oLSRc = \{ \lvert s \rvert > \rR \} $ and consider the $ \nN $-particle
gradient dynamics inside $ \SR = \{ \lvert s \rvert < \rR \} $ with normal reflection on $ \partial \SR $.
Denoting these finite-volume systems by $ \XRN $, we first take the limit
$ \nN \to \infty $ with $ \rR $ fixed to obtain an infinite-particle
dynamics $ \XR $ in $ \SR $, which is reversible with respect to the
corresponding conditional measure. We then let $ \rR \to \infty $ and prove
that $ \XR $ converges in law to $ \X $, where
$ \X $ is the unique strong solution to the Coulomb ISDE \eqref{:ISDE}.

%GTP[
The other approach is based on $ N $-particle approximations. 
In this approach, solutions $ \XN $ of the finite-dimensional SDE 
\eqref{:10h} converge in the path space to the solution $ \X $ of the ISDE (\tref{l:16}).

More precisely, we regard $ \XN $ as an $ \RdN $-valued process
by introducing frozen dummy particles for each $ i > \nN $,
which do not affect the dynamics.

In \tref{l:16}, we establish that
\begin{align*}
\X = \limi{\nN } \XN 
\quad \text{in law in } C ( [ 0 , \infty ) ; \Rd )^{ \N }
.\end{align*}

%G]

Although the second derivation is the most fundamental, it relies
on the first derivation,
on the convergence of a related family of Dirichlet forms,
and on a sandwich-type convergence theorem built upon these results.
The uniqueness of solutions to the ISDE is indispensable
for implementing the sandwich-type convergence.

These results justify the ISDE and provide a simulation scheme for the limiting stochastic dynamics.

%G]
%G[---
\smallskip

In \cite{o.Gin}, the vanishing of self-diffusion was proved for the Ginibre interacting Brownian motion.
The decisive mechanism behind this phenomenon is the number rigidity of the stationary RPF.
In contrast, in \cite{o.p}, it was proved that the self-diffusion coefficient is strictly positive for any potential in Ruelle's class with convex hardcore in dimensions $ d \ge 2 $.

Taken together, these results highlight a structural distinction between Coulomb interactions and Ruelle-type interactions.
With the present results in place, it is natural to formulate the following conjecture:

%G]

\smallskip 
\noindent \textbf{Conjecture.} 
For translation-invariant Coulomb interacting Brownian motions in two dimensions, the self-diffusion coefficient vanishes at every inverse temperature $\beta > 0$. 
In contrast, for every $d \ge 3$, we conjecture that Coulomb interacting Brownian motions are always diffusive, i.e., their self-diffusion coefficient is strictly positive for all $\beta > 0$. 

\smallskip 
%G[---

In \cite{thoma.23}, Thoma proved number rigidity for Coulomb RPFs in the case
$ d = 2 $, and established the absence of number rigidity for Coulomb RPFs
in all dimensions $ d \ge 3 $.
This dichotomy strongly suggests that the self-diffusion coefficient
is strictly positive in dimensions $ d \ge 3 $.

If confirmed, this phenomenon would vividly highlight the profound role
played by the long-range nature of the two-dimensional Coulomb potential,
and would firmly place Coulomb systems within the universality paradigm
of singular interacting diffusions.

%G]

\medskip
\noindent 
\textbf{Organization of the paper.} 
%\noindent 
The remainder of the paper is organized as follows.
In \sref{s:(}, we state our main results and discuss related work and positioning.
\sref{s:)} presents examples.
Sections \ref{s:2} and \ref{s:3} 
establish the existence of the logarithmic derivative of the equilibrium measure and provide its explicit representation. 
\sref{s:4} proves non-collision estimates for solutions of ISDEs.
Sections \ref{s:5}--\ref{s:7} introduce Dirichlet form techniques and construct weak solutions under both lower and upper schemes.
\sref{s:8} establishes the existence of unique strong solutions, while \sref{s:9} demonstrates the uniqueness of Dirichlet forms and completes the proof of our diffusion theorem. 
\sref{s:X} shows the convergence of finite-particle dynamics to the infinite-particle ISDE dynamics.
Finally, the Appendix (Section~\ref{s:A}) collects auxiliary results on configuration spaces, Dirichlet forms, and the general theory of ISDEs.

%G]--------------------------------------------------- 

\section{Setup and main theorems }\label{s:(}

We now formulate the problem in detail and state our main results (Theorems \ref{l:11}--\ref{l:16}). 

Throughout the paper, we occasionally attach labels directly to constants
to facilitate later references.
These labels are independent of equation numbering and are used consistently
whenever such constants reappear.

Moreover, when an unlabeled configuration
$ \sss $ is written together with points $ s^i $,
we implicitly use a labeled representation $ \sss = \sum_i \delta_{s^i} $.

%G[---
\subsection{Set up and logarithmic derivatives}
Let $ \SR = \{ s \in \Rd \, ;\, \vert s \vert < \rR \} $. 
Define the configuration space on $ \Rd $ by 
\begin{align} &\notag 
\sSS = \Bigl\{ \sss \, ;\, \sss = \sum_{i} \delta_{\si } , \sss (\SR ) < \infty \quad \text{ for all } \rR \in \mathbb{N} \Bigr\} 
.\end{align}
We equip $ \sSS $ with the vague topology, under which it becomes a Polish space, i.e., homeomorphic to a complete and separable metric space.

%GTP[
A function $ f:\sSS \to \mathbb{R} $ is called \emph{smooth} if $ f_{\Rs }^m $ is smooth on $ \SRm $ for all $ \rR,m \in \mathbb{N} $ and $ \sss $, 
where $ f_{\Rs }^m $ is the $ \SRm $-representation of $ f $ defined in \dref{d:fun}. 
It is called \emph{local} if $ f $ is $ \sigma[\piR] $-measurable for some $ \rR \in \mathbb{N} $, where 
$ \piA (\sss ) = \sss (\cdot \cap A ) $ for $ A \subset \Rd $. 
%GTP]
% 
We write
\begin{align}\notag
\db &= \{ f: \sSS \to \mathbb{R} \,;\, f \ \text{is $\mathscr{B}(\sSS )$-measurable and smooth}\}
,\\ \notag 
\dc &= \{ f \in \db \,;\, f \ \text{is local}\}
,\\ \notag 
\dbb &= \{ f \in \db \,;\, f \ \text{is bounded}\}, \quad 
\dcb = \{ f \in \dc \,;\, f \ \text{is bounded}\}.
\end{align}

A probability measure $\nu$ on $(\sSS,\mathscr{B}(\sSS ))$ is called an \emph{RPF} on $\Rd$. 
For an RPF $ \nu $ with the one-point correlation function $ \rho _{\nu }^1 $, we define 
the one-reduced Campbell measure $ \nuone $ of $ \nu $ by 
\begin{align} & \label{:10v}
\nuone (dx d\sss ) = \rho _{\nu }^1 (x) \nux (d\sss )dx 
,\end{align}
where $ \nux = \nu ( \cdot + \delta_x \vert \sss (\{ \x \}) \ge 1 ) $ is the reduced Palm measure of $ \nu $ conditioned at $ \x \in \Rd $. 
We set $ \Llocone (\nuone ) = \bigcap_{\rR =1}^{\infty} L^1 ( \nuone (\cdot \cap \SRSS ))$. We refer to \dref{d:corfun} for the definition of correlation functions. 

Let $ \dcbone := \CziRd \ot \dcb $ and $ \dbbone := \CziRd \ot \dbb $. 

\smallskip 
We now quote the concept of the logarithmic derivative $ \dnub $ of $ \nu $ \cite{o.isde}. 
\begin{definition}\label{d:11} \thetag{i}
The logarithmic derivative $ \dnuc $ of $\nu $ is an $ \Rd $-valued function such that 
$ \dnuc \in \Llocone (\nuone ) $ and that, for all $ f \in \dcbone $, 
\begin{align} \label{:10w} 
\int _{\RdSS } f (x,\sss ) \dnuc (x,\sss ) \nuone (dx d\sss ) = 
 - \int _{\RdSS } \nablax f (x,\sss ) \nuone (dx d\sss ) 
.\end{align}
\thetag{ii} 
 The $ \bullet $-logarithmic derivative $ \dnub $ of $\nu $ is the $ \Rd $-valued function defined similarly to the logarithmic derivative, but with $ \dcbone $ replaced by $ \dbbone $. 
\end{definition} 

 A $ \bullet $-logarithmic derivative is a logarithmic derivative because $ \dcbone \subset \dbbone$. Both logarithmic derivatives coincide, which follows from \lref{l:39}. 

Intuitively, the logarithmic derivative $ \dnu (\xs )$ represents the force received by the tagged particle at $ x $ from (infinitely many) other particles $ \sss = \sum_i \delta_{\si }$. 
Informally, $ \dnu $ is defined as the differential of the logarithm of the Hamiltonian if the potential gives the equilibrium state $ \nu $. 
One can easily calculate the logarithmic derivatives of canonical Gibbs measures with Ruelle's class potentials using the DLR equation \cite{o.isde}. 
Even if the DLR equation does not hold, $ \dnu $ can be well defined---this is the case for the Coulomb potentials. 

%G[
The logarithmic derivative $ \dnu $ of $ \nu $ plays a vital role because, when it exists, it allows us to derive the stochastic dynamics of an infinite-particle system with the reversible equilibrium distribution $ \nu $. This connection leads to the formulation of solutions $ \X = (\xX ^i )_{i\in \N }$ through ISDEs, underscoring the power of this concept in understanding complex systems \cite{o.isde,o.rm,os.ld}: 
\begin{align}\label{:10y}&
X_t^i - X_0^i = B_t^i + \half \int_0^t \dnu ( X_u^i , \sum_{j\ne i }^{\infty} \delta_{X_u^j} ) du 
,\quad i\in\N 
,\end{align}
where $ \{ B^i \} $ are independent standard Brownian motions. 

%G]

%G[

\subsection{Assumptions}
We begin by introducing the assumptions used throughout the paper. 

For $\epsilon > 0$, we define $ \Rde = \{ (\x , \y )\in \SR ^2 ; \vert \xy \vert > \epsilon \}$. 

\smallskip 

\noindent 
\As{A1} 
Let $\{ \ON \}$ be an increasing sequence of open sets with $\bigcup_N \ON = \Rd $. 

\noindent \thetag{i} 
Suppose $ \Phi \in C^2 (\Rd )$. 
For each $ \nN , \rR \in \N $, let $\Phi^N \in C_b^2(\ON)$ satisfy 
\begin{align}\label{:11e} &&
\Phi ( \x ) &= \limi{N} \PhiN ( \x ) 
,&\text{in $ C_b^2(\SR ) $}
,\\\label{:11E} &&
\Phi^N(\x) & = \infty 
, &\x \notin \ON 
,\\ \label{:11g} &&
\Delta \PhiN (\x ) & \ge 0 
,&\x \in \ON 
.\end{align}

\noindent \thetag{ii} 
Let $\Psi^N : (\Rd)^2 \to \R \cup \{ \infty \}$ satisfy, for each $ \rR \in \N ,\, \e > 0 $, 
\begin{align} \label{:11f}&&
 \Psi ( \xy ) &
= \limi{N} \PsiN ( \x , \y ) 
, %\quad 
&\text{in $ C_b^{2} (\Rde ) $}
,\\ \label{:11h} &&
\PsiN ( \x , \y ) & = 0 
,%\quad 
&( \x , \y ) \notin \ONtwo 
,\\ \label{:11i} &&
 \Delta_x \PsiN ( \x , \y ) &
= - \cref{;1a}^{\nN } \delta_{\y } ( \x ) + \psiN (\x , \y )
,%quad 
& ( \x , \y ) \in \ONtwo 
,\end{align}
where $ \Ct ^{\nN }\label{;1a}$ is a positive constant, $ \delta_{\y } $ is the delta measure at $ \y $, and $\psi^N \in C_b(\ONtwo ) $ satisfies, for some constant $\Ct \label{;32!}$, 
\begin{align} &\label{:11*} 
0 \le \psiN \le \cref{;32!} / \nN 
.\end{align}

%GTP[
\smallskip 

\noindent \As{A2} \thetag{i} 
The RPF $ \muN $ admits the labeled density given in \eqref{:10e},
where $ \beta > 0 $ is arbitrary and fixed throughout this paper.

\noindent \thetag{ii} 
There exist an $ \lzz > 0 $ and a constant $ \Ct \label{;A2} $ such that, 
\begin{align}\label{:11+}
\limsupN
\int_{ \sSS } \sss ( \SR ) \, \muN ( d \sss )  \le
\cref{;A2} \rR ^{ \lzz } ,\quad \rR \in \N 
.\end{align}

\noindent \thetag{iii} 
There exists a weakly convergent subsequence of $ \{ \mu^N \} $, denoted by $ \{ \mu^{\Nn} \} $, with limit $ \mu $ such that 
\begin{align}
 \label{:11n}
& \mu \bigl( \{ \sss (\Rd ) = \infty \} \bigr) = 1
.\end{align}

%GTP]]
%GTP[[

\begin{remark}\label{r:A2}
\noindent \thetag{i}
From \eqref{:11+}, the family $\{\muN \}_{\nN \in\N }$ is tight, and 
\begin{align} \label{:11)} & 
\int_{\sSS} \sss(\SR)\, \mu(d\sss ) \le 
\cref{;A2} \rR ^{ \lzz } ,\quad \rR \in \N 
.\end{align}
This estimate is used in the construction of the logarithmic derivative of $\mu$.

\noindent \thetag{ii} 
Assumption \eqref{:11n} merely excludes the trivial finite-particle case. 

\noindent \thetag{iii}
We do not assume \emph{a priori} that $\mu$ admits a density or correlation functions;
these properties follow from the above assumptions (see \pref{l:2'}).
\end{remark}

%GTP]]

\medskip

For a function $ f ( x ) $, we set $ f^{(\mathbf{i})} = \partial ^k f / \partial \x ^{\mathbf{i}}$ for $ \mathbf{i} \in \mathbf{I}(k)$, where 
\begin{align}& \label{:11l} 
\mathbf{I} (k) = \{ \mathbf{i} = (i_1,\ldots,i_d) \in ( \zN )^d \, ;\, i_1+\cdots+i_d= k \} 
\end{align}
and $ \x ^{\mathbf{i}} = x_1^{i_1}\cdots x_d^{i_d}$ for $ x = (x_1\ldots,x_d) \in \Rd $. 
Let $ \IIIl = \sqcup_{k = 0 }^{ l } \mathbf{I}( k ) $ and 
\begin{align}\label{:11L}&
 \mathscr{S}_{\rR }^l = 
\{ \x \in \SR ; \det [ \, \xji \, ]_{\mathbf{i},\mathbf{j} \in \IIIl } \not= 0 \}
, \ \xji = (\x _1^{ j_1}) ^{i_1} \cdots (\x _d ^{ j_d}) ^{i_d} 
.\end{align}

We introduce the following assumption. 

When $ \PsiN = \Psi $, \As{A3} below holds automatically. 
Indeed, \eqref{:11q}, \eqref{:11o} and \eqref{:11p} are obvious, and 
\eqref{:11j} follows from \lref{l:3@}. 
\smallskip

\noindent \As{A3} 
%G[---
For the subsequence $ \Nn $ in \As{A2} \thetag{iii}, 
there exists an $ \lz \in \N $ satisfying $ \lz > \lzz - d $ 
such that the following hold:
\smallskip

%G]
\smallskip

\noindent \thetag{i} 
For each $ \mathbf{i} \in \mathbf{I}( \lz + 1 ) $, $ \rR \in \N $, and $ \epsilon > 0 $, 
\begin{align}\label{:11q}&
\Psi ( \xy ) = \limin \PsiNn ( \x , y ) \quad \text{ in $ C_b^{\lz + 2} ( \Rde ) $}
,\\ &\label{:11j} 
	 \int_{\sSS } \Big(
	\sum _{\si \not\in \sS _{\rR + \e }}^{\infty} 
	 \sup_{\n \in \N } 
\big\lVert (-\nablax \PsiNn )^{(\mathbf{i})} ( \x , \si ) \big\rVert _{\SIXONn }
\Big)
 \mu (d \sss ) < \infty 
.\end{align}
\noindent \thetag{ii} 
For each $ \mathbf{i} \in \mathbf{I}( \lz + 1 ) $, $ \rR \in \N $, and $ \mu $-a.s.\,$ \sss $, 
\begin{align} \label{:11o}&
\limin 		
\sum _{\si \not\in \oLSR }^{ \Nn } 
 \Big\lVert 
( \nablaPsi )^{(\mathbf{i})} ( \x - \si ) - 
(\nablax \PsiNn )^{(\mathbf{i})} ( \x , \si ) \Big\rVert_{\6 } = 0 
,\\\label{:11p} & 
\limin 
\sum _{\si \not\in \oLSR }^{ \Nn } 
\Big( 
 \nablaPsi (\x - \si ) - 
\nablax \PsiNn (\x , \si ) \Big) = 0 \ \text{ for some } 
 \xjX 
.\end{align}

\smallskip

These three assumptions are sufficient to obtain an explicit formula for the coefficient of the ISDE arising from $ \mu $ and to construct the stochastic dynamics as the unique, strong solution of the ISDE. 
To obtain the ISDE in translation-invariant form, we further assume the following: 

\smallskip 
\noindent \As{A4} 
The following convergence holds in $ L_{\mathrm{loc}}^2(\muone )$, for $ \muone $-a.e.\,$ ( \xs ) $, or in 
$ C^1 (\oLSRse ) $ for $ \mu $-a.s.\,$ \sss $ and $ \e > 0 $: 
\begin{align}\label{:11s}&
\limiR \sum_{i} \Big( 1_{\SR } (\si ) - 1_{\SR } (\xsi ) \Big) \nablaPsi (\xsi )
= \nablaPhi (x) 
.\end{align}
Here $ \oLSRse = \{ \x \in \oLSR \,;\, \lvert \xsi \rvert \ge \e \text{ for all } i \} $. 

\begin{remark}\label{r:A3}
The assumptions imposed on the limit RPF are mild. 
We believe that they allow the construction of many examples of RPFs with Coulomb interactions to which our theorems apply.
\end{remark}

%G]

\smallskip
\noindent
\textbf{Unlabeled and labeling maps. }
Let
$ \map{\ulab }{\big( \sqcup_{k=0}^{\infty} (\Rd )^k \big) \sqcup \RdN }{\mathscr{M}(\Rd ) }$
be the unlabeled map defined by
\begin{align} \label{:02v}
\ulab \big( (x^i)_i \big) = \sum_i \delta_{x^i}
,\end{align}
where $ \mathscr{M}(\Rd ) $ is the set of all measures on
$ (\Rd , \mathscr{B}(\Rd ) ) $.
If $ k = 0 $, we set $ \ulab (x) = \mathsf{0} $ for $ x \in (\Rd )^0 $ by convention,
where $ \mathsf{0} $ denotes the zero measure.

%G]%G[

A label $ \map{\lab }{\sSS \backslash \{ \mathsf{0} \} }{(\sqcup_{k=1}^{\infty} ( \Rd )^k) \sqcup \RdN } $ is a measurable map such that
\begin{align} \notag & 
\ulab \circ \lab = \mathrm{id} 
.\end{align}

We write $ \lab ( \sss ) = ( \labi ( \sss ) )_i $. 
Throughout this paper, we fix a label such that the absolute values of particles are non-decreasing:
\begin{align}\label{:02x}&
\text{$ \vert \labi ( \sss ) \vert \le \vert \lab ^{i+1} ( \sss ) \vert $ for all $ i $ and 
$ \sss \in \sSS \backslash \{ \mathsf{0} \} $}
.\end{align}
By \eqref{:11n}, we have $ \lab ( \sss ) \in \RdN $ for $ \mu $-a.s.\,$ \sss $. 
For $ \sss \in \sSS $, we often write $ \lab ( \sss ) = ( s^i )_i $. 
For $ \yy \in \sSS $, we define $ \yyN = \mathsf{0} $ if $ \yy = \mathsf{0} $ and, for $ \yy \ne \mathsf{0} $, 
\begin{align}\label{:02y}& 
\yyN = \sum_{i=1 }^{ \yy ( \Rd ) \wedge \nN } \delta_{\labi ( \yy )} 
.\end{align}
%G]
%GTP[

\noindent 
{\bf Conditional RPFs and couplings. } 
Let $ \oLSRc = \{ \lvert s \rvert > \rR \} $ as before. 
Let $ \muRyyNN $ be the regular conditional probability of $ \muN $ given by 
\begin{align}\label{:02z}&
\muRyyNN = \muN ( \cdot \mid \pioLRc ( \sss ) = \pioLRc ( \yyN ) ) 
.\end{align}

From \As{A1} and \As{A2}, the labeled density $ \mN $ is continuous on $ \ON ^N $. 
Hence, $ \muRyyRN $ is well defined for all $ \yy \in \sSS $. 
Moreover, $ \muRyyNN $ admits the $ k $-density functions $ \mRyy ^{\nN , k} $. 
For $ \mu $-a.s.\,$ \yy $, the sequence $ \{ \yyN \}_{ \nN \in \N } $ defines a coupling for the family of RPFs $ \{ \muRyyRN \}_{ \nN \in \N } $. 
%GTP]

\subsection{Explicit formulas for logarithmic derivatives: \tref{l:11}}
In this subsection, we present explicit formulas for the logarithmic derivatives.
This result is crucial for proving the uniqueness of solutions
and the existence of strong solutions to the ISDE.

Let $  \mathbf{I}(k) $ as in \eqref{:11l}. Let $ \lz $ be as in \As{A3}. For $ \mathbf{i}\in \mathbf{I}(\lz + 1) $, let 
\begin{align} \label{:11d}
A_{\mathbf{i}, \sss } (\x )
&=
\sum_{\si \in \oLSRc }^{\infty} 
\int_0^1 
(-\nabla \Psi )^{(\mathbf{i})}( t \x - \si )
(\ell + 1 )(1-t)^{\ell }\, dt
.\end{align}
Here we use the integral form of the Taylor expansion of $ -\nabla \Psi (\xsi ) $ around $ \x = 0 $. 
By \eqref{:11j}, the infinite sum in \eqref{:11d} converges absolutely in $ \6 $ for $ \mu $-a.s.\,$ \sss $ 
(see \pref{l:35}). 
We define the function $ \Rsl \in \6 $ by
\begin{align}\notag 
\Rsl (\x )
&=
\sum_{\mathbf{i}\in \mathbf{I}(\lz + 1)}
\xiGamma A_{\mathbf{i}, \sss } (\x ) 
.\end{align}
Let $ \CRi $ be the constant in \eqref{:36a}. 
Let $  \III = \sqcup_{k=0}^{\lz } \mathbf{I}(k)$. Let 
\begin{align*}&
\oLSQse = \{ \x \in \oLSQ \,;\, \lvert \xsi \rvert \ge \e \text{ for all } i \} , \quad 
\mathbb{R}^d_{\neq}(\sss ) = \cap_i \{ x \in \mathbb{R}^d \,;\, x \ne \si \}
\end{align*}
Let $ \muone $ be the one-reduced Campbell measure of $ \mu $. 
\begin{theorem}[Explicit formulas for $ \dmu $]\label{l:11}
%\noindent 
\thetag{i}
Assume \As{A1}--\As{A3}. Then $ \mu $ admits a $ \bullet $-logarithmic
derivative $ \dmu \in \Lloctwo ( \muone ) $, which is given by, for each
$ \rR \in \mathbb{N} $ and for $ \muone $-a.e.\ $ (\xs ) \in \SR \times \sSS $,
\begin{align}\label{:11b}&
\dmu ( \xs ) =
\beta \Big(
- \nablaPhi ( \x )
+ \sum_{\si \in \oLSR } \frac{ \xsi }{ \lvert \xsi \rvert^{\da } }
+ \1 \CRi \x ^{\mathbf{i} }
+ \Rsl ( \x )
\Big)
.\end{align}

\noindent
\thetag{ii}
Assume \As{A1}--\As{A3}. Then, for $ \mu $-a.s.\,$ \sss $, each 
$ \qQ \in \N $, $ \mathbf{i}\in \III $, and $ \e > 0 $, 
\begin{align}&\label{:11!} 
\limiR \CRi = 0 , \quad % \mathbf{i}\in \III 
\limiR \big\lVert \Rsl \big\rVert_{\7 } = 0
,\\ \label{:11"} &
\dmu (\xs ) =
\beta \Big(
- \nablaPhi (\x )
+ \limiR \sum_{\si \in \oLSR } \frac{\xsi }{\lvert \xsi \rvert^{\da }}
\Big)
\quad \text{in $ \SEVEN $}
\end{align}
and $ \dmu (\xs ) $ admits a $ \muone $-version that is locally Lipschitz in $ x $ on $ \Rd _{\neq }(\sss ) $. 

\noindent
\thetag{iii}
Assume \As{A1}--\As{A4}.
Then $ \dmu $ can be expressed as
\begin{align}\label{:11c}
\dmu (\xs )
=
\limiR \beta \sum_{ \lvert \xsi \rvert \le \rR } \frac{\xsi }{\lvert \xsi \rvert^{\da }}
.\end{align}
The convergence in \eqref{:11c} is understood in the sense specified by \As{A4}.
\end{theorem}
%G]

%G]
 \begin{remark}\label{r:11} 
By construction, $ \CRi $ and $ \Rsl $ are independent of $ \pioLR (\sss ) $. 
\end{remark}

The existence of the logarithmic derivative $ \dmu $ yields a natural $ \mu $-reversible diffusion 
\cite{os.ld}. 
Using $ \dmu $, the ISDE for the labeled dynamics $ \X = (\xX ^i )_{i\in\N }$ associated 
with the $ \sSS $-valued diffusion $ \XX _t = \sum_i \delta_{\xX _t^i}$ can be written as in \cite{o.isde}: 
\begin{align}\label{:11x}& 
X_t^i - X_0^i
=
B_t^i
+ \half \int_0^t 
\dmu \Big( X_u^i , \sum_{j\ne i}^{\infty} \delta_{X_u^j} \Big)\, du
,\quad i \in \mathbb{N}
.\end{align}

From \eqref{:11"}, for each $ i \in \N $, the ISDE \eqref{:11x} can be rewritten as 
\begin{align}& \notag
X_t^i - X_0^i =
B_t^i
-\halfbeta \int_0^t \nablaPhi (X_u^i )\, du
+ \halfbeta \int_0^t 
\limiR \sum_{\substack{ \lvert X_u^j \rvert \le \rR \\ j\ne i }}
\frac{\Xij }{\lvert \Xij \rvert^{\da }}\, du
,\end{align}
and if \As{A4} holds, then \eqref{:11c} yields 
\begin{align}& \notag
X_t^i - X_0^i
=
B_t^i
+ \halfbeta \int_0^t 
\limiR \sum_{\substack{ \lvert \Xij \rvert \le \rR \\ j\ne i }}
\frac{\Xij }{\lvert \Xij \rvert^{\da }}\, du
,\quad i \in \N 
.\end{align}

Note that the drift coefficient $ \half \dmu (\xs ) $ is not defined on all of $ \RdSS $. 
A key point is therefore to ensure that the dynamics $ \X $ a.s.\,avoid the singular set, 
and to control the long-range divergence coming from the non-integrability of $ \nablaPsi $ at infinity. 
These issues are central in making the above ISDE formulations rigorous.

\subsection{Existence of weak solutions: \tref{l:12}}\label{s:25}
In this subsection, we state a generalized ISDE and present the corresponding weak solutions.

Let $ \map{\aaaa }{\Rd }{(\R )^{d^2}}$, where $ \aaaa = ( \aaaa _{kl})_{k,l=1}^d $, be a uniformly elliptic, bounded, symmetric matrix-valued function such that $ \aaaa \in C_b^2 (\Rd )$. 
Assume that there exists a constant $ \Ct \label{;12} \ge 1 $ satisfying
\begin{align}\label{:12o}
\cref{;12} ^{-1}\lvert \xi \rvert ^2 
\le \sum_{k,l=1}^d \aaaa _{kl} \xi_k \xi_l 
\le \cref{;12} \lvert \xi \rvert ^2 
, \quad \xi \in \Rd 
.\end{align}
Let $ \sigma = (\sigma _{k,l}) _{k,l=1}^d \in C_b^1(\Rd )$ be a matrix such that $ \sigma^t \sigma = \aaaa $. Define
\begin{align}\label{:12p}
 \mathrm{div\,} \aaaa = \Big( \sum_{k=1}^{d} \PD{\aaaa _{kl}}{x_k} \Big)_{l=1}^d ,\quad 
\bbb = \half \{ \aaaadmu \} 
.\end{align}

%G[---

We consider the following elliptic-type ISDE, which generalizes \eqref{:11x}:
\begin{align}\label{:12q}
X_t^i - X_0^i
=
\int_0^t \sigma (X_u^i)\, dB_u^i
+ \int_0^t \bbb \Big( X_u^i , \sum_{j \ne i}^{\infty} \delta_{X_u^j} \Big)\, du
,\quad i \in \N 
.\end{align}
Here, $ B^i = (B_1^{i},\ldots,B_d^{i}) $ is a $ d $-dimensional Brownian motion, and
\begin{align}\notag
\int_0^t \sigma (X_u^i)\, dB_u^i
=
\Big( \sum_{l=1}^d \int_0^t \sigma _{k,l} (X_u^i)\, dB_l^{i} (u) \Big)_{k=1}^d
.\end{align}

%G]
\smallskip

Loosely speaking, a pair $(\X ,\B )$ of an $ \RdN $-valued continuous process
$ \X =(X^i)_{i\in\N }$ and an $ \RdN $-valued $ \{\mathscr{F}_t\}$-Brownian motion
$ \B =(B^i)_{i\in\N }$, defined on a filtered probability space
$(\Omega,\mathscr{F},\{\mathscr{F}_t\},P)$, is called a weak solution of \eqref{:12q}
if \eqref{:12q} holds $P$-a.s. 
A weak solution $(\X ,\B )$ is called a strong solution if, in addition, $ \X $ is a measurable function of $ \B $. 
See \ssref{s:IFC} for precise definitions.

For $ \X = (X^i)_{i\in\N }$, the associated $ m $-labeled process $ \Xm $ is
\begin{align}\label{:12u}
 \Xm _t=(X_t^1,\ldots,X_t^m,\sum_{j>m}\delta_{X_t^j})
.\end{align}
We set $ \X _t^{[0]} = \sum_{i\in\N } \delta_{X_t^i} $ and $ \mu ^{[0]} = \mu $. 
We also write $ \XX _t = \sum_{i\in\N } \delta_{X_t^i} $. 

%G[---
Let $ \mum $ be the $ m $-reduced Campbell measure of $ \mu $. 
If the $ m $-point correlation function $ \rho^m $ of $ \mu $ exists, then
\begin{align}\label{:12y}
\mum = \rho^m (\mathbf{x} )\, \mu_{\mathbf{x}} (d\sss )\, d\mathbf{x}
,\end{align}
where $ \mu _{\mathbf{x}} = \mu (\cdot + \sum_{i=1}^m \delta_{x^i}\vert \sss (\{ x^i \}) \ge 1 ,\, 1 \le i \le m )$ is the $ m $-reduced Palm measure conditioned at $ \mathbf{x} = (x^1,\ldots,x^m) \in (\Rd )^m $. 
Even if $ \mu $ does not admit an $ m $-point correlation function, the $ m $-reduced Campbell measure can be defined by localization (see \eqref{:Camp} and \cite{Kal}). 

\begin{theorem}[Existence of weak solutions] \label{l:12}
Assume \As{A1}--\As{A3}. 
Then ISDE \eqref{:12q} has a weak solution $ \X = (X^i)_{i\in\N }$ starting at $ \mathbf{s} = \lab (\sss ) $ for $ \mu $-a.s.\,$ \sss $. 
Moreover, for each $ m \in \zN $, the associated $ m $-labeled process $ \Xm $ is a $ \mum $-symmetric and conservative diffusion.
\end{theorem}
% 

%G]

%G[

\subsection{Unique strong solutions: \tref{l:13}}

We next turn to the construction of unique strong solutions to the ISDE \eqref{:12q}. 
In \tref{l:13}, we establish the existence of strong solutions and prove their pathwise uniqueness under suitable conditions. 
To this end, we introduce the notion of the tail decomposition of the underlying RPF. 

%G[

Let $ \mathscr{T} ( \sSS ) $ be the tail $ \sigma $-field defined by
\begin{align}\notag %\label{:13h}
\mathscr{T} ( \sSS )
=
\bigcap_{ \rR = 1 }^{ \infty }
\sigma [ \piRc ]
.\end{align}
Let $ \mua = \mu ( \, \cdot \mid \mathscr{T} ( \sSS ) ) ( \mathsf{a} ) $
denote the regular conditional probability of $ \mu $
given $ \mathscr{T} ( \sSS ) $.
Then 
\begin{align} \notag %\label{:13i}
\mu
=
\int_{ \sSS } \mua \, \mu ( d \mathsf{a} )
.\end{align}
%G]

\begin{definition}[Tail decomposition]\label{d:15}
An RPF $ \mu $ is said to admit a tail decomposition if there exists a set 
$ \SSz \in \mathscr{B}(\sSS ) $ such that $ \mu (\SSz ) = 1 $ and, for all $ \mathsf{a} \in \SSz $, the following hold:
\begin{align}
\label{:13p}&
 \mua (\mathsf{A} ) \in \{ 0,1 \} 
 \quad \text{for all } \mathsf{A} \in \mathscr{T}(\sSS ),
\\
\label{:13q}&
 \mua ([\aaa ] ) = 1 ,
 \quad \text{where } 
 [\aaa ] = \{ \bb \in \SSz \,;\, \mua = \mub \},
\\
\label{:13r}&
 \mua \perp \mub \ \text{ on } \mathscr{T}(\sSS ) 
 \quad \text{whenever } \mua \ne \mub .
\end{align}
\end{definition}

This decomposition provides a natural framework for the construction of strong solutions and the proof of pathwise uniqueness.

\begin{remark}\label{r:13}
\noindent \thetag{i}
Let $ \SSz $ be as in Definition~\ref{d:15} and $ \aaa \in \SSz $. 
If $ \mathsf{A} \in \mathscr{T}(\sSS ) $ satisfies $ \aaa \in \mathsf{A} $, then $ \mua (\mathsf{A}) = 1 $ by \eqref{:13q}.

\noindent \thetag{ii}
As in \cite[Lem.\,14.2]{o-t.tail}, one can prove that $ \mu $ admits a tail decomposition.
\end{remark}

\medskip

Let $ \CRd = C([0,\infty); \Rd ) $. 
We equip $ \CRd $ with the topology of uniform convergence on compact time intervals and endow $ \CRdN $ with the corresponding product topology. 
With this topology, $ \CRdN $ is a Polish space.
Let
\begin{align}\notag
\SSsi
&=
\{ \sss \in \sSS \,;\, \sss (\Rd ) = \infty ,\ 
\sss (\{ s \}) \in \{ 0 , 1 \} \text{ for all } s \in \Rd \},
\\
\notag
\WSsi 
&=
\{ \ww \,;\, \ww \text{ is an } \SSsi\text{-valued continuous path on } [0,\infty) \}.
\end{align}
%G[
For $ \w = ( w^i )_{ i = 1 }^{ \infty } \in \CRdN $,
define the unlabeled path $ \ww := \upath (\w )$ by 
$ \ww _t = \upath (\w )_t 
=
\sum_{ i = 1 }^{ \infty } \delta_{ w_t^i }$. 
We set
\begin{align} \label{:13v}&
\WSsiNE = \WSsi \cap \upath (\CRdN ) 
,\\
\label{:13u}&
\IRT (\w )
=
\sup \bigl\{
i \in \mathbb{N}
\,;\,
\min_{t\in[0,T]} \lvert w^i(t) \rvert \le \rR
\bigr\}.
\end{align}
%G]
%G[

For an $ \RdN $-valued process $ \X = (X^i)_{i\in\N }$, let $ \XX = \upath (\X )$ denote the associated unlabeled process. If $ \XX \in \WSsiNE $, then $ \IRT (\X ) $ is well defined.
%G]

We impose the following conditions, from basic state-space regularity
to absolute continuity with respect to a reference measure $\nu$. 
\smallskip

\noindent 
\As{SIN} $ P ( \XX \in \WSsiNE ) = 1$. \\
\As{NBJ} $ P (\IRT (\X ) < \infty ) = 1 $ for each $ R , T \in \mathbb{N}$. \\
\As{AC}$_{\nu }$ $ P \circ \XX _t^{-1}\ll \nu $ for all $ 0 < t < \infty $.

\smallskip

Condition \As{SIN} ensures that the dynamics remain in the admissible configuration space.
Condition \As{NBJ} guarantees that, within any finite time interval, only finitely many particles enter a given bounded region.
Condition \As{AC}$_\nu$ ensures absolute continuity of the time marginals with respect to the reference measure $\nu$.

%G[
In addition, we impose two further conditions.
The condition \As{MF} is a mild measurability assumption,
whereas \As{IFC} is a substantive condition whose formulation
requires the finite-dimensional SDE scheme introduced later.
They are introduced in \ssref{s:US} and \ssref{s:IFC}, respectively. 
Precise definitions of strong solutions, unique strong solutions,
and pathwise uniqueness are given in Definitions~\ref{d:US4}--\ref{d:US6}.
%G]

The following \tref{l:13} establishes the existence and uniqueness of a
strong solution of \eqref{:12q} in the sense of \dref{d:US6}.
\begin{theorem}[Unique strong solutions]\label{l:13}
Assume \As{A1}--\As{A3}. Then the following hold. 

\noindent \thetag{i} 
%G[
For $ \mu $-a.s.\ $ \mathsf{a} $ and for $ \mua $-a.s.\ $ \sss $, 
the ISDE \eqref{:12q} admits a unique strong solution $ \X $ starting at $ \mathbf{s} = \lab ( \sss ) $ 
under the constraints \As{SIN}, \As{NBJ}, \As{AC}$_{\mua }$, \As{MF}, and \As{IFC}.
%G]

\noindent \thetag{ii}
For each $ m \in \zN $, the associated $ m $-labeled process $ \Xm $ is a $ \muam $-symmetric and conservative diffusion.
\end{theorem}
By \tref{l:US2}, we obtain the following corollary of \tref{l:13}. 
Let $ \xi _{\aaa } $ be an RPF such that $ \xi _{\aaa } \ll \mua $. 

\begin{corollary}\label{l:13A}
Assume \As{A1}--\As{A3}. 
Then, for $ \mu $-a.s.\ $ \aaa $, any weak solution of \eqref{:12q} with initial distribution 
$ \xi _{\aaa } \circ \lab ^{-1} $ is unique in law and pathwise unique 
under \As{SIN}, \As{NBJ}, \As{AC}$_{\mua }$, and \As{IFC}. 
\end{corollary}
%G]
From Theorems \ref{l:11}--\ref{l:13}, we obtain the following corollary.
\begin{corollary}\label{cor:12}
Assume \As{A1}--\As{A3}. 
Then the conclusions of Theorem \ref{l:12} and \tref{l:13} apply  to the ISDE \eqref{:12a}:
\begin{align}\notag 
& X_t^i - X_0^i
=
\int_0^t \sigma (X_u^i)\, dB_u^i
+\frac{\beta }{2} \int_0^t \mathrm{div\,} \aaaa (X_u^i)\, du 
\\  & 
+\frac{\beta }{2}\int_0^t 
\Big(
- \nablaPhi (X_u^i) + \limiR \sum_{\substack{\Lvert X_u^j \rvert \le \rR , \\ j \ne i }} 
\frac{X_u^i - X_u^j}{\lvert X_u^i - X_u^j \rvert ^d } 
\Big) 
\aaaa (X_u^i)\, du 
,\  i \in \N 
\label{:12a}
.\end{align}
\end{corollary}
%G]

\subsection{Diffusion processes in $\RdN $: \tref{l:1Y}} 
The analysis of fully labeled dynamics on the full label space $\RdN $ 
plays a crucial role in understanding infinite-particle systems. 
While the $ m $-labeled dynamics $\Xm $ provide finite-dimensional projections 
that admit invariant measures, they capture only partial aspects of
the underlying infinite system.
%%% ---

While the $ m $-labeled dynamics $\Xm $ provide projections to the tame infinite-dimensional spaces 
that admit invariant measures, they capture only partial aspects of 
the underlying infinite system.

%%%%
To obtain a complete description, it is necessary to go beyond finite labels
and establish the existence of a genuine diffusion on $\RdN $.
This step is indispensable, as it ensures that the dynamics are well defined
for the entire infinite system rather than merely for restricted subsystems.

Historically, the construction of fully labeled dynamics has been regarded as
a highly challenging problem.
Even for determinantal processes such as the Dyson model, establishing such
dynamics required delicate arguments based on random matrix theory.
For general Coulomb systems, the situation is more severe:
the long-range nature of the interaction and the absence of a classical
Gibbsian structure prevent a direct application of Dirichlet form methods
or invariant-measure techniques.
Against this backdrop, proving the existence of fully labeled conservative
diffusions represents a substantial advance.

In \tref{l:1Y}, we prove that the fully labeled process $\X $ 
constructed in \tref{l:13} is an $\RdN $-valued conservative diffusion.
Unlike the $ m $-labeled processes $\Xm $, the diffusion $\X $ 
is obtained without relying on the existence of an invariant measure, 
reflecting a fundamental difference between finite-label and full-label dynamics.
Nevertheless, the result guarantees that the fully labeled dynamics form a
legitimate stochastic process with infinite lifetime, thereby providing a
rigorous foundation for the dynamical theory of Coulomb systems.

Let $\mathbf{S}_{\star\star}\subset \RdN $ be the subset defined in \eqref{:9Bi}. 
 Let $\X $ be the solution of the ISDE \eqref{:12q} constructed in \tref{l:13}. 
Let $ \ulab $ denote the unlabeled map as in \eqref{:02v}. 
\begin{theorem}[Diffusion processes in $\RdN $] \label{l:1Y}
Assume \As{A1}--\As{A3}.
Then $\X $ is a conservative diffusion with state space
$\mathbf{S}_{\star\star}$ and 
$ \mu (\ulab ( \mathbf{S}_{\star\star} )) = 1$. 
\end{theorem}

Here, saying that $\X $ is a conservative diffusion with state space
$\mathbf{S}_{\star\star}$ means that the regular conditional law
\[
P_{\mathbf{x}}
:= P\bigl( \X \in \cdot \,\big|\, \X _0 = \mathbf{x} \bigr)
\]
admits a version such that the family
$\{ P_{\mathbf{x}} \}_{\mathbf{x}\in\mathbf{S}_{\star\star}}$
consists of probability measures on
$C([0,\infty);\mathbf{S}_{\star\star})$
satisfying the strong Markov property and having infinite lifetime
\cite{c-f}.
In particular, the diffusion law $P_{\mathbf{x}}$ is well defined for every
$\mathbf{x}\in\mathbf{S}_{\star\star}$.

The state space can, in principle, be extended to the whole $\RdN $
by freezing the dynamics outside $\mathbf{S}_{\star\star}$.

%G]---
\subsection{$ N $-particle approximation: \tref{l:16}}\label{s:28} 
We study finite-volume approximations of the infinite-particle dynamics. 
These approximations are realized in two complementary ways: 
by restricting the system to a bounded domain with reflecting boundary conditions (\pref{l:14}), 
and by truncating the system to $ N $ particles (\tref{l:16}). 
These theorems rigorously establish the convergence of these finite-volume approximations 
to the infinite-dimensional dynamics.

%G]%G[--

Let $ \dN $ be the logarithmic derivative of $ \muN $. 
By \eqref{:10e} and \As{A1}, 
\begin{align}\notag
\dN (\x , \sss ) 
= - \beta \Big( \nabla \PhiN ( \x ) + \sum_{i=1}^{\nN -1} \nablax \PsiN ( \x , \si ) \Big),
\quad \x \in \ON 
.\end{align}
Set $ \bbbN = \half (\mathrm{div\,} \aaaa + \dN \aaaa ) $. 
Then the SDE describing the $ \nN $-particle dynamics $ \XN = (X^{\Ni })_{i=1}^{\nN } $ is given by, 
for $ 1 \le i \le \nN $,
\begin{align}\notag 
 X _t^{\Ni } - X _0^{\Ni } 
 & = \int_0^t \sigma (X _u^{\Ni }) \, d B_u^i 
 + \int_0^t \bbbN \Big( X _u^{\Ni } , \sum_{j\ne i}^{\nN } \delta_{ X _u^{\Nj } } \Big) du 
\\ \notag 
&\quad \quad \quad + \int_0^t \anN ( X _u ^{\Ni } ) \, dL_{\nN }^{\Ni } (u), 
\\ \label{:14g}
L_{\nN }^{\Ni } (t) &= \int_0^t 1_{\partial \ON } ( X _u ^{\Ni } ) \, d L_{\nN }^{\Ni } (u) 
.\end{align}
Here $ \anN $ denotes the inward unit normal vector on $ \partial \ON $, 
and $ L_{\nN }^{\Ni } $ is a nonnegative, nondecreasing continuous process, 
called the local time on the boundary $ \partial \ON $ 
(cf.\ \cite[pp.~251--256]{fot.2}). 
If $ \ON = \Rd $, no local time term appears in \eqref{:14g}. The SDE then becomes 
\begin{align*}&
 X _t^{\Ni } - X _0^{\Ni } 
 = \int_0^t \sigma (X _u^{\Ni }) \, d B_u^i 
 + \int_0^t \bbbN \Big( X _u^{\Ni } , \sum_{j\ne i}^{\nN } \delta_{ X _u^{\Nj } } \Big) du 
.\end{align*}

We impose
\begin{align}\label{:14x}& 
\XN _0  = \labN (\sss ), \qquad \X _0 = \lab (\sss )
.\end{align}
Here $ \lab = ( \labi )_{ i } $ denotes the labeling map introduced in
\eqref{:02x}, and $ \labN = ( \lab ^i )_{ i = 1 }^{ N } $. 
The initial conditions of these solutions are coupled through the labeling map. 

%GTP]

Let $ \Ct \label{;14} $ be a positive constant. Let 
\begin{align}\label{:14y}&
\varphi \in C_b(\sSS ), \quad 
0 \le \varphi \le \cref{;14}, \quad 
\int_{\sSS } \varphi d\mu = 1 
.\end{align}
Define the RPF $ \varphi d\mu $ by $ (\varphi d\mu )(A ) = \int_{ A } \varphi d\mu $ for 
$ A \in \mathscr{B}(\sSS ) $. 
Let $ \XN $ and $ \X $ be the (unique strong) solutions of SDEs \eqref{:14g} and \eqref{:12q}, respectively. 
We assume:
\begin{align}\label{:14z}&
\XN _0 \elaw (\varphi d\mu ) \circ (\labN )^{-1} 
,\quad 
 \X _0 \elaw (\varphi d\mu ) \circ \lab ^{-1} 
.\end{align}
By setting $ \xX _t^{\nN , i } = \lab ^i ( \sss ) $ for all $ t \ge 0 $
and $ i > \nN $, we regard $ \XN = (\xX ^{\nN , i })$ as a $ \CRdN $-valued random variable. 
\begin{theorem}[$ N $-particle approximation]	\label{l:16}
Assume \As{A1}--\As{A3} and \eqref{:14y}--\eqref{:14z}. Then 
\begin{align} \notag  
\limin \XNn &= \X \quad \text{ in law in } \CRdN 
.\end{align}
\end{theorem}

\subsection{Related works and positioning}
We begin by recalling earlier developments on logarithmic derivatives,
which constitute a foundational concept in the theory of infinite-particle dynamics.

\smallskip
\noindent \textbf{Logarithmic derivatives.}
The concept of the logarithmic derivative for RPFs
was introduced in \cite{o.isde}.

In \cite{o.isde}, it was proved that if $ \nu $ admits a logarithmic derivative
$ \dnu $ and satisfies the quasi-Gibbs property, then \eqref{:10y}
admits a weak solution, and the associated unlabeled dynamics
$ \XX _t = \sum_i \delta_{ X_t^i } $
is a $ \nu $-reversible diffusion.

In \cite{os.ld}, it was proved that, for general random point fields,
the mere existence of a logarithmic derivative suffices to construct
such dynamics.

However, the mere existence of a logarithmic derivative
and the construction of a weak solution to the ISDE
do not guarantee uniqueness of solutions
or the existence of strong solutions.
Moreover, they do not allow a detailed analysis
of the qualitative behavior of the solutions.
Addressing these issues is therefore essential.

%G]

%%%%%%%%%%%%%%%%

%G[ ---
\smallskip
\noindent \textbf{Hyperuniformity.} 
The logarithmic derivatives for the sine$_{\beta }$, Airy$_{\beta }$ $( \beta = 1, 2, 4 )$, Bessel$_2$, and Ginibre RPFs have been computed \cite{o.isde,o.rm,o-t.airy,h-o.bes}. 
The former three are one-dimensional systems 
with two-dimensional Coulomb interaction potentials, and hence are not Coulomb RPFs. 
%G]

These computations require strong hyperuniformity of particle numbers within $ \SR $ for each $ \rR \in \mathbb{N}$, a property that must be established separately for each RPF.
% 
%G[
For example, in the case of the Ginibre RPF,
we use the following estimate
\cite{shirai,o-s.v}:
\begin{align}\label{:1Za}
\sup_{ \nN \in \N }
\mathrm{Var}^{ \muN } \bigl[ \sss ( \SR ) \bigr]
=
O ( \rR )
\quad \text{as } \rR \to \infty
.\end{align}
Such an estimate is called \emph{hyperuniform} because,
if $ \muN $ is replaced by $ \Lambda ^{ \nN } $,
then the right-hand side becomes
$ O ( \rR ^2 ) $ \cite{leb.hyp}.
Here $ \Lambda ^{ \nN } $ is the conditioned Poisson RPF defined by
\begin{align*}
\Lambda ^{ \nN }
=
\Lambda \Bigl(
\, \cdot \,\Big\vert\,
\sss ( \sS _{ \rR ( \nN ) } ) = \nN ,
\ \sss ( \Rd \backslash \sS _{ \rR ( \nN ) } ) = 0
\Bigr)
.\end{align*}
Here $ \Lambda $ is the Poisson RPF on $ \Rd $
with Lebesgue intensity $ dx $,
and $ \rR ( \nN ) $ is the positive number such that
the volume of $ \sS _{ \rR ( \nN ) } $
coincides with $ \nN $.
%G]

While Lebl\'{e} \cite{leb.hyp} established a hyperuniform estimate
for general $ \beta > 0 $, it is weaker than \eqref{:1Za} and is insufficient for our purposes.

%G[---
The proofs of the strong hyperuniform estimates for the above-mentioned examples 
rely on techniques specific to determinantal RPFs, 
and, except in the Ginibre case, do not extend to general Coulomb RPFs, 
which necessitates a new approach. 
%G]

\smallskip 

\noindent \noindent \textbf{Logarithmic derivatives of Coulomb RPFs.}
Until now, the logarithmic derivatives of Coulomb RPFs have been computed only in the case of the two-dimensional Coulomb RPF at inverse temperature $ \beta = 2 $, namely the Ginibre RPF \cite{o.isde}.

In contrast, we present here a new, remarkably simple and transparent method 
that yields uniform estimates for the logarithmic derivative of $ \muN $ 
\eqref{:1Zf} and for its conditional analogue \eqref{:1Zg}, 
valid simultaneously for all spatial dimensions $ d \ge 2 $ and all inverse temperatures $ \beta > 0 $.

The key insight is to exploit a certain harmonicity property of the interaction potential, which permits effective control over the infinite-dimensional interactions. Our approach is robust, conceptually clear, and provides a versatile tool for analyzing the stochastic dynamics of $ d $-dimensional Coulomb RPFs, thereby overcoming the limitations of all previously known methods.
%G[ -------------------------------------------------------------

\smallskip
\noindent
\textbf{Annealed and quenched estimates for the logarithmic derivative. }
We introduce a robust method for estimating the logarithmic derivative of $ \muN $
and refine it to conditional probabilities $ \muRyyN $,
exploiting a form of harmonicity of the interaction potential.

As a first step, we obtain the following uniform $ L^2 $-bound:
\begin{align}\label{:1Zf}
\sup_{\nN \in \N } \int_{\SQSS } \bigl| \dN ( \xs ) \bigr|^2 \, \muNone ( d\x d\sss ) < \infty
.\end{align}
This bound provides the control needed to define logarithmic derivatives
rigorously in the Coulomb setting (\lref{l:22}).

The proof of \eqref{:1Zf} is remarkably simple:
it relies only on \eqref{:11i}, \eqref{:11*}, and the basic definition of the logarithmic derivative. 

Moreover, the estimate \eqref{:1Zf} may be viewed as a form of hyperuniformity,
since the quantity in \eqref{:1Zf} diverges when $ \muNone $ is replaced by
the one-Campbell measure $ \Lambda ^{\nN , [1]} $ of the conditioned Poisson point process
$ \Lambda ^{\nN } $, even after excluding a neighborhood of $ \x $. 
This hyperuniformity is much weaker than that used in the previous works
\cite{k-o.du,k-o-t.udf,k-o-t.ifc,o.rm,o.isde,o-t.tail,o-t.airy,o-s.v},
which motivates the development in this paper of a new framework for the
analysis of stochastic dynamics associated with Coulomb random point fields. 

We next establish the \textbf{quenched analogue} of \eqref{:1Zf}:
\begin{align}\label{:1Zg}
\sup_{ \rR \ge \qQ + 1 }
\sup_{ \nN \in \N }
\int_{ \SQSS }
%	\bigl\lvert \dRyyRNN ( \xs ) \bigr\rvert ^2
\bigl\lvert \dN ( \xs ) \bigr\rvert ^2
\, \mu _{ \RyyN }^{ \nN , [1] } ( d\x d\sss )
< \infty
\quad \text{for $ \mu $-a.s.\,$ \yy $}
,\end{align}
where $ \yyN $ is defined in \eqref{:02y} (\lref{l:32}).

Inequality \eqref{:1Zg} plays a decisive role in the proof of \tref{l:11}.
Indeed, it fulfills the role previously played by strong hyperuniform estimates,
but with a crucial difference:
its proof is \textbf{significantly simpler}.

Earlier results relied on bounding averaged fluctuations
of the exterior configuration,
whereas here we directly control the external condition
in a quenched form.
This shift in perspective---replacing averaged estimates
by a quenched analysis---is one of the key conceptual innovations
of the present paper.

%G]

\smallskip 

\noindent \noindent \textbf{Interacting Brownian motions.}
Lang initiated the systematic study of interacting Brownian motions, which was subsequently developed by Fritz, Tanemura, and others in the case of finite-range interaction potentials \cite{lang.1,lang.2,fritz,tane.1}.
When the interaction becomes long-ranged, however, the problem of constructing solutions to ISDEs poses significant difficulties, even for potentials within Ruelle's class---namely, those that are superstable and integrable at infinity.
Later, the first author introduced the Dirichlet form approach, which made it possible to establish the existence of weak solutions to ISDEs for a broad family of interaction potentials \cite{o.tp,o.isde,o.rm,o.rm2}. This framework encompasses, in particular, the Ginibre interacting Brownian motion.
Nonetheless, the stronger questions of constructing pathwise unique strong solutions and establishing their uniqueness remained open.

In \cite{tsai.14}, Tsai succeeded in constructing pathwise unique strong solutions of ISDEs for the $ \sin_{\beta}$ RPFs with all $\beta \geq 1$. These processes correspond to infinite-particle systems in one dimension with logarithmic interactions, that is, with the two-dimensional Coulomb potential. In this sense, the ISDE in \cite{tsai.14} can be regarded as involving a Riesz potential. Tsai's proof relies on a delicate coupling mechanism that exploits special properties of the logarithmic interaction in one dimension. Such an approach, however, has no counterpart in higher dimensions with genuine Coulomb potentials.

Subsequently, in \cite{o-t.tail,k-o-t.ifc}, a general method for constructing pathwise unique strong solutions of ISDEs was developed. The central idea is to represent the ISDE as an infinite hierarchy of finite-dimensional SDEs evolving in a random environment. The well-posedness of this hierarchy is then reduced to a problem of tail triviality for the equilibrium measures of the associated unlabeled dynamics.

%G[------
To apply this framework, one requires both an equilibrium state (RPF)
and suitable regularity properties of the corresponding logarithmic derivative.

The construction of Coulomb RPFs was established in \cite{a-s,thoma.24}.
By exploiting the quenched bound \eqref{:1Zg},
we apply a Taylor expansion method adapted to the Coulomb interaction potential
and prove the following explicit representation of the logarithmic derivative
(\tref{l:11}):
\begin{align*}&
\dmu (\xs ) =
\beta \Big(
- \nablaPhi (\x )
+ \sum_{\si \in \oLSR } \frac{\xsi }{\lvert \xsi \rvert^{\da }}
+ \1 \CRi \x ^{\mathbf{i}}
+ \Rsl (\x )
\Big)
.\end{align*}
Using this representation, we further show that the associated ISDE satisfies the key assumption
known as the IFC condition, which is required in the above theory (\tref{l:13}).
%G]
%G[

\smallskip
\noindent
\textbf{Further related works.}
Katori and Tanemura proposed an alternative approach
to constructing stochastic dynamics with logarithmic interaction. 
Their method, based on space--time correlation functions,
is restricted to the special case of inverse temperature $ \beta = 2 $ in one-dimensional systems.
For further details, see \cite{KT07b} and the references therein. 

Moreover, \cite{o-t.sm} investigated in detail
the relationship between the logarithmic derivative
and the space--time correlation approach.
It was rigorously established that the stochastic dynamics
constructed by these two seemingly distinct methods
are, in fact, identical.
This coincidence highlights the fundamental role
of the logarithmic derivative
in unifying different perspectives
on infinite-particle dynamics with logarithmic interactions.
%G]

%G[[
In \cite{a-m}, Assiotis and Mirsajjadi constructed a non-equilibrium dynamics whose equilibrium distribution is the inverse Bessel RPF, and they identified the ISDE governing this dynamics.
Their method is algebraic in nature, relying on intertwining theory and the use of derivatives of the so-called characteristic polynomial. 
This construction bears a striking resemblance to the framework of logarithmic derivatives developed in \cite{o.isde}.

In \cite{b-kawa}, Bufetov and Kawamoto studied intertwining relations for Laguerre processes with parameter $ \alpha $ across different particle numbers. 
They computed the logarithmic derivatives of the corresponding finite-particle systems and pointed out that the construction of the infinite-particle limit process remains an open problem, to be addressed in subsequent work.

%G[

In \cite{kawa.22}, Kawamoto constructed a unique strong solution to the ISDE associated with the generalized sine random point field at inverse temperature $ \beta = 2 $, combining random matrix theory with the general ISDE theory \cite{o-t.tail}.

%G]

In \cite{e-t}, Esaki and Tanemura constructed a unique strong solution to a jump-type ISDE associated with a jump-type infinite-particle system with long-range interaction, and constructed the interacting Cauchy processes. 

In \cite{suzu.erg}, Suzuki established the ergodicity of the unlabeled dynamics of the Ginibre interacting Brownian motion, thereby providing a concrete confirmation of the general ergodicity theory.

\smallskip
\noindent
%G[----
\textbf{Key difficulties and techniques.} 
A major obstacle underlying this program is the pronounced and extremely strong long-range interactions 
characteristic of Coulomb systems. 
Previous methods relied on delicate hyperuniform estimates,
which are available only in the unique determinantal case $ d = 2 $ and $ \beta = 2 $. 

Our approach replaces this reliance
with new analytic estimates for the logarithmic derivative,
together with a quenched formulation that yields robust $ L^2 $ bounds.
This combination provides a conceptually simple yet powerful framework,
which enables us to establish strong well-posedness 
of Coulomb interacting Brownian motions 
in all spatial dimensions $ d \ge 2 $ and for all inverse temperatures $ \beta > 0 $.

%G]

%G[
\smallskip
\noindent
\textbf{Structure of the proofs.}
The proofs of the main theorems are distributed as follows: 
\tref{l:11} is proved in Section~\ref{s:3},
\tref{l:12} in Section~\ref{s:7},
\tref{l:13} in Section~\ref{s:8},
 \tref{l:1Y} in Section~\ref{s:9}, and 
 \tref{l:16} in \sref{s:X}. 

%G]
%G[

\smallskip
\noindent
\textbf{Sketch of the proof strategy.} 
The main theorems are proved in Sections \ref{s:2}--\ref{s:X},
and the overall strategy is as follows.

In \sref{s:2}, we establish the existence of the logarithmic derivative.
A key ingredient is a robust uniform $ L^2 $-estimate for the logarithmic derivative,
which exploits the harmonicity of the Coulomb potential.

Section~\ref{s:3} strengthens this result by proving a quenched estimate
and deriving an explicit representation of the logarithmic derivative.

In Section~\ref{s:4}, we establish non-collision properties of ISDE solutions
by direct SDE techniques.
These arguments do not rely on Dirichlet form theory;
nevertheless, the resulting non-collision property is later used
to prove uniqueness of Dirichlet forms.

%G]
%G[

Sections~\ref{s:5}--\ref{s:7} form the technical core of the paper:
we introduce Dirichlet form schemes for $ m $-labeled dynamics
and construct infinite-volume processes under both lower and upper Dirichlet forms.

In \sref{s:8}, using the explicit representation of the logarithmic derivatives,
we verify the IFC condition ensuring consistency of the dynamics 
and upgrade the constructions to pathwise unique strong solutions of the ISDEs.

In \sref{s:9}, using the uniqueness of solutions of the ISDE,
we establish uniqueness of extensions of the Dirichlet forms
and prove the existence of the fully labeled diffusion process in $ \RdN $.

In \sref{s:X}, we prove the convergence of finite-particle systems with reflecting boundary conditions 
using the convergence of the approximating lower Dirichlet forms. 
Finally, combining the uniqueness of extensions of Dirichlet forms with a sandwich argument, 
we prove that the $ \nN $-particle dynamics converge 
to the infinite-particle ISDE dynamics. 

%G]
%[GTP
\section{Concrete examples}\label{s:)}

We present concrete examples of RPFs that satisfy assumptions \As{A1}--\As{A4}.
Throughout this section, we assume $ d \ge 2 $ and $ \beta > 0 $.
%GTP] 

\begin{example}[\cite{thoma.24}]\label{d:Gauss} 
Let $\Phi $ be such that $ 0 \le \Delta \Phi $ is bounded on $ \Rd $, and set 
\begin{align*}&
\PhiN(x) = \nN^{2/d}\,\Phi\!\left(\frac{ \x + a_N }{\nN^{1/d}}\right) 
.\end{align*}
Here $ a_N $ is a convergent sequence in $ \Rd $ with limit $ a $,
and $ a $ belongs to the droplet, 
that is, the support of the equilibrium measure associated with the potentials $ \Phi $ and $ \Psi $.
Take $\PsiN = \Psi$ and $\ON = \mathbb{R}^d$. 
Then \As{A1} \thetag{ii} and \As{A3} are automatically satisfied. 
Moreover, \As{A1} \thetag{i} is satisfied in a wide class of examples. 

\As{A2} \thetag{i} is obvious,
and \thetag{ii}--\thetag{iii} follow from \cite{thoma.24}.
Indeed, in \cite{thoma.24},
a convergent subsequence was constructed,
yielding the limiting Coulomb RPF $ \mu $,
which possesses bounded correlation functions of all orders,
together with other regularity properties.
Moreover, the assumptions in \cite{thoma.24}
are weaker than those stated above;
see \cite[\thetag{1.25}, Th.\,6]{thoma.24}.

If $ d = 2 $ and $ \Phi (\x ) = \half \lvert x \rvert ^2 $, then the resulting ISDE is 
\begin{align}& \notag
X_t^i - X_0^i =
B_t^i
-\halfbeta \int_0^t X_u^i \, du - \frac{\beta }{2} a t 
+ \halfbeta \int_0^t 
\limiR \sum_{\substack{ \lvert X_u^j \rvert \le \rR \\ j\ne i }}
\frac{\Xij }{\lvert \Xij \rvert^{2}}\, du
.\end{align}
In this way, the term $ -  \frac{\beta }{2} a t $ appears, canceling the effect of the macroscopic position $ a $ in the droplet. 
\end{example}

We now present two further specific examples of RPFs that satisfy assumptions \As{A1}--\As{A4}. 
These examples exhibit a form of ``invariance'' under certain transformations, which follows from the invariance of the corresponding finite-particle systems. 
As a consequence, the one-point correlation function is constant for each $N$-particle system.

\begin{example}[Periodic approximation]\label{d:per} 
Let $ \mathbb{T} = (-\tfrac{1}{2}, \tfrac{1}{2}]^d $ denote the unit torus. 
Set $\PhiN = \Phi = 0$, and let $\Psi_{\mathbb{T}}(x,y)$ denote the Green's function of 
$-\cref{;per}\,\Delta$ on $\mathbb{T}$ subject to periodic boundary conditions, 
where $\Ct\label{;per}$ is the positive constant determined according to 
\eqref{:11f}. 
Set $ \ON = (-\tfrac{\nN ^{1/d}}{2},\,\tfrac{\nN ^{1/d}}{2})^d $ and 
$ \PsiN(\x,\y) = \Nd \,\Psi_{\mathbb{T}}\!\left(\frac{\x}{\nN ^{1/d}},\,\frac{\y}{\nN ^{1/d}}\right)$. 
By construction, the measure $\mu$ is invariant under translations on $\Rd $. 
\end{example}

\begin{example}[Sphere approximation]\label{d:sp} 
Let 
$ \mathbb{S}=\{\hat{x}\in\mathbb{R}^{d+1}:\lvert \hat{x}\rvert=1\}$ 
denote the unit sphere, and let $\mathbf{e}_{d+1}=(0,\dots,0,1)$ denote the north pole. 
Define the stereographic projection 
$ 
\varpi:\mathbb{S}\setminus\{\mathbf{e}_{d+1}\}\to\mathbb{R}^d
$ 
by
\[
\varpi(\hat{x})=\frac{\hat{x}'}{\,1-\hat{x}_{d+1}\,}, 
\qquad \hat{x}=(\hat{x}',\hat{x}_{d+1})\in\mathbb{R}^d\times\mathbb{R},
\]
so that $\varpi(\hat{x})$ is the intersection of the line through $\mathbf{e}_{d+1}$ 
and $\hat{x}$ with the hyperplane $\mathbb{R}^d\times\{0\}$. 
The inverse map $\varpi^{-1}:\mathbb{R}^d\to\mathbb{S}\setminus\{\mathbf{e}_{d+1}\}$ is
\[
\varpi^{-1}(x )=\frac{(2x_1,\ldots,2x_d,\,|x|^2-1)}{1+|x|^2}, \qquad x\in\mathbb{R}^d,
\]
with Jacobian $ J(x )={2^d}/ {(1+|x|^2)^d} $. Define 
$ \Psi_{\mathbb{S}} (\hat{\x } , \hat{\y } ) $ as the Green's function of $ - \cref{;sp}\,\Delta_{\mathbb{S}} $  on $ \mathbb{S}$, 
where $\Delta_{\mathbb{S}}$ is the Laplace--Beltrami operator and $\Ct\label{;sp}$ is determined according to \eqref{:11f} and \eqref{:11i}. 

Define $ \ON = \mathbb{R}^d $ and 
\begin{align*}&
\PhiN(\x) = - \Nd \, \log J(\varpi^{-1} (\x/\nN ^{\frac{1}{d}}) )
, \\&
\PsiN(\x,\y)= 
\Nd \,\Psi_{\mathbb{S}} \big(
\varpi^{-1} ( {\x}/{\nN ^{\frac{1}{d}}}) ,\,
\varpi^{-1} ( {\y}/{\nN ^{\frac{1}{d}}} ) \big)
.\end{align*}Note that, by construction, the measure $\mu$ is invariant under the full Euclidean group $E(d)$, i.e., under all translations, rotations, and reflections. 
The resulting diffusion $ \X $ is invariant under the full Euclidean group $ E(d) $.

\end{example}
\begin{remark}
The construction of the Coulomb RPFs can be extended to homogeneous spaces and,
more generally, to compact or weighted Riemannian manifolds,
including models beyond the Gaussian, periodic, and spherical cases.
These generalizations will be developed in future work.
\end{remark}

\section{Construction of the logarithmic derivative $\dmu $ of $\mu $}\label{s:2}

In this section we establish the existence of the logarithmic derivative $\dmu $ of $\mu $.
The section is organized into four parts.

In \ssref{s:2A}, we prove a uniform $L^2 $-bound for the logarithmic derivatives $\dN $ of the $N $-particle systems.
In \ssref{s:2B}, we derive a uniform $L^2 $-bound for the logarithmic derivatives of $\muRN $.
In \ssref{s:2C}, we prove the strong convergence of the density functions of $\muRN $ as $\nN \to \infty $ for each fixed $\rR \in \N $.
Finally, in \ssref{s:2D}, we construct the logarithmic derivative $\dmu $ of $\mu $.

All results in \sref{s:2} are proved under Assumptions \As{A1}--\As{A3}.

%GTP]

\subsection{$L^2$-uniform bound of the logarithmic derivative $\dN$ of $\muN$} \label{s:2A}
Let $ \muN $ be as in \As{A1}--\As{A3}. 
Let $\muNone $ denote the one-reduced Campbell measure associated with $\muN $, i.e., the specialization of the definition \eqref{:10v} (given for a general RPF $\nu $) to the case $\nu =\muN $.

Let $ \mN $ be the labeled density function of $ \muN $ given by \eqref{:10e}. 
Let 
\begin{align}\label{:21w}&
\dN (\xs ) 
= 
\begin{cases}
\nablax \log \mN (\x , s^1,\ldots , s^{\nN - 1}) , &\x \in \ON , \ \sss (\Rd ) = \nN - 1 
,\\
0, & 
\text{otherwise}
,\end{cases}
\end{align}
where $ \sss = \sum_{i=1}^{\nN - 1 } \delta_{\si } $. 
By \eqref{:10e} and \eqref{:21w}, % we have 
\begin{align}\label{:21x}&
\dN (\xs ) = - \beta \Big( \nabla \PhiN ( \x ) + \sum_{i=1}^{\nN -1} \nablax \PsiN ( \x , \si )
\Big) 
,\quad \x \in \ON 
.\end{align}

%GTP[
\begin{lemma} \label{l:21} 
For any $ \h \in \CziON \ot \dbb $, 
\begin{align}\notag
- \int _{\RdSS } \h \, ( \mathrm{div}_x \dN ) \, d\muNone =
\beta \int _{\RdSS } \h \Big( \Delta \PhiN 
+ \sum_{i=1}^{\nN -1 } \psiN ( \x , \si ) \Big) d\muNone .
\end{align}
\end{lemma}
\begin{proof}
By \As{A1}, we have $ \PhiN (\x ) = \infty $ for $ \x \notin \oL{\ON } $, and
\begin{align}\label{:21f}&
\Delta _{\x } \PsiN ( \x , \y )
= - \cref{;1a}^{\nN } \delta_{\y } ( \x ) + \psiN (\x , \y )
\quad \text{on } \ON \times \ON 
.
\end{align}Moreover, by \As{A1} and \As{A2}, the density $ \mN $ vanishes on collision configurations,
that is, $ \mN (\mathbf{x}) = 0 $ whenever $ x^i = x^j $ for some $ i \ne j $. 
Hence the singular delta terms in \eqref{:21f} do not contribute, and we obtain
\[
 \int _{\RdSS } \h \sum_{i=1}^{\nN -1 } \Delta_{\x} \PsiN ( \x , \si ) \, d\muNone
=
 \int _{\RdSS } \h \sum_{i=1}^{\nN -1 } \psiN ( \x , \si ) \, d\muNone .
\]
Combining this with \eqref{:21x}, the assertion follows.
\PFEND

%GTP]

\smallskip 

For each $ \qQ \in \N $, let $ \map{\fQ }{\Rd }{\R }$ be a cut-off function such that 
\begin{align}\notag &% \label{:22z} &
\fQ \in \CziRd , \quad 
\fQ ( \x ) = 
\begin{cases}
1 & \x \in \SQ 
\\
0 & \x \notin \SQQ 
\end{cases}
, \ 
 0 < \fQ (\x ) < 1 , \ \x \in \SQQ \backslash \oLSQ %,\ \quad 
,\\ \label{:22z}& 
 \lvert \nabla \log \fQ \rvert \in \CziRd , \quad 
0 < \lvert \nabla \log \fQ (\x ) \rvert , \ \x \in \SQQ \backslash \oLSQ 
.\end{align}
In equation \eqref{:22z}, we adopt the convention such that 
\begin{align*}&
\text{$ \lvert \nabla \log \fQ \rvert = 0 $ on $ \x \notin \SQQ $}, \quad 
\lim_{\x \to \partial \SQQ }\lvert \nabla \log \fQ (\x ) \rvert = 0 
.\end{align*}

Let $ \Ct \label{;22b} $, $ \Ct \label{;22c} $, and $ \Ct \label{;22z} $ be the quantities defined by 
\begin{align} \notag & 
 \cref{;22b} (\nN ) = 
\int_{\RdSS } \fQ \lvert \nabla \log \fQ \rvert^2 \muNone (d\x d\sss )
,\\ & \notag 
 \cref{;22c} (\nN ) = 
\int_{\ON \ts \sSS }\fQ \Big( \Delta \PhiN + \sum_{i=1}^{\nN -1} \psiN (\x , \si ) \Big) \muNone (d\x d\sss )
,\\ \label{:22s}&
\cref{;22z} = 
\limsupi{\nN } \frac{1}{4} \big\{ \sqrt{\cref{;22b} (\nN ) } + \sqrt{\cref{;22b} (\nN ) + 4 \beta \cref{;22c} (\nN ) } \big\}^2 
.\end{align}

\begin{lemma} \label{l:2!} 
For each $ \qQ $, the quantity $ \cref{;22z} $ is finite. 
\end{lemma}

\begin{proof} 
By \eqref{:22z}, $ \cref{;22b} (\nN ) $ are nonnegative and bounded. 

By \eqref{:11g} and \eqref{:11*}, $ \cref{;22c} (\nN ) $ are non-negative. 
By \eqref{:11e} and \eqref{:22z}, 
\begin{align}\label{:11a}&
\limsupi{\nN } \sup_{\x \in \ON } \fQ (\x ) \Delta \PhiN (\x ) < \infty 
.\end{align}
By \eqref{:11*}, for $ \muNone $-a.e.\,$ (\xs )$, 
\begin{align}\label{:11A}&
\Big\lvert \sum_{i=1}^{\nN -1} \psiN (\x , \si ) \Big\rvert \le \frac{\nN - 1 }{\nN } \cref{;32!} \le \cref{;32!}
\end{align}
By \eqref{:11a} and \eqref{:11A}, $ \cref{;22c} (\nN ) $ are bounded. 
 Hence by \eqref{:22s}, $ \cref{;22z} $ is finite. 
\PFEND

\begin{lemma} \label{l:22} 
For each $ \qQ \in \N $, 
\begin{align}&\label{:22a}
\limsupi{\nN } \int_{\RdSS } \fQ |\dN |^2 \muNone (d\x d\sss )\le \cref{;22z}
.\end{align}
\end{lemma}
\PF 
Let $ \SQ \subset \ON $. 
Applying \eqref{:10w} to $ \muN $ yields 
\begin{align}&\notag %\label{:}&
\int_{\RdSS }\fQ |\dN |^2 \muNone (d\x d\sss )= 
\int_{\RdSS } \big( \fQ \dN , \dN \big)_{\Rd } \muNone (d\x d\sss )
\\ \label{:21y}& = 
 -\int_{\RdSS } \mathrm{div}_x\, (\fQ \dN ) \muNone (d\x d\sss ) 
%\quad \text{ by \eqref{:21y}}
\end{align}
Using \lref{l:21}, we rewrite the last term as 
\begin{align}\notag 
&= - 
 \int_{\RdSS } \big( \nablax \fQ , \dN \big)_{\Rd } \muNone (d\x d\sss )
 - 
\int_{\RdSS } \fQ ( \mathrm{div}_x \dN ) \muNone (d\x d\sss )
\\ \notag& = 
- \int_{\RdSS } \fQ 
 \big( \nablax \log \fQ , \dN \big)_{\Rd }
% \sumpd ( \partial_p \log \fQ ) \dpN 
\muNone (d\x d\sss )
 \\\label{:22i}& + 
\beta \int_{\RdSS }\fQ \big( \Delta \PhiN +\sum_{i=1}^{\nN -1} \psiN (\x , \si ) \big) 
\muNone (d\x d\sss )
%(d\x d\sss )
% \ \text{ by \lref{l:21}} 
.\end{align}

Let $ x_{\nN } = \{ \int_{\RdSS }\fQ |\dN |^2 \muNone (d\x d\sss )\}^{1/2} $. By \eqref{:22s}, \eqref{:21y}, \eqref{:22i}, $ \fQ \ge 0 $, and the Cauchy--Schwarz inequality, we obtain 
\begin{align}\label{:22k}&
x_{\nN } ^2 \le \cref{;22b} (\nN ) x_{\nN } + \beta \cref{;22c} (\nN ) 
.\end{align}
Solving \eqref{:22k} yields 
\begin{align*}&
 \x _{\nN } \le \sqrt{\cref{;22b} (\nN ) } + \sqrt{\cref{;22b} (\nN ) + 4 \beta \cref{;22c} (\nN ) } 
.\end{align*}
Hence, we obtain \eqref{:22a} from \eqref{:22s}. 
\PFEND

\subsection{$ L^2$-uniform bound of logarithmic derivative of $ \muRN $}\label{s:2B}

We continue to work under Assumptions \As{A1}--\As{A3}.

 Let $ \muRN = \muN \circ \pioLR ^{-1}$. 
We regard $ \muRN $ as a probability measure on $ \oLSSR $, which is the configuration space over $ \oLSR $. 

Let $ \dRN $ be the logarithmic derivative of $ \muRN $, that is, for each $ \FRone $-measurable $ \h \in \GRone $, 
$ \FRone = \mathscr{B}(\oLSR ) \ts \FR $ and $ \FR = \sigma [\pioLR ]$
\begin{align}\label{:23r}&
 \int _{\sSsS } \h \dRN \muRNone (d\x d\sss ) = - \int _{\sSsS } \nablaxh \muRNone (d\x d\sss ) 
.\end{align}
Here $ \muRNone $ is the one-reduced Campbell measure of $ \muRN $. 

%G[---

We regard $ \dRN $ as an $ \FRone $-measurable function on $ \Rd \ts \sSS $
satisfying
$ \dRN (\xs ) = 0 $ for $ \x \notin \oLSR $ and
$ \dRN (\xs ) := \dRN (\x , \pioLR (\sss ) ) $.
We regard $ \muRNone $ as a measure on
$ ( \Rd \ts \sSS , \mathscr{B} ( \Rd ) \ts \FR ) $
with $ \muRNone ( \SRc \ts \sSS ) = 0 $.
Then \eqref{:23r} can be written as follows: for each
$ \FRone $-measurable $ \h \in \GRone $,
\begin{align}\label{:23s}&
\int_{\RdSS } \h \dRN \muNone (d\x d\sss )
=
- \int_{\RdSS } \nablax \h \muNone (d\x d\sss )
.\end{align}
We often use the latter interpretation of $ \dRN $,
as it allows us to regard $ \dRN (\xs ) $ as a function defined on
$ \Rd \ts \sSS $ independently of $ \rR $.
%G]
\begin{remark}\label{r:21} 
If we replace $ \GRone $ by $ C^{\infty} (\oLSR ) \ot \dcb $ as the space of test functions in \eqref{:23r}, then a local-time-type singular drift appears in the representation of $ \dRN $, indicating the reflection of particles on $\partial \SR $. 
\end{remark}

 Let $ \mN $ be as in \eqref{:10e}. 
Then $ \muRN $ has the labeled density $ \mRN $ defined on 
 $ \sqcup_{k=0}^{\nN } \oLSRk $ such that, for $ 0 \le k \le \nN $ and $ k + l = \nN $, 
\begin{align} \notag %\label{:23o}
\mRN ( \sR ) &= 
\frac{1}{\mathscr{Z}_{\rR }^{\nN }} \int_{ \oLSRxCl } \mN ( \sR , \tR ) d\tR 
,\quad \sR \in \oLSRk 
%	,\quad \text{ for } \sR \in \oLSRk 
.\end{align}

\noindent 
Here $ \mathscr{Z}_{\rR }^{\nN } $ is the normalization. 
Because $ \mN ( \mathbf{\x } ) $ is a symmetric function of $ \mathbf{\x } = (\x ^i)_{i=1}^N $, $ \mRN $ is well defined. 

For $ (\xs ) \in \SR \ts \oLSS $ and $ \sss (\oLSR ) \le \nN - 1 $, $ \dRN $ satifies 
\begin{align}\label{:23q}&
\dRN (\xs ) = \nablax \log \mRN (\x , \sR ) 
%,\quad \sss (\oLSR ) \le \nN - 1 
,\end{align}
and by $ \dRN (\xs ) = 0 $ for $ \sss (\oLSR ) \ge \nN $, where $ \sss = \ulab (\sR ) $. 

\begin{definition}\label{d:2M}
Let $ 0 \le k \le \nN - 1 $ and $ k + l = \nN - 1 $. 
Let $ \sss = \ulab (\sR )$ and $ \mathsf{t} = \ulab (\tR ) $. 
We set, for $ \sss \in \oLSSR $ satisfying $ \sss (\oLSR ) = k $, 
\begin{align}\label{:23t}& 
 \langle f \rangle _{\rR }^{\nN } (\xs ) = 
 \frac{1}{ \int _{\oLSRxCl } \mN ( \x , \sR , \tR ) d\tR }
 \int _{\oLSRxCl } ( \check{f} \mN ) ( \x , \sR , \tR ) d\tR 
.\end{align}
Here $ \map{f}{\oLSR \ts \sSS ^{\nN -1 } } {\R }$, $ \sSS ^{\nN - 1 } = \{ \mathsf{u} \in \sSS ; \mathsf{u} (\Rd ) = \nN - 1 \} $, and 
$ \map{\check{f}}{\oLSR \ts (\Rd )^{\nN - 1 }}{\R }$ is the function satisfying that 
$ \check{f} (\x , \sR , \tR ) = f ( \xst ) $ and that $ \check{f} (\x , \sR , \tR ) $ is symmetric in $ ( \sR , \tR ) $. 
\end{definition}

The following identity plays an important role in our analysis. 
\begin{lemma} \label{l:23}
Let $ \dRN $ and $ \dN $ be as in \eqref{:23q} and \eqref{:21w}, respectively. Then 
\begin{align}\label{:23a}&
 \dRN (\xs ) = \langle \dN \rangle _{\rR }^{\nN } (\xs ) 
.\end{align}
\end{lemma}
\PF 
Let $ k + l = \nN - 1 $. Let $ \sss $, $ \sR $, and $ \tR $ be as in \eqref{:23t}. 
Then 
\begin{align} \notag 
& \dRN (\xs ) =
\frac{1}{ \int _{\oLSRxCl } \mN ( \x , \sR , \tR ) d\tR } \nablax \int _{\oLSRxCl } 
 \mN ( \x , \sR , \tR ) d\tR 
\ \text{ by \eqref{:23q}}
\\ \notag & = 
\frac{1}{ \int _{\oLSRxCl } \mN ( \x , \sR , \tR ) d\tR }
\int _{\oLSRxCl } \nablax \mN ( \x , \sR , \tR ) d\tR 
\\ \notag 
& = 
\frac{1}{ \int _{\oLSRxCl } \mN ( \x , \sR , \tR ) d\tR }
 \int _{\oLSRxCl } \{ \nablax \log \mN ( \x , \sR , \tR ) \} \mN ( \x , \sR , \tR ) d\tR 
\\ \notag & = 
\langle \dN \rangle _{\rR }^{\nN } (\xs ) 
\ \text{ by \eqref{:21w}, \eqref{:23t}}
.\end{align}
Thus, we obtain \eqref{:23a}. 
\PFEND

We obtain a uniform $ L^2$-bound of the logarithmic derivatives $ \dRN $ of $ \muRN $. 
As we explained around \eqref{:23s}, we consider $ \dRN $ as a function on $ \RdSS $. 
\begin{lemma} \label{l:24} 
Let $ \fQ $ be as in \eqref{:22z}. For $ \rR \ge \qQ + 1 $, 
\begin{align}&\label{:24a}
\limsupi{\nN } \int_{\RdSS } \fQ |\dRN |^2 \muNone (d\x d\sss ) \le \cref{;22z} 
.\end{align}
\end{lemma}
\PF 
Recall that $ \muNone ( d\x d\sss ) = \rhoNone (\x ) \muxN (d\sss ) d\x $. Then 
\begin{align}\notag 
\int_{\RdSS }
\fQ |\dRN |^2 & \muNone (d\x d\sss )= \int_{\RdSS }
\fQ 
\big| 
\langle \dN \rangle _{\rR }^{\nN } \big|^2 \muNone (d\x d\sss )
\quad \text{ by \lref{l:23}}
\\\notag & \le 
\int_{\RdSS } \fQ \langle |\dN |^2 \rangle _{\rR }^{\nN } \muNone (d\x d\sss ) 
\quad 	
\text{ by Jensen's inequality}
\\\notag & = 
\int_{\RdSS } \fQ |\dN |^2 
\muNone (d\x d\sss )
\quad %&& 
\text{ by \eqref{:22z}, \eqref{:23t}}		% 0\le \fQ \le 1 
.\end{align}
Hence, we obtain \eqref{:24a} from \lref{l:22}. 
\PFEND

\subsection{Strong convergence of density functions of $ \muRN $} \label{s:2C}

We continue to work under Assumptions \As{A1}--\As{A3}.

Let $ \LambdaR $ be the Poisson RPF with intensity measure $ 1_{\SR }d\x $. 
Let $ \LambdaRone $ be the one-reduced Campbell measure of $ \LambdaR $. 

We denote by $ \0 $ the Radon--Nikodym density of $ \muRNone $ with respect to $ \LambdaRone $. 
The existence of $ \0 $ is obvious from \eqref{:10e}. 
Let 
\begin{align}\notag 
& \oLSSRkone = \{ (\xs ) \in \sSsS ;\, \sss (\oLSR ) = k \} 
,\\\label{:25z}
& \mRkNone = \0 1_{\oLSSRkone } , \quad 
 \mRklNone = \mRkNone \wedge \hhhl 
.\end{align}
We regard $ \oLSSRkone $ as a subset of $ \oLSRSSR $, where 
$ \oLSSR $ is the configuration space over $ \oLSR $. 
Let $ \muRone $ be the one-reduced Campbell measure of $ \muR = \mu \circ \pioLR ^{-1} $. 
\begin{lemma} \label{l:25} 
Let $ \rR \in \N $ such that $ \rR \ge \qQ + 1 $. 
Let $ k , l \in \{ 0 \} \cup \N $. 

\noindent \thetag{i}
$ \{ \fQ \mRklNone \}_{\n \in \N } $ is relatively compact in $ L^2 (\LambdaRone ) $. 

\noindent 
\thetag{ii} 
Let $ \Nn $ be as in \As{A2}. 
There exist functions $ \mRklone $ satisfying 
\begin{align} \label{:25a}
\fQ \mRklone &= 
\limin \fQ \mRklNnone \text{ in } L^2 (\LambdaRone ) 
,\\\label{:25b}
\mRklone &= \mRkllone \wedge \hhhl 
.\end{align}

\noindent 
\thetag{iii} 
$ \muRone $ has the Radon--Nikodym density 
$ \mRone = ({d\muRone }/{d \LambdaRone }) $, which is defined by 
$ \mRone \wedge l = \mRklone $ on $ \oLSSRkone $. 
\end{lemma}

\begin{proof} 
Note that $ \mRklNone (\xs ) $ is symmetric in $ (x, \sR )$ when regarded as a function on $ \oLSR ^{k+1} $ by $ \sss = \ulab (\sR )$. Let 
\begin{align} &\notag 
 \DDDRone [ \h ] (\xs )= \half \Big( \sum_{\x \in \SR }\vert \nablax \h (\xs ) \vert^2 + 
 \sum_{\siSR }\vert \nablasi \h (\xs ) \vert^2 \Big) 
.\end{align}
Using the permutation invariance of $ \mRklNone $ in $ (x, \sR )$ and noting that $ k + 1 $ particles exist in $ \oLSR $ for $ (\xs ) \in \oLSSRkone $, we obtain 
\begin{align*}
& \int_{\sSsS }\DDDRone [ \fQ \mRklNone ] \,  d\LambdaRone
\\& \le 
\int_{\sSsS } 
2\Big(\DDDRone [ \fQ ] \big( \mRklNone \big) ^2 + 
\fQ ^2 \DDDRone [ \mRklNone ] 
\Big) \,  d\LambdaRone
\\& = 
\int_{\sSsS } \lvert \nabla \fQ \rvert ^2 \big( \mRklNone \big) ^2 
 \,  d\LambdaRone+ 2 \int_{\sSsS } 
\fQ ^2 \DDDRone [ \log \mRklNone ] \big( \mRklNone \big) ^2 \,  d\LambdaRone
\\&
 = 
\int_{\sSsS } \lvert \nabla \fQ \rvert ^2 \big( \mRklNone \big) ^2 
d\LambdaRone + \int_{\sSsS } 
 \fQ ^2 ( k + 1 ) |\dN |^2 \big( \mRklNone \big) ^2 \,  d\LambdaRone
\\ \notag & \le 
\hhhl ^2 \int_{\sSsS } \lvert \nabla \fQ \rvert ^2 \,  d\LambdaRone + 
 ( k + 1 ) \hhhl \int_{\sSsS }\fQ \lvert \dN \rvert ^2 \mRklNone \,  d\LambdaRone
\quad \text{by \eqref{:22z}, \eqref{:25z}}
.\end{align*}
Hence, combining this with \eqref{:22z} and \lref{l:22} yields the following: 
\begin{align}\label{:25i}&
\limsupi{\nN } \int_{\sSsS }\DDDRone [ \fQ \mRklNone ] \,  d\LambdaRone
< \infty 
.\end{align} 

As before, we regard $ \fQ \mRklNone (\xs ) $ as a function on $ \SRkk $. 
Note that $ \fQ \mRklNone = 0 $ on $ ( \partial \SR ) \ts \SRk $ from \eqref{:22z} and $ \rR \ge \qQ + 1 $. 
Hence, from the Rellich embedding theorem and \eqref{:25i}, there exists $ q >1 $ such that 
$ \{\fQ \mRklNone \}_{\nN \in \N } $ is relatively compact in $ L^q (\oLSR ^{k+1} , d \mathbf{x} )$, 
 where we take $ 1< q < n / (n-1) $ for $ n = d ( k+1 ) $. 

According to this and \eqref{:22z}, $ \{\fQ \mRklNone \}_{\n \in \N } $ is relatively compact in $ L^q ( \LambdaRone ) $. 
From \eqref{:25z}, $ \{ \fQ \mRklNone \}_{\nN \in \N } $ is bounded in $ L^{\infty} (\LambdaRone ) $. 
Hence, $ \{\fQ \mRklNone \}_{\nN \in \N } $ is relatively compact in $ L^2 ( \LambdaRone ) $. 
This proves \thetag{i}. 

%G[

From \thetag{i} and a diagonal argument, there exists a subsequence of $ \{ \Nn \} $,
denoted by the same symbol, and a limit $ \{ \widetilde{\m }_{ \qQ , \Rkl }^{ [ 1 ] } \} $ such that
\begin{align}\label{:25j}
\limi{\n } \fQ \mRklNnone =
\widetilde{\m }_{ \qQ , \Rkl }^{ [ 1 ] } 
\quad \text{in } L^2 ( \LambdaRone ) 
\end{align}
for all $ \rR \ge \qQ + 1 $ and $ k , l \in \{ 0 \} \cup \N $. 

From the weak convergence
$ \limi{\n } \muNn = \mu $
given in \As{A2},
we obtain the uniqueness of limit points of
$ \{ \fQ \mRklNnone \}_{ \n \in \N } $.
Hence the full sequence converges to
$ \widetilde{\m }_{ \qQ , \Rkl }^{ [ 1 ] } $
in \eqref{:25j}.

In view of \eqref{:22z}, we define, for $  \qQ + 1 \le \rR $, 
\begin{align}\label{:25k}
\mRklone =
\begin{cases}
\fQ^{-1} \widetilde{\m }_{\qQ , \Rkl }^{[1]} & \text{on } \SQQ \ts \oLSSR ,\\
0 & \text{on } (\SR \backslash \SQQ ) \ts \oLSSR .
\end{cases}
\end{align}

By \eqref{:25j} and \eqref{:25k}, we obtain \eqref{:25a}.
By \eqref{:25z} and \eqref{:25a}, we obtain \eqref{:25b}.
This proves \thetag{ii}.

%GTP]%GTP[---

Using the consistency \eqref{:25b}, we define the function $\mRkone $ by
\begin{align*}
\mRkone (\xs ) \wedge \hhhl = \mRklone (\xs ) .
\end{align*}
By the decomposition $\oLSR \ts \oLSSR = \sqcup_{k=0}^{\infty} \oLSSRkone $, 
we define the function $\mRone $ by $ \mRone (\xs ) = \mRkone (\xs ) $,  $(\xs ) \in \oLSSRkone $.
We define $\muRone = \mRone \,  d\LambdaRone$.
Then $\muRone $ is the one-reduced Campbell measure of $\muR = \mu \circ \pioLR^{-1} $.
This proves \thetag{iii}.
%GTP]
\PFEND

\begin{proposition} \label{l:2'} 
\noindent \thetag{i}
$ \mu $ has a $ k $-density function $ \mRk $ on $ \oLSRk $ for each $ \rR , k \in \N $. 

\noindent \thetag{ii} 
$ \mu $ has a one-point correlation function $ \rho ^1 $. 
\end{proposition}
\begin{proof} 
Let $ \mRone (\xs ) $ be as in \lref{l:25} \thetag{iii}. 
Let 
$$ \mRone (\x , \sR ) := \mRone (\x , \ulab (\sR ) )
.$$
%-GTP[---
The restriction of $ \mRone ( \x , \sR ) $ to $ \oLSRk $ is symmetric for each $ k \in \N $.
Hence, we construct the $ k $-density function
$ \mRk $ of $ \mu $ on $ \oLSRk $ from $ \mRone $ by the formula
\begin{align}\notag &%
\mRk ( \x ^1 , \x ^2 , \ldots , \x ^k )
=
\cref{;2'f}
\, \mRone ( \x ^1 , \x ^2 , \ldots , \x ^k )
.\end{align}
The constant $ \Ct \label{;2'f} $
is determined by the normalization.
This proves \thetag{i}.

%GTP]

% 

The one-point correlation function $ \rho ^1 $ is given by the formula
\begin{align} & \notag %\label{:2(f}&
 \rho ^1 (\x ) = \wm _{\Rone } (\x ) + \sum_{j = 2 }^{\infty}
\frac{1 }{(j-1)!} \int_{\oLSR ^{j-1}}\wm _{\rR , j } (\x , \x ^2, \ldots,\x ^{j}) d\x ^{2}\cdots d\x ^{j}
.\end{align}
By \eqref{:11)}, $ \int_{\sSS } \sss (\SR ) \mu (d\sss ) < \infty $. 
Hence, the right-hand side is finite. 
Hence, $ \rho ^1 $ is the one-point correlation function of $ \mu $. 
This proves \thetag{ii}. 
\PFEND

\subsection{The logarithmic derivative $ \dmu $ of $ \mu $} \label{s:2D} 

We continue to work under Assumptions \As{A1}--\As{A3}.

This subsection constructs the logarithmic derivative $\dmu $ of $\mu $. 
The construction proceeds in two steps.
First, we define the logarithmic derivative $\dR $ of $\muR $ as the limit of the logarithmic derivatives of $\muR^{\Nn } $ as $\Nn \to \infty $. 
Next, we define $\dmu $ as the limit of $\dR $ as $R \to \infty $.
%GTP[---

Let $ \oLSSRkone $ and $ \mRNone $ be as in \eqref{:25z}. 
Let $ \oLSSRklNone = \{ (\xs ) \in \oLSSRkone \, ; \, 0 < \mRNone ( \xs ) < \hhhl \} $. 
We define
\begin{align}\label{:26h}
\dRklN ( \xs )
&=
\begin{cases}
\nablax \log \mRNone ( \xs ) & (\xs ) \in \oLSSRklNone ,
\\
0 & \text{otherwise}.
\end{cases}
\end{align}
We define the values of $ \dRklN ( \xs ) $ on the boundary $ \partial \SR \ts \oLSSR $ 
as limits of the values from the interior $ \SRSSR $. 

Let $ \mRone $ be as in \lref{l:25}. 
Replacing $ \mRNone $ by $ \mRone $, we define $ \oLSSRklone $ and $ \dlogRkl $ similarly.
%GTP]
Let $ \{ \Nn \} $ be as in \As{A2}. 
\begin{lemma} \label{l:2"}
For each $ k , l \in \zN $ and $ \rR \in \N $, 
\begin{align}\label{:26k}
\dlogRkl ( \fQ \mRklone )^{1/2} 
&=\limin \dRklNn ( \fQ \mRklNnone )^{1/2} 
\ \text{ weakly in $ L^2 (\LambdaRone )$} 
,\\ \label{:26l}
 \dlogRkl (\xs )& = \dlogRkll (\xs ) \quad \text{ for } (\xs ) \in \oLSSRklone 
.\end{align}
\end{lemma}
\begin{proof} 
From \lref{l:24}, \eqref{:25z}, and \eqref{:26h}, we easily deduce that 
\begin{align}\notag %\label{:}&
\{ \dRklNn ( \fQ \mRklNnone )^{1/2} \} _{\n }
\end{align}
is bounded in $ L^2 (\LambdaRone )$. 
Hence, we have a subsequence of $ \{ \Nn \} $, denoted by the same symbol, and a limit $ \widehat{\dlog}_{\Rkl } $ such that 
\begin{align}\label{:26m}
\widehat{\dlog}_{\Rkl } = \limin \dRklNn ( \fQ \mRklNnone )^{1/2} & \quad \text{ weakly in $ L^2 (\LambdaRone )$} 
.\end{align}
From \lref{l:25}, we deduce 
\begin{align} \label{:26n}
( \fQ \mRklone )^{1/2} = \limin ( \fQ \mRklNnone )^{1/2}&
\quad \text{ in } L^4 (\LambdaRone ) 
.\end{align}
Hence, from \eqref{:26m} and \eqref{:26n}, we obtain for any $ \h \in \CziSR \ot \dcb $
\begin{align} \notag &
 \int_{\sSsS }\widehat{\dlog}_{\Rkl } ( \fQ \mRklone )^{1/2} \h \,  d\LambdaRone
\\ \notag = &
 \limin \int_{\sSsS }
\Big( 
\dRklNn ( \fQ \mRklNnone )^{1/2}
\Big) ( \fQ \mRklNnone )^{1/2} \h \,  d\LambdaRone
\\ \notag
=&
\limi{\n }
\int_{ \sSsS } 
\dRklNn \,
\big( \h \fQ \big) \mRklNnone \, d \LambdaRone
\\ \notag =&
\limi{\n }
\Big(
-
\int_{ \sSsS }
\nablax ( \h \fQ )
\, \mRklNnone
\, d \LambdaRone
\Big)
\quad \text{by \eqref{:26h}}
\\ \notag
=&
-
\int_{ \sSsS }
\nablax ( \h \fQ )
\, \mRklone
\, d \LambdaRone 
\quad \text{by \eqref{:22z}, \eqref{:25a}}
\\ \label{:26u}
=&
\int_{ \sSsS }
\dlogRkl \, \h \, \fQ \, \mRklone 
\, d \LambdaRone
.\end{align}
Hence, $ \widehat{\dlog}_{\Rkl } = \dlogRkl ( \fQ \mRklone )^{1/2} $. 
Together with \eqref{:26m}, this yields \eqref{:26k}.
%G]

%GTP[--

From \eqref{:26h}, $ \dRklN $ satisfies $ \dRklN ( \xs ) = \dRkllN ( \xs ) $
for $ ( \xs ) \in \oLSSRklNone $ and for all $ l \in \N $.
This property is inherited by $ \dlogRkl $
from \eqref{:26m}--\eqref{:26u}.
Hence, we obtain \eqref{:26l}.
\PFEND

%GTP]

%G[---
Let $ \mathsf{M}_{ \rR }^{ [ 1 ] } = \{ (\xs ) \in \oLSRSSR \,;\, 0 < \mRone ( \xs ) < \infty \}$. 
We define 
\begin{align} \label{:2)x} &
\dR ( \xs ) = 
\begin{cases}
\nablax \log \mRone ( \xs )  & (\xs ) \in  \mathsf{M}_{ \rR }^{ [ 1 ] }
,\\
0 &\text{otherwise}. 
\end{cases}
\end{align}
We define the values of $ \dR ( \xs ) $ on the boundary $ \partial \SR \ts \oLSSR $ 
as limits of the values from the interior $ \SRSSR $. 
\begin{lemma} \label{l:2)} 
For all $ \rR \ge \qQ + 1 $, 
\begin{align} \label{:2)z} 
\limsupi{\nN } \int_{\sSsS }& \lvert \fQ \dRN \rvert ^2 \mRNone \,  d\LambdaRone\le \cref{;22z} 
,\\ \label{:2)a}
 \dR (\fQ \mRone )^{1/2} &= \limin \dRNn (\fQ \mRNnone )^{1/2} 
\quad \text{ weakly in $ L^2 ( \LambdaRone ) $}
,\\\label{:2)b} 
(\fQ \mRone )^{1/2} &= \limin (\fQ \mRNnone )^{1/2} 
\quad \text{ in $ L^2 ( \LambdaRone ) $}
.\end{align}
\end{lemma}
\begin{proof}
Recall that $ d \muRNone = \mRNone d \LambdaRone $. By \lref{l:24}, we obtain \eqref{:2)z}. 

From \eqref{:2)z}, $ \{ \dRN (\fQ \mRNone )^{1/2} \}_{\nN \in \N }$ is relatively compact under the weak convergence in $ L^2 ( \LambdaRone ) $. 
We easily see 
\begin{align}\label{:26o}
\dR ( \xs )
=
\limil
\sum_{ k = 0 }^{ \infty }
\dlogRkl ( \xs )
,\quad
\sum_{ k = 0 }^{ \infty }
\lvert \dlogRkl ( \xs ) \rvert 
\le 
\lvert \dR ( \xs ) \rvert
.\end{align}
From \lref{l:2"} and \eqref{:26o}, the limit points are unique and coincide with $ \dR ( \fQ \mRone )^{1/2} $. This proves \eqref{:2)a}. 

By \As{A2} \thetag{iii}, we have $ \mu =\limin \muNn $ weakly. Hence by \eqref{:11+}, 
\begin{align*}&
\int_{\sSsS } \fQ \mRone d \LambdaRone = 
\limin \int_{\sSsS } \fQ \mRNnone d \LambdaRone < \infty 
.\end{align*}
Hence, from \lref{l:25},
\begin{align}\notag &
\limin \int_{\sSsS } 
\big\lvert ( \fQ \mRone )^{1/2} - ( \fQ \mRNnone )^{1/2} \big\rvert ^2 d \LambdaRone 
\\ =\notag & 
\limin \int_{\sSsS } \Big( \fQ \mRone + \fQ \mRNnone 
-2 ( \fQ \mRNnone )^{1/2} ( \fQ \mRone )^{1/2} \Big) d \LambdaRone = 0 
.\end{align}
This proves \eqref{:2)b}. 
\PFEND
\begin{lemma} \label{l:26} 
For all $ \rR \ge \qQ + 1 $, 
\begin{align} \label{:26a} &
\int_{\sSsS } | \fQ \dR |^2 d\muRone \le \cref{;22z} 
,\\\label{:26'} &
 \int_{\sSsS } \h \dR d \muRone = \limin \int _{\sSsS } \h \dRNn d \muRNnone 
,\\\label{:26c} &
 \int_{\sSsS } \h \dR d \muRone = - \int _{\sSsS } \nablaxh d \muRone 
\end{align}
for all $ \FRone $-measurable $ \h = (\fQ \hat{\f })\ot \g \in \dcbone $, where $ \FRone = \mathscr{B}(\oLSR ) \ts \sigma(\pioLR ) $. 
\end{lemma}
\PF 
From \lref{l:25}, $ d \muRone = \mRone d \LambdaRone $. Then by \lref{l:2)}, 
\begin{align} \notag 
 \int _{\sSsS } | \fQ \dR |^2 d \muRone & 
= \int _{\sSsS } | \fQ \dR |^2 \mRone d \LambdaRone 
\\ \notag &
\le \liminfn \int _{\sSsS } 
\lvert \fQ \dRNn \rvert ^2 \mRNnone \,  d\LambdaRone\quad \text{ by \eqref{:2)a}} 
\\ & \notag 
\le \cref{;22z} \quad \text{ by \eqref{:2)z}}
.\end{align}
This implies \eqref{:26a}. By \lref{l:2)}, 
\begin{align} \notag &
\int _{\sSsS } \h \dR d \muRone 
= \int _{\sSsS } \h \dR \mRone \,  d\LambdaRone
\\ \notag = & 
\int _{\sSsS } (\hat{\f }\ot \g ) \fQ ^{1/2}
 \dR (\mRone )^{1/2} (\fQ \mRone )^{1/2} \,  d\LambdaRone
\\\notag =&
\limin \int _{\sSsS } (\hat{\f }\ot \g ) \fQ ^{1/2} \dRNn (\mRNnone )^{1/2} (\fQ \mRNnone )^{1/2}d\LambdaRone 
\ \text{ by \eqref{:2)a}, \eqref{:2)b}}
\\\notag = &
\limin \int _{\sSsS } \h \dRNn \mRNnone \,  d\LambdaRone
\quad \text{ by } \h = (\fQ \hat{\f })\ot \g 
\\ \label{:26b} =& 
 \limin \int _{\sSsS } \h \dRNn d \muRNnone 
.\end{align}
This yields \eqref{:26'}. From \eqref{:23s} and \As{A2}, we obtain 
 \begin{align} \notag 
\limin \int _{\sSsS } \h \dRNn d \muRNnone 
 = & - \limin \int _{\sSsS } \nablaxh d \muRNnone 
\quad \text{ by \eqref{:23s}}
\\ \label{:26t} = & 
 - \int _{\sSsS } \nablaxh d \muRone 
\quad \text{ by \As{A2} \thetag{ii}}
.\end{align}
Combining \eqref{:26b} and \eqref{:26t}, we obtain \eqref{:26c}. 
\PFEND

A function satisfying \eqref{:26c} for all $ \h $ in $ \dcbone $ is called a logarithmic derivative of $ \muR $. 
We readily see that $ \dR $ is such a function. 
We now construct the logarithmic derivative $ \dmu $ of $ \mu $ using the consistency of $ \dR $. 

In the following theorem, we regard $ \dR $ as a function on $ \RdSS $, denoted by the same symbol, and define
$ \dR (\xs ) = 0 $ for $ \x \notin \SR $, and $ \dR (\xs ) := \dR ( \x , \pi _{\oLSR } ( \sss ) ) $.
\begin{proposition} \label{l:27} 
There exists a limit of $ \dR $ for $ \muone $-a.e.\,and in $ \Llocmone $, which we denote by $ \dmu $. 
The limit $ \dmu $ is the $ \bullet $-logarithmic derivative of $ \mu $. 
\end{proposition} 
%G[	--
\PF 
By \eqref{:26c} and $ \muR = \mu \circ \pioLR ^{-1} $, 
we deduce that, for each $ \FRone $-measurable $ \h = \f \ot \g $ in $ \dcbone $ with $ \f $ in $ \CziSR $, 
\begin{align} \label{:27f} 
\int_{\RdSS } \h \dR \, d \muone 
= 
- \int_{\RdSS } \nablax \h \, d \muone 
= 
\int_{\RdSS } \h \dRR \, d \muone 
. \end{align}
For a non-negative $ \f $ satisfying $ \int_{\Rd } \f \, d \x = 1 $, 
consider the probability measure $ \f \ot 1 \, \muone $. 
By \eqref{:27f}, $ \dR \f $ is an $ \{ \FRone \} $-martingale under $ \f \ot 1 \, \muone $. 
Hence the first assertion follows from the martingale convergence theorems and \eqref{:26a}. 
 
By the first assertion and \eqref{:27f}, for each $ \h $ in $ \dcb $, 
\begin{align} \label{:27b} 
\int_{\RdSS } \h \dmu \, d \muone 
&= \limi{\rR } \int_{\RdSS } \h \dR \, d \muone = 
- \int_{\RdSS } \nablax \h \, d \muone 
. \end{align} 
By \lref{l:39}, \eqref{:27b} holds for all $ \h $ in $ \dbb $. 
Hence $ \dmu $ is the $ \bullet $-logarithmic derivative of $ \mu $. 
\PFEND 
%G]

\section{Explicit Formulas for $\dmu $: Proof of \tref{l:11}}\label{s:3}
%G[---
In this section, we present explicit formulas for the logarithmic derivative 
$ \dmu $ and complete the proof of \tref{l:11}. 
Although \pref{l:27} established the existence of a logarithmic derivative 
$ \dmu $, from which the existence of a weak solution to the ISDE can be deduced, 
it does not address the uniqueness of the solution 
nor the existence of a strong solution. 
To resolve these issues, we need an explicit representation of $ \dmu $. 
%G]

Unless stated otherwise, all results in \sref{s:3} are proved under Assumptions \As{A1}--\As{A3}.

\subsection{Construction of the logarithmic derivative $\dRyy $ of $\muRyy$} \label{s:31}

The goal of this subsection is to construct the logarithmic derivative of 
$\muRyy = \mu(\cdot \mid \pioLRc (\sss ) = \pioLRc (\yy)) $. 
Let $ \muN $ be the distribution of the $ N $-particle system described in \sref{s:1}. 
For $ \muN $-a.s.\,$ \yy \in \sSS $, let 
\begin{align} &\notag %
\muRyyN = \muN (\cdot \vert \pioLRc (\sss ) = \pioLRc (\yy ))
.\end{align}
Let $ \rho _{ \Ryy }^{N,1}$ be the one-point correlation function of $ \muRyyN $ (on $ \oLSR $). 
Let $ \muxRyyN $ be the reduced Palm measure of $ \muRyyN $ conditioned at $ \xx $. 
Let 
\begin{align} \label{:32p}&
\muRyyNone (d\x d\sss ) = \rho _{ \Ryy }^{N,1} (x) \muxRyyN (d\sss ) d \x 
\end{align}
 be the one-reduced Campbell measure of $ \muRyyN $. 
From \As{A1} and \As{A2}, $ \rho _{ \Ryy }^{N,1} (\x )$ is continuous in $ \x $ on $ \oLSR \cap \ON $. 

Let $\anest = \{ \ak \}_{\qqq \in \mathbb{N}}$ be a family of increasing sequences
$\ak = \{ \akR \}_{\rR \in \mathbb{N}}$ of natural numbers satisfying \eqref{:ak}. 
Let $ \Ki [\ak ] $ be as in \eqref{:CUTw}. 

By \As{A2}, $ \{ \muN \}_{\nN \in \N } $ is tight. 
Hence, for each $ \qqq \in \N $, there exists an $ \ak = \{ \ak (\rR ) \}_{\rR \in \N } $ satisfying 
\begin{align}\label{:32u}&
\inf_{\nN \in \N } \muN (\Ki [\ak ] ) > 1- \frac{1}{\qqq }
.\end{align}
Then $ \{ \Ki [\ak ] \}_{\qqq \in \N } $ becomes an increasing sequence of compact sets in $ \sSS $.

From \eqref{:32u}, $ \muN (\cup_{\qqq \in \N } \Ki [\ak ] ) = 1 $. 
Let 
\begin{align}&\label{:32v}
\muRyyNr = \mu _{\RyyN }^{N } (\cdot \cap \Ki [\ak ] ) 
.\end{align}
%	GTP[---
Let $\muRyyNoner $ be the one-reduced Campbell measure of $\muRyyNr $.
Let $ \fQ \in \CziRd $ be as in \eqref{:22z}, and let $\qQ < \rR $.
Fix $\yy \in \cup_{\qqq \in \N } \Ki[\ak ] $, and choose $\qqq \in \N $ such that
$\yy \in \Ki[\ak ] \backslash \Ki[\akkk ] $, where $ \Ki[a_0] = \emptyset $ by convention. 

Note that for each $\yy \in \cup_{\qqq \in \N } \Ki[\ak ] $, $ \qqq $ is uniquely determined. 
We introduce the quantities 
$ \Ct \label{;42a} $, $ \Ct \label{;42b} $, and $ \Ct \label{;42c} $ as follows:
%GTP]
\begin{align}&\notag 
\cref{;42a} (\qQ , \rR ,\nN , \qqq ) = \int_{\RdSS } \fQ \lvert \nabla \log \fQ \rvert^2 
\muRyyNoner (d\x d\sss ) 
,\\ & \notag  
\cref{;42b} (\qQ , \rR , \nN , \qqq ) = 
\int_{ \ON \ts \sSS } \fQ \Big( \Delta \PhiN 
+ \sum_{i=1}^{N -1} \psiN (\x , \si ) \Big) 
\muRyyNoner (d\x d\sss ) 
,\\ \label{:32c}& 
\cref{;42c} (\qQ , \qqq ) = 
 \frac{1}{4} \sup_{ \rR > \qQ }
\sup_{\nN \in \N } 
\Big\{ \sqrt{\cref{;42a} }+ \sqrt{\cref{;42a} + 4 \beta \cref{;42b} } \Big\}^2 
.\end{align}
%GTP[---
By \eqref{:11*}, we have $\psiN \ge 0 $.
Hence $\cref{;42a} + 4 \beta \cref{;42b} \ge 0 $.
%GTP]
\begin{lemma} \label{l:3<}
For each $ \qQ \in \N $ and all $ \yy \in \cup_{\qqq \in \N } \Ki [\ak ] $, $ \cref{;42c} (\qQ , \qqq ) $ is finite.
\end{lemma}
%[GTP---
\begin{proof}
From \eqref{:CUTw} and \eqref{:32v}, we have, for all $ \qQ < \rR $, $ \nN \in \N $, and $ \yy \in \sSS $,
\begin{align}\notag %\label{:3<f}
\int_{ \sSS } \sss ( \SQQ ) \, \muRyyNr ( d\sss )
\le \ak ( \qQ + 1 )
.\end{align}
Combining this with \eqref{:32c}, \eqref{:11h}, and \eqref{:11*}, we obtain, for
$ \yy \in \Ki [ \ak ] \backslash \Ki [ \akk ] $,
\begin{align}\notag 
\cref{;42a} &
\le
\Big\{
\sup_{ \x \in \Rd }
\lvert \fQ ( \x ) \rvert
\lvert \nabla \log \fQ ( \x ) \rvert ^2
\Big\}
\ak ( \qQ + 1 )
,\\\notag 
\cref{;42b} &
\le
\Big\{
\sup_{ \x \in \Rd }
\lvert \fQ ( \x ) \rvert
\lvert \Delta \PhiN ( \x ) \rvert
+
\sup_{ ( \x , \y ) \in \ON ^2 }
\lvert \psiN ( \x , \y ) \rvert
(\nN - 1 )
\Big\}
\ak ( \qQ + 1 )
\\&\notag 
\le
\Big\{
\sup_{ \x \in \Rd }
\lvert \fQ ( \x ) \rvert
\lvert \Delta \PhiN ( \x ) \rvert 
+
\cref{;32!}
\Big\}
\ak ( \qQ + 1 )
.\end{align}
This, together with \eqref{:22z} and \eqref{:32c}, proves \lref{l:3<}.
\PFEND

%GTP]

We define, for $ \muN $-a.s.\,$ \yy $ and $ \x \in \SR $, 
\begin{align} \label{:2!y} 
\dRyyN (\xs ) 
= \dN ( \x , \4 ) 
.\end{align}
Here $ \piA (\sss ) = \sum_{\si \in A } \delta_{\si } $ and $ \oLSR = \{ | s | \le \rR \} $. 

Let $ \yyN $ be as in \eqref{:02y}.  
By \eqref{:21x} and \eqref{:2!y}, we obtain, for $ \x \in \ON $,  
\begin{align}\label{:2!z}
\dRyyRNN 
= - \beta \Big\{ \nabla \PhiN ( \x ) 
+ \sum_{\si \in \oLSR } \nablax \PsiN ( \x , \si ) 
+ \sum_{ \yi \not\in \oLSR }^{ \nN } \nablax \PsiN ( \x , \yi ) 
\Big\}
.\end{align}
The following result is a local, quench version of Lemmas \ref{l:22} and \ref{l:24}. 
\begin{lemma} \label{l:32} 
 Let $ \rR \ge \qQ + 1 $. Then for $ \yy \in \cup_{\qqq \in \N } \Ki [\ak ]$, 
\begin{align}\notag &% \label{:32d}&
\limsupi{\nN } 
\int_{\RdSS } \fQ \big\lvert \dRyyRNN \big\rvert ^2 \mu _{\RyyN }^{\nN , [1]} (d \x d\sss ) \le 
\cref{;42c} (\qQ , \qqq ) 
.\end{align}
\end{lemma}
\PF 
 \eqref{:32c} corresponds to \eqref{:22s} used in the proof of \lref{l:22}. 
 The remainder of the proof of \lref{l:32} is similar to that of \lref{l:22}, and is therefore omitted. 
\PFEND

Let $ \mRyyNone $ be the Radon--Nikodym density of $ \muRyyNone $ with respect to $ \LambdaRone $. 
Such densities $ \mRyyNone $ exist by \eqref{:10e}. 
Define 
\begin{align} &\notag %
 \mRyyklNone = \big( \mRyyNone 1_{\oLSSRkone } \big) \wedge \hhhl 
,\end{align}
where $ \oLSSRkone = \{ (\xs ) \in \oLSRSSR ;\, \sss (\oLSR ) = k \} $. 

By definition, $ \muRyy = \mu (\cdot \vert \pioLRc (\sss ) = \pioLRc (\yy ) ) $ is an RPF on $ \Rd $ concentrated at $ \pioLRc (\yy )$ on $\oLSRc $. 
We regard $ \muRyy $ as an RPF on $ \oLSR $ as well as that on $ \Rd $. 
Let $ \muRyyone $ be the one-reduced Campbell measure of $ \muRyy $. 
\begin{lemma} \label{l:33} 
Let $ \rR \ge \qQ + 1 $, $ k , l \in \{ 0 \} \cup \N $. 
For $ \mu $-a.s.\,$ \yy $, the following hold:  

\noindent 
 \thetag{i}
$ \{ \fQ \mRyyklNone \}_{\nN \in \N } $ is relatively compact in $ L^2 (\LambdaRone ) $. 

%GTP[---

\noindent 
\thetag{ii} There exist functions $ \mRyyklone $ satisfying 
\begin{align} \notag 
\limin \fQ \mRyyklNnone & = \fQ \mRyyklone \text{ in } L^2 (\LambdaRone ) 
,\\ \notag 
\mRyyklone &= \mRyykllone \wedge \hhhl 
.\end{align}

\noindent 
\thetag{iii} 
$ \muRyyone $ has the Radon-Nikodym density $ \mRyyone = ({d\muRyyone }/{d \LambdaRone })$, 
which is defined by 
$ \mRyyone \wedge \hhhl = \mRyyklone $ on $ \oLSSRkone $. 
\end{lemma}

\begin{proof}
\lref{l:33} follows from \lref{l:32} using similar arguments as for the proof of \lref{l:25}. 
Hence, the proof is omitted. 
\PFEND

We proceed in the same manner for the conditional measures. 
Replacing $ \mRNone $ by $ \mRyyNone $ (resp.\,$ \mRyyone $) in \eqref{:26h},
we define $ \oLSSRyyklNone $ and $ \dRyyklN $ (resp.\,$ \oLSSRyyklone $ and $ { \dlog }_{ \Ryykl } $) similarly. 
\begin{lemma} \label{l:3"}
For $ \rR \ge \qQ + 1 $, $ k , l \in \zN $, and $ \mu $-a.s.\,$ \yy $, 
\begin{align} \notag %	\label{:3"a}
\limin \dRyyklNn ( \fQ \mRyyklNnone )^{1/2} &= {\dlog}_{\Ryykl } ( \fQ \mRyyklone )^{1/2}
\quad \text{ weakly in $ L^2 (\LambdaRone )$} 
,\\ \notag %\label{:3"b} 
 {\dlog}_{\Ryykl } (\xs ) & = {\dlog}_{\Ryykll }(\xs ) \quad \text{ for } (\xs ) \in \oLSSRyyklone 
.\end{align}
\end{lemma}

\begin{proof}
The proof of \lref{l:3"} is identical to that of \lref{l:2"}, using Lemmas \ref{l:32} and \ref{l:33}. 
We omit the details.
\PFEND

Replacing $ \mRone $ by $ \mRyyone $ in \eqref{:2)x}, we define $ \dRyy $ analogously to $ \dR $.
\begin{lemma} \label{l:3)} 
For all $ \rR \ge \qQ + 1 $ and $ \mu $-a.s.\,$\yy \in \Ki[\ak ] \backslash \Ki[\akkk ] $, 
\begin{align} \label{:3)x} 
\limsupi{\nN } \int_{\RdSS } \lvert & \fQ \dRyyN \rvert ^2 \mRyyNone \,  d\LambdaRone\le 
\cref{;42c} (\qQ , \qqq ) 
,\\\label{:3)a}
\dRyy (\fQ \mRyyone )^{1/2} &= \limin \dRyyNn (\fQ \mRyyNnone )^{1/2} 
\quad \text{ weakly in $ L^2 ( \LambdaRone ) $}
,\\ \label{:3)b} 
(\fQ \mRyyone )^{1/2} &= 
\limin (\fQ \mRyyNnone )^{1/2} 
\quad \text{ in $ L^2 ( \LambdaRone ) $}
.\end{align}
\end{lemma}

\begin{proof} 
\lref{l:3)} corresponds to \lref{l:2)} and follows from Lemmas \ref{l:32}--\ref{l:3"} using the same argument as for the derivation of \lref{l:2)} from Lemmas \ref{l:24}--\ref{l:2"}. Hence, the proof is omitted here. 
\PFEND

\begin{lemma} \label{l:34} Let $ \rR \ge \qQ + 1 $ and $ \qqq \in \N $. 
For $ \mu $-a.s.\,$\yy \in \Ki[\ak ] \backslash \Ki[\akkk ] $, the following hold: 
\begin{align} \label{:34!} &
\int_{\RdSS } \lvert \fQ \dRyy \rvert ^2 d\muRyyone \le \cref{;42c} (\qQ , \qqq ) 
 \\ \label{:34'} &
\int_{\RdSS } \h \dRyy d \muRyyone = 
\limin \int_{\RdSS } \h \dRyyNn d\muRyyNnone 
,\\\label{:34"} &
 \int_{\RdSS } \h \dRyy d \muRyyone = - \int _{\RdSS } \nablaxh d \muRyyone 
\end{align}
for all $ \FRone $-measurable $ \h = (\fQ \hat{\f })\ot \g \in \dcbone $. 
\end{lemma}

\begin{proof}
Note that Lemmas \ref{l:32}--\ref{l:34} correspond to Lemmas \ref{l:24}--\ref{l:26}, respectively. 
Hence, \lref{l:34} follows from Lemmas \ref{l:32}--\ref{l:3)} using similar arguments as for the proof of \lref{l:26}. Hence, the proof is omitted here. 
\PFEND

\subsection{Explicit formulae of the logarithmic derivative $ \dbR $ of $ \muRyy $} 
We continue to work under Assumptions \As{A1}--\As{A3}.

Let $ \mathbf{I} (k) $ be as in \eqref{:11l}. 
Let $ f^{(\mathbf{i})} = \partial ^k f / \partial \x ^{\mathbf{i}}$, 
where $\partial \x ^{\mathbf{i}}= (\partial x_1)^{i_1}\cdots (\partial x_{d})^{i_d}$. 
Let $ \III = \sqcup_{k = 0 }^{\lz } \mathbf{I}( k ) $ as in \tref{l:11}. 

From \eqref{:11h}, $ -\nablax \PsiN ( \x , \y ) = 0 $ for $ ( \x , \y ) \notin \oL{\ONtwo } $. 
%G---
Using Taylor expansions of
$ - \nablax \PsiNn ( \x , \y ) $
and
$ - \nablaPsi ( \x - \y ) $
in $ \x $
at the origin $ \x = 0 $
for each $ \lvert \y \rvert > \rR $ and $ ( \x , \y ) \in \ONtwo $,
we obtain, for $ \lvert \x \rvert \le \rR $,
\begin{align}\notag
- \nablax \PsiNn ( \x , \y )
&=
\1
\xiGamma
( - \nablax \PsiNn )^{ ( \mathbf{i} ) } ( 0 , \y )
+
\Rln ( \x , \y )
,\\\label{:35p}
- \nablaPsi ( \x - \y )
&=
\1
\xiGamma 
( - \nablaPsi )^{ ( \mathbf{i} ) } ( - \y )
+
\Rl ( \x , \y )
.\end{align}
Here $ \Rln $ and $ \Rl $ are defined by
\begin{align}\notag
&\Rln ( \x , \y )
=
\sum_{ \mathbf{i} \in \mathbf{I} ( \lz + 1 ) }
\xiGamma 
\int_0^1 
(-\nabla \PsiNn )^{(\mathbf{i})}( t \x , \y )
(\ell + 1 )(1-t)^{\ell }\, dt 
,\\\label{:35r} &
\Rl ( \x , \y ) =
\sum_{ \mathbf{i} \in \mathbf{I} ( \lz + 1 ) }
\xiGamma 
\int_0^1 
(-\nabla \Psi )^{(\mathbf{i})}( t \x - \y )
(\ell + 1 )(1-t)^{\ell }\, dt 
.\end{align}

 Let $ \Rylnn $, $ \Ryln $, and $ \mathscr{R} _{\rR + \e ,\yy }^{\lz }$ be such that, for $ \mu $-a.s.\,$ \yy = \sum_i \delta_{\yi }$, 
\begin{align} \notag &
\Rylnn (\x ) = \3 \Rln (\xixyi ) , \quad \Ryln (\x ) = \3 \Rl (\xixyi ) ,
\\\label{:35z} & 
\mathscr{R} _{\rR + \e ,\yy }^{\lz } (\x ) 
 = 
\sumyiRecN 
 \Rl (\xixyi ) 
,\quad 
\rR \in \N ,\ \epsilon \ge 0 
.\end{align}
\begin{proposition} \label{l:35} % \thetag{i}
For each $ \rR \in \N $ and $ \mu $-a.s.\,$ \yy $, we obtain 
\begin{align}& \label{:35e}
\int_{\sSS } \lVert \mathscr{R} _{\rR + \e ,\yy }^{\lz } \rVert _{\6 } 
d \mu (\yy ) < \infty 
\quad \text{ for each $ \e > 0 $}
,\\&\label{:35d} 
\limi{\rR } \lVert \Ryl \rVert _{\7 } = 0 \quad \text{ for each $ \qQ \in \N $}
,\\ 
&\label{:35c} 
\Ryl = \limin \Ryln = \limin \Rylnn 
 \quad \text{ in }\6 
,\\\label{:35!}&
\Ryl = \rrr _{\rR , \pioLRc (\yy ) }^{\lz } 
.\end{align}
\end{proposition}
\begin{proof} 
From \eqref{:35z}, \eqref{:35r}, and \eqref{:11j}, we obtain, for each $ \e > 0 $, 
 \begin{align} \notag &
\int_{\sSS } \lVert \mathscr{R} _{\rR + \e ,\yy }^{\lz } \rVert _{\6 } d \mu (\yy ) \le 
\int_{\sSS } \sumyiRecN 
\lVert \Rl (\xixyi ) \rVert _{\6 } d \mu (\yy ) 
& \ \text{ by \eqref{:35z}}
&\\ 
\notag \le &
 \sum_{\mathbf{i} \in \mathbf{I}(\lz + 1 ) } 
\frac{\rR ^{\lz + 1 }}{\iGamma }
 \int_{\sSS } \sumyiRecN 
\big\lVert (- \nablaPsi )^{(\mathbf{i})} ( \xyi ) \big\rVert _{\6 } 
 d \mu (\yy ) &  \ \text{ by \eqref{:35r}}
&
.\end{align}
The last term is finite by \eqref{:11j}. This yields \eqref{:35e}. 

From the above inequality, we obtain, for $ \mu $-a.s.\,$ \yy = \sum_i \delta_{\yi }$, 
\begin{align} & \notag 
\sumyiRecN 
 \lVert \Rl (\xixyi ) \rVert _{\6 } < \infty 
.\end{align}
Since $ \yy (\oLSRe ) < \infty $, this yields, for $ \mu $-a.s.\,$ \yy = \sum_i \delta_{\yi }$, 
\begin{align} & \label{:35g} %
\9 
 \lVert \Rl (\xixyi ) \rVert _{\6 } < \infty 
.\end{align}
This together with \eqref{:35z} yields \eqref{:35d}. 

%G[---

From \eqref{:35r} and \eqref{:11j}, for each $ \e > 0 $, we obtain
\begin{align}\notag &
\int_{\sSS }
\sumyiRecN
\supnN \lVert \Rln (\xixyi ) \rVert _{\SIXONn }
d \mu (\yy )
\\ \notag \le &
\sum_{\mathbf{i} \in \mathbf{I}(\lz + 1 ) }
\frac{\rR ^{\lz + 1 }}{\iGamma }
\int_{\sSS }
\sumyiRecN
\supnN
\big\lVert (-\nablax \PsiNn )^{(\mathbf{i})} ( \x , \yi )
\big\rVert _{\SIXONn }
d \mu (\yy )
\\ \notag < &
\infty
.\end{align}
As in \eqref{:35g}, for $ \mu $-a.s.\,$ \yy = \sum_i \delta_{\yi }$,
\begin{align}\label{:35h}&
\9
\supnN \lVert \Rln (\xixyi ) \rVert _{\SIXONn } < \infty
.\end{align}
%G]
By \eqref{:35r}, \eqref{:35z}, and \eqref{:35h}, the sequence $ \{ \Rylnn \} $ converges in $ \SIX $. 
Using \eqref{:11p} in addition, we have
\begin{align}\notag
& 
\limi{\n } \lVert \Ryln - \Rylnn \rVert _{\6 }
\\\notag & \le 
\limi{\n } \3 \lVert\Rl (\xixyi ) - \Rln (\xixyi ) \rVert _{\6 } \quad \text{by \eqref{:35z}}
\\\notag & \le 
\limi{\n } \sum_{\mathbf{i} \in \mathbf{I}(\lz + 1 ) }
\frac{\rR ^{\lz + 1 }}{\iGamma }
\3 \lVert 
(- \nablaPsi )^{(\mathbf{i})} ( \x - \yi ) - 
(- \nabla \PsiNn )^{(\mathbf{i})} ( \x , \yi ) 
 \rVert _{\6 }
\\\notag & = 0 \quad \text{by \eqref{:11o}}
.\end{align}
Collecting these, we obtain the second equality in \eqref{:35c}. 
By \eqref{:35z}, $ \Ryl = \limin \Ryln $ if the limit exists. 
Thus, we obtain \eqref{:35c}. 

By \eqref{:35z}, $ \Ryln $ is independent of $ \pioLR (\yy )$. 
Hence, \eqref{:35c} yields \eqref{:35!}. 
\PFEND

Let $ \III = \sqcup_{k = 0 }^{\lz } \mathbf{I}( k ) $ as in \tref{l:11}. 
For $ \mathbf{i} = (i_1,\ldots,i_d) \in \III $, define 
\begin{align}\notag %&\label{:35t}
\CRyinn &= \Big( \frac{1}{\iGamma } \Big) \3 (-\nablax \PsiNn )^{(\mathbf{i})} ( 0 , \yi ) 
,\\ 
\label{:35u}
\CRyiN &= 
\Big( \frac{1}{\iGamma } \Big) 
\3 (- \nablaPsi )^{(\mathbf{i})} (0- \yi ) 
.\end{align}
From \eqref{:35p}--\eqref{:35z} and \eqref{:35u}, 
\begin{align} \label{:35v}
\3 - \nablax \PsiNn (\x , \yi ) = & \1 \CRyinn \x ^{\mathbf{i}} + \Rylnn (\x ) 
,\\\label{:35y} 
 \3 - \nablaPsi (\xyi ) = & \1 \CRyiN \x ^{\mathbf{i}} + \Ryln (\x ) 
.\end{align}

%-----------------
%GTP[

\begin{proposition} \label{l:36} 
Let $ \mathbf{i} \in \III $ and $ \rR \in \N $. 
For $ \mu $-a.s.\,$ \yy $ such that $ \yy (\oLSR ) \ge 1 $, 
there exists a finite constant $ \CRyii $ satisfying 
\begin{align} \label{:36a}
\CRyii &= \limin \CRyinn = \limin \CRyiN 
,\\ \label{:36y}
\CRyii &= \CR _{\rR , \pioLRc (\yy ) }^{\mathbf{i}} 
,\\ \label{:36z}
\limiR \CRyii &= 0 
,\\ \label{:3'b}
\1 \CRyii \x ^{\mathbf{i}} + \Ryl (\x ) 
&=\limin \3 - \nablax \PsiNn ( \x , \yi ) \quad \text{in $ \SIX $}
.\end{align}
\end{proposition}

%GTP[ 
\begin{proof} 
From \eqref{:35v}, \eqref{:2!z}, \eqref{:35c}, \eqref{:11e}, and \eqref{:11f}, 
\begin{align} \notag &\limin
\Big(\beta \Big( \1 \CRyinn \x ^{\mathbf{i}}\Big) - \dRyyNn (\xs ) \Big) 
\\\notag =& \limin \Big( 
\beta \Big( - \Rylnn (\x ) + \3 - \nablax \PsiNn (\x , \yi ) \Big) - \dRyyNn (\xs ) \Big) \quad \text{ by \eqref{:35v} }
\\ \notag = &\limin 
 \beta \Big( - \Rylnn (\x ) + \nabla \PhiNn (\x ) + \sum_{\si \in \oLSR } \nablax \PsiNn ( \x , \si ) \Big) 
 \quad \text{ by \eqref{:2!z} }
\\ \notag %\label{:36q} 
=&
 \beta \Big( - \Ryl (\x ) + \nablaPhi (\x ) + \sum_{\si \in \oLSR } \nablax \Psi ( \x - \si ) \Big) 
\quad \text{by \eqref{:35c}, \eqref{:11e}, \eqref{:11f}}
.\end{align}
The convergence above holds in $ C^1 (\oL{\sS }_{\rRe } (\sss ) ) $ for each $ \e > 0 $ and all $ \sss $ such that $ \sss = \sum_{\si \in \oLSR } \delta _{\si } $, where $ \oL{\sS }_{\rRe } (\sss ) = \{ \x \in \oLSR ; \lvert \xsi \rvert \ge \e \text{ for all }i \} $. 

By \lref{l:3)} and the above equation, we obtain, weakly in $ L^{1}(\LambdaRone )$, 
\begin{align}\notag &
\limin \Big( \1 \CRyinn \x ^{\mathbf{i}}\Big) \fQ \mRyyone = 
\limin \Big( \1 \CRyinn \x ^{\mathbf{i}}\Big) \fQ \mRyyNnone 
%\quad \text{ by \eqref{:3)b}}
\\ &\notag = 
\limin \Big( \Big( \1 \CRyinn \x ^{\mathbf{i}}\Big) - \dRyyNn (\xs ) + \dRyyNn (\xs ) \Big) \fQ \mRyyNnone 
\\& \notag %\label{:36+} 
= 
\Big(
 \Big( - \Ryl (\x ) + \nablaPhi (\x ) + \sum_{\si \in \oLSR } \nablax \Psi ( \x - \si ) \Big) 
+ \dRyy (\xs ) \Big)
\fQ \mRyyone 
.\end{align}

From $ \muRyyone (\oLSR \ts \oLSSR ) > 0 $, we have $ \int _{\oLSR \ts \oLSSR } \mRyyone d \LambdaRone > 0 $. Hence, the family 
$ \{ \x ^{\mathbf{i}} \fQ (\xs ) \mRyyone (\xs ) ; \mathbf{i} \in \III \}$ 
is linearly independent in $ L^{1}(\LambdaRone )$.

From this and the display equation right above, each numerical sequence $ \{\CRyinn \}_{\n \in \N }$ has a finite limit $ \CRyii $ satisfying 
\begin{align} \notag & 
\limin \1 \CRyinn \big( \x ^{\mathbf{i}} \fQ \mRyyone \big) = \1 \CRyii \big( \x ^{\mathbf{i}} \fQ \mRyyone \big) 
,\end{align}
which implies the first equality in \eqref{:36a}. 
%GTP]

Let $ \xji $ for $ \xjX $ be as in \eqref{:11L}. 
For each $ \mathbf{j} \in \III $, 
\begin{align} \notag & 
\limin \1 ( \CRyinn - \CRyiN ) \xji 
\\ = \notag 
& \limin \Big(\1 \CRyinn \xji + \Rylnn (\xj ) \Big) 
- \Big( \1 \CRyiN \xji + \Ryln (\x ) \Big) \quad \text{ by \eqref{:35c}}
\\\notag = & 
\limin \3 \Big( - \nablax \PsiNn (\x , \yi ) + \nablaPsi (\x - \yi ) \Big) 
\quad \text{ by \eqref{:35v}, \eqref{:35y}}
\\ \notag %\label{:36f}
=\, & 0 \quad \text{by \eqref{:11o}}
.\end{align}

Let $ \mathbf{x}_{\mathbf{j}}= ( \xji ) _{\mathbf{i} \in \III } $ for $ \mathbf{j} \in \III $. 
Then the vectors $ \{ \mathbf{x}_{\mathbf{j} } \}_{\mathbf{j} \in \III } $ are linearly independent in $ \R ^{\III }$ 
for each $ \xjX $ by \eqref{:11L}. 
Hence, using the display equation right above, we find that, for each $ \mathbf{i} \in \III $, the numerical sequence $ \CRyiN - \CRyinn $ converges to zero as $ \n \to \infty $ for $ \mu $-a.s.\,$ \yy $. From this and $ \limin \CRyinn = \CRyii < \infty $ for each $ \mathbf{i} \in \III $, we obtain \eqref{:36a}. 

Clearly, $ \CRyiN $ and $ \CRyinn $ are independent of $ \pioLR (\yy )$. 
Hence, \eqref{:36a} yields \eqref{:36y}. 
From \eqref{:35u} and \eqref{:36a}, we have 
\begin{align} \notag &% \label{:36!}&
\CRyii = \limin \CRyiN =
\Big( \frac{1}{\iGamma } \Big) 
 \limin \3 
 (- \nablaPsi )^{(\mathbf{i})} (0- \yi ) 
.\end{align}
Because the sum in the display equation right above converges, we deduce 
\begin{align} & \notag %\label{:36"}&
\limiR \CRyii = 
\Big( \frac{1}{\iGamma } \Big) 
\limiR\Big( \limin \3 (- \nablaPsi )^{(\mathbf{i})} (0- \yi ) \Big) = 0 
.\end{align}
This implies \eqref{:36z}. 
From \eqref{:35v}, \eqref{:36a}, and \eqref{:35c}, we have \eqref{:3'b}. 
\PFEND

Let $ \oLSSR $ be the configuration space over $ \oLSR $. Let 
\begin{align}\label{:21z}&
 \oLSSRe = \{ ( \xs ) \in \oLSRSSR \,;\, \lvert \x - \si \rvert \ge \e ,\ \sss ( \oLSR ) = k \} 
.\end{align}
We regard $ f ( \xs ) \in \8 $ as a function of the variables 
$ ( \x , s^i ) _{ i = 1 }^{ k } $, 
symmetric in $ ( s^i ) _{ i = 1 }^{ k } \in \oLSRk $, and equip $ \8 $ with the $ C^1 $-norm. 
Here $ \8 = C^1(\SR )$ for $  k = 0 $ by convention. 
%G]
%G[
\begin{theorem}\label{l:3!}
Let $ \dRyy $ be defined in \lref{l:3)}. 

\noindent \thetag{i} For each $ \rR \in \N $ and for $ \mu $-a.s.\,$ \yy $, the identity
\begin{align}\label{:36c}&
\dRyy
=
\beta \Big\{
- \nablaPhi ( \x )
- \sum_{\si \in \oLSR } \nabla \Psi ( \x - \si )
+ \1 \CRyii \x ^{\mathbf{i} }
+ \Ryl ( \x )
\Big\}
\end{align}
holds for all $ ( \xs ) \in \SR \ts \sSS $ such that $ \x \ne \si $ for all $ i $. 

\smallskip
\noindent \thetag{ii}
For $ \mu $-a.s.\,$ \yy $ and each $ \rR , k \in \N $ and each $ \e > 0 $, we have
\begin{align}\label{:1!c}&
\dRyy ( \xs ) = \limi{\n } \dRyyRNNn ( \xs ) \quad \text{in } \8
.\end{align}
Here $ \Nn $ and $ \dRyyRNN $ are as in \As{A2}\,\thetag{iii} and \eqref{:2!z}, respectively. 
\end{theorem}
%GTP]

%GTP[ 
\begin{proof} 
Let $\SSRm = \{ \sss \in \sSS \,;\, \sss (\SR ) = m \} $.
Let $\h = \f \ot \g \in \GRone $ be an $\FRone $-measurable function such that
$\g (\sss ) = 0 $ on $\bigsqcup_{m=M}^{\infty} \SSRm $ for some $M \in \N $, and
$\h (\xs ) = 0 $ whenever $\lvert \xsi \rvert \le \epsilon $ for some $i $ and some $\epsilon > 0 $. Then
\begin{align}\label{:36h}
\Big( %\sum_{\sioLSR } - 2
\sum_{\si \in \oLSR } \nabla \Psi ( \x - \si )
 \Big) \h (\xs ) \in \GRone .
\end{align}

Let $\muRyyNnone $ and $\yyNn $ be as in \eqref{:32p} and \eqref{:02y} with $\nN = \Nn $, respectively.
By \eqref{:11e}, \eqref{:11f}, \eqref{:3)b}, and \eqref{:36h}, we obtain
\begin{align} 	\notag &
 \int_{\RdSS } 
\Big( \nablaPhi ( \x ) + \sum_{\sioLSR } \2 \Big) \h (\xs ) 
d\muRyyone 
\\& \label{:36k} = \limin \int_{\RdSS } 
\Big( \nablaPhiNn ( \x ) + \sum_{\sioLSR } \nablax \PsiNn ( \x , \si )	 \Big) 	
 \h (\xs ) d\muRyyNnone 
.\end{align}
Hence, we obtain 
\begin{align} & \notag 	
\int_{\RdSS }
\Big( 
 \dRyy (\xs ) + \beta \Big\{ \nablaPhi ( \x ) + 
\sum_{\sioLSR } \2 
\Big\} \Big) 
\h (\xs ) d\muRyyone 
\\\notag 
= & \limin \int_{\RdSS } 
\Big( 
 \dRyyNn (\xs ) + 
\\&\notag 
\beta \Big\{ \nablaPhiNn ( \x ) + 
\sum_{\sioLSR } \nablax \PsiNn ( \x , \si )		
\Big\} \Big) 
 \h (\xs ) d\muRyyNnone 
\quad \text{ by \eqref{:34'}, \eqref{:36k}}
\\ \notag = &
\limin \int_{\RdSS } 
\beta \Big( - \3 \nablax \PsiNn ( \x , \yi ) \Big) \h (\xs ) d\muRyyNnone 
\quad \text{ by \eqref{:2!z}} 
\\ \notag = & 
\Big\{ \int_{\RdSS } \beta \Big( \1 \CRyii \x ^{\mathbf{i}} + \Ryl (\x ) \Big) \h (\xs ) d\muRyyone \Big\} 
\quad \text{by \eqref{:3)b}, \eqref{:3'b}}
.\end{align}
This yields \eqref{:36c} for $ \muRyyone $-a.s.\,$ (\xs )$. % This yields \thetag{i}. 
The right-hand side of \eqref{:36c} is continuous on 
$ \{(\xs ) \in \oLSR \ts \oLSSR ; \x \ne \si \} $. 
Hence, there exists a $ \muRyyone $-version of $ \dRyy $ such that 
\eqref{:36c} holds for all $ \{(\xs ) \in \oLSR \ts \oLSSR ; \x \ne \si \} $. 
This yields \thetag{i}.

For $ \mu $-a.s.\,$ \yy $, by \eqref{:2!z} and \eqref{:35v}, we have, $ \x \in \ONn $, 
\begin{align}\notag
\dRyyNn (\xs )
&=
\beta \Big\{
- \nablaPhiNn (\x )
- \sum_{\si \in \oLSR } \nablax \PsiNn (\x , \si )
+ \1 \CRyinn \x ^{\mathbf{i}}
+ \Rylnn (\x )
\Big\}
.\end{align}
Hence, by \eqref{:11e} and \eqref{:11f},
together with Propositions \ref{l:35} and \ref{l:36},
we obtain
\begin{align}\notag
\limi{\n }
\dRyyNn (\xs )
&=
\beta \Big\{
- \nablaPhi (\x )
- \sum_{\si \in \oLSR } \2
+ \1 \CRyii \x ^{\mathbf{i}}
+ \Ryl (\x )
\Big\}
.\end{align}
in $ \8 $ for each $ \e > 0 $. From this and \thetag{i}, we obtain \thetag{ii}. 
\PFEND

\subsection{Explicit formulas for $ \dmuyy $ and $ \dmu $: Proof of \tref{l:11}}\label{s:3C}
We continue to work under Assumptions \As{A1}--\As{A3}.
% G[

Let $ \dRyy $ be as in \lref{l:3)}. The consistency of $ \dRyy $ in the following lemma is crucial for demonstrating the explicit expression of the logarithmic derivatives of $ \mu $ in \tref{l:11}. 

%G]

\begin{lemma} \label{l:37}
For $ \mu $-a.s.\,$ \yy $ and each $ \rR \in \N $, we have the following: 
\begin{align}\label{:37b}&
\dRyy (\xs ) = \dlog _{\rR + 1 , \yy } ( \xs + \pi _{\oL{\sS }_{\rR +1} \backslash \oLSR } (\yy ))
\end{align} 
for all $ (\xs ) \in \oLSRSSR $ such that $ \x \ne \si $ for each $ i $. 
\end{lemma}

\begin{proof}
From \eqref{:2!z}, we have 
\begin{align}\label{:37f}&
\dRyyNn (\xs ) = \dlog _{\rR + 1 , \yyNn }^{\Nn } ( \xs + \pi _{\oL{\sS }_{\rR +1} \backslash \oLSR } (\yyNn ))
.\end{align}
Hence, applying \eqref{:1!c} to \eqref{:37f}, we obtain \eqref{:37b}. 
\PFEND

Let $  \SSz $ be as in \dref{d:15}. 
We introduce the equivalence relation in $ \SSz $ such that 
$ \aaa \sim_{\mathscr{T}} \bb $ if and only if \eqref{:13q} holds. 

Let $ [\yy ] = \{ \zz \in \SSz ; \yy \sim_{\mathscr{T}} \zz \} $. 
Let  $ \SSoneyy  = \{ (\xs ) ; \delta_x + \sss \in [\yy ]\}$ and 
\begin{align}\label{:38u}&
 \SSneoneyy  = \{ (\xs ) \in \SSoneyy  ; x \ne \si \text{ for all }  i ,\ \sss \in \SSs  \}
, \quad  \sss = \sum_i\delta_{\si }
,\end{align}
where $ \SSs = \{ \sss \in \sSS \, ;\, \sss (\{ \x \} ) \in \{ 0,1 \} \text{ for all } \x \in \Rd \} $. 

For $ \mu $-a.s.\,$ \yy $, we define the function $ \dmubu (\xs ) $ on $  \SSneoneyy  $ by 
\begin{align}\label{:38z}
\dmubu (\xs ) & = \dlog _{\rR , \sss } (\x , \pioLR (\sss ))
.\end{align}
According to \eqref{:37b}, $ \dmubu (\xs ) $ is well defined. Indeed, from \eqref{:37b}, we have 
\begin{align*}&
\dlog _{\rR , \sss } (\x , \pioLR (\sss )) = 
\dlog _{\rR + 1 , \sss } ( \x , \pioLR (\sss ) + \pi _{\oL{\sS }_{\rR +1} \backslash \oLSR } (\sss ))= 
\dlog _{\rR + 1 , \sss } ( \x , \pi _{\oL{\sS }_{\rR +1}} (\sss ))
.\end{align*}

We take a version of $ \dmubu (\xs ) $ such that, for any $ \yy \in \SSz $, 
\begin{align}\notag &%\label{:38a}&
 \dmubu (\xs ) = \dmubz (\xs ) \quad \text{ for all }  (\xs ) \in \SSneoneyy  ,\, \zz \in [\yy ] 
.\end{align}
We also set $ \dmubu (\xs ) = 0 $ for $ \yy \notin \SSz $. 
Because the equivalence relation $ \sim_{\mathscr{T}} $ gives the partition of $ \SSz $ and $ \mu (\SSz ) = 1 $, such a version of $ \dmubu$ is well defined for all $ (\xs ) \in \Rd \ts \sSS $. 
Thus, for any $ (\xs ) \in \RdSS $, we define $\mathfrak{d} $ by
\begin{align}\label{:38x}
\mathfrak{d} (\xs ) = \dmubu (\xs ) \quad \text{for }  (\xs ) \in  \SSneoneyy  ,\,  \yy \in \sSS 
.\end{align}

%GTP]

Let $ \mu _{\yy } = \mu (\, \cdot \, | \mathscr{T}(\sSS ) )(\yy )$ be the tail decomposition of $ \mu $ as in \dref{d:15}. 
Let $ \muone _{\yy } $ be the one-reduced Campbell measure of $ \mu _{\yy } $. 
\begin{proposition} \label{l:38}
 \thetag{i} For $ \mu $-a.s.\,$ \yy $, $ \mathfrak{d}$ is 
the $ \bullet $-logarithmic derivative $ \dmuyy $ of $ \muyy $. 

\noindent 
\thetag{ii} 
$ \mathfrak{d} $ is the $ \bullet $-logarithmic derivative $ \dmu $ of $ \mu $ and satisfies \eqref{:11b}. 
\end{proposition}

%GTP[

\PF
From \eqref{:34"} and \eqref{:38z}, we obtain that, for all $ \h \in \dcbone $,
\begin{align}\label{:38g}
\int_{\RdSS} \h \, \dmubu \, d\muyyone = - \int_{\RdSS} \nablax h \, d\muyyone .
\end{align}
By \lref{l:39}, \eqref{:38g} holds for all $ \h \in \dbbone $, and hence
$ \dmubu $ is the $ \bullet $-logarithmic derivative of $ \muyy $. 
Combining this with \eqref{:38x}, we obtain \thetag{i}.

Using $ \muone = \int_{\sSS} \muyyone \, \mu(d\yy) $ and \eqref{:38g}, 
we further obtain, for all $ \h \in \dbbone $, 
\begin{align} \notag %
\int_{\RdSS} \h \, \mathfrak{d} \, d\muone
= - \int_{\RdSS} \nablax h \, d\muone 
.\end{align}
Hence, $ \mathfrak{d} $ is the $ \bullet $-logarithmic derivative of $ \mu $. 

Finally, from this, \eqref{:38z}, and \eqref{:38x}, for $ \mu $-a.s.\,$ \yy $ and
$ \muone_{\yy} $-a.e.\,$ (\xs ) $,
\begin{align} \notag &%
\dmu(\xs)
= \dlog_{\rR , \sss}\bigl(\x , \pioLR (\sss )\bigr).
\end{align}
Together with \eqref{:36c}, this yields \eqref{:11b} and completes the proof of \thetag{ii}.
\PFEND

%GTP]
\noindent {\bf Proof of \tref{l:11}. } 
\noindent 
We obtain the $ \bullet $-logarithmic derivative $ \dmu \in \Lloctwo (\muone ) $ denoted by \eqref{:11b} from Propositions \ref{l:27} and \ref{l:38}. 
We will prove the remaining claims of \thetag{i} later.

From \eqref{:36z}, $ \limiR \CRi = 0 $. From \eqref{:35d}, $ \limi{\rR } \Rsl = 0 $ in $ \7 $. 
Hence, we deduce \eqref{:11!}. 
From \eqref{:11b} and \eqref{:11!}, we have, for $ \mu $-a.s.\,$ \sss $,
\begin{align}\notag %\label{:38k}
\dmu &(\xs ) = \beta \Big( - \nablaPhi ( \x ) 
- 
\sum_{\sioLSR } \2 + \1 
 \CRi \x ^{\mathbf{i}}+ \Rsl (\x ) \Big) 
%		\ \text{ for all $ \rR \in \N $}
\\&\notag =
\beta \Big( - \nablaPhi ( \x ) + \limiR \Big(
- \sum_{\sioLSR } \2 + \1 
 \CRi \x ^{\mathbf{i}}+ \Rsl (\x ) \Big) \Big) 
\\&\notag 
=\beta \Big( - \nablaPhi ( \x ) + \limiR \Big( - \sum_{\sioLSR } \2 \Big) \Big) 
\quad \text{ in $ \SEVEN $ by \eqref{:11!}}
\end{align}
for all $ \qQ , \rR \in \N $ and each $ \epsilon > 0 $. 
Hence, we obtain \eqref{:11"}. Thus, we have obtained \thetag{ii}. 

Because \eqref{:11"} holds for all $ \e > 0 $, $ \dmu (\xs ) $ has a $ \muone $-version such that $ \dmu (\xs ) $ is locally Lipschitz continuous in $ \x $ on $ \Rd _{\ne }(\sss ) := \{ \x \in \Rd ; \x \ne \si \}$. 
Thus, we obtain \thetag{i}. 

%GTP[

In addition to \As{A1}--\As{A3}, we assume \As{A4} in the proof of \thetag{iii}. 
From \eqref{:11"} and \eqref{:11s}, it follows that, for $ \muone $-a.e.\ $x$, 
\begin{align} 
\notag 
\dmu (\xs ) = & -\beta \nablaPhi (\x ) - \beta \limiR \sum_{i} 1_{\SR } (\si ) \2 
\quad \text{by \eqref{:11"}} 
\\ \notag 
= & -\beta \nablaPhi (\x ) - \beta 
\limiR \sum_{i} \Big( 1_{\SR } (\si ) - 1_{\SR } (\x - \si ) \Big) \2 
\\ \notag &
- \beta \limiR \sum_{i} 1_{\SR } (\x - \si ) \2 
\\ \notag 
= & - \beta \limiR \sum_{i} 1_{\SR } (\x - \si ) \2 
\quad \text{by \eqref{:11s}}
.
\end{align}
This proves \thetag{iii}.
\qed
%GTP]

As mentioned before \As{A3}, \As{A3} is automatically satisfied when $ \PsiN = \Psi $. 
The following lemma verifies the only nontrivial condition in \As{A3}.
\begin{lemma} \label{l:3@}
Let $ \lzz $ be as in \As{A2}. 
Let $ d + \lz > \lzz $. 
Then \eqref{:11j} holds. 
\end{lemma}

\begin{proof}
By \eqref{:11)},  $ \int_{\sSS} \sss(\SR)\, \mu(d\sss ) \le 
\cref{;A2} \rR ^{ \lzz } $, $\rR \in \N $. 
Hence, 
\begin{align} \notag &
\int_{ \sSS }
\sum _{\si \notin \oLSRe } ^{\infty} 
\Big( \sup_{ \x \in \SR }
\frac{1}{ \lvert \xsi \rvert^{ d + \lz } } 
\Big) 
\, d \mu ( \sss )
\\  \le &\notag 
\int_{\sSS } \Big( \frac{\sss (\SRR \backslash \oLSRe )}{\e ^{ d + \lz }} + 
\sum_{r=1}^{\infty} \sum_{\rR + r \le \vert \si \vert < \rR + r + 1 } 
\frac{1}{ \lvert | \si | - \rR \rvert^{ d + \lz } } 
\Big)\, d \mu ( \sss )
\\\le & \notag 
\frac{\cref{;A2} (\rR + 1) ^{\lzz }}{\e ^{ d + \lz }} + 
\int_{\sSS } \sss (\sS _{\rR + 2 } \slash \SRR ) d \mu (\sss )+ \\\notag &
\int_{\sSS }
\sum_{r=1}^{\infty} \sss (\sS _{\rR + r + 1}\backslash \sS _{\rR + r })
\Big(\frac{1}{r^{ d + \lz }} - \frac{1}{(r+1)^{ d + \lz }} \Big) d \mu (\sss )
,\end{align}
which is finite by \eqref{:11)}.  
This together with $ \PsiN = \Psi $ and \eqref{:10a} proves \eqref{:11j}.
\PFEND

\section{Non-collision in ISDEs with infinitely many particles} \label{s:4}
%G[ ---
In this section, we prove the non-collision property of the stochastic dynamics
for infinite-particle systems described by solutions of the ISDE \eqref{:12q}.
Importantly, we do not rely on properties of the associated quasi-regular
Dirichlet forms, such as capacity.

The main result of this section, \pref{l:41}, applies to weak solutions of the ISDE \eqref{:12q}.
These solutions are independent of Dirichlet forms.
We apply \pref{l:41} to the solutions $ \uLX $ and $ \uLX_{\aaa} $ constructed in \ssref{s:6B} and \ssref{s:81} via lower Dirichlet forms.

All results in \sref{s:4} are proved under \As{A1}--\As{A3}.

\begin{proposition} \label{l:41}
Let $ \X = (\xX ^i )_{i\in\N }$ be a weak solution of \eqref{:12q}. 
We define 
$ \XXXi = \big( X_u^i , \sum_{ j \ne i }^{ \infty } \delta_{ X_u^j } \big) $. 
Assume that 
\begin{align} \label{:41z} 
 E\Big[ \int_0^{T} \sumiN 1_{\SR } (X_u^i ) \Lvert \bbb ( \XXXi )\Rvert ^2 du \Big] < \infty 
&,\quad T , \rR \in \N 
%&\ \text{ for all $ T , \rR \in \N $}
,\\\label{:41a}
E[ 0 \vee (- \log \rvert X_0^i-X_0^j \Rvert ) ; X_0^i , X_0^j \in \SR ] < \infty 
&,\quad i \ne j , \rR \in \N 
.\end{align}
Let $ \XX _t = \sum_{i\in\N } \delta_{\xX _t^i} $ and $ \SSs = \{ \sss \in \sSS \, ;\, \sss (\{ \x \} ) \in \{ 0,1 \} \text{ for all } \x \in \Rd \} $. 
Then 
\begin{align}\label{:41b}&
P ( \XX _t \in \SSs \text{ for all } 0 \le t < \infty ) = 1 
.\end{align}
\end{proposition}
\begin{proof}
Let $ i \ne j \in \N $ be fixed. 
Without loss of generality, we can assume $ X_0^i , X_0^j \in \SR$. 
We divide the case into two parts: $ d \ge 3 $ and $ d=2$. 

Suppose $ d \ge 3 $. Let 
\begin{align}&\notag %		\label{:41d}&
\tauRe = T \wedge 
\inf \{ t \ge 0 ; \Lvert X_{t}^i-X_{t}^j \Rvert \le \e \} \wedge \min_{k=i,j} \inf\{ t \ge 0 ; \ X_t^k \not\in \SR \} 
 .\end{align}
Here, we suppress $ i , j , \rR , T \in \N $ from the notation of $ \tauRe $. 

Let $ \sumij = \sum_{(k,l) = (i,j), (j,i)} $. 
From \eqref{:12p}, \eqref{:12q}, and It\^{o}'s formula, 
\begin{align} \notag &
E[ - \log \Lvert X_{\tauRet }^i-X_{\tauRet }^j \Rvert ] - E[ - \log \Lvert X_0^i-X_0^j \Rvert ] 
\\&\notag = 
 - 
\sumij 
E\Big[ \int_0^{\tauRet } 
\Big( \frac{\Xkl }{ \Lvert \Xkl \rvert ^{2}}, \bbb \XXXk \Big) _{\Rd } du \Big] 
\\& \label{:41e} - 
\frac{d-2}{2} \sumij 
E\Big[\int_0^{\tauRet } 
\Big( \frac{\Xkl }{ \Lvert \Xkl \rvert ^{2}} 
\aaaa (X_u^k) , 
\frac{\Xkl }{ \Lvert \Xkl \rvert ^{2}} \Big) _{\Rd } 
du \Big] 
.\end{align}
Note that, for $ X_0^i , X_0^j \in \SR $, we have 
$ 0 \wedge ( - \log \Lvert X_{\tauRet }^i-X_{\tauRet }^j \Rvert ) \ge - \log 2\rR $. 
Combining this with \eqref{:12o} and \eqref{:41e}, we have 
\begin{align} \notag &
E[ 0 \vee ( - \log \Lvert X_{\tauRet }^i-X_{\tauRet }^j \Rvert ) ] + 
\Dtwo E\Big[\int_0^{\tauRet } \frac{1} { \Lvert \Xij \Rvert ^{2}}du \Big] 
\\\notag & \le 
 - E[ 0 \wedge ( - \log \lvert X_{\tauRet }^i-X_{\tauRet }^j \Rvert ) ] 
+ E[ 0 \vee (- \log \Lvert X_0^i-X_0^j \Rvert ) ] 
\\& \label{:41g} \quad \quad \quad + 
\sumij 
E\Big[ \int_0^{\tauRetkl }
 \Big( \frac{\Xkl }{ \lvert \Xkl \rvert ^{2}}, 
\bbb \XXXk \Big) _{\Rd } du \Big] 
.\end{align}

Let $ \Ct \label{;K5} $ and $ \Ct \label{;K5a}$ be the positive constants defined by 
 \begin{align} \notag &
 \cref{;K5} = E[ 0 \vee (- \log \Lvert X_0^i-X_0^j \Rvert ) ] + \log(2\rR )
 , \\ \label{:41p}&
 \cref{;K5a} = 2 
 E\Big[ \int_0^{T} \sumiN 1_{\SR } (X_u^i ) \lvert \bbb ( \XXXi )\rvert ^2 du \Big] ^{1/2}
.\end{align}
By \eqref{:41a} and \eqref{:41z}, we have $ \cref{;K5} < \infty $ and $ \cref{;K5a} < \infty $. 
Let 
\begin{align} \notag & 
x_{\epsilon}=E[ 0 \vee ( - \log \Lvert X_{\tauRet }^i-X_{\tauRet }^j \Rvert ) ] 
,\\&\label{:41q}
y_{\epsilon} = E\Big[\int_0^{\tauRet } \frac{1} { \Lvert \Xij \rvert ^{2}}du \Big] 
.\end{align}
By \eqref{:41g}--\eqref{:41q} and the Schwarz inequality, we have 
\begin{align}\label{:41i}& 
x_{\epsilon} + \Dtwo y_{\epsilon} \le \cref{;K5} + \sqrt{ y_{\epsilon} } \cref{;K5a}
.\end{align}
By \eqref{:41i} and $ d \ge 3 $, we have $ \limsupz{\epsilon}x_{\epsilon} < \infty $ and $ \limsupz{\epsilon}y_{\epsilon} < \infty $. 
Hence by Fatou's lemma and $ \limsupz{\epsilon}x_{\epsilon} < \infty $, for each $ t \ge 0 $, 
\begin{align} \notag 
E[ 0 \vee ( - \log \Lvert X_{\tauRzt }^i-X_{\tauRzt }^j \Rvert ) ] 
& \le \liminfz{\epsilon} E[ 0 \vee ( - \log \Lvert X_{\tauRet }^i-X_{\tauRet }^j \Rvert ) ] 
\\ \label{:41h}& 
= \liminfz{\epsilon}x_{\epsilon} < \infty 
.\end{align}
Hence, from Fatou's lemma and \eqref{:41h}, we obtain, for each $ i \ne j \in \N $, 
\begin{align}& \notag 
E[ 0 \vee ( - \log \Lvert X_{\tauRz }^i-X_{\tauRz }^j \Rvert ) ; \tauRz > 0 ] 
\\& \notag \le 
E[ \liminfi{t} \big( 0 \vee ( - \log \Lvert X_{\tauRzt }^i-X_{\tauRzt }^j \Rvert ) \big) ; \tauRz > 0 ] 
\\& \le 
 \liminfi{t} E[0 \vee ( - \log \Lvert X_{\tauRzt }^i-X_{\tauRzt }^j \Rvert ) ; \tauRz > 0 ] < \infty 
\label{:41I}
.\end{align}
Let $ \tilde{\tau} _{\rR , T }^{i,j} = \inf\{ 0 \le t \le T ; X_t^i \notin \SR \text{ or } X_t^j \notin \SR \} $. 
By \eqref{:41I}, 
\begin{align*}&
P ( X_t^i = X_t^j \text{ for some } 0 < t < \tilde{\tau} _{\rR , T }^{i,j} ; \tauRz > 0 ) = 0
%\quad \text{ for all $ \rR , T , i \ne j \in \N $}
\end{align*}
for all $ \rR , T , i \ne j \in \N $. This implies \eqref{:41b} for $ d\ge 3 $. 

We next suppose $ d = 2$. 
Applying It\^{o}'s formula to $ \log ( 0 \vee (-\log \lvert \x \rvert ) ) $ instead of $ - \log \lvert \x \rvert $ in \eqref{:41e} yields \eqref{:41b} for $ d=2$ in a similar fashion. 
For the sake of completeness, we present further details.

Let $ \vep $ be such that $ \vep ( \x ) = 0 \vee - \log \lvert \x \rvert $ for $ \x \ne 0 $ and $ \vep (0) = \infty $. 
Applying It\^{o}'s formula to $ \logvep = \log ( 0 \vee - \log \lvert \x \rvert ) $, we have 
\begin{align}\notag &
E[ \logvep ( X_{t \wedge \tauRe }^i-X_{t \wedge \tauRe }^j ) ] - E[ \logvep ( X_{0}^i-X_{0}^j ) ] 
\\&\notag =
 - \sumij E\Big[ \int_0^{\tauRet } 
\frac{ 1}{ \vep (\Xkl ) }
\Big( \frac{\Xkl }{\Lvert \Xkl \Rvert ^2}, 
\bbb \XXXk \Big) _{\Rd } du \Big] 
\\ \notag &%\label{:41j}&
 - \sumij 
E\Big[\int_0^{\tauRet } 
\frac{ 1 }{ \vep ( \Xkl )^2 } 
\Big( \frac{\Xkl }{\Lvert \Xkl \Rvert ^2} 
 \aaaa (X_u^k) , 
\frac{\Xkl }{\Lvert \Xkl \Rvert ^2} \Big) _{\Rd } 
du \Big] 
.\end{align}
From \eqref{:12o} and the display equation right above, we obtain 
\begin{align} \notag &
E[ \logvep ( X_{t \wedge \tauRe }^i-X_{t \wedge \tauRe }^j ) ] + 
\frac{1}{\cref{;12}}
E\Big[\int_0^{\tauRet } 
\frac{ 1 }{ \vep ( \Xij )^2 \Lvert \Xij \Rvert ^2 } 
du \Big] 
\\ \notag & \le 
E[ \logvep ( X_{0}^i-X_{0}^j ) ] 
\\& \label{:41k} 
 - \sumij 
E\Big[ \int_0^{\tauRet } 
\frac{ 1}{ \vep (\Xkl ) }
\Big( \frac{\Xkl }{\Lvert \Xkl \Rvert ^2}, 
\bbb \XXXk \Big) _{\Rd } du \Big] 
.\end{align}

Let $ x_{\epsilon} , y_{\epsilon} \ge 0 $ be such that 
\begin{align} \notag &
x_{\epsilon}= E[ \logvep ( X_{t \wedge \tauRe }^i-X_{t \wedge \tauRe }^j ) ] 
,\\ \label{:41l}& 
 y_{\epsilon} = 
E\Big[\int_0^{\tauRet } 
\frac{ 1}{ \vep ( \Xij )^2 \Lvert \Xij \Rvert ^2 } 
du \Big] 
.\end{align}
Let $ \Ct \label{;K5c} = E[ \vep ( X_{0}^i-X_{0}^j ) ] $. 
By \eqref{:41a}, $ \cref{;K5c} < \infty $. From \eqref{:41k} and \eqref{:41l}, 
\begin{align}\label{:41m}&
x_{\epsilon} + y_{\epsilon} \le \cref{;K5c} + \sqrt{ y_{\epsilon} } \cref{;K5a}
. \end{align}
From \eqref{:41m}, we have $ \limsupz{\epsilon}x_{\epsilon} < \infty $ and $ \limsupz{\epsilon}y_{\epsilon} < \infty $. 
Hence, using Fatou's lemma, 
we obtain that, for each $ \rR \in \N $ and $ t \ge 0 $, 
\begin{align}\notag % \label{:41h}&
E[ 0 \vee ( - \log \rvert X_{\tauRzt }^i-X_{\tauRzt }^j \lvert ) ] 
% \\ \notag &
& \le \liminfz{\epsilon}
E[ 0 \vee ( - \log \rvert X_{\tauRet }^i-X_{\tauRet }^j \lvert ) ] 
\\ \notag &
= \liminfz{\epsilon}x_{\epsilon} < \infty 
.\end{align}
This corresponds to \eqref{:41h} for $ d \ge 3 $. 
The rest of the proof for $ d = 2 $ is the same as the case $ d = 3 $. Thus, we obtain \eqref{:41b} for $ d = 2 $. 
\PFEND

\begin{proposition}	\label{l:42}
Let $ \X = (\xX ^i )_{i\in\N }$ be a weak solution of \eqref{:12q}. Assume 
\begin{align}\label{:42y}& P \circ \X _0^{-1} \ll \mu \circ \lab ^{-1}
%\text{ is absolutely continuous with respect to $ \mu \circ \lab ^{-1}$}
\end{align}
and \eqref{:41z}. Then \eqref{:41b} holds. 
\end{proposition}
\begin{proof}
From \eqref{:41z}, \eqref{:42y}, and Fubini's theorem, 
\begin{align} &\notag 
E\Big[ \int_0^{T} \sumiN 1_{\SR } (X_u^i ) \Lvert \bbb ( \XXXi )\Rvert ^2 du \vert \X _0 = \lab (\sss ) \Big] < \infty 
 \quad \text{ for $ \mu $-a.s.\,$ \sss $} 
.\end{align}
Hence, applying \pref{l:41} to $ \X $ such that $ \X _0 = \lab (\sss )$, we have 
\begin{align} & \notag 
P ( \XX _t \in \SSs \text{ for all } 0 \le t < \infty \,\vert \, \X _0 = \lab (\sss ) ) = 1 \quad \text{ for $ \mu $-a.s.\,$ \sss $} 
.\end{align}
Integrating the left-hand side with respect to $ \mu \circ \lab ^{-1}$ yields \eqref{:41b}. 
\PFEND
%GTP[ ---

Let $\Xm = (X^1,\ldots,X^m, \sum_{j>m} \delta_{X^j}) $ be an $m $-labeled process.
Suppose that $\Xm $ is a weak solution to the SDE
\begin{align}\notag
X_t^i - X_0^i =
\int_0^t \sigma (X_u^i) \, dB_u^i +
\int_0^t \bbb (\XXXi ) \, du
,\quad 1 \le i \le m 
.\end{align}
Note that this SDE is the restriction of the ISDE \eqref{:12q} to the first $m $ components.
%GTP]
Let $ \labm (\sss ) = ( \lab^{1}(\sss ) ,\ldots, \lab^m(\sss ), \sum_{i>m}\delta_{\labi (\sss )} ) $ for a label $ \lab = (\labi )_{i\in\N }$.

\begin{proposition}	\label{l:44}
Let $ 2 \le m < \infty $. Let $ \Xm $ be as above. 
Assume that 
\begin{align}&\notag %	\label{:44z} & 
 E\Big[ \int_0^{T} \sum_{i=1}^m 1_{\SR } (X_u^i )
 \Lvert \bbb ( \XXXi )\Rvert ^2 du \Big] 
< \infty 
,\quad \ T , \rR \in \N 
,\\ &\notag %\label{:44y}& 
P \circ (\X _0^{[m]})^{-1} \ll \mu \circ (\labm ) ^{-1}
.\end{align}
Then $ P ( \XX _t^m \in \SSs \text{ for all } 0 \le t < \infty ) = 1 $, where 
$ \XX _t^m = \sum_{i=1 }^m \delta_{\xX _t^i} $. 
\end{proposition}
\begin{proof}
The proof is the same as \pref{l:42}, so is omitted.
\PFEND

%G[---

\section{Dirichlet forms in finite and infinite volume}\label{s:5}

A symmetric Dirichlet form $ ( \E , \mathscr{D} ) $ on $ \Lnu $ 
is understood in the standard sense \cite{fot.2,c-f} (see \dref{d:A2} for details).

To prove Theorem \ref{l:12}--\ref{l:16}, we employ the Dirichlet form approach
\cite{k-o-t.udf,k-o-t.ifc,o.dfa,o.tp,o.isde,o-t.tail}.
We introduce Dirichlet forms associated with the ISDE \eqref{:12q}
and construct infinite-volume Dirichlet forms as limits of
finite-volume ones.
These finite-volume forms naturally split into two classes,
called the lower and upper Dirichlet forms.

Using the family of lower Dirichlet forms, we prove the closability of the
limiting infinite-volume form.
On the other hand, we show that the upper Dirichlet form is quasi-regular,
which ensures the existence of the associated diffusion. 

Unless stated otherwise, all results in \sref{s:5} are proved under Assumptions \As{A1}--\As{A3}. 

\subsection{Convergence of Dirichlet forms}\label{s:5a}
%G]

 For $ \mu $-a.s.\,$ \yy $, let $ \muRyy = \mu (\cdot \vert \pioLRc (\sss ) = \pioLRc (\yy ) ) $ as defined previously. 
Let $ \muRyy ^{[m]}$ be the $ m $-reduced Campbell measure of $ \muRyy $ (see \eqref{:Camp}). 
Let $ \oLSSR $ be the configuration space over $ \oLSR $. 
Let $ \oLSSRk = \{ \sss \in \oLSSR ; \sss (\oLSR ) = k \} $. 
\begin{lemma} \label{l:51}
For $ \mu $-a.s.\,$ \yy $ and each $ \rR \in \N $, the following hold: 

\noindent \thetag{i} 
For $ m , k \in \zN $ such that $ m + k \ge 1 $, 
$ \muRyy ^{[m]}$ has a bounded and continuous density $ \mRyy ^{[m]} $ with respect to $ \LambdaR ^{[m]} $ on $ \oLSR ^m \ts \oLSSRk $. 

\noindent \thetag{ii} $ \muRyy $ has a one-point correlation function $ \rho _{\Ryy }^1 $. 
\end{lemma}

\begin{proof}
Let $ \wm _{\Ryy }^{m+k} $ be the $ (m+k)$-labeled density of $ \muRyy $. 
From \eqref{:34"} and the definition of the one-reduced Campbell measure, we obtain 
\begin{align} &\notag %\label{:51b}&
\nablaxi \log \wm _{\Ryy }^{m+k} ( \x ^1 ,\ldots , \x ^{m+k} ) 
= \dRyy ( \x ^i , \sum_{ \x ^j \in \oLSR , j\ne i }^{m+k} \delta_{\x ^j } ) 
\end{align}
in the distributional sense on $ \SR ^{m+k} $, with test functions in
$ C_0^{\infty} (\SR ^{m+k} ) $. 

By \eqref{:36c}, the logarithmic derivative $ \dRyy $ admits the representation
\begin{align}& \notag %\label{:51c}
 \dRyy (\xs ) = \beta \Big( - \nablaPhi ( \x ) - \sum_{\sioLSR } \nablaPsi ({\xsi }) + 
\1 \CRyii \x ^{\mathbf{i}} + \Ryl (\x ) \Big)
.\end{align}
By \eqref{:3'b}, $ \1 \CRyii \x ^{\mathbf{i}} + \Ryl (\x ) \in \SIX $. 
Combining these yields \thetag{i}. 

From \eqref{:11)}, $ \int_{\sSS } \sss (\SR ) \mu (d\sss ) < \infty $. 
Then $ \int_{\sSS } \sss (\SR ) \muRyy (d\sss ) < \infty $ for $ \mu $-a.s.\,$ \yy $. 
From this, \thetag{ii} follows by the same argument as in \pref{l:2'} \thetag{ii}. 
\PFEND

 \begin{remark}\label{r:51} 
 Lebl\'{e} \cite[Th.\,1]{leb.DLR} proved the existence of continuous, local density of $ \mu $ in case of $ d = 2 $. 
\end{remark}

For $ f \in \db $ and $\sss \in \sSS $, let 
$ f_{ \rR ,\sss } $ and $\mathbf{x}_{\rR } (\sss )$ be as in \dref{d:fun}. 
Let 
\begin{align} 	\notag &%	&		 \label{:52q}
\DDDaR [f,g] (\sss ) = 
\frac{1}{2} \sum_{ \siSR } 
 \Big( ( \nablasi f_{ \rR ,\sss } ) 
\aaaa , 
\nablasi g_{ \rR ,\sss } \Big)_{\Rd } 
(\mathbf{x}_{\rR } (\sss )) 
.\end{align}
We introduce the $ m $-labeled carr\'{e} du champs on $ \RdmSS $. 
For $ m\ge 1$,
\begin{align} & \label{:52s} 
\DDDaRm [f,g] = \half 
1_{\SRm }
\sum_{i=1}^m \big( ( \nablaxi f ) \aaaa , \nablaxi g \big)_{\Rd }
 + \DDDaR [f,g]
.\end{align}
Let $ \mum $ be the $ m $-reduced Campbell measure of $ \mu $ defined in \eqref{:Camp}. Let 
\begin{align}&\notag 
\ERmum (f,g) = \int_{\RdmSS } \DDDaRm [f,g] d\mum 
,\\& \notag %\label{:52w} 
\dRbmum = \{ f \in \dbm ; \ERmum (f,f) < \infty ,\, f \in \Lmm \} 
,\\& \label{:52v} 
\dRcmum = \{ f \in \dRbmum ;
 \text{ $ f $ is $ \mathscr{B}(\SRm ) \ts \sigma [\piR ] $-measurable}\} 
.\end{align}
%[G
Although $ \ERmum $, $ \dRbmum $, and $ \dRcmum $ depend on $ \aaaa $ and $ \mu $, 
we suppress this dependence in the notation. 
Replacing $ \mum $ with $ \muRyym $, we write $ \ERyym $ and $ \dRyybm $.
%G]
\begin{lemma}	\label{l:52}
 Let $ \rR \in \N $ and $ m \in \{ 0 \} \cup \N $. 
\\\thetag{i} 
$ (\ERyym , \dRyybm )$ is closable on $ L^2 (\muRyym )$ for $ \mu $-a.s.\,$ \yy $. 
\\\thetag{ii} 
$ (\ERmum , \dRbmum )$ is closable on $ \Lmm $. 
\\ \thetag{iii} 
$ (\ERmum , \dRcmum ) $ is closable on $ \Lmm $. 
\end{lemma}
%G[
\begin{proof}
Let $ \oLSSRk = \{ \sss \in \sSS ; \sss (\oLSR ) = k \} $ and 
$ \muRyymk = \muRyym (\cdot \cap  \oLSR ^m \ts \oLSSRk )$. 
Let 
\begin{align}\label{:52f}
\ERyymk (f,g)
= \int_{\RdmSS} \DDDaRm [f,g] \, d\muRyymk .
\end{align}
By \lref{l:51} \thetag{i}, $ \muRyymk $ has a bounded and continuous density on $ \oLSR ^m \ts \oLSSRk $ 
for each $ k \in \zN $. 
Hence, $ (\ERyymk , \dRyybm ) $ is closable on $ L^2 (\muRyymk ) $.

Let $ \CnfRk = \{ (\xs ) \in \RdmSS ; \sss (\oLSR ) = k \} $. 
Then $ \RdmSS $ can be decomposed as the disjoint union 
$ \RdmSS = \bigsqcup_{k\in\zN} \CnfRk $. 
Note that 
\begin{align}\label{:52g}&
 \muRyym = \sumi{k}  \muRyymk ,\quad  \muRyymk ( (\CnfRk )^c ) = 0 
.\end{align}

%G[
By \eqref{:52f} and \eqref{:52g}, 
$ ( \ERyymk , \dRyybm ) $ is closable on $ L^2 ( \muRyym ) $ for each $ k $. 
Since $ ( \ERyym , \dRyybm ) $ is the countable sum of such closable forms, 
$ ( \ERyym , \dRyybm ) $ is closable on $ L^2 ( \muRyym ) $. 
For the case $ m=0 $, see \cite[Lem.\,3.2]{o.dfa}. 
%G]

%G[
By \thetag{i}, we obtain the superposition 
$ ( \ERstarm , \uL{\mathscr{D}} _{\Rstar }^{[m]} ) $, 
with respect to $ \yy $ under $ \mu $, of the closures 
$ ( \ERyym , \uLdRyym ) $ 
of 
$ ( \ERyym , \dRyybm ) $ on $ L^2 ( \muRyym ) $. 
That is, $ \uLEDRstarm $ and $ ( \ERstarm , \oL{\mathscr{D}} _{\Rstar }^{[m]} ) $ 
is the closed form on $ \Lmm $ defined by
\begin{align}\notag &
\ERstarm ( f , g ) = \int_{\sSS } \ERyym ( f , g ) \, \mu ( d \yy )
,\\& \notag 
\uL{\mathscr{D}} _{\Rstar }^{[m]} = 
\Big\{
f \in \bigcap_{\text{$ \mu $-a.s.\,$ \yy $}} \uLdRyym 
\, ; \, f \in \Lmm , \ 
\ERstarm ( f , f ) < \infty
\Big\}
.\end{align}

By the superposition theorem for closed forms \cite[Prop.\,3.1.1]{b-h}, 
this superposition yields a closed form on $ \Lmm $, which extends $ ( \ERmum , \dRbmum ) $. 
Hence, \thetag{ii} follows from \lref{l:ext2}.

%G]

Finally, $ (\ERmum , \dRbmum ) $ is clearly an extension of 
$ (\ERmum , \dRcmum ) $. 
Therefore, \thetag{iii} follows from \thetag{ii} by \lref{l:ext2}.
\PFEND 
%G]

\begin{remark}\label{r:52}
In \cite{os.ld}, it was proved that the existence of logarithmic derivative implies the closability of bilinear forms in \lref{l:52}. This result is a general theory applicable to all RPFs $ \nu $ with logarithmic derivative. 
Using this, we can prove \lref{l:52}. We present here direct proof using the explicit representation of the logarithmic derivative for $ \mu $ in \tref{l:11}. 
\end{remark}

Let $ (\ERyym , \uLdRyym )$, $ \EuLdRmum $, and 
$ (\ERmum, \dORmum)$ be the closures of the closable forms appearing in  \lref{l:52} 
\thetag{i}, \thetag{ii}, and \thetag{iii}, respectively. 

These closed forms are symmetric Dirichlet forms (cf.\ \dref{d:A2}). 
In the following lemma, we verify that they are strongly local and quasi-regular. 
See \ssref{s:ext} for the definitions of strong locality and quasi-regularity.

\begin{lemma} \label{l:53}
Let $ \rR \in \N $ and $ m \in \{ 0 \} \cup \N $. 
\\\thetag{i} 
$ (\ERyym , \uLdRyym )$ 
 is a strongly local, quasi-regular Dirichlet form on $ L^2 (\muRyym )$ for $ \mu $-a.s.\,$ \yy $, where we regard $ \muRyym $ as a measure on $ \RdmSS $. 
\\\thetag{ii}
$ \EuLdRmum $ is a strongly local, quasi-regular Dirichlet form on $ \Lmm $. 
\\\thetag{iii} $ (\ERmum ,\dORmum ) $ is a strongly local closed form on $ \Lmm $. 
\end{lemma}
\begin{proof} %G[
Let $ ( \ERyymk , \uLd _{\Ryy,k}^{[m]} ) $ be the closure of the closable form 
$ ( \ERyymk , \dRyybm ) $ on $ L^2 ( \muRyymk ) $, 
introduced in the proof of \tref{l:52}. 
Then $ ( \ERyymk , \uLd _{\Ryy,k}^{[m]} ) $ is a strongly local, quasi-regular Dirichlet form 
on $ L^2 ( \muRyymk ) $ satisfying the reflecting boundary condition. 
The properly associated diffusion is frozen outside $ \CnfRk $, 
and $ \CnfRk $ is its invariant set. 
Moreover, the sets $ \CnfRk $ are mutually disjoint for $ k \in \zN $ and 
$ \RdmSS = \sqcup_{k=0}^{\infty} \CnfRk $. 
%G]

The Dirichlet form $(\ERyym , \uLdRyym )$ coincides with the countable sum
of the closed forms $(\ERyymk , \uLd _{\Ryy,k}^{[m]})$, namely,
\begin{align}\notag
\ERyym 
= \sumi{k} \ERyymk , \qquad
\uLdRyym 
= \Big\{
f \in \bigcap_{k=0}^{\infty} \uLd _{\Ryy,k}^{[m]} \,;\,
\sumi{k}  \ERyymk (f,f) < \infty
\Big\}.
\end{align}
%G[ 

By a gluing argument, the diffusions associated with 
$ ( \ERyymk , \uLd _{\Ryy,k}^{[m]} ) $, $ k \in \zN $, 
can be glued together to construct on $ \RdmSS $ a diffusion properly associated with 
$ ( \ERyym , \uLdRyym ) $. 
Thus the form $ ( \ERyym , \uLdRyym ) $ on $ L^2 ( \muRyym ) $ admits a properly associated diffusion.

%---

Hence, by the general theory of quasi-regular Dirichlet forms, 
$ ( \ERyym , \uLdRyym ) $ 
is strongly local and quasi-regular on $ L^2 ( \muRyym ) $. 
This proves \thetag{i}.

%GTP] %GTP[---

Let $ \CnfRy = \{ ( \xs ) \in \RdmSS ; \piRc ( \sss ) = \piRc ( \yy ) \} $. 
We denote by the same symbol the associated partition of $ \RdmSS $. 
With $ \CnfRk $ replaced by $ \CnfRy $ and the countable sum of 
$ ( \ERyymk , \uLd _{\Ryy,k}^{[m]} ) $ replaced by the superposition of 
$ ( \ERyym , \uLdRyym ) $, 
\thetag{ii} follows in the same way as \thetag{i}.

Since the associated carr\'{e} du champ consists of differentials,
the form $ (\ERmum , \dORmum ) $ is strongly local,
which yields \thetag{iii}.
%GTP}
\PFEND

We examine the infinite-volume Dirichlet forms. 
Clearly, $ \DDDaR [ f , f ] (\sss ) $ is non-decreasing in $ \rR $ for all $ f \in \db $ and $ \sss \in \sSS $. 
Hence, we define 
\begin{align}\label{:52r}&
\DDDa [f , f ] ( \sss )= \limi{ \rR } \DDDaR [ f , f ] (\sss ) \le \infty 
.\end{align}
We set the carr\'{e} du champ $ \DDDa [ f , g ] (\sss ) $ by polarization for 
 $ \sss $ such that $ \DDDa [ f , f ] (\sss ) < \infty $ and $ \DDDa [ g , g ] (\sss ) < \infty $. 
For $ m \in \{ 0 \} \cup \N $ and $ f , g \in \dbm $, let 
\begin{align} \notag &
\DDDam [f,g] = 
\half \sum_{i=1}^m \big(( \nablaxi f ) \aaaa , \nablaxi g \big)_{\Rd } + \DDDa [f,g] 
,\\\notag &
\Emum (f,g) = \int_{\RdmSS } \DDDam [f,g] d\mum 
,\\ &\label{:53q}
 \uLdm = \{f \in \bigcap_{\rR =1}^{\infty} \uLdRmum ; \limiR \ERmum (f,f) < \infty \} 
.\end{align}

%G[

The notion of convergence in the strong resolvent sense
is defined before \lref{l:ext3}.

\begin{lemma} \label{l:54}
For $ m \in \{ 0 \} \cup \N $, the following hold: 
\\
\thetag{i}
$ (\Emum , \uLdm ) $ is a closed form on $ \Lmm $ and is the limit of
$ \EuLdRmum $
in the strong resolvent sense on $ \Lmm $.
\\
\thetag{ii}
$ (\Emum , \cup_{\rR \in \N } \dORmum ) $ is closable on $ \Lmm $.
Its closure $ (\Emum , \dOmum ) $ is the limit of 
$ (\Emum , \dORmum ) $ in the strong resolvent sense on $ \Lmm $. 
\end{lemma}
%G]

%
\begin{proof} 
It is clear that $ \EuLdRmum $ is an extension of $ ( \ERRmum , \dRRmum ) $ 
(see \dref{d:ext} for the notion of extension). 
Hence, by \eqref{:52v}--\eqref{:53q}, $ ( \Emum , \uLdm ) $ is the increasing limit of $ \EuLdRmum $ 
in the sense of \lref{l:ext3}. 
Hence, we obtain \thetag{i} from \lref{l:ext3} \thetag{i}. 

Since $ (\Emum , \uLdm ) $ is an extension of 
$ (\Emum , \cup_{\rR \in \N } \dORmum ) $, 
the first assertion of \thetag{ii} follows from \thetag{i} and \lref{l:ext2}. 

It is clear that $ (\ERRmum , \dORRmum ) $ is an extension of 
$ (\ERmum , \dORmum ) $. 
Hence, $ (\ERmum , \dORmum ) $ is decreasing with limit 
$ (\Emum , \cup_{\rR \in \N } \dORmum ) $. 

By definition, $ (\Emum , \dOmum ) $ is the closure of the maximal closable part of 
$ (\Emum , \cup_{\rR \in \N } \dORmum ) $ on $ \Lmm $. 
Hence \thetag{ii} follows from \lref{l:ext3} \thetag{ii}. 
\PFEND

We refer to $(\Emum , \uLdm )$ and $(\Emum , \dOmum ) $ as the lower and upper Dirichlet forms for $ \mum $, respectively. 

\subsection{Dirichlet forms in infinite volume} \label{s:5z}
We now introduce new bilinear forms in addition to the ones we had before. 

We define 
\begin{align} & \notag 
 \dbmum = \{ f \in \dbm ; \Emum (f,f) < \infty ,\, f \in \Lmm \} 
,\\& \notag %\label{:5z}
 \dcmum = \{ f \in \dcm ; \Emum (f,f) < \infty ,\, f \in \Lmm \} 
.\end{align}
Although $ \dbmum $ and $  \dcmum $ depend on $ \aaaa $ and $ \mu $, 
we suppress this dependence in the notation. 
\begin{lemma} \label{l:55} 
 For each $ m \in \zN $, the following hold: 
\\\thetag{i} 
$(\Emum , \dbmum ) $ is closable on $ \Lmm $.
\\\thetag{ii} 
$(\Emum , \dcmum ) $ is closable on $ \Lmm $. 
\end{lemma}
\PF 
Clearly, $(\Emum , \uLdm )$ is an extension of $(\Emum , \dbmum ) $. 
By \lref{l:54} \thetag{i}, $(\Emum , \uLdm )$ is a closed form on $ \Lmm $. 
Hence, \thetag{i} follows from this and \lref{l:ext2}. 
Note that $(\Emum , \dbmum ) $ is an extension of $(\Emum , \dcmum ) $. 
Hence, \thetag{ii} follows from \thetag{i} and \lref{l:ext2}. 
\PFEND

Let $(\Emum , \oL{\dbmum } ) $ and $(\Emum , \oL{\dcmum } ) $ be the closures of $(\Emum , \dbmum ) $ and $(\Emum , \dcmum ) $ on $ \Lmm $, respectively. 

\begin{lemma} \label{l:56} 
For each $ m \in \zN $, the following hold. 
 \\\thetag{i} 
 $ (\Emum , \oL{\dbmum } ) = (\Emum , \uLdm ) $. 
 \\\thetag{ii} 
 $ (\Emum , \oL{\dcmum }) = (\Emum , \dOmum ) $. 
\end{lemma}
\PF 
Let $ \dRbmum $ be as in \eqref{:52v}.  
Since $ \dbmum \subset \dRbmum $ for all $ \rR \in \N $ and $ \Emum ( f , f ) < \infty $ for all $ f \in \dbmum $, we obtain $ \dbmum \subset \uLdm $. 
It is not difficult to show that $ \dbmum $ is dense in $ \uLdm $. 
	This yields \thetag{i}. 

From \lref{l:54} \thetag{ii}, $ (\Emum ,\dOmum ) $ is the closure of $ (\Emum , \cup_{\rR \in \N } \dORmum ) $ on $ \Lmm $. 
Note that $ (\Emum , \cup_{\rR \in \N } \dORmum ) $ is an extension of $ (\Emum , \dcmum ) $. 
Hence, $ (\Emum ,\dOmum ) $ is an extension of $ (\Emum , \dcmum ) $. 
From \cite[Lem.\,2.5]{k-o-t.udf}, $ (\Emum , \dcmum ) $ is dense in $(\Emum , \dOmum )$. 
	Hence, we have \thetag{ii}. 
\PFEND

\begin{lemma} \label{l:57}
For each $ m \in \zN $, 
\\\thetag{i} 
$(\Emum , \uLdm ) $ is a strongly local Dirichlet form on $ \Lmm $. 
\\\thetag{ii} 
$(\Emum , \dOmum ) $ is a strongly local, quasi-regular Dirichlet form on $ \Lmm $. 
\end{lemma}
%G[

\PF 
By \lref{l:54} \thetag{i}, $(\Emum , \uLdm ) $ is a closed form on $ \Lmm $. 
It is clearly symmetric, and by \eqref{:53q},
\begin{align}\notag 
\Emum (f,g) 
= \int_{\RdmSS } \DDDam [f,g] \, d\mum 
.\end{align}
Since $ \DDDam $ is the carr\'{e} du champ given by differentials, 
$(\Emum , \uLdm ) $ satisfies \eqref{:A2a} and is strongly local. 
This proves \thetag{i}. 

By \tref{l:QR2}, $ (\Emum , \oL{\dcmum }) $ is a quasi-regular Dirichlet form on $ \Lmm $. 
Moreover, by \lref{l:56} \thetag{ii}, 
$ (\Emum , \oL{\dcmum }) = (\Emum , \dOmum ) $. 
Hence $ (\Emum , \dOmum ) $ is quasi-regular on $ \Lmm $. 
Its strong locality follows in the same way as in \thetag{i}. 
\PFEND

%G]
%G[----
\section{Weak solutions of ISDEs for the lower Dirichlet form}\label{s:6}

The purpose of \sref{s:6} is to construct weak solutions $ \uLX $
to the ISDEs \eqref{:12q}
associated with the lower Dirichlet form
$ ( \Emum , \uLdm ) $ on $ \Lmm $.

In \sref{s:6A}, we construct the $ m $-labeled process $ \uLXRm $,
which is properly associated with the lower Dirichlet form
$ \EuLdRmum $ on $ \Lmm $.
Utilizing $ \uLXRm $, we then construct the fully labeled process
$ \uLXR $
and show that $ \uLXR $
is a weak solution of the ISDEs \eqref{:14t} and \eqref{:14u} (\lref{l:61}): 
\begin{align}\notag
\uLXRi (t) - \uLXRi (0)
&=
\int_0^t \sigma ( \uLXRi (u) ) d B_u^i
+ \int_0^t \bbb \bigl( \uLXRi (u) , \sum_{j \ne i}^{\infty} \delta_{ \uLXRj (u) } \bigr) du
\\ \notag
&\quad\quad \quad \quad \quad \quad 
+ \int_0^t \anR ( \uLXRi (u) ) d L_{\rR }^{i} (u),
\\ \label{:14t}
L_{\rR }^{i} (t)
&=
\int_0^t 1_{\partial \SR } ( \uLXRi (u) ) d L_{\rR }^{i} (u),
\quad i \le \sss (\oLSR )
,\\ \label{:14u}
\uLXRi (t) - \uLXRi (0)
&=
0 
,\quad \sss (\oLSR ) < i < \infty 
.\end{align}

Furthermore, we define the fully labeled process $ \uLX $
as the limit of $ \uLXR $
and show that the $ m $-labeled process
$ \uLX^{[m]} $, derived from $ \uLX $,
is associated with the lower Dirichlet form
$ ( \Emum , \uLdm ) $ on $ \Lmm $
(\tref{l:6A}).
In \sref{s:6B}, we establish that $ \uLX $
indeed solves the ISDEs \eqref{:12q}
(\tref{l:6B}).
%G]

 We impose
 \begin{align}\label{:60a}& 
 \uLXR (0) = \lab (\sss )
 .\end{align}
Here $ \lab = ( \labi )_{ i } $ denotes the labeling map introduced in \eqref{:02x}. 
The initial conditions of these solutions are coupled through the labeling map. 
We assume
\begin{align}\label{:60b}&
 \uLXR (0) \elaw (\varphi d\mu ) \circ \lab ^{-1} 
\end{align}

All results in \sref{s:6} are proved under Assumptions \As{A1}--\As{A3}, \eqref{:14y}, and \eqref{:60b}.

%G[

\subsection{The $ m $-labeled process for the lower Dirichlet form}\label{s:6A}

We briefly recall standard notions from Dirichlet form theory.
A Markov process $ \{ X_t \} $ is said to be associated with
a Dirichlet form $ ( \E , \mathscr{D} ) $ on $ \Lnu $
if $ \{ X_t \} $ is associated with the $ L^2 $-Markovian semigroup
$ T_t $ induced by $ ( \E , \mathscr{D} ) $ on $ \Lnu $,
that is,
$ E_x [ f ( X_t ) ] = T_t f ( x ) $
for $ \nu $-a.s.\,$ x $, all $ t \ge 0 $, and all $ f \in \Lnu $.

We say that $ \{ X_t \} $ is properly associated with $ ( \E , \mathscr{D} ) $ on $ \Lnu $
if, in addition, $ ( \E , \mathscr{D} ) $ is quasi-regular and $ E_x [ f ( X_t ) ] $
is a quasi-continuous $ \nu $-version of $ T_t f ( x ) $.
If $ ( \E , \mathscr{D} ) $ is strongly local,
then the associated Markov process is a diffusion,
that is, a continuous Markov process with the strong Markov property.
See \ssref{s:ext} or \cite{c-f} for strong locality and quasi-continuity.

%G]
Let $ \EuLdRmum $ be the quasi-regular, strongly local Dirichlet form on $ \Lmm $ as in \lref{l:53}. 
Let $ \uLXRm $ be the diffusion properly associated with $ \EuLdRmum $ on $ \Lmm $. 
For $ \sss (\SRover ) \ge m $, we take 
\begin{align}	\notag &%	\label{:61(}&
 \uLXRm (0) = ( \lab ^1 (\sss ) ,\ldots , \lab ^m (\sss ) , \sum_{i > m } \delta_{\labi (\sss )})
.\end{align}
Here $ \lab (\sss ) = (\labi (\sss ))_{i}$ is the label defined as \eqref{:02x}.

Let $ W^{[m]}:= C([0,\infty);(\Rd )^m\ts \sSS )$ $ ( m \ge 1 ) $ and $ W^{[0]}:= C([0,\infty); \sSS )$. 
For $ \w = (w^i) _{i\in\N }$, let $ \w ^{[m]}= (w^1\ldots,w^m, \ww ^{m*}) \in W^{[m]}$, where 
$ \ww ^{m*} (t) = \sum_{j > m } \delta_{w ^j (t)}$. 
Let $ \map{\upathmn }{W^{[n]}}{W^{[m]}}$ for $ m \le n $ such that 
\begin{align} &\notag %		\label{:61e}&
\upathmn (\w ^{[n]}) = (w^1,\ldots,w^m, \sum_{m < i \le n} \delta_{w^i} + \ww ^{n*} ) 
.\end{align}

Let $ \uL{P}_{\rR }^{[m]} $ be the distribution of $ \uLXRm $. 
We can easily show that $ \EuLdRmum $, $ m \in \zN $, are consistent in the sense that 
\begin{align}\label{:61f}&&
\uL{P}_{\rR }^{[m]} = \uL{P}_{\rR }^{[n]} \circ (\upathmn ) ^{-1} 
,\quad 0 \le m \le n < \infty 
.\end{align}
We take the space where $ \uLXRm $ is defined as $ W^{[m]}$. 
From \eqref{:61f}, we can write $ \uLXR = (\uLXRi )_{i\in \N } $. 
Here, $ \uLXRi = w^i $ and $ (w^i)_{i\in \N } \in \CRdN $. 
We set $ \uLXRm = ( \uLXR ^m , \uLXXR ^{m*} ) $, where 
$ \uLXR ^m = (\uLXRi )_{i=1}^m $ and $ \uLXXR ^{m*} (t) = \sum_{i>m}^{\infty} \delta_{\uLXRi (t) }$. 

Let $ ( B ^i )_{i\in \N } $ be a sequence of independent $ d $-dimensional standard Brownian motions. 
Let $ r_t $ be the time-reversal operator on the path space on $ [0,\infty)$ such that 
$ r_t (\omega)(s) = \omega (t-s)$ if $ 0\le s \le t $ and $ r_t (\omega)(s) = \omega (0)$ if $ t \le s $. 
Let $ \mathscr{C}^{[m]} =
\{ ( \mathbf{x} , \sss ) \in \RdmSS \,;\, \x ^k = \si \text{ for some } k , i \} $,
where $ \mathbf{x} = (\x ^k )_{k=1}^m $ and $ \sss = \sum_i \delta_{ \si } $.

\begin{lemma} \label{l:61} 
For all $\rR \in \N $, the following hold: 

\noindent \thetag{i} 
$ \uLXRm $ is a weak solution of \eqref{:14t} and 
$ \uLXXR ^{m*} (t) = \uLXXR ^{m*} (0) $ for all $ t \ge 0 $.

\noindent \thetag{ii} 
$ \uLXR $ is a weak solution of ISDE \eqref{:14t} and \eqref{:14u}. 

\noindent \thetag{iii}
Assume \eqref{:14z}. Then for $ i \in \N $, 
\begin{align}\label{:61g}& 
 \uLXRi (t) -\ \uLXRi (0) = 
\half \Big\{ 
\int_0^t \sigma ( \uLXRi (u) ) dB_u^i - 
\int_0^t \sigma ( \uLXRi (u) ) dB_u^i \circ r_t \Big\} 
.\end{align}
%G[ ---
\thetag{iv} 
The set $ \mathscr{C}^{[m]} $ has zero capacity with respect to $ \EuLdRmum $ on $ \Lmm $. 
\end{lemma}
\begin{proof} 
Without loss of generality, we may assume that $ m $ is the number of particles in $ \oLSR $. 
Indeed, particles outside $ \oLSR $ remain fixed, and the case $ m $ smaller than the number of particles in $ \oLSR $ reduces to the case of equality.

The Fukushima decomposition is the counterpart of It\^{o}'s formula in the theory of Dirichlet forms \cite[Th.\,4.2.6]{c-f}. Because $ \EuLdRmum $ is a quasi-regular Dirichlet form on $ \Lmm $, we can use the Fukushima decomposition. 
We regard $ \xione $, $ i = 1, \ldots, m $, as a function on $ (\Rd )^m \ts \sSS $ in an obvious manner. 
Applying the Fukushima decomposition to $ \xione $, we see that $ \uLXR ^m = (\uLXRi )_{i=1}^m $ satisfies 
\begin{align} & \label{:61G}&
 \uLXRi (t) - \uLXRi (0) = M_t^{[\xione ]} + N_t^{[\xione ]} ,
\quad \text{ $ 1 \le i \le m$}
,\end{align}
where $ M^{[\xione ]}$ is the martingale additive functional of finite energy and 
$ N ^{[\xione ]}$ is the continuous additive functional of zero energy of the Fukushima decomposition for the additive functional $ A_t ^{[\xione ]} = \uLXRi (t) - \uLXRi (0) $. 

Using a straightforward calculation, we obtain that (cf.\,\cite[pp.164-165]{c-f}) 
\begin{align}& \label{:61h}
 M_t^{[\xione ]} = \int_0^t \sigma ( \uLXRi (u) ) dB_u^i
\end{align}
and that the zero energy additive functional $ N ^{[\xione ]}$ satisfies 
\begin{align}\label{:61H}&
N_t^{[\xione ]} = 
 \int_0^t \bbb ( \uLXRi (u) , \sum_{j\ne i}^{\infty} \delta_{ \uLXRj (u) } ) du 
 + \int_0^t \anR (\uLXRi (u) )dL_{\rR }^i (u)
.\end{align}
We can prove \eqref{:61H} similarly to \cite[Example 5.2.2]{fot.2}. In \cite[Example 5.2.2]{fot.2}, it was assumed that $ \aaaa $ is the unit matrix; however, generalizing to the current case is straightforward. 
From \eqref{:61h} and \eqref{:61H}, $ \uLXR ^m $ satisfies \eqref{:14t}. 
Because $ \EuLdRmum $ has no energy outside $ \oLSR $, we deduce $ \uLXXR ^{m*} (t) = \uLXXR ^{m*} (0) $ for all $ t \ge 0 $. Thus, we obtain \thetag{i}. 

We next consider $ \uLXR = (\uLXRi )_{i\in\N }$. From \eqref{:61G}--\eqref{:61H}, $ \uLXRi $ 
satisfy \eqref{:14u} for $ i =1,\ldots, m $. 
For $ i > m $, SDE \eqref{:14u} obviously holds because 
$ \uLXXR ^{m*} (t) = \sum_{i>m}^{\infty} \delta_{\uLXRi (t) }$ is frozen outside $ \oLSR $ under the dynamics of $ \EuLdRmum $. 
This implies \thetag{ii}. 

From the Lyons--Zheng decomposition \cite[Th.\,6.7.2]{c-f}, we write $ \uLXRi $ as the sum of 
the martingale additive functionals: 
\begin{align}\label{:61y}&
 \uLXRi (t) - \uLXRi (0) = 
\half \Big\{ M_t^{[\xione ]} - M_t^{[\xione ]} \circ r_t \Big\} 
.\end{align}
Indeed, from \cite[Th.\,6.7.2]{c-f}, \eqref{:61y} holds $ P_{\mum }$-a.e.\,up to lifetime, where $ P_{\mum }$ is the distribution of $ \uLXRm $ satisfying $ \uLXRm (0) \elaw \mum $. By \eqref{:14z}, the distribution of $ \uLXRm \ll P_{\mum }$. 
Hence, \eqref{:61y} holds. 

From \eqref{:61h} and \eqref{:61y}, $ $we obtain \eqref{:61g}. This implies \thetag{iii}. 

Clearly, $ \mum (\mathscr{C}^{[m]} ) = 0 $. 
Because the $ m $-density function of $ \muRyy $ is bounded on $ \uL{\sS }_{\rR }^m $, $ d \ge 2 $, and $ \EuLdRmum $ has no energy outside $ \oLSR $, we obtain \thetag{iv}. 
\PFEND

\begin{lemma} \label{l:62}
The following hold:
\begin{align}\label{:62i} & 
\sup_{ \rR , i \in \mathbb{N} } E [ \lvert \ \uLXRi (t) - \uLXRi (u) \rvert ^4 ] 
 \le \cref{;76} \lvert t - u \rvert ^ 2 , \text{ } 0 \le t,u < \infty 
,\\ \label{:62j}&
\limi{a} \liminfi{ R }
P ( \max_{1\le i \le m } \maxT \lvert \ \uLXRi (t) \rvert \le a ) = 1 
,\quad m , T \in \mathbb{N} 
%	\quad \text{ for each } m , T \in \mathbb{N} 
,\\&\label{:62k}
\limil \inf_{ R \in \mathbb{N} }
P ( \IRT ( \uLXR ) \le l ) = 1 
,\quad T \in \mathbb{N} 
	%\quad \text{ for each } T \in \mathbb{N} 
.\end{align}
Here, $ \Ct \label{;76} > 0 $ is a constant. 
\end{lemma}
\begin{proof}
From \eqref{:14y} and \eqref{:60b}, we can assume $ \varphi = 1 $ without loss of generality. 
We have 
$ \langle \int_0^t \sigma ( \uLXRi (u) ) dB_u^i \rangle_t = \int_0^t \sigma ^t \sigma (\uLXRi (u) ) du $ and $ \sigma ^t \sigma = \aaaa $. From \eqref{:12o}, $ \aaaa $ is uniformly elliptic and bounded. 
Hence, from \eqref{:61g} and the martingale inequality, we have \eqref{:62i}--\eqref{:62k}. 
 The proof of \eqref{:62k} is not as easy as other claims. 
We can prove \eqref{:62k} in the same fashion as Lemma 8.7 in \cite[pp.\,1212--1214]{o-t.tail} and the detail is omitted here. 
\PFEND

For a fully labeled process $ \uLX = (\uLxX ^i )_{i\in \N }$, let $ \uLX ^{[m]} = (\uLX ^m , \uLXX ^{m*} ) $ be the $ m $-labeled process defined by $ \uLX ^m = (\uLxX ^i )_{i=1}^m $ and $ \uLXX _t^{m*} = \sum_{i>m}^{\infty} \delta_{\uLxX _t^i } $. 
Let $(\Emum , \uLdm )$ be the Dirichlet form as in \lref{l:57}. 
\begin{theorem}	\label{l:6A}
There exists an $ \uLX $ satisfying the following: 

\noindent \thetag{i} 
For each $ m \in \zN $, $ \uLXm $ is the $ \mum $-symmetric, conservative, continuous Markov process associated with $(\Emum , \uLdm ) $ on $ \Lmm $.

\noindent \thetag{ii} 
For each $ m \in \zN $, $ \uLXm $ satisfies 
\begin{align}		\label{:6Aa}&
 \uLXm = \limiR \uLXRm \quad \text{ in law in $ W^{[m]}$} 
.\end{align}

\noindent \thetag{iii} 
 $ \uLX $ satisfies 
\begin{align} \label{:6Ab}&
 \uLX = \limiR \uLXR \quad \text{ in law in $ \CRdN $}
.\end{align}
In particular, each tagged particle $ \uLxX ^{i } $ of $ \uLX = (\uLxX ^i )_{i\in \N }$ does not explode. 
\end{theorem}
\begin{proof} 
Let $ \W ^m = \CRd ^m $ and $ \WRN = \CRdN $ with the product topology. 
From \eqref{:14z}, $ \uLXR (0) \elaw (\varphi d\mu ) \circ \lab ^{-1} $. 
Combining this with \eqref{:62i}, we deduce the tightness of $ \ \uLXRi $ in $ \CRd $ for each $ i \in \N $. 
Hence, we obtain the tightness of $ \uLXR ^m $ and $ \uLXR $ in $ \W ^m $ and $ \WRN $, respectively, because 
$ \W ^m $ and $ \WRN $ are endowed with the product topology. 

 Combining the tightness of $ \uLXR $ in $ \WRN $ with 
\eqref{:62k} and using \lref{l:top3}, we obtain that $ \uLXXR ^{m*} $ is tight in $ \CiSS $. 
Hence, $ \uLXRm = (\uLXR ^m , \uLXXR ^{m*} ) $ is tight in $ W^{[m]} $. 

Let $ \uLX $ be any limit point of $ \uLXR $. 
According to \lref{l:54} \thetag{i}, $ \uLXRm $ converges to $ \uLXm $ in law in $ W^{[m]}$ 
and $ \uLXm $ is a $ \mum $-symmetric and continuous Markov process associated with $(\Emum , \uLdm ) $ on $ \Lmm $. 
In particular, $ \uLXm $ are unique in law and the convergence in \eqref{:6Aa} holds. 
Because $ \uLXRm = (\uLXR ^m , \uLXXR ^{m*} ) $ is tight in $ W^{[m]} $, the limit continuous Markov process $ \uLXm $ in \eqref{:6Aa} is conservative. 

Because \eqref{:6Aa} holds for all $ m \in \N $, the convergence in \eqref{:6Ab} holds and $ \uLX $ is unique in law. 
The second statement of \thetag{iii} follows from the first since $ \uLX $ is a $ \WRN $-valued random variable by \eqref{:6Ab}. 
\PFEND

\subsection{Weak solutions of ISDEs for the lower Dirichlet form} \label{s:6B}

Let $ \uLX $ be as in \tref{l:6A}. \ssref{s:6B} aims to prove that $ \uLX $ is a weak solution of ISDE \eqref{:12q} satisfying \As{SIN}, \As{NBJ}, and \As{AC}$_{\mu }$. 

Let $ \uLXR $ be as in \lref{l:61}. Let 
\begin{align} &\notag 
 \XXXRiu = 
(\XRiu , \sum_{j\ne i}^{\infty} \delta_{\XRju } ) 
,\quad 
\uLXXXi = (\uLXi _u , \sum_{j\ne i}^{\infty} \delta_{\uL{\xX } _u^j} ) 
.\end{align}
Let $ \aaaa \in C_b^2 ( \Rd ) $ and $ \bbb = \half \{ \aaaadmu \} $ 
be as in \eqref{:12o} and \eqref{:12p}.
Let $ \sigma \in C_b^1 ( \Rd ) $ be a matrix-valued function such that $ \sigma ^t \sigma = \aaaa $,
%G[---
\begin{lemma}\label{l:63}
The following convergence holds in law in $ \CRdN $. 
\begin{align}\label{:63a}
\limiR
\pL
\int_0^{\cdot}
\sigma ( \uLXRi ( u ) ) \, d B_u^i
\pR
&=
\pL
\int_0^{\cdot}
\sigma ( \uLxX _u^i ) \, d B_u^i
\pR
,\\\label{:63z}
\limiR
\pL
\int_0^{\cdot}
\bbb ( \XXXRiu ) \, du
\pR
&=
\pL
\int_0^{\cdot}
\bbb ( \uLXXXi ) \, du
\pR
.\end{align}
\end{lemma}
%G]
\PF 
%G[---
Let $ \langle M \rangle_t $ be the quadratic variation process of a continuous martingale $ M_t $.
By \eqref{:6Ab}, $ \sigma ^t \sigma = \aaaa $, and $ \aaaa \in C_b^2 ( \Rd ) $,
we obtain, for all $ 0 \le t < \infty $,
\begin{align}\notag
&
\limiR
\pL
\big\langle
\int_0^{ \cdot }
\sigma ( \uLXRi ( u ) ) \, d B_u^i
\big\rangle_t
\pR
=
\limiR
\pL
\int_0^{ \cdot }
\aaaa ( \uLXRi ( u ) ) \, du
\pR
\\\label{:63g}
&
\quad\quad\quad\quad\quad\quad
=
\pL
\int_0^{ \cdot }
\aaaa ( \uLxX _u^i ) \, du
\pR
=
\pL
\big\langle
\int_0^{ \cdot }
\sigma ( \uLxX _u^i ) \, d B_u^i
\big\rangle_t
\pR
.\end{align}
From \eqref{:63g}, we obtain \eqref{:63a},
because convergence in law in $ \CRd $
of continuous martingales follows from that of their quadratic variations.

From \tref{l:11}, we have
$ \dmu \in \Lloctwo ( \muone ) $.
Then $ \bbbQ := 1_{ \oLSQ } \bbb \in \Lmone $
for any $ \qQ \in \N $.
%G]
%
%G[
Hence for any $ \epsilon > 0 $, there exists $ \bbbQe \in C_b(\RdSS )$ such that 
$ \bbbQe = 1_{\oLSQ } \bbbQe $ and 
\begin{align} \label{:63h} &
\| \bbbQ - \bbbQe \|_{\Lmone } < \epsilon 
.\end{align}
Hence for any $ T , i \in \N $, 
\begin{align} \notag 
 E \big[\int_0^T \big\vert ( \bbbQ - \bbbQe ) ( \XXXRiu )\big\vert du \big] 
%\\ &\notag 
&\le E \big[ \int_0^T 
\sumiN \big\vert ( \bbbQ - \bbbQe ) ( \XXXRiu )\big\vert du \big] 
% \\ &\notag \le \cref{;14} T \muone (\SQ \ts \sSS )^{1/2} \| ( \bbbQ - \bbbQe ) \|_{\Lmone } 
\\&\label{:63i}
\le \cref{;14} T \muone (\SQ \ts \sSS )^{1/2} \epsilon 
.\end{align}
Here we used \eqref{:14y} and \eqref{:60b}, the fact that $ \uLXR ^{[1]} $ is a $ \muone $-symmetric Markov process, and \eqref{:63h}. 
%
%
%G[
Similarly, we have, for any $ T , i \in \N $, 
\begin{align}\label{:63j}&
 E \big[ \int_0^T \big\vert ( \bbbQ - \bbbQe ) ( \uLXXXi ) \big\vert d u \big] 
\le \cref{;14} T \muone (\SQ \ts \sSS )^{1/2} \epsilon 
.\end{align}

Note that $ \bbbQe \in C_b(\RdSS )$. For any $ F \in C_b(\CRd )$, we have 
\begin{align} \label{:63k} &
 F \Big( \int_0^{\cdot } \bbbQe ( w (u) , \ww (u)) d u \Big) \in C_b(\CRdSS ) 
.\end{align}
Hence, from \eqref{:6Ab} and \eqref{:63k}, we obtain, for each $ i \in \N $, 
\begin{align} &\notag %
\limiR E \Big[ F \Big( \int_0^{\cdot } \bbbQe ( \XXXRiu )d u \Big) \Big] = 
E \Big[ F \Big( \int_0^{\cdot } \bbbQe ( \uLXXXi ) d u \Big) \Big] 
.\end{align}
From this, \eqref{:63i}, and \eqref{:63j}, we easily obtain \eqref{:63z}. 
\PFEND

\begin{lemma} \label{l:64} 
For each $ i , T \in \mathbb{N} $, 
\begin{align}\label{:64a}&
\limiR P \Big( \maxT \Big\lvert \int_0^t \anR (\uLXRi (u) ) d L_{\rR }^{i} (u)\Big\rvert > 0 \Big) = 0 
.\end{align}
\end{lemma}

\PF 
From \eqref{:6Ab}, $ \{ \uLxXQi (\cdot) - \uLxXQi (0) \}_{\qQ \in \N }$ is tight in $ \CRd $. 
Hence, 
\begin{align} & \notag 
\limiR P \pL \maxT \lvert \uLXRi ( t ) - \uLXRi (0) \rvert \ge \frac{\rR }{3} \Big ) 
\\& \label{:64f} \le 
\limiR P \pL \supQ \big( \maxT \lvert \uLxXQi ( t ) - \uLxXQi (0) \rvert \big) \ge \frac{\rR }{3} \Big ) = 0
.\end{align}
By \lref{l:63}, $ \{ \int_0^{\cdot} \sigma ( \uLxXQi (u) ) d B_u^i \}_{\qQ \in \N }$ and 
$ \{ \int_0^{\cdot} \bbb \XXXQiu du \}_{\qQ \in \N }$ are tight in $ \CRd $. 
Similarly as \eqref{:64f}, we obtain 
\begin{align}\notag & 
\limiR P \pL \maxT \lvert \int_0^{t} \sigma ( \uLXRi (u) ) d B_u^i du \rvert \ge \frac{\rR }{3} \Big ) = 0 
,\\\label{:64h}&
\limiR P \pL \maxT \lvert \int_0^{t} \bbb ( \XXXRiu )du \rvert \ge \frac{\rR }{3} \Big ) = 0 
.\end{align}
From ISDE \eqref{:14t} and $ \uLXXR ^{\diai } (u) = \sum_{ j\not=i } \delta_{\uLxXRj (u) } $, we have 
\begin{align}\notag 
\int_0^t \anR (\uLXRi (u) )dL_{\rR }^i (u) = & \uLXRi (t) - \uLXRi (0) - \int_0^t \sigma (\uLXRi (u) ) d B_u^i 
\\ \label{:64j} &
- \int_0^t \bbb ( \XXXRiu )du 
.\end{align}
Note that $ L_{\rR }^{i} $ only increases when $ \uLXRi (u)$ is on the boundary $ \partial \SR $ from the second equation of \eqref{:14t}. 
Hence, from this and \eqref{:64f}--\eqref{:64j}, we obtain 
\begin{align} &\notag 
 P \Big( \maxT \Big\lvert \int_0^t \anR (\uLXRi (u) ) d L_{\rR }^{i} (u)\Big\rvert > 0 \Big) 
\le 
 P \Big( \maxT \lvert \uLXRi ( t ) - \uLXRi (0) \rvert \ge \frac{\rR }{3} \Big) 
\\&\notag% \quad \quad \quad \quad \quad \quad \quad \quad \quad \quad 
+ 
 P \Big( \maxT \lvert \int_0^{t} \sigma ( \uLXRi (u) ) d B_u^i du \rvert \ge \frac{\rR }{3} 
 \Big) 
+ P \pL \maxT \lvert \int_0^{t} \bbb ( \XXXRiu )du \rvert \ge \frac{\rR }{3} \Big ) 
.\end{align}
Taking $ \rR \to \infty $ and using \eqref{:64h}, we obtain \eqref{:64a}. 
\PFEND

Let $ \mathbb{X}_{\rR }^i $ and $ \mathbb{X}^i $ be $ C([0,\infty ); (\Rd )^4 )$-valued random variables such that 
\begin{align}\notag 
\mathbb{X}_{\rR }^i &= 
 \Big(\uLXRi (\cdot ) - \uLXRi (0), \int_0^{\cdot} \sigma ( \uLXRi (u ) ) d B_u^i , 
 \int_0^{\cdot} \bbb ( \XXXRiu )du , \int_0^{\cdot } \anR (\uLXRi (u ) ) d L_{\rR }^{i}(u ) \Big) 
,\\ \label{:65x} \mathbb{X}^i &= 
\Big(\uLxX _{\cdot}^i - \uLxX _0^i, 
\int_0^{\cdot} \sigma ( \uLxX _u^i ) d B_u^i , \int_0^{\cdot} \bbb ( \uLXXXi ) du ,\mathbf{0} \Big) 
.\end{align}
\begin{lemma}\label{l:65} 
We obtain 
\begin{align} \label{:65a} 
 ( \mathbb{X}^i )_{i \in \N } =\limiR & ( \mathbb{X}_{\rR }^i )_{i \in \N } 
\quad \text{ in law in $ C([0,\infty ); (\Rd )^4 )^{\N } $}
.\end{align}
\end{lemma}

\begin{proof} 
We can rewrite $ ( \mathbb{X}^i )_{i \in \N }$ as the $ C([0,\infty ); (\Rd )^{\N })^4 $-valued process: 
\begin{align*}& 
 ( \mathbb{X}^i )_{i \in \N }= 
 ( \uLX _{\cdot} - \uLX _0 , \pL \int_0^{\cdot} \sigma ( \uLxX _u^i ) d B_u^i \pR , 
\pL \int_0^{\cdot} \bbb ( \uLXXXi ) du \pR , 
(\mathbf{0})_{i \in \N } )
\end{align*}
We rewrite $ ( \mathbb{X}_{\rR }^i )_{i \in \N } $ similarly. 
The convergence of $ \uLXR (\cdot ) - \uLXR (0) $ to $ \uLX (\cdot ) - \uLX (0) $ follows from \eqref{:6Ab} and \eqref{:14z}. 
The convergence of the other terms follows from Lemmas \ref{l:63} and \ref{l:64}. 
From these, we easily obtain \eqref{:65a}. 
\PFEND

\begin{theorem}\label{l:6B} 
$ \uLX $ is a weak solution of \eqref{:12q} 
satisfying $ \uLX (0) \elaw ( \varphi \mu )\circ \lab ^{-1} $, \As{SIN}, \As{NBJ}, and \As{AC}$_{\mu }$. 
\end{theorem}
\begin{proof} 
Let $ \{ h_k \} $ be an increasing sequence of positive numbers such that $ \lim_k h_k = \infty $. 
Let $ \sS _{h_k } =\{ s \in \Rd ; \lvert s \rvert < h_k \} $. 
For $ \mathbf{w} = (w _i )_{i =1} ^4\in \CRd ^4 $, let 
\begin{align}& \label{:6Bf}
 \tau _k (\mathbf{w})= \inf \{ t > 0 ; \mathbf{w} ( t ) \notin \sS _{h_k }^4 \} 
.\end{align}
Let $ \mathbf{V}_k $ be the set of continuity points of the map $ \mathbf{w} \mapsto \tau _k (\mathbf{w} ) $: 
\begin{align*}&
\mathbf{V}_k = \{ \mathbf{w} \in C([0,\infty ); (\Rd )^4 ) 
; \mathbf{w} \mapsto \tau _k (\mathbf{w} ) \text{ is continuous at }\mathbf{w}\}
.\end{align*}
We can and do take $ \{ h_k \} $ such that, for all $ i , k \in \N $, 
\begin{align}\label{:6Bh}&
P ( \mathbb{X}^i \in \mathbf{V}_k ) = 1 
%\quad \text{ for all $ i , k \in \N $}
.\end{align}

Let $ F _k \in C_0 ((\Rd )^4 )$ such that 
$ F _k (\x , \y , \z , u ) = \xy - \z - u $ if $ (\x , \y , \z , u ) \in \sS _{h_k }^4 $. 
Using \lref{l:65} and \eqref{:6Bh}, we obtain, for each $ i , k \in \N $ and $ t \in [0 , \infty )$,
\begin{align}\label{:6Bi}&
 F _k (\XtaukX ) = \limiR F _k (\XtaukXR ) \ \text{ in law} 
.\end{align}
 
Recall that $ \uLXR = (\uLXRi )_{i=1}^{\infty}$ satisfies SDE \eqref{:14t}. 
By the definition of $ F _k $, \eqref{:65x}, \eqref{:6Bf}, and \eqref{:14t}, we obtain, for each $ \rR , i , k \in \N $ and $ t \in [0 , \infty )$, 
\begin{align}\notag &
F _k (\XtaukXR ) = 
\uLXRi (\taukXR ) - \uLXRi (0) -
 \int_0^{\taukXR } \sigma (\uLXRi (u) ) d B_u^i 
\\& \notag 
\quad -
 \int_0^{\taukXR }
\bbb ( \XXXRiu )du 
- \int_0^{\taukXR } \anR (\uLXRi (u) )dL_{\rR }^i (u) = 0
.\end{align}
From this and \eqref{:6Bi}, we obtain, for each $ i , k \in \N $ and $ t \in [0 , \infty )$, 
\begin{align} \notag & 
 F _k (\XtaukX ) = 0 
,\end{align}
which implies that 
\begin{align} \notag & 
\uLxX _{\taukX }^i - \uLxX _0^i - \int_0^{\taukX } \sigma (\uLxX _u^i ) d B_u^i 
- \int_0^{\taukX } \bbb ( \uLXXXi ) du 
= 0 
.\end{align}
By \eqref{:65x} and \lref{l:65}, the equation above holds as the identity of continuous processes. 
Hence, using $ \limi{k}\tau _k (\mathbb{X}^i ) = \infty $ yields ISDE \eqref{:12q}. 
 
We proceed with \As{SIN}. Let $ \WSsiNE $ be as in \eqref{:13v}. 
We will prove $ \uLX \in \WSsiNE $ a.s., that is, each tagged particle neither explodes nor collides with any other particle. 

From \tref{l:6A} \thetag{iii}, $ \uLxX ^i$ does not explode for each $ i \in \N $. 

To prove the non-collision property of $ \uLX $, we use \pref{l:42}. 
From \tref{l:11}, $ \dmu \in \Lloctwo (\muone ) $. 
From \tref{l:6A}, the one-labeled process $ \uLX ^{[1] } $ given by $ \uLX $ is a $ \muone $-symmetric Markov process. 
Hence, the $ \muone $-symmetry of $ \uLX ^{[1] } $ and \eqref{:14y} and \eqref{:60b} imply that, for all $ \rR , T \in \N $, 
\begin{align} \notag &
 E\Big[ \int_0^{T} \sumiN 1_{\SR } (\uLxX _u^i ) \lvert \bbb ( \uLXXXi ) \rvert ^2 du \Big] 
% \\ \label{:7Af} & 
\le \cref{;14} T \int_{\RdSS } 1_{\SR } (\x ) \lvert \bbb (\xs ) \rvert ^2 d\muone < \infty 
.\end{align}
From this, we obtain \eqref{:41z}. 
Note that \eqref{:42y} follows from \eqref{:14y} and \eqref{:60b}. 

Thus, we have verified the assumptions \eqref{:41z} and \eqref{:42y} of \pref{l:42}. 
Hence, we obtain the non-collision property \eqref{:41b}. 
Thus, we obtain \As{SIN}. 

From \eqref{:61g} and \tref{l:6A}, we have, for each $ i \in \N $, 
\begin{align}\label{:6Bo}& 
 \uLxX _t^{i } - \uLxX _0^{i } = 
\half \Big\{ 
\int_0^t \sigma ( \uLxX _u^i ) dB_u^i - \int_0^t \sigma ( \uLxX _u^i ) dB_u^i \circ r_t \Big\} 
.\end{align}
From \eqref{:12o}, \eqref{:6Bo}, and the martingale inequality, we have \As{NBJ}. 
Indeed, we can prove \As{NBJ} in the same fashion as Lemma 8.7 in \cite[pp.\,1212--1214]{o-t.tail}. 
Assumption \thetag{8.81} in \cite{o-t.tail} follows from \eqref{:11)} in the present paper. 
The detail is omitted here. 

Assumption 
\As{AC}$_{\mu }$ is obvious because $ \uLXX $ is associated with the reversible Dirichlet form 
$ (\Emu , \uL{\mathscr{D}}^{\mu } )$ on $ \Lm $ and $ \uLXX _0 \elaw \varphi \mu $. 
\PFEND

%G[---

\section{Weak solutions of ISDEs for the upper Dirichlet form: Proof of \tref{l:12}}\label{s:7}

The aim of \sref{s:7} is to prove \tref{l:12}.
In \sref{s:7!}, we construct the $ m $-labeled diffusion $ \oLXm $.
In \sref{s:7B}, we then construct the associated fully labeled process
$ \oLX $
and show that $ \oLX $ is a weak solution of \eqref{:12q}
satisfying \As{SIN}, \As{NBJ}, and \As{AC}$_{\mu }$.
These conditions are essential for establishing pathwise uniqueness
and the existence of strong solutions to the ISDE \eqref{:12q}.

Unlike in \sref{s:6},
the upper Dirichlet form $ ( \Emum , \dOmum ) $
is quasi-regular.
Therefore, we construct the $ m $-labeled diffusion $ \oLXm $
directly from the infinite-volume upper Dirichlet form.
This approach differs markedly from the finite-volume approximation
used in \sref{s:6}.

All results in \sref{s:7} are proved under Assumptions \As{A1}--\As{A3}. 

\subsection{The $ m $-labeled process for the upper Dirichlet form}\label{s:7!}

Let $ ( \Emum , \dOmum ) $ be the Dirichlet form on $ \Lmm $ as in \lref{l:57}. 
Let $ \capam $ denote the capacity with respect to $ ( \Emum , \dOmum ) $
on $ \Lmm $.
We write $ \mathrm{Cap}^{\mu } $ for the case $ m = 0 $.
Define $ \map{ \ulabmz }{\RdmSS }{\sSS } $ by
$ \ulabmz ( \mathbf{x} , \sss ) = \ulab ( \mathbf{x} ) + \sss $.

\begin{lemma}\label{l:71}
Let $ A \subset \sSS $.
If $ \mathrm{Cap}^{\mu } (A) = 0 $, 
then, for each $ m \in \N $, 
\begin{align}\notag
\capam \bigl( ( \ulabmz )^{-1} (A) \bigr)
= 0
%\quad \text{for all $ m \in \N $}
.\end{align}
\end{lemma}

\begin{proof}
The proof follows from the argument of \cite[Lem.\,4.1]{o.tp}.
\PFEND

%G]

\begin{theorem}	\label{l:7A}

\noindent \thetag{i} 
There exists a $ \mum $-symmetric, conservative diffusion $ \oLXm $ properly associated with $(\Emum , \dOmum ) $ on $ \Lmm $ for each $ m \in \zN $. 

\noindent \thetag{ii} 
None of the labeled particles $\oLxX^i$, where $i=1,\ldots,m$, in $\oLXm$ explode or collide with other labeled particles: 
\begin{align}\label{:7Aa}&
P (\oLxX ^i \in C([0,\infty);\Rd ), i=1,\ldots,m ) = 1 
,\\\label{:7Ab} & 
P (\oLxXti \ne \oLxXtj \text{ for all } t\in [0,\infty ), 1\le i, j \le m , i\ne j ) = 1 
.\end{align} 
%G[--- 
\noindent \thetag{iii}
For each $ 1 \le i \le m $,
\begin{align}\label{:7Ac}&
\oLxXti - \oLxXzi
=
\int_0^{ t }
\sigma ( \oLxXui ) \, d B_u^i
+
\int_0^{ t }
\bbb ( \oLXXXi ) \, du
,\\\label{:7Ad}&
\oLxXti - \oLxXzi
=
\half
\Big\{
\int_0^{ t }
\sigma ( \oLxXui ) \, d B_u^i
-
\int_0^{ t }
\sigma ( \oLxXui ) \, d B_u^i \circ r_t
\Big\}
.\end{align}
Here $ r_t $ is the time-reversal operator
defined before \lref{l:61}, and
\begin{align}\label{:7Ae}
\oLXXXi
=
\big(
\oLxXui ,
\sum_{ j \ne i }^{ \infty }
\delta_{ \oLxXuj }
\big)
.\end{align}
%G]
\end{theorem}

\begin{proof} %
From \lref{l:57} \thetag{ii}, $(\Emum , \dOmum ) $ is a quasi-regular, strongly local Dirichlet form on $ \Lmm $. 
Hence, we obtain a $ \mum $-symmetric diffusion $ \oLXm $ properly associated with $(\Emum , \dOmum ) $ on $ \Lmm $ by the general theory of Dirichlet forms \cite{c-f}.

 Because $(\Emum , \dOmum ) $ is a quasi-regular Dirichlet form, we use the Fukushima decomposition (cf.\,\cite[Th.\,4.2.6]{c-f}). 
We regard the coordinate function $ x^i $ as a function on $ \RdmSS $ in an obvious manner for $ 1 \le i \le m $. 
Then $ x^i \in \dOmumloc $. 
Although $ x^i \notin \dOmum $, we can apply the Fukushima decomposition to $ x^i $ by localization. 

Consider the additive functional $ A^{[x^i]}(t) := \oLxXti - \oLxXzi $ of $ \oLX ^{[m]}$. 
Let $ \zeta $ be the lifetime of the diffusion $ \oLXm $. 
Applying the Fukushima decomposition to $ x^i $, we have a decomposition of $ \oLX ^{[m]}$ such that, for $ 1 \le i \le m $, 
\begin{align}\label{:7Ak}&
 \oLxX _{\zetat }^i - \oLxXzi = M _{\zetat }^{[x^i]} + N_{\zetat }^{[x^i]} 
,\end{align}
where $M^{[x^i]} $ is the martingale additive functional of finite energy and $ N^{[x^i]}$ is the continuous additive functional of zero energy \cite{fot.2,c-f}. 
By a standard calculation (cf.\,\cite[pp.164-165]{c-f}), we can show that $ M^{[x^i]} $, $ 1 \le i \le m $, are 
martingales such that 
\begin{align}\label{:7Al}&
 M _{\zetat }^{[x^i]} = 
\int_0^{\zetat } \sigma ( \oLxXui ) dB_u^i 
.\end{align}
For any $ g \in \dcbm $, 
\begin{align}\label{:7An}
- \Emum (x^i, g ) &= \int_{(\Rd )^m \ts \sSS } \bbb ( \si , \sum_{j\ne i}^m \delta_{s^j } +
 \sum_{ k >m} \delta_{s^k}) g ( \sss ^{[m]} )
\mum (d\sss ^{[m]})
.\end{align}
Here, we set $ \sss ^{[m]} = (s^1,\ldots, s^m ,\sum_{j>m} \delta_{\sj })$. 
Using \eqref{:7Ae} and \eqref{:7An} together with localization and \cite[Th.\,5.2.4]{fot.2}, we have 
\begin{align}&\label{:7Ap}
 N _{\zetat }^{[x^i]} = 
\int_0^{\zetat } \bbb ( \oLXXXi ) du 
.\end{align}
Hence, from \eqref{:7Ak}, \eqref{:7Al}, and \eqref{:7Ap}, we obtain 
\begin{align}\label{:7Aq}&
 \oLxX _{\zetat }^i - \oLxXzi = 
\int_0^{\zetat } \sigma ( \oLxXui ) dB_u^i 
 + \int_0^{\zetat } \bbb ( \oLXXXi ) du 
.\end{align}
Applying the Lyons--Zheng decomposition \cite[p.284]{c-f} to $ x^i$, we see, by \eqref{:7Aq}, 
\begin{align}\label{:7Ar}&
 \oLxX _{\zetat }^i - \oLxXzi = \half \Big\{ 
\int_0^{\zetat } \sigma ( \oLxXui ) dB_u^i - \int_0^{\zetat } \sigma ( \oLxXui ) dB_u^i \circ r_t \Big\} 
.\end{align}

Let 
$ \oLX ^m = (\oLxX ^i )_{i=1}^m $ and 
$ \oLXX _t^{m*} = \sum_{j>m}^{\infty} \delta_{\oLxXtj } $. 
Then $ \oLXm = (\oLX ^m , \oLXX ^{m*} )$. 
Let 
\begin{align} &\notag %		\label{:7!f}&
\zeta_{\mathrm{tp}} = \sup \{ t \ge 0 ; \sup_{0\le u \le t } \lvert \oLX _u^m \rvert < \infty \} 
,\\ & \notag %				\label{:7!g} &
\zeta_{\mathrm{cnf}} = \sup \{ t \ge 0 ; \sup_{0\le u \le t } \oLXX ^{m*}_u (\oLSR ) < \infty \text{ for all } \rR \in \N \} 
.\end{align}
Then $ \zeta = \zeta_{\mathrm{tp}} \wedge \zeta_{\mathrm{cnf}} $. 
Suppose $ \zeta < \infty $. 
Taking $ t \to \infty $ in \eqref{:7Ar} yields 
 \begin{align} &\notag %
 \oLxX _{\zeta }^i - \oLxXzi = \half \Big\{ 
\int_0^{\zeta } \sigma ( \oLxXui ) dB_u^i - \int_0^{\zeta }
\sigma ( \oLxXui ) dB_u^i \circ r_t \Big\} 
.\end{align}
The right-hand side is finite. Hence, $ \oLxX _{\zeta }^i $ is finite. 
By $ \lim_{t\to\zeta }\oLxX _{\zetat }^i = \oLxX _{\zeta }^i $, $ X^i $ does not explode at finite time $ \zeta $. 
Hence, we obtain $ \zeta = \zeta_{\mathrm{cnf}} $. 

Because $(\Emu , \dOmu ) $ is a quasi-regular Dirichlet form on $ \Lm $, there exists an increasing sequence of compact sets $ \mathsf{K}_{\qqq } $ in $ \sSS $ such that 
\begin{align}\label{:7At}&
\limi{n} \mathrm{Cap}^{\mu } (\mathsf{K}_{\qqq }^c) = 0 
.\end{align}

Let $ \anest =\{ \ak \}_{\qqq \in\mathbb{N}} $ and $ \Ki [\ak ]$ be as in \eqref{:ak} and \eqref{:CUTw}, respectively. 
By the property of compact sets of $ \sSS $ \cite{Kal}, there exists such an $ \{ \ak \}_{\qqq \in\mathbb{N}} $ satisfying 
\begin{align}\label{:7Au}&
\mathsf{K}_{\qqq } \subset \mathsf{K}[\ak ] := \cap_{\rR = 1}^{\infty}\{ \sss\in \sSS ; \sss (\SR ) \le \akR \} 
.\end{align}
Hence, from \eqref{:7At} and \eqref{:7Au}, 
\begin{align}\label{:7Av}&
\mathrm{Cap}^{\mu } ((\cup_{\qqq \in \N } \mathsf{K}[\ak ])^c ) = 0 
.\end{align} 
Using \eqref{:7Av} and \lref{l:71}, we have 
\begin{align} &\notag % 
\capam ( (\ulabmz ) ^{-1} ( \cup_{\qqq \in \N } \mathsf{K}[\ak ])^c ) = 0 
.\end{align} 
Hence, using the general theory of Dirichlet forms \cite{fot.2,c-f}, we obtain 
\begin{align} &\notag %
\oLPm{} ( \oLXX _t^{m* } \in (\ulabmz ) ^{-1} ( \cup_{\qqq \in \N } \mathsf{K}[\ak ])^c
\text{ for some } 0 \le t < \infty ) = 0 
.\end{align}
Here $ \oLPm{} $ is the diffusion measure associated with $ ( \Emum , \dOmum ) $ on $ \Lmm $. 

From this, we deduce $ \zeta_{\mathrm{cnf}} = \infty $, $ \oLPm{} $-a.s. 
We have already proved that $ \zeta = \zeta_{\mathrm{cnf}}$. Therefore, $ \zeta = \infty $, $ \oLPm{} $-a.s. 
Thus $ \oLXm $ is conservative, which implies \thetag{i}. 

We obtain \eqref{:7Aa} from $ \zeta_{\mathrm{tp}}= \zeta = \infty $. 
Since $ \oLXm $ is $ \mum $-symmetric and conservative by \thetag{i}, for all $ \rR , T \in \N $, 
there exists a constant $ \label{;7!} \Ct $ such that 
\begin{align}\notag % \label{:7!k} &
 E\Big[ \int_0^{T} \lvert \bbb ( \oLXXXi ) \rvert ^2 du ; \oLX _0^m \in \SRm \Big] &
\le 
 E\Big[ \int_0^{T} \sum_{i=1}^m \lvert \bbb ( \oLXXXi ) \rvert ^2 du ; \oLX _0^m \in \SRm \Big] 
\\\notag & 
 \le \cref{;7!}
 T \int_{\SR \ts \sSS } \lvert \bbb (\xs ) \rvert ^2 d\muone < \infty 
.\end{align}
From this and \pref{l:44}, we obtain \eqref{:7Ab}. 
This proves \thetag{ii}.

By $ \zeta = \infty $, \thetag{iii} follows from \eqref{:7Aq} and \eqref{:7Ar}. This yields \thetag{iii}. 
\PFEND
%GTP

\begin{remark}\label{r:5B} 
In the proof of \tref{l:7A}, we applied the Lyons--Zheng decomposition 
directly to the limit process $ \oLXm $, since the upper Dirichlet form 
$(\Emum , \dOmum )$ is quasi-regular by \lref{l:57} \thetag{ii}. 

At that point, however, the quasi-regularity of the lower Dirichlet form 
$(\Emum , \uLdm )$ had not yet been established. 
For this reason, in the proof of \tref{l:6A} we instead applied the 
Lyons--Zheng decomposition to the finite-volume Dirichlet forms $ \EuLdRmum $, which are known to be quasi-regular by \lref{l:53}. 
\end{remark}
%GTP]
\subsection{Weak solutions of ISDEs for the upper Dirichlet form and the proof of \tref{l:12}} \label{s:7B}

In \ssref{s:7!}, we have constructed the $ m $-labeled process $ \oLXm $. 
The construction of the fully labeled process $ \oLX $ for these partially labeled processes have not yet done. 
The goal of the subsection is to construct the fully labeled process $ \oLX $ and prove that it satisfies ISDE \eqref{:12q} (\tref{l:7B}) and complete the proof of \tref{l:12}. 

The proof in the present section is different from that of \ssref{s:6B}. 
The upper Dirichlet form Dirichlet form $(\Emum , \dOmum ) $ on $ \Lmm $ is quasi-regular. 
Hence, we will utilize the Dirichlet form theory and the general theorems developed in \cite{o.tp,o.isde}.

For a label $ \lab = (\labi )_{i\in\N }$, let $ \lab ^{[0]} (\sss ) = \sss $ and 
$ \labm (\sss ) = ( (\labi (\sss ) )_{i=1}^m , \sum_{j>m}\delta_{\labj (\sss )} ) $. 
Let $ \map{\lpath }{\WSsiNE }{\CRdN }$ be as in \lref{l:top2}. We define 
\begin{align}&\notag % \label{:77!}&
\lpath ^{[m]} ( \ww ) = (\lpath ^1 ( \ww ) , \ldots, \lpath ^m ( \ww ) , \sum_{j>m}^{\infty} \delta_{\lpath ^j ( \ww ) } )
.\end{align}
Here, for $ \mathbf{w} = (w^i)$, we define the measure-valued path $ \ww $ by $ \ww _t = \sum_i \delta_{w_t^i}$. 
From \lref{l:top2}, $ \lpath (\ww )_0 = \lab (\ww _0 )$. Hence, $ \lpath ^{[m]} (\ww )_0 = \labm (\ww _0 )$. 

Let $ \oLPm{\labm (\sss ) } $ be the diffusion measure of $ \oLXm $ staring at $ \labm (\sss ) $ 
properly associated with $(\Emum , \dOmum ) $ on $ \Lmm $. 
By \lref{l:57}, $(\Emum , \dOmum ) $ is a quasi-regular, strongly local symmetric Dirichlet form on $ \Lmm $. 

\begin{lemma}[{\cite[Th.\,2.4]{o.tp}}] \label{l:77} 
%	Assume \As{A1}--\As{A3}. 
For $ \mu $-a.s.\,$ \sss $, 
\begin{align}\label{:77a}&
\oLPm{\labm (\sss ) } = \oL{P} ^{\mu }_{\sss } \circ (\lpath ^{[m]} )^{-1 }
,\quad m \in \zN 
.\end{align}
\end{lemma}
\PF 
\lref{l:77} follows from the argument of \cite[Th.\,2.4]{o.tp}. 
We briefly explain below the difference between \lref{l:77} and \cite[Th.\,2.4]{o.tp}.

In \cite[Th.\,2.4]{o.tp}, the following were assumed: \\
\thetag{i} $(\E ^{[k]} , \dc ^{[k]}) $ is closable on $ L^{2}(\mu ^{[k]}) $ for $ k = 0 $ and $ m $. 
\\\thetag{ii} The $ m $-point correlation function of $ \mu $ is locally bounded for each $ m \in \N $.
\\\thetag{iii} $ \mathrm{Cap} (\SSs ^c ) = 0 $. 
\\\thetag{iv} each tagged particle of $ \oLXX $ does not explode. 

In \lref{l:77}, \thetag{i} follows from \lref{l:55}. 

In \cite[Th.\,2.4]{o.tp}, \thetag{ii} was used only for proving the quasi-regularity of $(\Emum , \dOmum ) $. 
 Thus, we can replace \thetag{ii} by \lref{l:57}.

In \cite[Th.\,2.4]{o.tp}, \thetag{iii} was used to prove 
the non-collision of the tagged particles $ \{ \oLxX^i \}_{ i \in \N } $ of $ \oLXm $ 
via \cite[Lem.\,4.1]{o.tp}. Hence, \tref{l:7A} \thetag{ii} applies.

\thetag{iv} also follows from \tref{l:7A} \thetag{ii}. 

%GTP]

Thus, \lref{l:77} can be proved by following the proof of \cite[Th.\,2.4]{o.tp}. 
\PFEND 

For $ \w = (w ^i)\in \CRdN $, let $ \upathm (\w ) = (w^1,\ldots,w^m, \sum_{i>m} \delta_{w ^i})$. 
\begin{theorem} \label{l:7B} 

\noindent \thetag{i} 
There exists a family of probability measures $ \oLP{\mathbf{x} }^{\N }$ on 
$ \CRdN $ such that, for $ \mu $-a.s.\,$ \sss $, 
\begin{align}\label{:78f}&\quad \quad \quad 
\oLPm{\labm (\sss ) } = \oLP{\lab(\sss ) }^{\N } \circ ( \upathm )^{-1} 
,\quad \m \in \zN 
.\end{align}

\noindent \thetag{ii} 
Let $ \oLX = \w $. Then $ \oLX $ under $ \oL{P}_{ \varphi \mu \circ \lab ^{-1} }^{\N }$ is a weak solution of \eqref{:12q} with initial distribution $ \varphi \mu \circ \lab ^{-1} $ that satisfies \As{SIN}, \As{NBJ}, and \As{AC}$_{\mu }$. 
\end{theorem}
\begin{proof}
From \eqref{:77a}, the family of distributions $\oLPm{\labm (\sss ) }$ is consistent. 
Hence, using the Kolmogorov construction theorem, we obtain \thetag{i}. 

%GTP[
Let $ \oLXm $ be as in \tref{l:7A}. Then its first $ m $ components form a weak solution of \eqref{:12q}. 
By the consistency \eqref{:78f}, $ \oLX $ is a weak solution of \eqref{:12q}.

From \eqref{:7Ac}, for $ \oLXm $, none of the $ m $-labeled particles explode or collide with the other labeled particles for each $ m \in \N $. 
Combining this with the consistency \eqref{:78f}, each tagged particle $ \oLxX ^i $ of $ \oLX = ( \oLxX ^i )_{ i \in \N } $ neither explodes nor collides with the others. 
We have thus obtained \As{SIN}.

We obtain \As{NBJ} from the martingale decomposition \eqref{:7Ad}.
The proof is analogous to that of \tref{l:6B} and is omitted. 
%GTP]

Finally, \As{AC}$_{\mu }$ is obvious because the unlabeled dynamics $ \oLXX = \upathz (\oLX ) $ are associated with the Dirichlet form on $ \Lm $, where $ \oLXX _t = \sum_{i\in\N } \delta_{\oLxXti }$. 
\PFEND

%G[---

\noindent {\bf Proof of \tref{l:12}. }
\noindent
From \tref{l:7A}, $ \oLX ^{ [ m ] } $ is a $ \mum $-symmetric,
conservative diffusion for each $ m \in \zN $.
From \tref{l:7B}, there exists a fully labeled process $ \oLX $
associated with $ \oLX ^{ [ m ] } $,
and $ \oLX $ is a weak solution of \eqref{:12q}
with initial distribution
$ \varphi \, \mu \circ \lab ^{ -1 } $.
Here we may take any function $ \varphi $
satisfying \eqref{:14y}.
Using Fubini's theorem,
we obtain $ \oLX $ such that
$ \oLX $ is a weak solution of \eqref{:12q}
starting at $ \lab ( \sss ) $
for $ \mu $-a.s.\,$ \sss $.
This completes the proof.
\qed
%G]
\section{Unique strong solutions of ISDEs: Proof of Theorem \ref{l:13}} \label{s:8}

To prove \tref{l:13}, we utilize a general result developed in \cite{o-t.tail,k-o-t.ifc}. 

A key assumption for the uniqueness of solutions is the tail triviality of the associated RPFs.
Starting from the tail decomposition of $ \mu $, we introduce the tail decomposition of Dirichlet forms in \sref{s:81}.
Using this, we complete the proof of \tref{l:13} in \sref{s:82}. 

Unless stated otherwise, all results in \sref{s:5} are proved under Assumptions \As{A1}--\As{A3}. 

\subsection{Tail decomposition of Dirichlet forms and weak solution} \label{s:81} 
 
Let $ \mua $, $ \aaa \in \SSz $, be the tail decomposition of $  \mu $ given in \dref{d:15}. 
Then $ \mu = \int \mua d\mu (\aaa )$ and $ \mu (\sSS \backslash \SSz ) = 0 $. 
For all $ \rR \in \N $, $ \mu $-a.s.\,$ \aaa $, and $ \mua $-a.s.\,$ \yy $, let 
\begin{align} & \notag %\label{:81z}
 \muaRy = \mua (\cdot \vert \pioLRc (\sss ) = \pioLRc (\yy )) 
.\end{align}
Let  $ \muaRym $ be the $ m $-reduced Campbell measure of  $ \muaRy $. 
Let $ \oLSSR $ be the configuration space over $ \oLSR $ and $ \oLSSRk = \{ \sss \in \oLSSR ; \sss (\oLSR ) = k \} $. 
\begin{lemma} \label{l:81} 
For $ \mu $-a.s.\,$ \aaa $, $ \mua $-a.s.\,$ \yy $, and each $ \rR \in \mathbb{N}$, the following hold:  

\noindent 
\thetag{i} 
Let $ m \in \zN $. Then $ \muaRy ^{[m]}$ has a density $ \maRy ^{[m]} $ with respect to $ \LambdaR ^{[m]} $ such that $ \maRy ^{[m]} $ is bounded and continuous on $ \oLSR ^m \ts \oLSSRk $ for each $ k \in \zN $. 

\noindent \thetag{ii} $ \muaRy $ has a one-point correlation function $ \rho _{\aRy }^1 $. 
\end{lemma}

\begin{proof}
Let $ \dRyy $ and $ \dlog _{\aRy } $ be the logarithmic derivatives of $ \muRyy $ and $ \mu _{\aRy } $, respectively. 
From \pref{l:38}, $ \dRyy ( \xs ) = \dlog _{\aRy } ( \xs ) $ for $ \muaone $-a.e.\,$ (\xs )$ for $ \mu $-a.s.\,$ \aaa $. 
Hence, from \tref{l:3!}\,\thetag{i}, 
\begin{align} \notag &%\label{:81f}&
 \dlog _{\aRy } (\xs ) = \beta \Big( - \nablaPhi ( \x ) - \sum_{\sioLSR } \nablaPsi ({\xsi }) + 
\1 \CRyii \x ^{\mathbf{i}} + \Ryl (\x ) \Big)
.\end{align}
In \lref{l:51}, the existence of a bounded continuous density 
$ \mRyy ^{[m]} $ of $ \muRyy ^{[m]} $ was derived from the explicit representation 
\eqref{:36c} of $ \dRyy $. 
Using the above representation of $ \dlog _{\aRy } $, 
assertion \thetag{i} follows by the same argument. 

Assertion \thetag{ii} can be proved analogously to \lref{l:51} \thetag{ii}. 
\PFEND

Lemmas \ref{l:82}--\ref{l:83} below correspond to Lemmas \ref{l:52}--\ref{l:57}, respectively. 
The proof of Lemma \ref{s:8}.$ k $ is similar to that of Lemma \ref{s:5}.$ k $ for $ 1 \le k \le 6 $; therefore, we omit those proofs. 

%G[

Let $ \muam $ be the $ m $-reduced Campbell measure of $ \mua $. 
By replacing $ \mum $ with $ \muam $ and $ \muaRym $ in \eqref{:52v}, we define 
$ (\EaRm , \daRbm ) $ and $ (\EaRyym , \daRyybm ) $, respectively. 
%G]
\begin{lemma}	\label{l:82}
For $ \mu $-a.s.\,$ \aaa $, $ \rR \in \N $, and $ m \in \zN $, the following hold: 
\\\thetag{i} 
$ (\EaRyym , \daRyybm )$ 
is closable on $ L^2 ( \muaRym )$ for $ \mua $-a.s.\,$ \yy $. 
\\\thetag{ii} 
$ ( \EaRm , \dRbmuma )$ 
 is closable on $ \Lmuam $. 
\\ \thetag{iii} 
$ ( \EaRm , \dRcmuma ) $ is closable on $ \Lmuam $. 
\end{lemma}

Let 
$ (\EaRyym , \uLdaRyym )$ be the closure of 
$ (\EaRyym , \daRyybm )$ on $ L^2 (\muaRym )$. 
Let 
$ \EuLdRmuam $ and 
$ (\EaRm ,\dORmuma ) $ be 
the closures of $ (\EaRm ,\dRbmuma )$ and $ (\EaRm , \dRcmuma ) $ on $ \Lmuam $, respectively. 
\begin{lemma} \label{l:8!}
For $ \mu $-a.s.\,$ \aaa $, $ \rR \in \N $, and $ m \in \zN $, the following hold: 
\\\thetag{i} 
$ (\EaRyym , \uLdaRyym ) $ 
 is a strongly local, quasi-regular Dirichlet form on 
 $ L^2 (\muaRym )$ for $ \mu $-a.s.\,$ \yy $. 
\\\thetag{ii}
$ \EuLdRmuam $ 
is a strongly local, quasi-regular Dirichlet form on $ \Lmuam $. 
\\\thetag{iii} $ (\EaRm ,\dORmuma ) $ is a strongly local closed form on $ \Lmuam $. 
\end{lemma}

Let $ \EuLdmuam $ be the increasing limit of $ \EuLdRmuam $, $ \rR \in \N $. 
\begin{lemma}	\label{l:8"}
For $ \mu $-a.s.\,$ \aaa $ and $ m \in \zN $, the following hold: 

\noindent \thetag{i} 
The form $ \EuLdmuam $ is a closed form on $ \Lmuam $ and is the strong resolvent
limit of $ \EuLdRmuam $ on $ \Lmuam $.

\noindent \thetag{ii} $ ( \Emuam , \cup_{\rR \in \N } \dORmuma ) $ is closable on $ \Lmuam $. 
The closure $ ( \Emuam , \dOmuma ) $ is the limit of $ ( \Emuam ,\dORmuma ) $ 
in the strong resolvent sense on $ \Lmuam $. 
\end{lemma}

Let $ \dbmuma = \{ f \in \dbm ; \Emuam (f,f) < \infty , f \in \Lmuam \} $. 
We define $ \dcmuma $ by replacing $ \db $ with $ \dc $. 
\begin{lemma} \label{l:8'} 
For $ \mu $-a.s.\,$ \aaa $ and $ m \in \zN $, the following hold: 
\\\thetag{i} 
$( \Emuam , \dbmuma ) $ is closable on $ \Lmuam $. 
\\\thetag{ii} 
$( \Emuam , \dcmuma ) $ is closable on $ \Lmuam $. 
\end{lemma}

Let $( \Emuam , \uLdmuamb ) $ and $( \Emuam , \oLdcmuma ) $ be the closure of 
$( \Emuam , \dbmuma ) $ and $( \Emuam , \dcmuma ) $ on $ \Lmuam $, respectively. 
\begin{lemma}	\label{l:84}
For $ \mu $-a.s.\,$ \aaa $ and each $ m \in \zN $, the following hold: \\
\thetag{i} $ ( \Emuam , \uLdmuamb ) = ( \Emuam ,\uLdmuam ) $. \\
\thetag{ii} $ ( \Emuam , \oLdcmuma ) = ( \Emuam ,\oLdmuam ) $. 
\end{lemma}
\begin{lemma} \label{l:83}
For $ \mu $-a.s.\,$ \aaa $ and $ m \in \zN $, the following hold: 
\\\thetag{i} 
$\EuLdmuam $ is a strongly local Dirichlet form on $ \Lmuam $. 
\\\thetag{ii}
$( \Emuam , \dOmuma ) $ is a strongly local, quasi-regular Dirichlet form on $ \Lmuam $. 
\end{lemma}

Similarly as $ \dRyy $ for $ \muRyyone $ in \lref{l:34}, we obtain  $ \dlog _{\aRy }$ for $ \muaRy ^{[1]} $. 
In \lref{l:8=}, we prove that $ \dlog _{\aRy }$ coincides with $ \dmu $ under $ \muaRy ^{[1]}$. 
\begin{lemma} \label{l:8=}
For $ \mu $-a.s.\,$ \aaa $, all $ \rR \in \N $, and $ \mua $-a.s.\,$ \yy $, $ \dlog _{\aRy } $ satisfies %the following: 
\begin{align}\label{:8=b}& 
\dlog _{\aRy } (\xs ) = \dmu (\xs + \pioLRc (\yy ) ) \quad \text{ for $ \muaRy ^{[1]}$-a.s.\,$ (\xs )$}
.\end{align}
\end{lemma}
\begin{proof}
By $ \muaRy ^{[1]} \mua (d\yy ) \mu (d\aaa ) = \muone $, we obtain, for $ h \in \GRone $,
\begin{align}\notag & %	\label{:81c}&
\int h (\xs ) \dlog _{\aRy } (\xs ) d\muaRy ^{[1]} (\xs ) 
\mua (d\yy ) \mu (d\aaa ) 
\\\notag &
= - \int \nablax h (\xs ) d\muaRy ^{[1]}(\xs ) \mua (d\yy ) \mu (d\aaa ) 
\quad \text{ by definition}
\\\notag &
= - \int \nablax h (\xs ) d\muone (\xs ) 
\quad \text{ by Fubini's theorem}
\\\notag &
= \int h (\xs ) \dmu (\xs ) d\muone (\xs ) 
\quad \text{ by definition}
 \\&\notag 
 = \int h (\xs ) \dmu (\xs ) d\muaRy ^{[1]} (\xs ) \mua (d\yy ) \mu (d\aaa ) 
\quad \text{ by Fubini's theorem}
.\end{align}
This implies \eqref{:8=b}. 
\PFEND

Let $ \EuLdRmuam $ be as in \lref{l:8!}. 
Let $ \uLX _{\aR }^{[m]} $ be the diffusion properly associated with $ \EuLdRmuam $ on $ \Lmuam $.% 
\begin{proposition}	\label{l:8A}
In addition to \As{A1}--\As{A3}, we assume \eqref{:14z}. 
Then for $ \mu $-a.s.\,$ \aaa $, there exists a fully labeled process $ \uLX _{\aaa } $ satisfying the following: 

\noindent \thetag{i} 
For each $ m \in \zN $, the associated $ m $-labeled process $ \uLXm _{\aaa }$ is the $ \muam $-symmetric and conservative continuous Markov properly associated with 
$ \EuLdmuam $ on $ \Lmuam $. 

\noindent \thetag{ii} 
For each $ m \in \zN $, the $ m $-labeled process $ \uLXm _{\aaa }$ and 
$ \uLX _{\aaa } $ satisfy 
\begin{align} \notag %	
\uLX _{\aaa } ^{[m]} &=\limiR \uLX _{\aR } ^{[m]} \quad \text{ in law in 
$ C([0,\infty);(\Rd )^m\ts \sSS )$}
,\\ \notag %	
 \uLX _{\aaa } &= \limiR \uLX _{\aR } \quad \text{ in law in $ \CRdN $}
.\end{align}
In particular, each tagged particle $ \uLxX _{\aaa }^i $ of $ \uLX _{\aaa }$ does not explode. 
\end{proposition}
\begin{proof}
The proof is the same as that of \tref{l:6A} and is omitted. 
\PFEND
\begin{proposition}	\label{l:8B}
For $ \mu $-a.s.\,$ \aaa $, the following hold. 

\noindent \thetag{i} 
There exists a $ \muam $-symmetric, conservative diffusion $ \oL{\X }_{\aaa }^{[m]} $ properly associated with $( \Emuam , \dOmuma ) $ on $ \Lmuam $ for each $ m \in \zN $. 

\noindent \thetag{ii} 
None of the labeled particles $\oLxX _{\aaa }^i$, $ 1 \le i \le m $, in $ \oL{\X }_{\aaa }^{[m]} $ explode or collide with other labeled particles. 

\noindent 
\thetag{iii} 
We write 
$ \oL{\X }_{\aaa }^{[m]} = (\oLxX _{\aaa }^1 ,\ldots, \oLxX _{\aaa }^m , \sum_{j>m }^{\infty} \delta_{\oLxX _{\aaa }^j} ) $. 
For each $ 1 \le i \le m $, 
\begin{align} &\notag % \label{:85b}&
 \oLxX _{\aaa }^i (t) - \oLxX _{\aaa }^i (0) = 
\int_0^t \sigma ( \oLxX _{\aaa }^i (u) ) dB_u^i 
 + \int_0^t \bbb ( \oLxX _{\aaa }^i (u) , \sum_{j\ne i}^{\infty} \delta_{\oLxX _{\aaa }^j (u)} ) du 
,\\& \notag 
 \oLxX _{\aaa }^i (t) - \oLxX _{\aaa }^i (0) =\half \Big\{ 
\int_0^t \sigma ( \oLxX _{\aaa } ^i (u) ) dB_u^i - \int_0^t \sigma ( \oLxX _{\aaa }^i (u) ) dB_u^i \circ r_t \Big\} 
,\end{align}
where $ r_t $ is the time reversal operator defined before \lref{l:61}. 
\end{proposition}
\begin{proof}
The proof is similar to that of \tref{l:7A}. Hence, it is omitted. 
\PFEND
Let $ \oLX _{\aaa } $ be the fully labeled process constructed in the same fashion as $ \oLX $ in \tref{l:7B}. 
Then $ \oL{\X }_{\aaa }^{[m]} $ in \tref{l:8B} is the $ m $-labeled process of $ \oLX _{\aaa } $. 
\begin{proposition}\label{l:86} 
For $ \mu $-a.s.\,$ \aaa $, 
$ \uLX _{\aaa }$ and $ \oLX _{\aaa } $ are weak solutions of \eqref{:12q} with initial distribution $ \varphi \mua \circ \lab ^{-1}$ satisfying \As{SIN}, \As{NBJ}, and \As{AC}$_{\mua }$. 
\end{proposition}
\begin{proof}
The proof of \pref{l:86} is identical to the combination of those for 
Theorems \ref{l:6B} and \ref{l:7B}. 
Hence, it is omitted. 
\PFEND
\subsection{Unique strong solution: Proof of \tref{l:13}} \label{s:82} 
From \pref{l:86}, we obtain weak solutions $ \uLX _{\aaa }$ and $ \oLX _{\aaa } $ of \eqref{:12q} with initial distribution $ \varphi \mua \circ \lab ^{-1}$, which satisfy \As{SIN}, \As{NBJ}, and \As{AC}$_{\mua }$. 

Our goal is to prove that $ \uLX _{\aaa }^{[m]} $ and $ \oLX _{\aaa }^{[m]} $
satisfy \As{IFC}. To this end, we apply \tref{l:IFC1}. More precisely, we
verify assumptions \Ass{B1} and \Ass{B2} of \tref{l:IFC1} for
$ \uLX _{\aaa } $ and $ \oLX _{\aaa } $. These assumptions concern the set
$ \Ha = \bigcup_{\mmm \in \NNNthree } \Ha _{\mmm } $ introduced in \eqref{:CUTu}--\eqref{:CUTc}. 

See \ssref{s:IFC2} for details on \Ass{B1} and \Ass{B2}. 
Once these assumptions are verified, \As{IFC} follows from \tref{l:IFC1}. 

Recall that \Ass{B1} is a condition imposed on a general RPF $ \nu $. 
\begin{lemma}\label{l:8X}
For $ \mu $-a.s.\ $ \aaa $, the processes $ \uLX _{\aaa } $ and $ \oLX _{\aaa } $ satisfy \Ass{B1} with $ \nu $ replaced by $ \mua $. 
\end{lemma}
\begin{proof} %G[
Let $ \Ki [\ak ] $ be the compact set defined in \eqref{:CUTw}. 
We can easily find a sequence $ \anest = \{ a_{\qqq } \}_{\qqq \in \mathbb{N} } $
as in \eqref{:ak} satisfying \eqref{:CUTy} and
\begin{align}\label{:8Xf}&
\sum_{\qqq = 1 }^{\infty} \qqq ^2 \mua ( \Ki [\ak ]^c ) < \infty
.\end{align}
Then \eqref{:8Xf} and \eqref{:CUTy} yield $ \mua ( \sSS \backslash \bigcup_{\qqq } \Kakk ) = 0 $. 
This proves \Ass{B1}\,\thetag{i}.

We take a $ \muaone $-version of $ \dmua $ as in \tref{l:11} and \pref{l:38}. 
Here $ \yy $ in \pref{l:38} is replaced by $ \aaa $. 
Then, for $ \mua $-a.s.\,$ \sss $, $ \dmua (\xs )$ is locally Lipschitz in $ \x $ on $ \Rd _{\neq }(\sss ) $. 
Let $ \SSneoneaa $ be as in \eqref{:38u}. Let 
\begin{align}\label{:8Xe}&
 \SSsde = \ulabzone (\SSneoneaa )
,\end{align}
where $ \ulabzone (x,\sss ) = \delta_x + \sss $. 
Then \Ass{B1}\,\thetag{ii} is satisfied. 
%GTP[---

In what follows, we denote both $ \uLX _{\aaa } $ and $ \oLX _{\aaa } $ by
the same symbol $ \X $, and denote by $ \XX $ the unlabeled dynamics
associated with $ \X $.
%GTP]

Let $ \SSs = \{ \sss \in \sSS ; \sss (\{ \x \} ) \in \{ 0,1 \} \text{ for all } \x \in \Rd \} $. 

Note that $ \Ki [\ak ] $ is increasing concerning $ \qqq $ from \eqref{:CUTy}. 
 Let 
\begin{align} \notag &%\label{:}&
 \kappa _{\qqq } = \inf \{ t > 0 ; \XX _t \notin \Ki [\ak ] \} 
.\end{align}
According to \eqref{:CUTt}--\eqref{:CUTc}, condition \As{B1} \thetag{iii} follows from
\begin{align}\label{:8Xp}
\pP ( \XX _t \in \SSs \text{ for all } 0 \le t < \infty ) = 1 ,
\\\label{:8Xq}
P \bigl( \lim_{\qqq \to \infty} \kappa _{\qqq } = \infty \bigr) = 1 
.\end{align}
We have already proved \eqref{:8Xp} in \pref{l:41}. Thus, it only remains to prove \eqref{:8Xq}. 
To verify \eqref{:8Xq}, we use \lref{l:IFC)}. 

We check all assumptions of \lref{l:IFC)} for $ \mua $: 
\Ass{UB}, \Ass{MF}, \Ass{BX}, $ \{\mathbf{S}_{\mua }\}$, and \eqref{:IFC3g}. See \ssref{s:IFC2}. 

\Ass{UB} follows from \eqref{:12o}. 

 \Ass{MF} is clear since $ \XX $ is associated with a Dirichlet form on $ \Lma $. 

$ \{\mathbf{S}_{\mua }\}$ is also clear because the $ m $-labeled process $ \Xm $ given by $ \X $ is associated with a symmetric Dirichlet form on $ \Lmuam $. 

Let $ \nu = \varphi \mua $ and $ \qQ _{\nu } $ be the probability such that $ \XB $ is defined on $ (\Omega , \mathcal{F}, \qQ _{\nu } , \{ \mathcal{F}_t \} )$. 
\Ass{BX} makes the following statement: for $ \qQ _{\nu } $ -a.s., 
\begin{align}\label{:8Xd}&
\sigma [\mathbf{B}_s ; s\le t] \subset \sigma [\X _s ;s\le t ] , \quad t \ge 0 
.\end{align}
 Let $ \mathbf{M} = (M^i) _{i=1}^{\infty}$ such that $ M_t^i = \int_0^t \sigma (X_u^i) dB_u^i $. 
Because $ \aaaa = \sigma ^t \sigma $, $ \sigma \in C_{b}^1 (\Rd )$, and $ \aaaa $ is uniformly elliptic, we can represent $ \mathbf{B}=(B^i)_{i=1}^{\infty}$ as a function of $ \mathbf{M}$. Thus, $ \mathbf{B} $ is a $ \sigma [ \mathbf{M}_s;s\le t ]$-adapted process. Because $ \XB $ is a solution of ISDE \eqref{:12q}, 
\begin{align}\label{:8Xe}&
M_t^i = 
\int_0^t \sigma (X_u^i) dB_u^i = X_t^i - X_0^i - 
\int_0^t \bbb ( X_u^i , \sum_{j\ne i}^{\infty} \delta_{X_u^j}) du 
.\end{align}
From \eqref{:8Xe}, $ \mathbf{M}$ is a $ \sigma [\X _s ;s\le t ]$-adapted process. 
Hence, we obtain \eqref{:8Xd}.

By \lref{l:IFC5} and \eqref{:8Xf}, we may use \eqref{:IFC3g}. 
Thus all assumptions of \lref{l:IFC)} are satisfied, and hence \eqref{:8Xq} follows. 
This proves \Ass{B1}\,\thetag{iiii}. 
\PFEND

%G[---

Let $ \Han $ be as in \eqref{:CUTc}. 
We choose $ \mathbf{a} $ in the definition of $ \Han $ as in \lref{l:8X} so that $ \Han $ satisfies \Ass{B1}. 

\begin{lemma} \label{l:8Y} 
For $ \mu $-a.s.\ $ \aaa $, $ \Han $ satisfies \Ass{B2}. 
\end{lemma}
%G]
\begin{proof}
Let $ (\sigma , \bbb ) $ be the coefficient of ISDE \eqref{:12q}. 
We define $ \sigma ^m $ and $ \bbb ^m $ by 
\begin{align} \label{:87d}&
\sigma ^m (\mathbf{x}, \sss ) = (\sigma (\x ^i ))_{i=1}^m , \quad 
 \bbb ^m (\mathbf{x}, \sss ) = \Big( \bbb ( \x ^i, \sum_{j\ne i}^m \delta _{ \xj } + \sss ) \Big)_{i=1}^m
\end{align}
for $ \mathbf{x} = (x^1,\ldots,x^m)$ and $ \sss \in \sSS $. 

%GTP[
We verify that $ \sigma^m $ and $ \bb^m $ in \eqref{:87d}
satisfy \eqref{:B2a} and \eqref{:B2b}.
%GTP]

By assumption, $ \sigma \in C_b^1(\Rd ) $ and that $ \sigma $ is independent of $ \sss $. 
Hence, $ \sigma ^m $ is Lipschitz continuous, which yields \eqref{:B2a}. 

By \eqref{:35c}, we have $ \Rsl \in \SIX $. From \eqref{:11b}, we obtain 
\begin{align} &\notag %\label{:87f}&
\dmu (\xs ) = \beta \Big( - \nablaPhi ( \x ) - \sum_{\sioLSR } \2 + 
\1 \CRi \x ^{\mathbf{i}}+ \Rsl (\x ) \Big)
.\end{align}
Hence, $ \dmu $ is Lipschitz continuous on $\Han $ from \eqref{:CUTu} and \eqref{:CUTc}. 
Then from $ \aaaa \in C_b^2 (\Rd ) $, $ \bbb = \half\{\aaaadmu \} $ is Lipschitz continuous on $ \Han $. 
Hence, by \eqref{:87d}, $ \bbb ^m $ is Lipschitz continuous on $ \Han $, which yields \eqref{:B2b}. 
\PFEND

\begin{proposition} \label{l:87} 
For $ \mu $-a.s.\,$ \aaa $, $ \uLX _{\aaa }$ and $ \oLX _{\aaa } $ satisfy \As{IFC}. 
\end{proposition}
\begin{proof}
By \lref{l:8X}, $ \uLX _{\aaa }$ and $ \oLX _{\aaa } $ satisfy \Ass{B1}. 
By \lref{l:8Y}, $ \Han $ satisfies \Ass{B2}. 
Hence, \As{IFC} follows from \tref{l:IFC1}. 
\PFEND

\begin{theorem}\label{l:8Z}
For $ \mu $-a.s.\ $ \aaa $, the following hold:

\smallskip
\noindent \thetag{i}
For $ \varphi \mua \circ \lab ^{-1} $-a.s.\ $ \mathbf{s} $, 
the processes $ \uLX _{\aaa } $ and $ \oLX _{\aaa } $ are the unique strong solutions of \eqref{:12q} 
starting at $ \mathbf{s} $ under the constraints 
\As{SIN}, \As{NBJ}, \As{AC}$_{\mua }$, \As{MF}, and \As{IFC}. 
Moreover, the labeled dynamics $ \uLX _{\aaa } $ and $ \oLX _{\aaa } $ are indistinguishable.

\smallskip
\noindent \thetag{ii}
The $ m $-labeled processes 
$ \uLX _{\aaa }^{[m]} $ and $ \oLX _{\aaa }^{[m]} $, 
induced by $ \uLX _{\aaa } $ and $ \oLX _{\aaa } $, respectively, 
are conservative, $ \muam $-symmetric diffusions properly associated 
with the same strongly local, quasi-regular Dirichlet form on $ \Lmuam $:
\begin{align}\label{:87h}&
( \Emuam , \uLdmuam ) = ( \Emuam , \oLdmuam )
.\end{align}
Here $ ( \Emuam , \uLdmuam ) $ and $ ( \Emuam , \oLdmuam ) $ are given by \lref{l:8"}.
\end{theorem}

\begin{proof}
Applying \tref{l:US1} to $ \mua $, we shall prove \tref{l:8Z}. Hence, we verify the assumptions of \tref{l:US1} to $ \mua $. 
Note that $ \varphi \mu $ has a tail decomposition such that $ (\varphi \mu )_{\aaa } = \varphi \mua $ for $ \mu $-a.s.\,$ \aaa $. 
By construction, $ \mua $ is tail trivial. 

Applying \pref{l:86} to $ \mua $, we have weak solutions $ \uLX _{\aaa } $ and $ \oLX _{\aaa } $ to \eqref{:12q} with initial distribution $ \varphi \mua \circ \lab ^{-1}$ that satisfy \As{SIN}, \As{NBJ}, and \As{AC}$_{\mua }$. 
From \pref{l:87}, $ \uLX _{\aaa } $ and $ \oLX _{\aaa } $ satisfy \As{IFC}. 
Thus, all assumptions of \tref{l:US1} are fulfilled. 
Hence, according to \tref{l:US1}, $ \uLX _{\aaa } $ and $ \oLX _{\aaa } $ are unique strong solutions of \eqref{:12q} under the constraints \As{SIN}, \As{NBJ}, \As{AC}$_{\mua }$, \As{MF}, and \As{IFC}. 
In particular, $ \uLX _{\aaa } $ and $ \oLX _{\aaa } $ are indistinguishable. 
This implies \thetag{i}. 

Let $ \uLXm _{\aaa } $ and $ \oLXm _{\aaa } $ be the $ m $-labeled processes of $ \uLX _{\aaa } $ and $ \oLX _{\aaa } $, respectively. 
Because $ \uLX _{\aaa } $ and $ \oLX _{\aaa } $ are indistinguishable, we obtain that $ \uLXm _{\aaa } $ and $ \oLXm _{\aaa } $ are indistinguishable. 

From Propositions \ref{l:8A} and \ref{l:8B}, $ \uLXm _{\aaa } $ and $ \oLXm _{\aaa } $ are associated with the $ m $-labeled Dirichlet forms $ ( \Emuam , \uLdmuam ) $ and $ ( \Emuam , \oLdmuam ) $ on $ \Lmuam $, respectively. 
Because $ \uLXm _{\aaa } $ and $ \oLXm _{\aaa } $ are indistinguishable, we obtain \eqref{:87h}. 

From \tref{l:8B}, $ ( \Emuam , \oLdmuam ) $ is a strongly local, conservative, and quasi-regular Dirichlet form on $ \Lmuam $. Hence, from \eqref{:87h}, $ ( \Emuam , \uLdmuam ) $ is also a strongly local, conservative, and quasi-regular Dirichlet form on $ \Lmuam $. This implies \thetag{ii}. 
\PFEND

%G[
\begin{remark}\label{r:8Z}
The method of using the uniqueness of strong solutions of ISDEs
to prove the uniqueness of extensions of Dirichlet forms
for interacting Brownian motions
is due to Tanemura \cite{tane.2}.
This approach was subsequently adopted in \cite{k-o-t.udf}.
\end{remark}
%G]
%G[
\noindent {\bf Proof of \tref{l:13}. } 
Let $ \uLX _{\aaa } $ and $ \oLX _{\aaa } $ be as in \tref{l:8Z}. 
By \tref{l:8Z}, they are indistinguishable. 
Hence, we write $ \X := \uLX _{\aaa } = \oLX _{\aaa } $. 

By \tref{l:8Z}, for $ \mu $-a.s.\ $ \mathsf{a} $ and for $ \varphi \mua \circ \lab ^{-1} $-a.s.\ $ \mathbf{s} $, 
\eqref{:12q} admits a unique strong solution $ \X $ starting at $ \mathbf{s} $ under the constraints 
\As{SIN}, \As{NBJ}, \As{AC}$_{\mua }$, \As{MF}, and \As{IFC}. 
Since $ \varphi $ ranges over all functions satisfying \eqref{:14y}, this remains valid for 
$ \mathbf{s} = \lab ( \sss ) $ for $ \mua $-a.s.\ $ \sss $. 

For each $ m \in \zN $, the $ m $-labeled process $ \Xm $ induced by $ \X $ is properly associated with the strongly local, quasi-regular Dirichlet form in \eqref{:87h} on $ \Lmuam $. 
Hence, $ \Xm $ is a $ \muam $-symmetric diffusion. 
By \pref{l:8A}, $ \Xm $ is conservative. 
This completes the proof of \tref{l:13}. 
\qed
%G]

\section{The diffusion on $ \RdN $: Proof of \tref{l:1Y}}\label{s:9}

The aim of this section is to prove the uniqueness of extensions of Dirichlet forms 
and to construct the diffusion on $ \RdN $.
We first establish the uniqueness of Dirichlet forms and, using this result, show that the fully labeled process
$ \X = ( \xX ^i )_{ i \in \N } $ constructed in \tref{l:13} is a conservative diffusion on $ \RdN $.
This completes the proof of \tref{l:1Y}. 

Unless stated otherwise, all results in \sref{s:9} are proved under Assumptions \As{A1}--\As{A3}.

\subsection{Uniqueness of extensions of Dirichlet forms}\label{s:9A}
We now analyze the Dirichlet forms directly associated with solutions of the ISDE \eqref{:12q} 
to prove the uniqueness of extensions of Dirichlet forms asserted in \tref{l:9A}.
%G]

Let $ \mua $ be as in \dref{d:15}. Let $ \muam $ be the $ m $-Campbell measure of $  \mua $. Recall that 
\begin{align} &  \label{:58y}
\muam = \int_{\sSS } \muaRym \mu (d \yy )
. \end{align}%G[
%G[
Let $ ( \EaRyym , \uLdaRyym ) $ be as in \lref{l:8!}. 
Let $ \uLEDaRstarm $ be the superposition of $ ( \EaRyym , \uLdaRyym ) $ 
with respect to $ \yy $ under the measure $ \mua $:
%G]
\begin{align} \notag %\label{:1Xs}
& \EaRstarm ( f , g ) = \int_{\sSS } \EaRyym ( f , g ) \mua (d\yy ) , 
\\\notag & 
\uLDaRstarm = \{ f \in \bigcap_{\mua\text{-a.s.\,}\yy \in [\aaa ] } \uLdaRyym ; 
\EaRstarm ( f , f ) < \infty , f \in L^2 (\muam ) \} 
,\end{align}
where $  [\aaa ]  $ is the set in \dref{d:15}. We recall that $ \mua ( [\aaa ]  ) = 1 $. 
By the general theorem on the superposition of closed forms \cite{b-h}, 
$ \uLEDaRstarm $ is a closed form on $ L^2 (\muam ) $. Let $ \EuLdRmuam $ be as in \lref{l:8!}.

\begin{lemma} \label{l:93}
Let $ m \in \{ 0 \} \cup \N $. For $ \mu $-a.s.\,$ \aaa $ and all $ \rR \in \N $, 
\begin{align} & \label{:93a} 
\EuLdRmuam = \uLEDaRstarm 
.\end{align}
\end{lemma}
\begin{proof}
Let $ \uLXaRm $ and $ \uLX _{\aRy }^{[m]} $ be the $ m $-labeled diffusions associated with the Dirichlet forms 
$ ( \EaRm , \uLdaRm ) $ on $ \Lmuam $ and $ ( \EaRyym , \uLdaRyym ) $ on $ L^2 ( \muaRym ) $, respectively. 
Let $ \uLXaR $ and $ \uLX _{\aRy } $ be the fully labeled processes associated with the diffusions 
$ \uLXaRm $ and $ \uLX _{\aRy }^{[m]} $, respectively. 
By the argument preceding \lref{l:61}, such processes exist and are unique. 

By \lref{l:61}, both fully labeled processes $ \uLXaR $ and $ \uLX _{\aRy } $ 
are weak solutions of the pair of SDEs \eqref{:14t} and \eqref{:14u}. 
From the representation of the logarithmic derivative $ \dRyy (\xs ) $ given by \eqref{:2!z}, the pair of SDEs \eqref{:14t} and \eqref{:14u} has a unique strong solution.

Indeed, the coefficients given by the logarithmic derivative are smooth outside 
$ \mathscr{C}^{[1]} = \{ ( \xs ) \in \RdSS \,;\, \x = \si \text{ for some } i \} $, where $ \sss = \sum_i \delta _{\si } $, and each particle is subject to the reflecting boundary condition on $ \partial \SR $. 

Hence, \eqref{:14t} admits a unique strong solution until the solution hits $ \mathscr{C}^{[1]} $, which never occurs by \lref{l:61} \thetag{iv}. 
Here, \lref{l:61} \thetag{iv} is clearly valid for $ ( \E _{\aRy }^{[1]} , \uLd _{\aRy }^{[1]} ) $ on 
$ L^2 ( \muaRy ^{[1]} ) $ as well, similarly to $ ( \E _{\aR }^{[1]} , \uLd _{\aR }^{[1]} ) $ on 
$ L^2(\mua ^{[1]}) $. 
Thus, \eqref{:14t} admits a unique strong solution. 

Equation \eqref{:14u} means that all particles outside $ \oLSR $ are frozen. 
Hence, \eqref{:14u} clearly admits a trivial unique strong solution. 

%G[

Let $ \CnfRy $ be as in the proof of \lref{l:53}. 
Then $ \CnfRy $ is an invariant set of the diffusion $ \uLX _{\aRy }^{[m]} $ properly associated with 
$ ( \EaRyym , \uLdaRyym )  $ on $ L^2 ( \muaRym ) $. 

Recall that $ \uLEDaRstarm $ is the superposition of $ ( \EaRyym , \uLdaRyym )  $ 
with respect to $ \yy $ under $ \mua $, and that $ \CnfRy $ yields a partition of $ \RdmSS $. 
Hence the family of diffusions $ \uLX _{\aRy }^{[m]} $ associated with 
$ ( \EaRyym , \uLdaRyym )  $ on $ L^2 ( \muaRym ) $ 
gives rise to a diffusion associated with $ \uLEDaRstarm $ on $ \Lmuam $. 
Here we used \eqref{:58y}.  
%G]

As observed at the beginning of the proof, 
$ \uLX _{\aRy }^{[m]} $ is the unique solution of \eqref{:14t} and \eqref{:14u}. 
Hence, for each $ \yy \in \sSS $, the diffusion obtained from 
$ \uLEDaRstarm $ is this unique solution.
%G]

Recall that $ \uLXaR $ is a weak solution of SDEs \eqref{:14t} and \eqref{:14u}. 

Hence, the diffusion $ \uLX _{\aRy }^{[m]} $, $ \yy \in \sSS $, obtained from $ \uLEDaRstarm $ and the diffusion $ \uLXaRm $ associated with $ ( \EaRyym , \uLdaRyym ) $ coincide. 
This implies the identity of the Dirichlet forms in \eqref{:93a}. 
\PFEND

We see that $ \uLEDaRstarm $, $ \rR \in \N $, is increasing. 
Let $ \uLEDastarm $ be the increasing limit of $ \uLEDaRstarm $. 
By \lref{l:ext3} \thetag{i}, $ \uLEDastarm $ is a closed form on $ L^2(\muam )$. 
Let $ \EuLdmuam $ be as  in \lref{l:8"}. 
\begin{lemma}	\label{l:94}
For each $ m \in \zN $ and $ \mu $-a.s.\,$ \aaa $, 
\begin{align} \label{:94a} 
 ( \Emuam , \uLdmuam ) &= \uLEDastarm 
.\end{align}
\end{lemma}
\begin{proof}%G[
Equations \eqref{:94a} follows from \eqref{:93a}. 
%G]
\PFEND

Recall that $ \uLEDaRstarm $ is a closed form on $ L^2(\mu _{\aRy }^{[m]} ) $ and that 
$ \dcbm \subset  \uLdaRyym $. 
Then $ (\E _{\aRy }^{[m]} , \dcbm )$ is closable on $ L^2(\mu _{\aRy }^{[m]} ) $. 
Let $ ( \E _{\aRy }^{[m]} , \oLd _{\aRy }^{[m]}  ) $ be its closure. 
Let $ ( \E _{\aR \star }^{[m]} , \oLd _{\aR \star }^{[m]}  ) $ be the superposition of 
$ ( \E _{\aRy }^{[m]} , \oLd _{\aRy }^{[m]}  ) $ with respect to $ \yy $ under $  \mu $. 

Let $ ( \ER ^{[m]} , \oLd _{\star , \rR }^{[m]} ) $ and 
$ (\ER , \oLd _{\star , \rR \star }^{[m]})$  be the superpositions of 
$ \EoLdRmuam $ and $ \oLEDaRstarm $, respectively, with respect to $ \aaa $ under $ \mu $. 
\begin{lemma} \label{l:95}
Let $ m \in \{ 0 \} \cup \N $ and $ \rR \in \N $. Then the following hold:
\begin{align} & \label{:95a} 
\EoLdRmuam = \oLEDaRstarm 
,\\\label{:95b} & 
( \ER ^{[m]} , \oLd _{\star , \rR }^{[m]} ) 
= ( \ER ^{[m]} , \oLd _{\star , \rR \star }^{[m]} ) 
.\end{align}
\end{lemma}

\begin{proof}
Since the carr\'{e} du champ of $  \ER ^{[m]}$ vanishes outside $ \oLSRc $, we obtain \eqref{:95a}. 
Equation \eqref{:95b} follows immediately from \eqref{:95a}.
\PFEND

Let $ ( \Emuam , \dOmuma ) $ be the closed form given in \lref{l:8"}. 
Let $ ( \Emuam , \dom _{\aaa , \mathrm{upr} \star }^{[m]} ) $ be the decreasing limit of $ \oLEDaRstarm $. 
\begin{lemma} \label{l:96}
Let $ m \in \{ 0 \} \cup \N $ and $ \rR \in \N $. Then the following hold:
\begin{align}\label{:96a}&
( \Emuam , \dOmuma ) = ( \Emuam , \dom _{\aaa , \mathrm{upr} \star }^{[m]} ) 
.\end{align}
\end{lemma}

\begin{proof}
 \eqref{:96a} follows immediately from \lref{l:8"} and \eqref{:95b}.
\PFEND

Let $ ( \Emum , \dOmumstar ) $, 
$ ( \Emum , \dom _{\star , \mathrm{upr} \star }^{[m]} ) $, 
$ ( \Estarm , \dom _{\star , \mathrm{lwr}}^{[m]} ) $, and 
$ ( \E ^{[m]} , \dom _{\star , \mathrm{lwr} \star }^{[m]} ) $ 
be the superpositions of 
$ ( \Emuam , \dOmuma ) $, $ ( \Emuam , \dom _{\aaa , \mathrm{upr} \star }^{[m]} ) $, 
$ \EuLdmuam $, and $ \uLEDastarm $, respectively, 
with respect to $ \aaa $ under $ \mu $. 
\begin{align}\notag
\Estarm ( f , g ) 
& = \int _{\sSS } \E _{\aaa }^{[m]} ( f , g ) \, \mu ( d\aaa ) , 
\\\notag
\dOmumstar 
& = \Bigl\{ f \in \bigcap _{\mu \text{-a.s.\,}\aaa } \dom _{\aaa , \mathrm{upr}}^{[m]} \, ; \,
\Estarm ( f , f ) < \infty , \, f \in L^2 ( \mum ) \Bigr\} 
.\end{align}
%G]
%
Replacing $ ( \Emuam , \dOmuma ) $ with $ ( \Emuam , \dom _{\aaa , \mathrm{upr}* }^{[m]})$, 
$ \EuLdmuam $, and $ \uLEDastarm $, we obtain the superpositions 
$ ( \Estarm , \dom _{\star , \mathrm{upr}\star }^{[m]} ) $, 
$ ( \Estarm , \dom _{\star , \mathrm{lwr}}^{[m]} ) $ and 
$ ( \E ^{[m]} , \dom _{\star , \mathrm{lwr} \star }^{[m]} ) $, respectively.

The following theorem shows that these Dirichlet forms coincide and are strongly local and quasi-regular. 
\begin{theorem}[Uniqueness of extensions of Dirichlet forms]	\label{l:9A} 

\noindent \thetag{i} 
Let $ \SSz $ be the set in \dref{d:15}. For $ \mu $-a.s.\,$ \aaa \in \SSz $, 
\begin{align} \label{:9Aa} 
 ( \Emuam , \dOmuma ) = ( \Emuam , \dom _{\aaa , \mathrm{upr} \star }^{[m]} ) 		
=( \Emuam ,\uLdmuam ) 
= \uLEDastarm 
,\\ \label{:9Ab} 
 (\Emum , \dOmumstar  )  = ( \Emum , \dom _{\star , \mathrm{upr} \star }^{[m]} ) 
=  (\Emum , \dom _{\star , \mathrm{lwr} \star }^{[m]}  ) 
= (\Emum , \dom _{\star , \mathrm{lwr} \star }^{[m]} ) 
.\end{align} 

\noindent \thetag{ii} 
All Dirichlet forms in \eqref{:9Aa} and \eqref{:9Ab}  are strongly local and quasi-regular on $ \Lmuam $ and $ \Lmm $, respectively.  
\end{theorem}

\begin{proof} % [Proof of \tref{l:9A}] 
From \eqref{:96a}, \eqref{:87h}  and \eqref{:94a}, we obtain \eqref{:9Aa}. 
 Taking the superposition of \eqref{:9Aa} with respect to $ \aaa $ under $ \mu $ yields \eqref{:9Ab}. 
This proves \thetag{i}. 

By \lref{l:83} \thetag{ii}, $( \Emuam , \dOmuma ) $ is strongly local and quasi-regular on $ \Lmuam $. 
Hence, all Dirichlet forms in \eqref{:9Aa} are strongly local and quasi-regular. 

%GTP[ 
Let $ [ \aaa ] = \{ \bb \in \SSz ; \mua = \mub \} $ be as in \dref{d:15}. 
We have $ \mua ( [ \aaa ] ) = 1 $, and $ [ \aaa ] \cap [ \bb ] = \emptyset $ if $ [ \aaa ] \ne [ \bb ] $.
Note that $ [\aaa ]$ gives a partition of $ \SSz $ and $ \mu (\SSz ) = 1 $ by construction. 
%GTP]

Let $ \mathrm{Cap}_{\aaa } ^{[m]}$ be the capacity given by $( \Emuam , \dOmuma ) $ on 
$ L^2 (\Rdm \ts [\aaa ]  ,\mua ) $. 
There exists an increasing sequence of compact sets $ \mathsf{K}_{\aaa }^n $ satisfying 
\begin{align}& \label{:9Af} 
 \mathsf{K}_{\aaa }^n \subset [\aaa ] , \quad 
\limi{n} \mathrm{Cap}_{\aaa }^{[m]} (\{\Rdm \ts \mathsf{K}_{\aaa }^n \}^c) = 0 
.\end{align}
Indeed, \eqref{:9Af} follows from \As{Q1} in \dref{d:QR}. 
Here, we take, for each $ n \in \N $, 
\begin{align}\label{:9Ag}&
\mathsf{K}_{\aaa }^n = \mathsf{K}_{\bb }^n 
\quad \text{ if and only if }
[\aaa ] = [\bb ] 
.\end{align}

Let $ \mathsf{P} _{\aaa }^{[m]} $ be the diffusion properly associated with 
the Dirichlet form $ ( \Emuam , \dOmuma ) $ on $ \Lmuam $, 
which is strongly local and quasi-regular by \lref{l:83} \thetag{ii}. 
From \eqref{:9Af}, the state space of $ \mathsf{P}_{\aaa }^{[m]}$ can be taken to be disjoint for $ [\aaa ] \ne [\bb ]$. 
Thus, the collection of diffusions $ \mathsf{P}_{\aaa }^{[m]}$, $ \aaa \in \SSz $, gives the diffusion properly associated with $ (\Emum , \dOmumstar  ) $ on $ \Lmm $. 

Hence, from \cite[Th.\,1.5.3]{c-f}, $ (\Emum , \dOmumstar  ) $ on $ \Lmm $ is quasi-regular. 
The strong locality of $ (\Emum , \dOmumstar  ) $ is clear by construction. 
Thus, all Dirichlet forms in \eqref{:9Ab} are strongly local and quasi-regular on $ \Lmm $. 
This completes the proof of \thetag{ii}. 
\PFEND

\tref{l:9A} plays a crucial role in the proof of Theorems \ref{l:1Y}--\ref{l:16}. 
We present further applications of \tref{l:9A} in the following remark.
\begin{remark}\label{r:Xb} 
%G[
\noindent \thetag{i} 
The sub-diffusivity of tagged particles in the Ginibre interacting Brownian motion
($ d = \beta = 2 $) follows from \eqref{:9Aa} \cite{o.Gin}, reflecting the long-range nature
of the two-dimensional Coulomb potential. 
In contrast, tagged particles are always diffusive for Ruelle-class potentials with
a convex hardcore in dimensions $ d \ge 2 $ \cite{o.p}.

%G] 
%G[
\noindent \thetag{ii} 
Suzuki proved that the unlabeled Markov processes associated with a Dirichlet form are ergodic if the reference measure $ \mu $ is number rigid and the Dirichlet form is locally ergodic \cite{suzu.erg}. 
This theorem can be applied to the Ginibre RPF. 
In \cite{suzu.erg}, Suzuki considered the Dirichlet form that corresponds to $ (\E ^{[m]} , \dom _{\star , \mathrm{lwr} \star }^{[m]} ) $. 
By \tref{l:9A}, the continuous Markov process considered in \cite{suzu.erg} is an ergodic diffusion. 

\noindent \thetag{iii} Suzuki \cite{suzu.curv} studied curvature bounds for Dyson's Brownian motion, an infinite-particle system on $ \R $ with the two-dimensional Coulomb interaction at inverse temperature $ \beta = 2 $, that is, with a Riesz potential. He analyzed Dirichlet forms of the type $ (\E ^{[m]} , \dom _{\star , \mathrm{lwr} \star }^{[m]} ) $. Since our argument for \tref{l:9A} may be valid for Riesz potentials, his results can apply to Dyson's model studied in \cite{o.isde,o-t.tail} by the uniqueness of Dirichlet forms. 
%G]
\end{remark}
%---
\begin{remark}\label{r:93} 
Kawamoto proved the identity \eqref{:93z} for Dyson's model \cite{kawa.dyson}:
\begin{align} \label{:93z}&
( \Emum , \dOmum ) = ( \Emum , \dom _{\star , \mathrm{upr} }^{[m]} ) 
.\end{align}
This identity, together with the uniqueness theorem for solutions of ISDEs in \cite{o-t.tail}, 
implies that the tail $ \sigma $-field of $ \mu $ is invariant under the dynamics of the associated diffusion. 
Recently, Kawamoto has announced an extension of this identity to a broader class of RPFs 
(see \cite{kawa.tail}). 
It is therefore plausible that the same identity also holds for the Coulomb RPFs.
\end{remark}

\subsection{Proof of \tref{l:1Y}} \label{s:9B}
Utilizing the uniqueness results of solutions of ISDEs (\tref{l:8Z}) and Dirichlet forms (\tref{l:9A}), we establish \tref{l:1Y} in this subsection. The goal is to prove that the fully labeled process $ \X $ in \tref{l:13} is an $ \RdN $-valued diffusion. 

Let $ \uLX _{\aaa } $ and $ \oLX _{\aaa } $ be as in \tref{l:8Z}. 
In the proof of \tref{l:13}, we constructed $ \X $ as $ \X := \uLX _{\aaa } = \oLX _{\aaa } $ by proving $ \uLX _{\aaa } $ and $ \oLX _{\aaa }$ are indistinguishable. 
Let $ \Xm $ be the $ m $-labeled process of $ \X $, see \eqref{:12u}. 
%G[

In \tref{l:13}, we have already proved that $ \Xm $ is a $ \muam $-symmetric diffusion 
for $ \mu $-a.s.\,$ \aaa $. 
However, this does not necessarily imply that the family $ \Xm $ indexed by $ \aaa \in \sSS $ 
is a $ \mum $-symmetric diffusion. 

Indeed, let $ \widetilde{\mathbf{S}} _{\aaa }^{[m]} $ be an invariant set of the 
$ \muam $-symmetric diffusion $ \Xm $. 
Then 
$ \muam ( \RdmSS \backslash \widetilde{\mathbf{S}} _{\aaa }^{[m]} ) = 0 $. 
However, it does not necessarily hold that
\begin{align}\label{:9Bu}
\widetilde{\mathbf{S}} _{\aaa }^{[m]} \cap \widetilde{\mathbf{S}} _{\bb }^{[m]} = \emptyset 
\quad \text{whenever } [ \aaa ] \ne [ \bb ] 
.\end{align}
where $ [ \aaa ] $ is the equivalence class introduced in \eqref{:13q}.

%G]

%G[---

If \eqref{:9Bu} holds, then the family $ \Xm $ indexed by $ \aaa $ is a diffusion 
with state space $ \cup _{\aaa } \widetilde{\mathbf{S}} _{\aaa }^{[m]} $, 
because each $ \widetilde{\mathbf{S}} _{\aaa }^{[m]} $ is an invariant set 
of the conservative diffusion $ \Xm $. 
In the following theorem, we prove \eqref{:9Bu} for $ \mathbf{K} _{\aaa }^{[m]} $, 
defined in the proof, and show that the $ m $-labeled process $ \Xm $ 
is a diffusion with state space $ \RdmSS $.

%G]
% 
\begin{proposition} \label{l:9B} 
\noindent \thetag{i} For $ \mu $-a.s.\,$ \aaa $, 
$ \Xm $ is a $ \muam $-symmetric, conservative diffusion properly associated with the Dirichlet form on $ \Lmuam $ in \eqref{:9Aa} for each $ m \in \zN $. 

\noindent \thetag{ii}
$ \Xm $ is a $ \mum $-symmetric, conservative diffusion properly associated with the Dirichlet form on $ \Lmm $ in \eqref{:9Ab} for each $ m \in \zN $. 
\end{proposition}
\begin{proof}
By \tref{l:8Z} \thetag{ii}, $ \Xm $ is a $ \muam $-symmetric diffusion 
properly associated with $ ( \Emuam ,\uLdmuam ) $ on $ L^2 (\muam ) $. 
Hence by \eqref{:9Aa}, we obtain \thetag{i}.  

Let $ \mathsf{K}_{\aaa }^n $ be the compact set given in the proof of \tref{l:9A}. Let $ [\aaa ]$ be as in \eqref{:13q}. 
$ \mathsf{K}_{\aaa }^n $ is increasing in $ n $ for each $ \aaa $ and that $ \mathsf{K}_{\aaa }^n \subset [\aaa ] $ and 
\begin{align}\label{:9Bg}&
\Rdm \ts \mathsf{K}_{\aaa }^n \subset \Rdm \ts [\aaa ] 
,\quad n \in \N 
.\end{align}

Let $ \mathbf{K}_{\aaa }^{[m]} = \Rdm \ts \cup_{n=1}^{\infty} \mathsf{K}_{\aaa }^n $. 
By \eqref{:9Ag}, $ \mathbf{K}_{\aaa }^{[m]} = \mathbf{K}_{\bb }^{[m]} $ if and only if $ [\aaa ] = [\bb ] $. 
By \eqref{:9Bg}, $ \mathbf{K}_{\aaa }^{[m]} \cap \mathbf{K}_{\bb }^{[m]} = \emptyset $ if $ [\aaa ] \ne [\bb ]$. 

By the proof of \tref{l:9A}, $ \mathbf{K}_{\aaa }^{[m]} $ is the state space of the diffusion properly associated with $ ( \Emuam , \dOmuma ) $ on $ L^2 ( \Rdm \ts [\aaa ] ,\muam ) $. 
By \tref{l:8Z} \thetag{ii}, $ \uLX _{\aaa }^{[m]} $ is properly associated with $ ( \Emuam , \oLdmuam ) $. 
Thus, $ \mathbf{K}_{\aaa }^{[m]} $ is an invariant set of the diffusion $ \uLX _{\aaa }^{[m]}$.

Since $ \uLXm $ coincides with $ \uLX _{\aaa }^{[m]}$ on $ \mathbf{K}_{\aaa }^{[m]} $, $ \uLXm $ is a $ \mum $-symmetric, conservative diffusion properly associated with the superposition $ (\Emum , \dOmumstar  ) $ of $ ( \Emuam , \oLdmuam ) $ %
 on $ \Lmm $ with the state space $ \sqcup _{[\aaa ]} \mathbf{K}_{\aaa }^{[m]} $. 
Here $ \{ \mathbf{K}_{\aaa }^{[m]} \} _{ [\aaa ]}$ gives a partition of the state space into $ \mathbf{K}_{\aaa }^{[m]} $. Each $ \mathbf{K}_{\aaa }^{[m]} $ is an invariant sets of $ \uLXm $. % 
We have thus proved \thetag{ii}.
\PFEND
\noindent {\bf Proof of \tref{l:1Y}. }
\noindent 
Redefining $ \SSz $ in \pref{l:9B} if necessary, we may assume that 
its conclusion holds for all $ \aaa \in \SSz $, with $ \mu ( \SSz ) = 1 $. 

Let $ \X $ be the fully labeled process in \tref{l:13}. 
Let $ \XX = \X ^{[0]}$ be the zero-labeled process given by $ \X $. 
Recall that $ \mu = \mu ^{[0]} $. 
From \pref{l:9B} \thetag{ii}, $ \XX $ is a $ \mu $-reversible diffusion with state space $ \sSS _{\star \star } $ properly associated with the Dirichlet form 
$ (\E ^{\mu } , \dom _{\star , \mathrm{lwr} \star }^{\mu } )$ on $ \Lm $. 
$ \sSS _{\star \star } $ is given by 
\begin{align} &\notag % \label{:93!}&
 \sSS _{\star \star } = \cup_{\aaa \in \SSz } \cup_{n=1}^{\infty} \mathsf{K}_{\aaa }^n = \cup_{\aaa \in \SSz } \mathbf{K}_{\aaa }^{[0]}
,\end{align}
where $ \SSz $ is as in \dref{d:15} and $ \mathsf{K}_{\aaa }^n $ is given by \eqref{:9Af} and 
$ \cup_{n=1}^{\infty} \mathsf{K}_{\aaa }^n = \mathbf{K}_{\aaa }^{[0]}$. 

 From \tref{l:9A} \thetag{ii}, $ (\E ^{\mu } , \dom _{\star , \mathrm{lwr} \star }^{\mu } )$ on $ \Lm $ is strongly local and quasi-regular. 
The state space $ \sSS _{\star \star } $ satisfies $ \mathrm{Cap} (\sSS _{\star \star }^c ) = 0 $, 
 where $ \mathrm{Cap} $ is the capacity associated with $ (\E ^{[m]} , \dom _{\star , \mathrm{lwr} \star }^{[m]} ) $ on $ \Lm $. 
This follows from that $ \mathbf{K}_{\aaa }^{[0]} $ is an invariant set of the diffusion $ \uLXX _{\aaa } = \uLX _{\aaa }^{[0]}$ for each $ \aaa \in \SSz $. 
Because $ \XX = \uLXX _{\aaa }$, $ \mathbf{K}_{\aaa }^{[0]} $ is the invariant set of the diffusion $ \XX $ for each $ \aaa \in \SSz $.

We can take a quasi-continuous version of the diffusion measure $ P_{\xx }$ of $\XX $ such that $ P_{\xx }$ is defined for all $ \xx \in \sSS _{\star \star } $. Note that $ P_{\xx } ( \XX \in C([0,\infty ); \sSS _{\star \star } ) ) = 1 $ because $ \sSS _{\star \star } $ is the state space of the conservative diffusion $ \XX $.

Let $ \map{\lpath }{\WSsiNE }{\CRdN }$ be as in \lref{l:top2}. 
Let $ \upath $ be the map defined by $ \upath (\w ) = \ww $, where $ \ww (t) = \sum_i \delta_{w^i (t)}$ for $ \w = (w^i)_i $. 
$ \upath \circ \lpath $ is the identity map. Hence, $ \lpath $ is injective. We note that $ \X = \lpath (\XX ) $. 

The value $ \lpath (\ww ) (t) \in \RdN $ is determined by $ \lab (\ww (0) ) $ and $ \ww (u ) $, $ 0\le u \le t $. 
Hence, $ \lpath (\ww ) (t) $ is $ \mathscr{F}_t = \sigma [ \lab (\ww (0) )] \vee \sigma [ \ww (u), 0\le u \le t ]$-measurable. 
Thus, the time-shifted path $ \lpath (\ww ) (\cdot + a ) $ is determined by $ \lpath (\ww ) (a) $. 

Let $ \mathbf{S}_{\star \star } \subset \RdN $ be the set defined by 
\begin{align}\label{:9Bi}&
\mathbf{S}_{\star \star } = 
\{ \lpath (\ww )(t) ; \ww \in C([0,\infty ); \sSS _{\star \star } ) , 0 \le t < \infty \} 
.\end{align}
We note that $ \ulab (\mathbf{S}_{\star \star }) = \sSS _{\star \star }$.

Let $ P_{\mathbf{x}} = P (\X \in \cdot \vert \X _0 = \mathbf{x})$ be the distribution of $ \X $ starting at $ \mathbf{x} \in \ulab ^{-1} (\sSS _{\star \star })$. As before, let $ \mathsf{P} _{\xx } $ be the distribution of $ \XX $ starting at $ \xx $. 
We have 
\begin{align}\label{:9Bk}&
P_{\mathbf{x}} = \mathsf{P} _{\xx }\circ \lpath ^{-1} , \quad 
\mathbf{x} = \lpath (\XX _0 ) = \lab (\xx ), \quad \ulab (\mathbf{x}) = \xx 
.\end{align}
Let $ \sigma $ be an arbitrary $ \{\mathscr{F}_t \} $-stopping time such that $ \sigma < \infty $ a.s. 
Let $ \tau = \sigma \circ \lpath $ and $ \mathscr{G}_t = \lpath ^{-1} (\mathscr{F}_t ) $. 
We have $ \mathscr{G}_{\tau } = \lpath ^{-1} (\mathscr{F}_{\sigma }) $. 
 $ \lpath $ is injective and $ \XX $ is a diffusion with state space $ \sSS _{\star \star } $. 
For each $ \mathbf{x} \in \mathbf{S}_{\star \star } $ such that $ \mathbf{x} = \X _{ \sigma (\X ) }$, 
\begin{align}\notag &
P_{\mathbf{x}} ( \X _{t + \sigma (\X ) } \in \cdot 
\vert \mathscr{F}_{\sigma } ) \vert _{\mathbf{x} = \X _{ \sigma (\X ) }} 
\\ \notag &= 
\mathsf{P}_{\ulab (\mathbf{x}) } ( \lpath (\XX _{t + \sigma (\lpath (\XX ) )}) \in \cdot 
\vert \lpath ^{-1} (\mathscr{F}_\sigma ) ) 
 \vert _{\ulab (\mathbf{x} ) = \XX _{ \sigma (\lpath (\XX ) ) } }
\\&\notag = 
\mathsf{P}_{\xx } ( \lpath (\XX _{t + \tau (\XX ) }) \in \cdot \vert \mathscr{G}_{\tau } ) 
\vert _{\xx = \XX _{ \tau (\XX ) }}
\\&\notag = 
\mathsf{P}_{\xx } ( \lpath (\XX _{t }) \in \cdot ) 
\vert _{\xx = \XX _{ \tau (\XX ) }}
\\ &\notag %	\label{:93f} = 
= P_{\mathbf{x}} ( \X _{t } \in \cdot \vert \mathscr{F}_{\sigma } ) \vert _{\mathbf{x} = \X _{ \sigma (\X ) }} 
.\end{align}
Here, we used \eqref{:9Bk} for the second line, the strong Markov property of $ \XX $ for the forth line, and the injectivity of $ \lab $ and $ \lpath $ for the last line. Because $ \XX $ is $ \mu $-reversible and satisfies \As{SIN}, $ \XX $ has infinite lifetime. Hence, so is $ \X = \lpath (\XX )$. 
Thus, we conclude that $ \X $ is a conservative diffusion with state space $ \mathbf{S}_{\star \star }$ such that 
$ \mu (\ulab (\mathbf{S}_{\star \star })) = 1 $. 
\qed

\section{Finite-volume approximation: Proof of \tref{l:16}} \label{s:X}

All results in \sref{s:X} are proved under Assumptions \As{A1}--\As{A3}  and \eqref{:14y}--\eqref{:14z}. 

\subsection{Finite-domain approximation}
We now impose reflecting boundary conditions on $ \partial \SR $ 
in \eqref{:14g} and fix the particles outside the region 
$ \oLSR = \{ s \in \Rd ; |s| \le \rR \} $. 
We suppose $ \SR \subset \ON $. Then the dynamics are described as follows:
\begin{align}\notag
\uLxXRNi (t) - \uLXRNi (0) 
&= \int_0^t \sigma ( \uLXRNi (u) ) \, d B_u^i 
 + \int_0^t 
\bbbN \Big( \uLXRNi (u) , \sum_{j\ne i}^{\nN } \delta_{ \uLXRNj (u) } \Big) du 
\\ \notag
&\quad + \int_0^t \anR ( \uLXRNi (u) ) \, dL_{\rR }^{\Ni } (u),
\\ \label{:14p}
L_{\rR }^{\Ni } (t) &= \int_0^t 1_{\partial \SR } ( \uLXRNi (u) ) \, d L_{\rR }^{\Ni } (u),
\quad i \le \sss ( \oLSR ) 
,\\ \label{:14r}
\uLXRNi (t) - \uLXRNi (0) 
&= 0 
,\quad \sss ( \oLSR ) < i 
.\end{align}
Here $ \anR $ denotes the inward unit normal vector on $ \partial \SR $, 
and $ L_{\rR }^{\Ni } $ is the local time on the boundary $ \partial \SR $. 

%G] %G[---
Let $ \{ \Nn \}_{ n = 1 }^{ \infty } $ be the sequence in \As{A2} such that $ \muNn $ converges to $ \mu $. 
Taking the limit $ \Nn \to \infty $ in the SDE \eqref{:14p}--\eqref{:14r}, we obtain the ISDE 
\eqref{:14t}--\eqref{:14u}. 
  
%G]
%GTP[--
\begin{remark}\label{r:14} 
\thetag{i} 
The solutions of the above SDE remains fixed outside the region $ \oLSR $. 
This follows from \eqref{:02x} and \eqref{:14r}. 

\noindent \thetag{ii}
Let $\mathbf{s}=(s^i)_{i=1}^{\nN}$ satisfy $s^i\neq s^j$ for all $i\neq j$.
Then the SDE \eqref{:14p}--\eqref{:14r} with $\uLXRN(0)=\mathbf{s}$ admits a unique strong solution $\uLXRN$.
Indeed, the coefficients are smooth outside the collision set
$\mathscr{N}=\bigcup_{i\neq j}\{s^i=s^j\}$, and $\uLXRN$ does not hit
$\mathscr{N}$ since $d\ge2$; see \pref{l:41}.
\end{remark}
%GTP]

\begin{lemma} \label{l:X1} 

\noindent \thetag{i} The sequence $ \uLXRN $ $( \nN , \rR \in \N )$ is tight in $ \CRdN $.

\noindent \thetag{ii} 
The sequence $ \X ^{\nN , [m]} $ $ (\nN \in \N )$ is tight in $ \W (\RdmSS )$ for $ m \in \zN $. 
\end{lemma}

\begin{proof}
We suppose $ \nN \ge \sss (\oLSR ) $. 
We only consider the case $ i \le \sss (\oLSR ) $ because $ \uLxXRNi (t) $ is frozen for $ i > \sss (\oLSR ) $. 

The unlabeled dynamics $ \uLXX _{\rR }^{\nN } $ are $ \muRss ^{\nN } $-reversible. 
Suppose that $ \varphi = \mathrm{const}$. 
Then from the Lyons-Zheng decomposition \cite[p.284]{c-f}, 
\begin{align} &\notag %\label{:X1g}& 
\uLxXRNi (t) - \uLxXRNi (0) = 
\half \Big\{ \int_0^t \sigma ( \uLxXRNi (v) ) dB_v^i - \int_0^t \sigma ( \uLxXRNi (v) ) dB_v^i \circ r_t \Big\} 
.\end{align}
From this, \eqref{:12o}, \eqref{:14y}--\eqref{:14z}, and the martingale inequality, we obtain 
\begin{align} \notag %\label{:X2k}&
E_{\mu }[\lvert \uLxXRNi (t) - \uLxXRNi (u) \rvert ^4] \le \cref{;X2k} \lvert t - u \rvert ^2 
,\end{align}
where $ \Ct \label{;X2k} $ is a constant independent of $ \nN , \rR \in \N $. 
By the Lyons-Zheng decomposition, $ \XN $ is tight in $ \CRdN $. This implies \thetag{i}. 

Moreover, by the Lyons-Zheng decomposition, any limit point satisfies \As{NBJ}. 
From this and \lref{l:top3}, we obtain \thetag{ii}. 
\PFEND
% G[

Recall that, for $ \yy \in \sSS $, $ \yyNn = (\lab^1(\yy ),\ldots,\lab^{\Nn }(\yy ))$, where $ \lab = (\labi )_i $ is the label as in \eqref{:02x}. Thus, $ \yyNn $ is a function of $ \yy $. 
\begin{lemma}\label{l:X2} 
For $ \mu $-a.s.\,$ \yy $, the following hold:  

\noindent
\thetag{i} 
The Dirichlet form $ (\E _{\aRyNn }^{\Nnm }, \uLd _{\aRyNn }^{\Nnm })$ 
on $ L^2 (\mu _{\aRyNn }^{\Nnm }) $ converges to the Dirichlet form
$ ( \E _{\Ry } ^{[m]} , \uLd _{\Ry }^{[m]} ) $ on $ L^2 (\muRyy ^{[m]} ) $
in the strong resolvent sense 
after identifying $ L^2 (\mu _{\aRyNn }^{\Nnm }) $ and $ L^2 (\muRyy ^{[m]} ) $ 
with their images in $ L^2 ( \LambdaR ^{[m]} ) $
under the canonical isometric embeddings induced by the Radon--Nikodym derivatives.

\smallskip\noindent
\thetag{ii} 
The Dirichlet form $ (\E _{\aRyNn }^{\Nnm }, \oLd _{\aRyNn }^{\Nnm })$ 
on $ L^2 (\mu _{\aRyNn }^{\Nnm }) $ converges to the Dirichlet form
$ ( \E _{\Ry } ^{[m]} , \oLd _{\Ry }^{[m]} ) $ on $ L^2 (\muRyy ^{[m]} ) $
in the strong resolvent sense on $ L^2 ( \LambdaR ^{[m]} ) $ under the identification in \thetag{i}. 
\end{lemma}
%G]
%G[
\begin{proof}
Let $ \wm _{\aRyNn }^{ \Nn , k } $ and $ \wm _{\Ry }^k $ 
be the $ k $-density functions of $ \mu _{\aRyNn }^{ \Nn } $ and $ \mu _{\Ry } $, respectively, on 
$ \oLSRk $. By \eqref{:1!c}, $ \wm _{\aRyNn }^{ \Nn , k } $ converges to $ \wm _{\Ry }^k $ 
in $ C ( \oLSRk )$. 

Let
$ p _{\aRyNn }^{ \Nn , k } ( t , \mathbf{x} , \mathbf{y} ) 
\wm _{\aRyNn }^{ \Nn , k } ( \mathbf{y} ) \, d\mathbf{y} $
and
$ p _{\Ry }^{ k } ( t , \mathbf{x} , \mathbf{y} )
\m _{\Ry }^{ k } ( \mathbf{y} ) \, d\mathbf{y} $
be the transition probabilities of $ \uLXRNn $ and $ \uLXR $ on $ \oLSRk $, respectively. 
For each $  \rR \in \N $ and $ k \in \N $, these transition densities exist by the classical theory of heat kernels in divergence form. 
We set
\begin{align}\notag
\sS _{\ne }
& = \bigl\{
\mathbf{x} = ( x^i )_{ i = 1 }^{ k } \in \oLSRk \, ;\,
x^i \ne x^j \text{ for } i \ne j
\bigr\},
\quad 
\mathscr{O}
 = ( 0 , \infty ) \times \sS _{\ne }^{ 2 } .
\end{align}

By the classical theory of heat kernels in divergence form, 
$ p _{\aRyNn }^{ \Nn , k } $ and $ p _{\Ry }^{ k } $
are locally H\"older continuous on $ \mathscr{O} $, and $ p _{\aRyNn }^{ \Nn , k } $ converges to 
$ p _{\Ry }^{ k } $ uniformly on each compact subset of $ \mathscr{O} $. 
Moreover, $ \wm _{\aRyNn }^{ \Nn , k } $ converges to $ \wm _{\Ry }^k $ in $ C ( \oLSRk )$, as noted at the beginning of the proof. 
Consequently, the corresponding $ \alpha $-resolvent densities converge, and \thetag{i} follows. 

The Dirichlet form in \thetag{ii} is associated with a different boundary condition.
Nevertheless, its heat kernel converges locally uniformly in the same manner as that in \thetag{i}.
This implies the assertion of \thetag{ii}. 
\PFEND

Let $ \uLEDRNnmSTAR $ and $ \oLEDRNnmSTAR $ be the superpositions, with respect to $ \yy $ under $ \mu $, of 
$ ( \E _{\Ry }^{\Nnm } , \uLd _{\Ry }^{\Nnm } ) $ and 
$ ( \E _{\Ry }^{\Nnm } , \oLd _{\Ry }^{\Nnm } ) $, respectively.
\begin{lemma}\label{l:X3}	

\noindent
\thetag{i} 
The Dirichlet form 
$ \uLEDRNnmSTAR $ on $ L^2 ( \muR ^{\Nnm } ) $ converges to the Dirichlet form
$ ( \ER ^{[m]} , \uLd _{\Rstar }^{[m]} ) $ on $ L^2 ( \muR ^{[m]} ) $
in the strong resolvent sense on $ L^2 ( \LambdaR ^{[m]} ) $
after identifying $ L^2 ( \muR ^{\Nnm } ) $ and $ L^2 ( \muR ^{[m]} ) $
with their images in $ L^2 ( \LambdaR ^{[m]} ) $
under the canonical isometric embeddings induced by the Radon--Nikodym derivatives. 

\smallskip \noindent 
\thetag{ii}
The Dirichlet form $ \EDaRNnmSTAR $ on $ L^2 ( \muR ^{\Nnm } ) $ 
converges to the Dirichlet form $ ( \ER^{[m]} , \dORmumSTAR ) $ 
on $ L^2 ( \muR ^{[m]} ) $ in the strong resolvent sense on 
$ L^2 ( \LambdaR ^{[m]} ) $ under the identification in \thetag{i}.
\end{lemma}

\begin{proof} 
By the definition of superposition together with Fubini's theorem and Jensen's inequality, 
\lref{l:X3} follows from \lref{l:X2}. 
\PFEND
%G]

%G[
We define the particles with indices greater than $ N_n $ to remain fixed for all time. 
With this convention, the dynamics of the first $ N_n $ particles can be regarded as an infinite-particle system.
Under this convention, all solutions start from the same initial condition $ \lab ( \sss ) $ almost surely. 
\begin{theorem}[Finite-domain approximation]	\label{l:14}
Assume \As{A1}--\As{A3} and \eqref{:14y}--\eqref{:14z}. Then 
\begin{align}\label{:14a}
\limin \uLXRNn &= \uLXR \quad \text{ in law in } \CRdN 
,\\\label{:14b}
\limiR \uLXR &= \X \quad \text{ in law in } \CRdN 
.\end{align}
\end{theorem}
\begin{proof}
Since $ \uLXR $ is a solution of SDE \eqref{:14t}--\eqref{:14u}, it follows that $ \uLXR ^{[m]} $ is associated with $ ( \ER ^{[m]} , \uLd _{\Rstar  }^{[m]} ) $ on $ L^2 ( \muR ^{[m]} ) $. 
Hence, by \lref{l:X1} \thetag{i} and \lref{l:X3} \thetag{i}, we obtain \eqref{:14a}. 

By \tref{l:6A} \thetag{iii}, we obtain \eqref{:14b}, which completes the proof. 
\PFEND

\subsection{Finite-particle approximation: Proof of \tref{l:16} }

\medskip 
\noindent {\bf Proof of \tref{l:16}. } 
Let $ \XN $ be the solution of SDE \eqref{:14g} with \eqref{:14y}--\eqref{:14z}. 
For each $ m \in \zN $ with $  m \le\nN $, let 
\begin{align}&\notag 
\E ^{\Nm } (f,g) = \int_{\RdmSS } \DDDam [f,g] d \mu ^{ \Nm } 
.\end{align}
Then $ (\E ^{\Nm }, \dcbm )$ is closable on $ L^2(\mu ^{\Nm }) $. 
Let $ ( \E ^{\Nm } , \dom ^{\Nm } ) $ be its closure. 
The diffusion $ \X ^{\Nm } $ is properly associated with $ ( \E ^{\Nm } , \dom ^{\Nm } ) $ on $ L^2(\muNm ) $. 
By the Lyons--Zheng decomposition and \eqref{:14y}--\eqref{:14z},
$ \XN $ and $ \X ^{\Nm } $, $ N \in \N $, are tight in $ \CRdN $ and $\W ( \RdmSS ) $, respectively.

Hence, there exist subsequences of $ \XNn $ and $ \X ^{\Nnm } $, 
still denoted by the same symbols, and limits $ \widetilde{\X } $ and 
$ \widetilde{\X }^{[m]} $ such that $ \widetilde{\X }^{[m]} $ is the 
$ m $-labeled process of $ \widetilde{\X } $. Furthermore, for any nonnegative 
$ f , g \in \dcbm $ and $ \alpha > 0 $, 
\begin{align}\label{:X4s}&
\limi{\n } \int_{\RdmSS } f R^{\alpha , \Nn , [m]} g \, d (\mum )^2 = 
 \int_{\RdmSS } f \widetilde{R}^{\alpha , [m]} g \, d (\mum )^2 
.\end{align}
Here $  R^{\alpha , \Nn , [m]} $ denotes the $ \alpha $-resolvent of $ ( \E ^{\Nm } , \dom ^{\Nm } ) $ 
on $ L^2(\muNm ) $ and 
\begin{align*}&
\widetilde{R}^{\alpha , [m]} g (\mathbf{x},\sss )=
E_{(\mathbf{x},\sss )}[ \int_0^{\infty} \mathrm{e}^{-\alpha t}g (\widetilde{\X }_t^{[m]}) dt ]
.\end{align*}

At this stage, it is not known whether $ \widetilde{\X }^{[m]} $ is a Markov process. 
We now prove that this is indeed the case.
%G] 

Let $ \uLEDRNnmSTAR $ and $ \EDaRNnmSTAR $ be as in \lref{l:X3}. 

For symmetric bilinear forms $ A $ and $ B $, 
$ A \succ B $ means that $ A $ is an extension of $ B $ in the sense of \eqref{:ext1a}. 
Then it is clear that
\begin{align*}
\uLEDRNnmSTAR \succ ( \E ^{\Nnm } , \dom ^{\Nnm } ) 
\succ \EDaRNnmSTAR 
.\end{align*}
%G]
% 
Hence,
\begin{align}\label{:X4u}
\uLEDRNnmSTAR \le ( \E ^{\Nnm } , \dom ^{\Nnm } ) 
\le \EDaRNnmSTAR 
.\end{align}

By \lref{l:X3} \thetag{i}, $ \uLEDRNnmSTAR $ converges to 
$ (\ER ^{[m]} , \uLd _{\Rstar  }^{[m]} )$ in the strong resolvent sense as $ \n \to \infty $. 
By \lref{l:ext3} \thetag{i}, $ (\ER ^{[m]} , \uLd _{\Rstar  }^{[m]} )$ converges to the increasing limit $ ( \Emum , \dom \SLm  ) $ in the strong resolvent sense as $ \rR \to \infty $. 
Thus, the left closed form in \eqref{:X4u} converges to $ ( \Emum , \dom \SLm  ) $ 
in the strong resolvent sense in the two-step limits. 

Let $ f \in \oLd \aRNnm $. 
Then $ f $ is $ \mathscr{B}((\SR )^m) \ts \sigma [ \piR ] $-measurable. 
Moreover, $ \muR ^{ \Nnm } $ is absolutely continuous with respect to $ \muR ^{ [m] } $. 
Hence, we regard $ \EDaRNnmSTAR $ as a closed form on $ L^2 (\muR ^{ [m] }) $.

By \lref{l:X3} \thetag{ii}, $ \EDaRNnmSTAR $ converges to $ ( \ER ^{[m]} , \dORmumSTAR ) $ 
in the strong resolvent sense. Moreover, by \lref{l:ext3} \thetag{ii}, $ ( \ER ^{[m]} , \dORmumSTAR ) $ converges to the decreasing limit $ ( \Emum , \USm  ) $ in the strong resolvent sense on $ \Lmm $. 
Thus, the right closed form in \eqref{:X4u} converges to $ ( \Emum , \USm  ) $  
in the strong resolvent sense in the two-step limits. 

Hence, any limit point of
$ (\E ^{\Nnm } , \dom ^{\Nnm } ) $ lies between $ ( \Emum , \dom \SLm  ) $ and
$ ( \Emum , \USm  ) $  in the sense that, for any such nonnegative $ f ,g \in \dcb $, 
the quantity in the right-hand side of \eqref{:X4s} is between 
$ \int_{\sSS } f \uL{R}_{\star , \infty }^{\alpha , [m] } g d\mum $ 
and 
$ \int_{\sSS } f \oL{R}_{\star , \infty }^{\alpha , [m] } g d\mum $. 
Here $\uL{R}_{\star , \infty }^{\alpha , [m] } $ and $ \oL{R}_{\star , \infty }^{\alpha , [m] } $ are the 
$ \alpha $-resolvent of $ ( \Emum , \dom \SLm  ) $ and $ ( \Emum , \USm  ) $ , respectively. 

%G[
By \tref{l:9A}, these two Dirichlet forms coincide. We write their common form as 
$ ( \Emum , \dom _{\star }^{[m]} ) $, 
and denote its $ \alpha $-resolvent on $ \Lmm $ by $ R_{\star }^{\alpha , [m] } $.
%G] 
Therefore, $ (\E ^{\Nnm } , \dom ^{\Nnm } ) $ converges to 
$ ( \Emum , \dom _{\star } ^{[m]} ) $ in the following sense. 
For any nonnegative $ f ,g \in \dcb $, 
\begin{align*}&
\limi{\n } 
\int_{\RdmSS } f R^{\alpha , \Nn , [m]} g \, d (\mum )^2 = 
\int_{\RdmSS } f R_{\star }^{\alpha , [m] } g \, d\mum 
.\end{align*}
%G[
Hence $ \widetilde{R}^{\alpha , [m]} = R_{\star }^{\alpha , [m]} $. 
This implies that $ \widetilde{\X }^{[m]} $ is a Markov process with resolvents $ \widetilde{R}^{\alpha , [m]} $, $ \alpha > 0 $, and that $ \widetilde{\X }^{[m]} = \X ^{[m]} $ in law. 
Hence
\begin{align}&
\limi{\n } \X ^{\Nnm } = \X ^{[m]} \quad \text{in finite-dimensional distributions}
.\end{align}
This, together with tightness in $ \CRdN $, completes the proof.
\qed
%G]

\begin{remark}\label{r:X1} 
The main difficulty in proving \tref{l:16} is that the state space of the $ N $-particle system depends on $ N $. 
In \cite{k-o.du}, using a general theorem on the convergence of Dirichlet forms in distinct state spaces \cite{k-s}, a result similar to \tref{l:16} was proved. 
In that work, the $ N $-particle dynamics is considered in domains $ S_{R_N } $ with reflecting boundary conditions such that $ R_N \to \infty $. 
However, this does not yield convergence of the particle systems in the whole space and is therefore unsatisfactory. 

The key idea in the proof of \tref{l:16} is the identification of a relation connecting the $ N $-particle Dirichlet form with finite-domain Dirichlet forms.
\end{remark}

\section{Appendices} \label{s:A} 
\subsection{Topology of the unlabeled and labeled path spaces} \label{s:top}

Let $ \SSsi = \{ \sss \in \sSS \, ;\, \sss (\Rd ) = \infty , \sss (\{ s \} ) \in \{ 0 , 1 \} \text{ for all } s \in \Rd \} $. 
Let $ \WSsi $ be the set of all $ \SSsi $-valued continuous paths defined on $ [0,\infty)$. 
%G[
\begin{remark}
The preprint \cite{o.gin} is a longer version of the published paper 
\cite{o.Gin}, containing additional topological and labeling results
that were omitted from the journal version.
\end{remark}
%G[
\begin{lemma}[{\cite[Lem.\,7.1]{o.gin}}]\label{l:top1}
Let $I_i$ be of the form $[0,b_i)$ or $(a_i,b_i)$ with $0 \le a_i < b_i \le \infty$.
Each $\ww \in \WSsi$ admits a representation
\begin{align}\notag
\ww_t = \sum_{i=1}^{\infty} \delta_{w^i(t)},
\quad w^i \in C(I_i;\Rd)
.\end{align}
The collection $\{(w^i,I_i)\}_{i\in\mathbb N}$
is uniquely determined up to relabeling.
\end{lemma}
%
%G]
For $\w=(w^i)_{i=1}^{\infty}$, define $\upath(\w)=\mathsf{w}$ by
$\mathsf{w}_t=\sum_{i=1}^{\infty}\delta_{w^i(t)}$.

\begin{lemma}[{\cite[Lem.\,7.2]{o.gin}}]\label{l:top2}
For any label $\lab$, there exists a unique map
$\lpath:\WSsiNE\to\CRdN$ such that
$\lab(\ww_0)=\lpath(\ww)_0$ and $\upath\circ\lpath(\ww)=\ww$.
\end{lemma}
%G]
%G[

Set $\|u \|_T=\sup_{0\le t\le T}|u(t) |$.
For $\mathbf u=(u^i)_{i\in\mathbb N}$ and $\mathbf v=(v^i)_{i\in\mathbb N}$, define 
\begin{align*}&
\varrho_{\mathrm{path}}(u,v)
= \sum_{T=1}^{\infty} \frac{1}{2^T}( 1 \wedge \| u - v \|_T ),
\quad
\varrho_{\mathrm{lpath}} ( \mathbf u , \mathbf v )
= \sum_{i=1}^{\infty} \frac{1}{2^i} 
\bigl( 1 \wedge \varrho_{\mathrm{path}} ( u^i , v^i ) \bigr)
.\end{align*}
Then $ ( \CRdN , \varrho_{\mathrm{lpath}} ) $ is a complete separable metric space.

Let $f_R\in C_0(\Rd)$ satisfy $0\le f_R\le1$, $f_R=1$ on $\SR$, and
$f_{R+1}=0$ on $\SRR^c$. For $\sss=\sum_i\delta_{\si}$, set
$f_R\sss=\sum_i f_R(\si)\delta_{\si}$.
Let $\varrho_{\mathrm{Prh}}$ denote the Prohorov metric on $\sSS$
\cite[p.~11]{Kal}.
The vague topology on $\sSS$ is induced by the complete separable metric 
 \begin{align}\notag
 \varrho_{\mathrm{vague}}(\sss,\sss')
 =\sum_{R=1}^{\infty} \frac{1}{2^R} 
 \bigl(1\wedge \varrho_{\mathrm{Prh}}(f_R\sss,f_R\sss')\bigr)
 .\end{align}
Then $\CiSS$ is a complete separable metric space under the metric 
\begin{align}\notag
\varrho_{\mathrm{upath}}(\ww,\ww')
=\sum_{T=1}^{\infty}\frac{1}{2^T}
\bigl(1\wedge \sup_{0\le t\le T}\varrho_{\mathrm{vague}}(\ww_t,\ww_t')\bigr)
.\end{align}

%G[
\noindent 
Let $\IRT$ be defined in \eqref{:13u}, and set
\begin{align}\notag
\mathcal{NBJ}
=\bigl\{\w\in\CRdN;\ \IRT(\w)<\infty\ \text{for all } R,T\in\mathbb N\bigr\}
.\end{align}
%G]

\begin{lemma}[{\cite[Lem.\,7.3]{o.gin}}] \label{l:top3} 
\thetag{i} $ \upath (\w ) \in \CiSS $ for $ \w \in \mathcal{NBJ} $. 
\\\thetag{ii} 
$ \map{\upath }{\mathcal{NBJ}}{C([0,\infty);\sSS )}$ is continuous. 
\end{lemma}

%G[
\subsection{Closability and extensions of Dirichlet forms, 
and strongly local, quasi-regular Dirichlet Forms}\label{s:ext}
%G]

In this subsection, we recall notions related to Dirichlet forms
from \cite{c-f,fot.2}.

Let $ \sS $ be a Polish space, that is, a topological space homeomorphic to a separable and complete metric space. Let $ \nu $ be a Radon measure on $ \sS $. 
We set $ \Lnu = L^2(\sS , \nu )$.

For non-negative symmetric bilinear forms $ (\E _1, \mathscr{D}_1 ) $ and $ (\E _2, \mathscr{D}_2 ) $, 
we say $ (\E _2, \mathscr{D}_2 ) $ is an extension of $ (\E _1, \mathscr{D}_1 ) $ if 
\begin{align}\label{:ext1a}&
 \mathscr{D}_1 \subset \mathscr{D}_2 , \quad \E _1 (f,f) = \E _2 (f,f) \quad \text{ for all }f \in \mathscr{D}_1 
.\end{align}

A non-negative, closed form $ (\E , \mathscr{D} ) $ on $ \Lnu $ is a densely defined bilinear form defined on $ \mathscr{D} $ that is complete under the inner product $ \E (f,g)+ (f,g) _{\Lnu }$. 

\begin{definition}[Closable and closure {\cite{fot.2}}]\label{d:ext} \thetag{i}
A non-negative symmetric bilinear form $ (\E , \mathscr{D}_0 ) $ densely defined on $ \Lnu $ is called closable on $ \Lnu $ if, for any $ \E $-Cauchy sequence $ f_n\in \mathscr{D}_0$ such that $ \lim \| f_n \| _{\Lnu } = 0 $ as $ n \to \infty $, 
it holds that $ \limi{n}\E (f_n , f_n ) = 0$. 

\noindent \thetag{ii} 
If $ (\E , \mathscr{D}_0) $ is closable on $ \Lnu $, then there exists 
a closed extension of $ (\E , \mathscr{D}_0) $. 
The smallest closed extension $ (\E , \mathscr{D} ) $ is called the closure of $ (\E , \mathscr{D}_0) $. 
\end{definition}
\begin{lemma} \label{l:ext2}
Let $ (\E _1, \mathscr{D}_1 ) $ and $ (\E _2, \mathscr{D}_2 ) $ be non-negative symmetric bilinear forms on $ \Lnu $. 
Let $ (\E _2, \mathscr{D}_2 ) $ be an extension of $ (\E _1, \mathscr{D}_1 ) $. Let $ (\E _2, \mathscr{D}_2 ) $ be closable on $ \Lnu $. It holds that $ (\E _1, \mathscr{D}_1 ) $ is closable on $ \Lnu $. 
\end{lemma}
\begin{proof}
If $ \{ f_n \} $ is a Cauchy sequence of $ (\E _1, \mathscr{D}_1 ) $, then $ \{ f_n \} $ is a Cauchy sequence of $ (\E _2, \mathscr{D}_2 ) $ by \eqref{:ext1a}. 
Then $ \limi{n} \E _2 (f_n , f_n ) = 0 $ because of the closability of $ (\E _2, \mathscr{D}_2 ) $ on $ \Lnu $. 
Hence, $ \limi{n} \E _1 (f_n , f_n ) = 0 $ from \eqref{:ext1a}. This completes the proof. 
\PFEND
%G[ ---
A sequence of closed, nonnegative bilinear forms
$ (\E ^n, \mathscr{D}^n) $ on $ \Lnu $
is said to converge to $ (\E , \mathscr{D}) $ on $ \Lnu $
in the strong resolvent sense
if the associated $\alpha$-resolvents $ G_{\alpha}^n $
converge strongly in $ \Lnu $ to the $\alpha$-resolvent $ G_{\alpha} $
of $ (\E , \mathscr{D}) $ for each $ \alpha > 0 $.
This is equivalent to the same convergence holding for some $ \alpha > 0 $
(cf.\,\cite{r-simon}).

In this case, the $ \Lnu $-semigroups $ T_t^n $
associated with $ (\E ^n, \mathscr{D}^n) $
converge strongly in $ \Lnu $ to the $ \Lnu $-semigroup $ T_t $
associated with $ (\E , \mathscr{D}) $
for each $ t \ge 0 $.

%G]

For non-negative symmetric bilinear forms $ (\E ^1,\mathscr{D}^1 )$ and $ (\E ^2,\mathscr{D}^2 )$, we write $ (\E ^1,\mathscr{D}^1 ) \le (\E ^2,\mathscr{D}^2 )$ if 
$ \mathscr{D}^1 \supset \mathscr{D}^2 $ and 
$ \E ^1 (f,f) \le \E ^2 (f,f) $ for all $ f \in \mathscr{D}^2 $. 
We write $ (\E ^1,\mathscr{D}^1 ) \ge (\E ^2,\mathscr{D}^2 )$ if 
$ (\E ^2,\mathscr{D}^2 ) \le (\E ^1,\mathscr{D}^1 )$. 
A sequence of symmetric non-negative bilinear forms $ \{ (\E ^n, \mathscr{D}^n) \} $ on $\Lnu $ is said to be increasing if $(\E ^n, \mathscr{D}^n) \le (\E ^{n+1}, \mathscr{D}^{n+1}) $ for all $ n $ and decreasing if $(\E ^n, \mathscr{D}^n) \ge (\E ^{n+1}, \mathscr{D}^{n+1}) $ for all $ n $. 

\begin{lemma}[{\cite[p.373, Th.\,S.16]{r-simon}, \cite{simon}}] \label{l:ext3} 
\thetag{i} 
Let $ \{ (\E ^n, \mathscr{D}^n) \}_{n\in\N } $ be an increasing sequence of non-negative symmetric closed bilinear forms on $\Lnu $. 
Let $ (\E ^{\infty}, \mathscr{D}^{\infty} )$ be the symmetric bilinear form defined by 
\begin{align}	\notag &%	\label{:ext3a}&
\E ^{\infty } (f,f) = \limi{n} \E ^n (f,f) , 
\quad 
\mathscr{D}^{\infty} =
 \{ f \in \bigcap_{n \in\N } \mathscr{D}^n \, ;\, \limi{n} \E ^n (f,f) < \infty \} 
.\end{align}
It holds that $ (\E ^{\infty}, \mathscr{D}^{\infty} )$ is a closed symmetric bilinear form on $ \Lnu $. 
$ \{ (\E ^n, \mathscr{D}^n) \}_{n\in\N } $ converges to $ (\E ^{\infty}, \mathscr{D}^{\infty} )$ in the strong resolvent sense on $ \Lnu $. 

\noindent \thetag{ii} 
Let $ \{ (\E ^n, \mathscr{D}^n) \}_{n\in\N } $ be a decreasing sequence of non-negative symmetric closed bilinear forms on $\Lnu $. 
Let $ (\E ^{\infty}, \mathscr{D}^{\infty} )$ be the symmetric bilinear form defined by 
\begin{align} &\notag % 	\label{:ext3b}&
\E ^{\infty } (f,f) = \limi{n} \E ^n (f,f) , 
\quad 
\mathscr{D}^{\infty} = \bigcup_{n \in\N } \mathscr{D}^n 
.\end{align}
It holds that 
$ \{ (\E ^n, \mathscr{D}^n) \}_{n\in\N } $ converges to $ (\E _{\mathrm{max}}^{\infty}, \mathscr{D}_{\mathrm{max}}^{\infty} )$ in the strong resolvent sense on $ \Lnu $. 
Here, $ (\E _{\mathrm{max}}^{\infty}, \mathscr{D}_{\mathrm{max}}^{\infty} )$ is the maximal closable part of $ (\E _{\mathrm{max}}^{\infty}, \mathscr{D}_{\mathrm{max}}^{\infty} )$ on $ \Lnu $. 
\end{lemma}

\begin{definition}[Dirichlet form {{\cite{fot.2}}}]\label{d:A2}
We say a closed non-negative symmetric bilinear form $ (\E , \mathscr{D} ) $ densely defined on $ \Lnu $ is 
a symmetric Dirichlet form if any $ u \in \dom $ satisfies 
\begin{align}& \label{:A2a}
v := \min \{ 1 , \max\{ 0 , u \} \} \in \dom \text{ and }
\E (v,v) \le \E (u,u)
.\end{align}
\end{definition}
For a symmetric Dirichlet form, there exists an associated symmetric Markovian $ L^2$-semi-group \cite{fot.2}. We now recall the concept of capacity \cite{fot.2,c-f}. 
\begin{definition}[Capacity {{\cite{c-f,fot.2}}}]\label{d:A3}
Let $ \mathcal{O} $ be the family of all open subsets of $ \sS $. 
For $ \aAA \in \mathcal{O} $, let 
$ \mathcal{L}_{\aAA ,1}=\{ f \in \mathscr{D} ; f \ge 1 \ \nu \text{-a.e.\,on }\aAA \} $. 
We set $ \mathcal{O}_0=\{ \aAA \in \mathcal{O} ; \mathcal{L}_{\aAA ,1} \ne \emptyset \} $. 
Let $ \E _1 = \E + (\cdot,*)_{\Lnu }$. 
For an open set $ \aAA \in \mathcal{O} $, we set 
\begin{align*}&
\capa (\aAA ) = 
\begin{cases}
\inf \{ \E _1(f,f); f \in \mathcal{L}_{\aAA ,1}\}, & \aAA \in \mathcal{O}_0 
\\
\infty , & \aAA \notin \mathcal{O}_0 
.\end{cases}
\end{align*}
For any set $ \bBB \subset \sS $, we set 
$ \capa (\bBB ) = \inf \{ \capa (\aAA ) ; \aAA \in \mathcal{O},\, \aAA \supset \bBB \} 
$. 
\end{definition}

Let 
 $ \mathscr{D}(\fFF _{k}) = \{ f \in \mathscr{D}; f = 0 \text{ $ \nu $-a.e.\,on } \fFF _{k}^c \} $ and 
$ \E _1 = \E + (\cdot , \cdot )_{\Lnu } $. 
An increasing sequence of closed sets $ \{ \fFF_{k} \} $ is called an $ \E $-nest if $ \cup _{k \ge 1 }\mathscr{D}(\fFF _{k}) $ is $ \E _1 $-dense. 
A function $ f $ is called $ \E $-quasi-continuous if for any $ \epsilon >0 $, there exists an open set $ \oOO $ with $ \capa (\oOO ) < \epsilon $ such that $ f \vert _{\sS \backslash \oOO }$ is finite and continuous. 
We call a subset $ \mathfrak{N} \subset \sS $ an $ \E $-polar set if 
there exists an $ \E $-nest $ \{ \fFF_{k} \} $ such that 
$ \mathfrak{N} \subset \cap_k ( \sS \backslash \fFF_{k} )$. 

\begin{definition}\label{d:QR}
A Dirichlet form $ (\E , \mathscr{D} ) $ on $ \Lnu $ is quasi-regular if: \\ 
\As{Q1} there exists an $ \E $-nest $ \{ \fFF_{k} , k \ge 1 \} $ consisting of compact sets; 

\noindent 
\As{Q2} 
there exists an $\E _1 $-dense subset of $\mathscr{D}$
whose elements admit $\E$-quasi-continuous $\nu $-versions;
\\\As{Q3} 
there exists $ \{ f_k , k \ge 1 \} \subset \mathscr{D} $ having $ \E $-quasi-continuous $ \nu $-versions $ \{ \tilde{f}_k , k \ge 1 \} \subset \mathscr{D} $ and an $ \E $-polar set $ \mathfrak{N} \subset \sS $ such that 
$ \{ \tilde{f}_k , k \ge 1 \} \subset \mathscr{D} $ separates the points on $ \sS \backslash \mathfrak{N} $. 
\end{definition}

A Dirichlet form $ (\E , \mathscr{D} ) $ on $ \Lnu $ is called strongly local if $ \E ( f , g ) = 0 $ for any $ f , g \in \mathscr{D} $ \ such that $ f $ is constant on a neighborhood of the support of $ g $ \cite{c-f}. 
A diffusion process is a family of Markov processes with continuous sample path and has the strong Markov property. We say a diffusion is conservative if it has an infinite lifetime. 
If the Dirichlet form is strongly local and quasi-regular and the state space is homeomorphic to a complete separable metric space, then there exists a properly associate diffusion exists \cite{c-f}. 
We refer to \cite{c-f} for the concept of properly associated. 

In general, a Dirichlet form is not necessarily symmetric. 
In the present paper, a Dirichlet form means a symmetric Dirichlet form.

\subsection{Representation of functions on $ \sSS $ and correlation functions} \label{s:fun}

\begin{definition}\label{d:corfun} 
A symmetric, locally integrable function $ \rho_\nu^m $ on $ \Rdm $
is called the \emph{$ m $-point correlation function}
of an RPF $ \nu $ (with respect to Lebesgue measure) if
\begin{align}\notag &%
\int_{ A_1^{ k_1 } \times \cdots \times A_n^{ k_n } }
\rho_\nu^m ( \x ^1 , \ldots , \x ^m ) \, d\x ^1 \cdots d\x ^m
=
\int_{ \sSS }
\prod_{ i = 1 }^{ n }
\frac{ \sss ( A_i ) ! }{ ( \sss ( A_i ) - k_i ) ! }
\, \nu ( d\sss )
,\end{align}
for any disjoint bounded measurable sets
$ A_1 , \ldots , A_n $
and nonnegative integers
$ k_1 + \cdots + k_n = m $.
Here, $ \sss ( A ) $ denotes the number of particles in $ A $
when $ \sss = \sum_i \delta_{ \si } $,
and the fraction is interpreted as zero
if $ \sss ( A_i ) - k_i < 0 $.
\end{definition}

%GTP[
Let $\anest = \{ \ak \}_{\qqq \in \mathbb{N}}$ be a family of sequences
$\ak = \{ \akR \}_{\rR \in \mathbb{N}}$ of natural numbers.
Assume that, for all $\qqq, \rR \in \mathbb{N}$,
\begin{align}\label{:ak}
 \akR < \akRR, \quad
 \akR < \akk(\rR), \quad
 \ak^{+}(\rR) < \akk(\rR),
\end{align}
where $\ak^{+}(\rR) = 1 + \akRR$. 
%GTP]] 
%GTP[
For $\ak = \{ \akR \}_{\rR \in \mathbb{N}}$, let 
\begin{align}\label{:CUTw} 	%\label{:52l}&
\Ki [\ak ] &= \bigcap_{\rR = 1}^{\infty} \{ \sss \in \sSS ; \sss (\SR ) \le \akR \} 
.\end{align}
Then $ \Ki [\ak ]$ is a compact set in $ \sSS $ for each $ \qqq \in \mathbb{N}$ such that 
\begin{align} & \label{:CUTy} 
 \Ki [\ak ] \subset \Kakk \subset \Ki [\akk ] 
.\end{align}
%GTP]

Let $ \nu $ be an RPF and $ m \in \N $. Let $ \nu _{\qqq } = \nu (\cdot \cap \Ki [\ak ] ) $. 
Suppose that $ \nu $ has density functions. 
Then by \eqref{:CUTw}, $ \nu _{\qqq } $ has the $ m $-point correlation function and the $ m $-reduced Campbell measure $ \num _{\qqq } $ similarly as \eqref{:12y}. 
We define $ \num $ as the increasing limit of $ \num _{\qqq } $: 
\begin{align}	\label{:Camp} & \quad 
\num ( A ) = \limi{\qqq } \num _{\qqq } ( A ) 
,\quad A \in \mathscr{B}(\RdmSS ) 
%	\quad \text{ for } A \in \mathscr{B}(\RdmSS ) 
.\end{align}
Note that $ \num $ is independent of the particular choice of compact sets $ \Ki [\ak ]$.

Let $ \| h \| := \| h \|_{L^1(\SQmSS , \num )} +\|\nablax h \|_{L^1(\SQmSS , \num )} $. 
We note that, in the definition of $ \| \cdot \| $, the differential 
is taken only with respect to the $ x $-variable, 
and not with respect to $ \sss \in \sSS $. 
This point is very different from the energy norms 
used in Theorems \ref{l:QR1} and \ref{l:QR2}. 
\begin{lemma} \label{l:39} 
The set $ \dcbm $ is dense in $ \dbbm $ with respect to $ \| \cdot \|$, $ \qQ \in \N $. 
\end{lemma}
\begin{proof}
Recall that $ \dcbm = \CziRdm \ot \dcb $ and $ \dbbm = \CziRdm \ot \dbb $. 

From the proof of Lemma 2.4 in \cite{o.dfa}, we easily see that both $ \dcb $ and $ \dbb $ are dense in $ L^2( \nu ) $. Hence, we can prove \lref{l:39} from this. 
\PFEND
%--

For $ \rR \in \N $ and $ m \in \{0\} \cup \N $, set 
$ \SSRm = \{ \sss \in \sSS ; \sss (\SR ) = m \} $. 
For $ A \in \mathscr{B}(\sSS )$, define $\pi_A(\sss ) = \sss(\cdot \cap A)$. 
Let $ \SRm $ denote the $m$-fold product of $\SR$. 
\begin{definition} \label{d:fun} \thetag{i} 
For $ \sss \in \sSS $, a coordinate of $\sss $ on $ \SR $ is a tuple 
$\mathbf{x}_{\rR }(\sss ) = (x_{\rR }^i(\sss ))_{i} 
\in \sqcup_{m=1}^{\infty} \SRm $ such that $ 
\piR (\sss ) = \sum_{i} \delta_{x_{\rR }^i(\sss )}$. 

\noindent \thetag{ii} The restriction $ \mathbf{x}_{\rR }^m (\sss ) $ of $ \mathbf{x}_{\rR }(\sss ) $ on $ \SRm $ is called the $ \SRm $-coordinate. 

\noindent \thetag{iii}
Let $ f:\sSS \to \mathbb{R} $ and $ \rR,m \in \mathbb{N} $. 
We say that $ f_{\Rs }^m $ is the $ \SRm $-representation of $ f $ if 
$ f_{\Rs }^m(\mathbf{x}) $ is a function on $ \SRm $ such that
\begin{align}\notag %\label{:10p}
f_{\rR,\cdot}^m : \sSS \times \SRm \to \mathbb{R}, 
\qquad (\sss,\mathbf{x}) \mapsto f_{\Rs }^m(\mathbf{x}),
\end{align}
and the following conditions hold:
\begin{align}
%\label{:10q}
\notag & f_{\Rs }^m(\mathbf{x}) \ \text{is symmetric in $\mathbf{x}$ for each $ \sss \in \SSRm $}
,\\
%\label{:10r}
\notag 
& f_{\rR,\sss(1)}^m(\mathbf{x}) = f_{\rR,\sss(2)}^m(\mathbf{x}) 
\ \text{if } \piRc (\sss(1)) = \piRc (\sss(2))
,\\
%\label{:10s}
\notag 
& f_{\Rs }^m(\mathbf{x}_{\rR } (\sss )) = f(\sss ) 
\ \text{for } \sss \in \SSRm,\\
%\label{:10t}
\notag & f_{\Rs }^m(\mathbf{x}) = 0 
\ \text{for } \sss \notin \SSRm.
\end{align}
We define $ f _{\Rs }$ on $ \sqcup_{m\in\N } \SRm $ by $ f _{\Rs } (\mathbf{x}) = f_{\Rs }^m(\mathbf{x}) $ for $ \mathbf{x} \in \SRm $. 
\end{definition}

Note that $f_{\Rs }^m$ is uniquely determined, and for every $\sss \in \sSS $,
\[
f(\sss ) = f_{\Rs }(\mathbf{x}_{\rR } (\sss )).
\]
By convention, $\mathbf{x}_{\rR }(\sss ) = \emptyset $ for $\sss \in \SSRz $, and $f_{\Rs } $ is constant on $\SSRz $. 

For a bounded set $A$, one defines $\mathbf{x}_A^m(\sss )$, $ f_{A}^m (\mathbf{x}) $ and $f_{A,\sss} (\mathbf{x}) $ analogously by replacing $\SR$ with $A$.
\subsection{Quasi-regularity of $ m $-labeled Dirichlet forms} \label{s:QR}
Let the coefficient $ \aaaa $ satisfy \eqref{:12o} and $ \DDDam $ be as \eqref{:52s}. 
Let 
\begin{align} \notag &
\Enum (f,g) = \int_{\RdmSS } \DDDam [f,g] d\num 
,\\ \notag & 
 \dcnum = \{ f \in \dcm ; \Enum (f,f) < \infty ,\, f \in \Lnum \} 
.\end{align}
Although $ \Enum $, $ \dRbmum $, and $ \dcnum $ depend on $ \aaaa $ and $ \mu $, 
we suppress this dependence in the notation. 
 We make assumptions: 
\smallskip

\noindent 
\thetag{A.1} 
$ ( \Enum , \dcnum ) $ is closable on $ \Lnum $. 

\smallskip 
\noindent \thetag{A.2} 
$ \int_{\sSS } \sss (\SR ) \nu (d\sss ) < \infty $ for each $ \rR \in \N $. 

\noindent \thetag{A.3} 
$ \nu $ has a bounded $ k $-density function on $ \oLSRk $ for each $ \rR , k \in \N $. 

%GTP[ 

Let $ \EDnum $ be the closure of $ ( \Enum , \dcnum ) $ on $ \Lnum $. 
We write $ \EDnu $ for $ m = 0 $ since $ \nu ^{[0]} = \nu $. 

\begin{theorem}[{\cite[Th.\,1]{o.dfa}}] \label{l:QR1}
Assume \thetag{A.1} for $ m = 0 $, \thetag{A.2}, and \thetag{A.3}. 
Then the Dirichlet form $ \EDnu $ on $ \Lnu $ is quasi-regular.
\end{theorem}
\begin{proof}
The theorem was proved in \cite[Th.\,1]{o.dfa} under the assumption that the diffusion matrix $ \aaaa $ is the identity. 
In the present setting, $ \aaaa $ is uniformly elliptic and bounded by \eqref{:12o}, 
and hence the same argument applies verbatim.
\PFEND
%GTP]
\begin{lemma} \label{l:QR5}
Assume \thetag{A.2}. Then $\dcnu $ is dense in $ \Lnu $.
\end{lemma}
\PF %G[----
The assertion was proved in the proof of Lemma~2.5 in \cite{k-o-t.udf}. 
 See the last line of the proof of Lemma~2.5 in \cite{k-o-t.udf}. 
%GTP]
\PFEND

In \tref{l:QR2}, we significantly relax the assumptions introduced in \tref{l:QR1}. 
Specifically, if $(\Emum, \dcnum)$ is found to be closable on $\Lnum $, then the existence of a density function is no longer necessary. 

%GTP[

\begin{theorem}\label{l:QR2}
Assume \thetag{A.1} and \thetag{A.2}. 
Then $ \EDm $ is a quasi-regular Dirichlet form on $ \Lnum $. 
\end{theorem}

\PF %---
Assumptions \As{Q1} and \As{Q3} follow from \thetag{A.1} and \thetag{A.2}
by exactly the same arguments as those used in the proofs of
\cite[Th.\,1]{o.dfa} and \cite[Lem.\,2.3]{o.tp}.
Thus, it only remains to prove \As{Q2}.

For $ m = 0 $, \thetag{A.3} is used only in the proof of
\cite[Lem.\,2.4]{o.dfa} to establish \lref{l:QR5}. 
Hence, \As{Q2} holds for $ m = 0 $.

For $ m \in \N $, the result follows from \cite[Lem.\,2.3]{o.tp} 
by reducing to the case $ m = 0 $,
and we therefore omit the proof.
\PFEND
%GTP]
\subsection{Weak solution and the IFC condition to ISDEs} \label{s:IFC} 
 In \ssref{s:IFC}, we quote the IFC condition from \cite{o-t.tail}. 

Let $ \SSsi = \{ \sss \in \sSS \, ;\, \sss (\Rd ) = \infty , \sss (\{ s \} ) \in \{ 0 , 1 \} \text{ for all } s \in \Rd \} 
$. 
Let $ \SSsde $ be a Borel set such that $ \SSsde \subset \SSsi $. 
Let $\map{\ulabzone }{\RdSS }{\sSS }$ such that $ \ulabzone (x,\sss ) = \delta_x + \sss $. 
Let 
$ \SSSsde = \ulab ^{-1} (\SSsde ) $ and $ \SSSsdeone =( \ulabzone ) ^{-1} (\SSsde ) $.

Let 
$ \map{\sigma }{\SSSsdeone }{\mathbb{R}^{d^2}}$ and $ \map{\bb }{\SSSsdeone }{\Rd }$ be Borel measurable functions. 
We consider the ISDE of $ \X =(X^i)_{i\in\mathbb{N}}$ with state space $ \SSSsde $ such that 
\begin{align}\label{:IFCa}&%&\notag 
dX_t^i = \sigma (X_t^i,\XX _t^{\diai }) dB_t^i + 
\bb (X_t^i,\XX _t^{\diai }) dt \ (i\in\mathbb{N}),\quad 
\X _t \in \SSSsde 
.\end{align}
Here, $ \XX _t^{\diai } = \sum_{ j\not=i } \delta_{X_t^j} $ and 
$ \mathbf{B}=(B^i)_{i\in\mathbb{N}}$ is an $\RdN $-Brownian motion. 

Let $ \mathcal{L}^{p} $ be the set of all measurable $ \{ \mathscr{F}_t \}_{t \ge 0 } $-adapted 
processes $ \alpha $ such that 
$ E[ \int_0^T \vert \alpha (t,\omega) \vert ^p dt ] < \infty $ for all $ T $. 
\begin{definition}[weak solution]\label{d:IFC} 
By a weak solution of ISDE \eqref{:IFCa}, we mean an $ \RdN \ts \RdN $-valued stochastic process $ \XB $ defined on a probability space $ (\Omega , \mathscr{F}, P )$ 
with a reference family $ \{ \mathscr{F}_t \}_{t \ge 0 } $ such that

\noindent 
\thetag{i} $ \X =(X^i)_{i=1}^{\infty} $ is an $ \{ \mathscr{F}_t \}_{t \ge 0 }$-adapted, $ \SSSsde $-valued continuous process. 
\\
\thetag{ii} $ \mathbf{B} = (B^i)_{i=1}^{\infty}$ is an $ \RdN $-valued 
\FtB 
with $ \mathbf{B}_0 = \mathbf{0}$, 
\\
\thetag{iii} 
the family of measurable $ \{ \mathscr{F}_t \}_{t \ge 0 } $-adapted 
processes $ \Phi $ and $ \Psi $ defined by 
\begin{align}\notag & 
\Phi ^i(t,\omega ) = \sigma 
(X_t^i(\omega ),\XX _t^{\diai }(\omega )) ,\quad 
\Psi ^i(t,\omega ) = \bb (X_t^i(\omega ),\XX _t^{\diai }(\omega ))
\end{align}
belong to $ \mathcal{L}^{2} $ and $ \mathcal{L}^1 $, respectively. 
Here, we can and do take a predictable version of $ \Phi ^i $ and $ \Psi ^i $ (see pp 45-46 in \cite{IW}). 
\\
\thetag{iv} with probability one, the process $ \XB $ satisfies for all $ t $ 
\begin{align}\notag & 
X_t^i - X_0^i = 
\int_0^t \sigma (X_u^i,\XX _u^{\diai }) dB_u^i 
 + 
\int_0^t 
\bb (X_u^i,\XX _u^{\diai }) du \quad (i\in\mathbb{N}) 
.\end{align}
\end{definition}

Let $ \X = (X^i)_{i\in\mathbb{N}}$ be a weak solution to \eqref{:IFCa} starting at $ \mathbf{s} = \lab (\sss )$. Here, $\lab $ is the label given in \eqref{:02x}. 
Let $ \XX $ such that $ \XX _t = \sum_i \delta_{\xX _t^i } $. 

Define $ \map{ \sigmaXms } {[0,\infty) \ts \Rdm }{\mathbb{R}^{d^2}}$ and 
 $ \map{ \bbbXms }{[0,\infty) \ts \Rdm }{\Rd }$ such that, for $ (u,\mathbf{v}) \in \Rdm $ and 
$ \mathsf{v} = \sum_{i=1}^{m-1} \delta_{v_i} \in \sSS $, where 
$\mathbf{v}=(v_i)_{i=1}^{m-1} \in (\Rd )^{m-1} $, 
\begin{align} &\notag %
 \sigmaXms ( t, (u, \mathbf{v})) = 
 {\sigma } (u , \mathsf{v} + \XX _t^{m*}) ,\quad 
\bbbXms ( t, (u, \mathbf{v})) = {\bb } (u , \mathsf{v} + \XX _t^{m*})
.\end{align}
Here, $\XX _t^{m*} = \sum_{i=m+1}^{\infty} \delta_{X_t^i }$. 
For $ \mathbf{s}=(\si )_{i\in\mathbb{N}} $, we set $ \sss _m^* = \sum_{i=m+1}^{\infty} \delta_{\si }$. 
Recall that $ \X _0 = \mathbf{s} = \lab (\sss )$. 
We have $ \XX _0^{m*} = \sss _m^* $ by construction. 

The coefficients $ \sigmaXms $ and $ \bbbXms $ depend on 
both $ \XX ^{m*} $ and the label $ \lab $. 
In particular, $ \XX $ gives a part of $ \sigmaXms $ and $ \bbbXms $. 
Let 
\begin{align} \notag &
\SSSsdemtw = 
\{ \mathbf{s}^m = (s^i)_{i=1}^m \in \Rdm ; 
 \ulab (\mathbf{s}^m) + \ww _t^{m*} \in \SSsde 
 \} 
,\end{align}
where $ \ww _t^{m*}=\sum_{i=m+1}^{\infty}\delta_{w_t^i}$ for 
$ \ww _t = \sum_{i=1}^{\infty} \delta_{w_t^i}$. 
By definition, $ \SSSsdemtw $ is a time-dependent domain in $ \Rdm $ given by $ \ww _t^{m*}$.

We consider the SDE to $ \mathbf{Y}^m=(Y^{m,i})_{i=1}^m $ with random environment $ \XX $ defined on $ \OFpsF $ such that 
\begin{align} &\notag 
dY_t^{m,i} = 
\sigmaXms (t, (Y_t^{m,i},\mathbf{Y}_t^{m,\diai })) dB_t^i + 
 \bbbXms (t, (Y_t^{m,i},\mathbf{Y}_t^{m,\diai })) dt
,\\ &\notag 
 \mathbf{Y}_t^m \in \SSSsdemt \quad \text{ for all } t 
,\\\label{:IFCh}&
 \mathbf{Y}_0^{m} = \mathbf{s}^m ,
\quad \text{ where $ \mathbf{s}^m=(s^i)_{i=1}^m $ for $ \mathbf{s}=(\si )\in\SN $}
.\end{align}
Here, $ \mathbf{Y}^{m,\diai } = (Y^{m,j})_{j\not=i}^m $ and 
$ \mathsf{Y}_t^{m,\diai }=\sum_{j\not=i}^m \delta_{Y_t^{m,j}}$.

For $ \X = (X^i)_{i\in\mathbb{N}}$, let $ \X ^{m*} = (0,\ldots,0,X^{m+1}, X^{m+2},...) $. 
The first $ m $ components of $ \X ^{m*} $ are constant paths $ 0 $. A triplet $ (\mathbf{Y}^m,\mathbf{B}^m ,\X ^{m*} )$ of continuous processes on $ \OFpsF $ satisfying \eqref{:IFCh} is called a weak solution. 

Let 
$ \W = \CRd $ and $\WRdzm = \{ \w \in \WRdm ; \w (0) =\mathbf{0} \}$.

Let $ \Btm = \sigma [\w (s) ; 0\le s \le t , \w \in \WRdm ]$. 

Let $ \Bt (\WWdm ) = \sigma [(\mathbf{v}(s) ,\w (s) ) ; 0\le s \le t ] $. 

We set $ \Ehatmt = \overline{\Bt (\WWdm ) }$, and $ \Ehatm = 
\overline{\mathscr{B} (\WWdm ) }$, where $ \oL{\cdot}$ denotes the completion with respect to 
$ \PPPm = \Ps \circ (\lBlhatm )^{-1} $.

\begin{definition}
\label{d:IFC2} 
\thetag{i} 
 $\mathbf{Y} ^{m}$ is called a strong solution of \eqref{:IFCh} for $ \XB $ \uPs\ if 
$ (\mathbf{Y}^{m},\mathbf{B}^m,\X ^{m*})$ satisfies \eqref{:IFCh} and 
there exists a $ \Ehatm $-measurable function $ \map{\Fms }{\WWdm }{ \WRdm } $ 
such that $ \Fms $ is $ \Ehatmt / \Btm $-measurable for each $ t $, and $ \Fms $ satisfies %
$ \mathbf{Y}^m = \Fms (\lBlhatm )$ $ \Ps $-a.s.. 

\noindent \thetag{ii} 
The SDE \eqref{:IFCh} is said to have a unique strong solution for $ \XB $ \uPs\ 
if there exists a function $ \Fms $ satisfying the conditions in \dref{d:IFC2} \thetag{i} and, 
for any weak solution $ (\hat{\mathbf{Y}}^m ,\mathbf{B}^m,\X ^{m*})$ of \eqref{:IFCh} \uPs, 
$ \hat{\mathbf{Y}}^m = \Fms (\lBlhatm ) $ for $ \Ps $-a.s. 
\end{definition}
By construction, the function $ \Fms $ is unique for $ \PPPm $-a.s.

We introduce the IFC condition of $ \XB $ defined on $ \OFPF $: 
\smallskip

\noindent \As{\iFc} \ 
The SDE \eqref{:IFCh} has a unique strong solution $ \Fms (\mathbf{B}^m,\X ^{m*})$ 
for $ \XB $ under $ \Ps $ for $ P\circ \X _0 ^{-1}$-a.s. $ \mathbf{s}$ 
for all $ m \in \mathbb{N}$, where $ \Ps = P (\cdot \vert \X _0 = \mathbf{s})$.

\subsection{A unique strong solution of ISDEs} \label{s:US}

\begin{definition}[uniqueness in law]\label{d:US1}
Uniqueness in law of weak solutions for \eqref{:IFCa} with initial distribution $ \nu $ holds if 
the laws of the processes $ \X $ and $ \X '$ in $ \WRN $ coincide for any 
 weak solutions $ \X $ and $ \X '$ with initial distribution $ \nu $. 
 \end{definition}

\begin{definition}[pathwise uniqueness]\label{d:US2}
Pathwise uniqueness of weak solutions of 
\eqref{:IFCa} holds if whenever $ \X $ and $ \X '$ are two weak solutions 
defined on the same probability space $ (\Omega , \mathscr{F} ,P )$ with the same reference family $ \{ \mathscr{F}_t \}_{t \ge 0 }$ and the same $ \RdN $-valued \FtB 
$ \mathbf{B} $ such that $ \X _0=\X _0' $ a.s., 
then 
\begin{align}\notag 
P (\text{$ \X _t=\X _t'$ for all $ t \ge 0 $}) = 1 
.\end{align} 
\end{definition}

Let $ \PBr $ be the distribution of an $ \RdN $-valued Brownian motion 
$ \mathbf{B} $ with $ \mathbf{B}_0 = \mathbf{0}$. 
Let $ \Bt (\PBr )$ and $ \mathscr{B}(\PBr ) $ be the completion of 
$ \sigma [ \w _s ; 0\le s \le t ,\, \w \in \WRNz ] $ and $ \mathscr{B}(\WRNz ) $
with respect to $ \PBr $, respectively. 
Let $ \Bt ^{\infty} = \sigma [ \w (s) ; 0\le s \le t , \w \in \WRN ] $. 
\begin{definition}[a strong solution starting at $ \mathbf{s}$] \label{d:US4} 
A weak solution $ \X $ of \eqref{:IFCa} 
with an $ \RdN $-valued $ \mathscr{F}_t $-Brownian motion 
$\mathbf{B}$ is called a strong solution starting at $ \mathbf{s}$ defined on $ \OFPF $ 
if $ \X _0=\mathbf{s}$ a.s.\, and if there exists a function 
$ \map{\Fs }{\WRNz }{\WRN }$ such that 
$ \Fs $ is $ \mathscr{B}(\PBr ) /\mathscr{B}(\WRN )$-measurable, and that 
$ \Fs $ is $ \Bt (\PBr ) /\Bt ^{\infty}$-measurable for each $ t $, and 
that $ \Fs $ satisfies 
\begin{align}\notag & 
\X = \Fs (\mathbf{B}) \quad \text{ a.s.}
\end{align}
\end{definition}

%G[

\begin{definition}[a unique strong solution starting at $ \mathbf{s} $]\label{d:US5}
The ISDE \eqref{:IFCa} admits a unique strong solution starting at $ \mathbf{s} $
if there exists a $ \mathscr{B}(\PBr )/\mathscr{B}(\WRN ) $-measurable mapping
$ \map{\Fs }{\WRNz }{\WRN } $ such that the following hold: 

\smallskip
\noindent
\thetag{i}
For any weak solution $ (\hat{\X }, \hat{\mathbf{B}} ) $ of \eqref{:IFCa}
starting at $ \mathbf{s} $, it holds that
\begin{align}\notag
\hat{\X } = \Fs (\hat{\mathbf{B}}) \quad \text{a.s.}
\end{align}

\noindent
\thetag{ii}
For any $ \RdN $-valued $ \{ \mathscr{F}_t \} $-Brownian motion $ \mathbf{B} $
defined on $ \OFPF $ with $ \mathbf{B}_0 = \mathbf{0} $,
the continuous process $ \Fs (\mathbf{B}) $
is a strong solution of \eqref{:IFCa} starting at $ \mathbf{s} $.
\end{definition}
%G]%G[
\begin{definition}[a unique strong solution under constraints]\label{d:US6}
Let \As{$ \bullet $} be a condition.
The ISDE \eqref{:IFCa} admits a unique strong solution starting at $ \mathbf{s} $ 
under the constraints \As{$ \bullet $} if there exists a
$ \mathscr{B}(\PBr )/\mathscr{B}(\WRN ) $-measurable mapping
$ \map{\Fs }{\WRNz }{\WRN } $ such that the following hold: 

\smallskip
\noindent
\thetag{i}
For any weak solution $ (\hat{\X }, \hat{\mathbf{B}} ) $ of \eqref{:IFCa}
starting at $ \mathbf{s} $ and satisfying \As{$ \bullet $},
it holds that
\begin{align}\notag
\hat{\X } = \Fs (\hat{\mathbf{B}}) \quad \text{a.s.}
\end{align}

\noindent
\thetag{ii}
For any $ \RdN $-valued $ \{ \mathscr{F}_t \} $-Brownian motion $ \mathbf{B} $
defined on $ \OFPF $ with $ \mathbf{B}_0 = \mathbf{0} $,
the continuous process $ \Fs (\mathbf{B}) $
is a strong solution of \eqref{:IFCa} starting at $ \mathbf{s} $
and satisfying \As{$ \bullet $}.
\end{definition}
%G]
Let $ \WSsi $ be as in \lref{l:top1}. Let 
\begin{align} &\notag 
 \WSsiNE = \{\ww = \{ ( w ^i , I_i ) \}_{i \in \mathbb{N}} \in \WSsi ; I_i = [0,\infty ) \text{ for all $ i \in \mathbb{N}$}\} 
,\\& \notag 
\IRT (\w ) = 
\sup\{ i \in \mathbb{N}; \min_{t\in[0,T]} \lvert w^i(t)\rvert \le \rR \}, 
\quad \text{where $ \w = (w^i) \in \WRN $}
.\end{align}
Let $ \XX $ be the unlabeled path of $ \X = (X^i)$ given by $ \XX _t = \sum_{i} \delta _{X_t^i}$. 

\smallskip 
\noindent 
\As{SIN} \,$ P ( \XX \in \WSsiNE ) = 1$. \\
\As{NBJ} $ P (\IRT (\X ) < \infty ) = 1 $ for each $ R , T \in \mathbb{N}$. \\
\As{AC}$_{\nu }$ $ P \circ \XX _t^{-1}\ll \nu $ for all $ 0 < t < \infty $. 

Recall the condition \As{IFC} introduced in \ssref{s:IFC}.

For a strong solution $ \Fs $, we impose the following condition: 

\noindent 
\As{MF} $ P (\Fs (\mathbf{B}) \in A )$ is $ \oL{\mathscr{B} (\RdN ) }$-measurable in $ \mathbf{s}$ for any $ A \in \mathscr{B} ( \WRN ) $, where the completion is taken with respect to $ {P \circ \X _0 ^{-1}}$.

\begin{theorem}[{\cite[Th.\,3.1, p.1157]{o-t.tail}}] \label{l:US1}
Assume that $ \nu $ is tail trivial. 
Assume that \eqref{:IFCa} has a weak solution $ \XB $ under $ P $ satisfying \As{SIN}, \As{NBJ}, \As{AC}$_{\nu }$, and \As{IFC}. 
It holds that \eqref{:IFCa} has a family of unique strong solutions $ \{ F_{\mathbf{s}} \} $ starting at $ \mathbf{s}$ for $ P\circ \X _0^{-1}$-a.s.\,$ \mathbf{s}$ under the constraints  \As{SIN}, \As{NBJ}, \As{AC}$_{\nu }$, \As{IFC}, and \As{MF}. 
\end{theorem}

\begin{theorem}[{\cite[Cor.\,3.2, p.1158]{o-t.tail}}] \label{l:US2}
Under the same assumptions as \tref{l:US1}, the following hold: 
\\\thetag{i}
The uniqueness in law of weak solutions of \eqref{:IFCa} with initial distribution $ \nu $ holds under the constraints  \As{SIN}, \As{NBJ}, \As{AC}$_{\nu }$, and \As{IFC}. 
\\\thetag{ii}
The pathwise uniqueness of weak solutions of \eqref{:IFCa} with initial distribution $ \nu $ holds under the constraints  \As{SIN}, \As{NBJ}, \As{AC}$_{\nu }$, and \As{IFC}. 
\end{theorem}

\subsection{Cut-off compact sets and a sufficient condition for \As{IFC}} \label{s:IFC2}
In \ssref{s:IFC2}, we present a sufficient condition for the IFC condition.

We localize the coefficients of \eqref{:IFCh} to prove the IFC condition. 
To do this, we introduce a sequence of compact subsets $ \Han $ in $ \RdmSS $. 
Let $ \SSsi $ be the set consisting of infinite configurations with no multiple points. 

Let $ \mathbf{x}=(\xone ,\ldots, \xm ) \in \Rdm $, $ \ulab (\mathbf{x}) = \sum_{i=1}^m \delta_{\xI }$, 
and $\sss =\sum_i \delta_{\si }$. We set 
\begin{align}\label{:CUTt} & 
\SSsi ^{[m]} = \{ (\mathbf{x},\sss )\in \RdmSS \, ;\, \ulab (\mathbf{x}) + \sss \in \SSsi \} 
.\end{align}
We set $ \RRr = \{ \x \in \Rd ; |\x | \le \rr \}^m $. Let $ j,k,l=1,\ldots,m $. Let 
\begin{align} \notag &
\RRprs = \big\{ \mathbf{x} \in \RRr \, ;\, \min_{j\not=k } |\xj -\xk | \ge 2^{-\pp } ,\ 
\inf_{l,i} |\x ^l -\si | \ge 2^{-\pp } \big\} 
,\\& \label{:CUTu}
\RRprCs = \big\{ \mathbf{x} \in \sS _{\rr }^m \, ;\, 
\inf_{j\not=k } |\xj -\xk | > 2^{-\pp } , \inf_{l,i} |\x ^l -\si | > 2^{-\pp } \big\} 
.\end{align}
Then $ \RRprCs $ is an open set and $ \RRprs $ is its closure in $ \Rdm $.

Let $ \anest =\{ \ak \}_{\qqq \in\mathbb{N}} $ and $ \Kakk $ be as in \eqref{:ak}--\eqref{:CUTy}. 
For $ \mmm = (\pqr ) $, let 
\begin{align} \notag 
 \Han =\Ha _{\pqr } & = 
\big\{ (\mathbf{x},\sss ) \in \SSsi ^{[m]} \, ;\, \ \mathbf{x} \in \RRprs ,\ \sss \in \Kakk \big\} 
,\\ \label{:CUTc}
\Ha  & = \bigcup_{\pqr   \in \NNNone } \Ha _{\pqr } 
.\end{align}

Although $ \Han = \Ha _{\pqr } $ depends on $ m \in \N $, we suppress the
dependence on $ m $ in the notation. Clearly, the set $ \Ha _{\pqr } $
is increasing in each parameter $ \pqr \in \N $. Hence, by monotonicity,
the limit of the following exit times $ \varsigma _{\pqr } $ as
$ \pqr \to \infty $, taken successively in $ \pp $, $ \qqq $, and $ \rr $,
is well defined.

\medskip

Let $ \SSsde $ be the set defined before \eqref{:IFCa}. 
By definition, $ \SSsde $ is the set such that the coefficient of \eqref{:IFCa} is defined on 
$ \SSsde ^{[1]} = (\ulabzone )^{-1} (\SSsde )$. 
Here we extend the domain of $ \ulab $ to $ \RdmSS $ such that $ \ulabmz (\mathbf{x},\sss ) = \ulab (\mathbf{x}) + \sss $. 

Let $ \XB $ be a weak solution of \eqref{:IFCa} defined on $ \OFPF $. 
Let $ \Xm $ be the associated $ m $-labeled process. 
Let $ \nu $ be the RPF in \As{AC}$_{ \nu }$. 

\smallskip 
\noindent 
\Ass{B1} There exists a sequence $ \anest = \{ a_{\qqq } \}_{\qqq \in \mathbb{N} } $
of sequences such that the following conditions hold.

\noindent \thetag{i} 
$\nu ( \sSS \backslash \bigcup_{\qqq } \Kakk ) = 0 $.

\noindent 
\thetag{ii} 
 The inclusion $ \ulabmz (\Ha ) \subset \SSsde $ holds. 

\noindent \thetag{iii} 
$ \Xm =(\X ^m,\XX ^{m*}) $  does not exit from $ \Ha = \bigcup_{ (\pqr ) \in \NNNthree }\Ha _{\pqr } $: 
\begin{align}& \notag %\label{:CUTh} 
\pP ( \limi{\rr }\limi{\qqq }\limi{\pp } \varsigma _{\pqr }(\Xm ) = \infty ) = 1 
.\end{align}
Here $ \varsigma _{\pqr } (\UV ) = \inf \{ t > 0 ; \UV _t \not\in \Ha _{\pqr } \} $ is the exit time from $ \Ha _{\pqr } $. 

\medskip 

Let $ \sigma $ and $ \bb $ be the coefficients in \eqref{:IFCa}. 
We set
\begin{align}\label{:IFC1d}&
 \sigma ^m (\mathbf{x},\sss ) = (\sigma ( x^i , \xx ^{\diai } + \sss ) )_{i=1}^m
,\quad 
 \bb ^m (\mathbf{x},\sss ) = (\bb ( x^i , \xx ^{\diai } + \sss ) )_{i=1}^m
,\end{align}
where 
$ \mathbf{x}=(x^i)_{i=1}^m \in \Rdm $ and $ \xx ^{\diai } = \sum_{j\ne i } \delta_{x^j}$. 

For $ \mmm = (\pqr ) \in \NNNthree $ and $ (\mathbf{x} ,\sss ) , (\mathbf{y} ,\sss ) \in \RRprCs $, we set 
$ (\mathbf{x} ,\sss ) \sim_{\mmm } (\mathbf{y} ,\sss ) $ 
if $ \mathbf{x} $ and $ \mathbf{y} $ are in the same connected component of 
$ \RRprCs $ and $ \sss \in \HmnPc $. 
Here, $ \Pi _2 $ is a projection $ \map{\Pi _2 }{\RdmSS }{\sSS }$ given by $ \Pi _2 (\mathbf{x},\sss ) = \sss $.

With these preparation, we then make the assumptions: 

\smallskip 

\noindent 
\Ass{B2} 	
For each $ \mmm = (\pqr ) \in \NNNthree $, there exists a constant $ \FFn $ such that 
\begin{align}&\label{:B2a} %\tag{\ref{:B2a}}
| \sigma ^m (\mathbf{x} , \sss ) - \sigma ^m (\mathbf{y},\sss ) | \le \FFn |\mathbf{x} - \mathbf{y} |
,\\ & \label{:B2b} %\tag{\ref{:B2b}} 
| \bbb ^m (\mathbf{x} , \sss ) - \bbb ^m (\mathbf{y},\sss ) | \le \FFn |\mathbf{x} - \mathbf{y} |
\end{align}
 for all $ (\mathbf{x} , \sss ) , (\mathbf{y} , \sss ) \in \Han $ such that 
$ (\mathbf{x} ,\sss ) \sim_{\mmm } ( \mathbf{y} , \sss ) $, where $ \Han $ is the set defined in \eqref{:CUTc} and $ \sim_{\mmm } $ is the equivalence relation defined after \eqref{:IFC1d}. 

\begin{theorem}\label{l:IFC1}
Suppose that \eqref{:IFCa} has a weak solution $ \XB $ satisfying $ \{ \mathbf{B1} \} $ and that $ \{ \mathbf{B2} \} $ holds. 
Then $ \XB $ satisfies \As{IFC}. 
\end{theorem}

\begin{proof}
\tref{l:IFC1} is essentially proved in \cite[Prop.\,11.1]{o-t.tail}.
Since the formulation there is slightly different, we explain the
correspondence between the assumptions.

(\textbf{A1})--(\textbf{A4}) in \cite{o-t.tail} are basic assumptions on
the dynamics, which have already been verified in the present paper in a
form adapted to the present setting.
(\textbf{B1}) in \cite{o-t.tail} corresponds to
$ \nu ( \sSS \backslash \bigcup_{\qqq } \Kakk ) = 0 $, which follows from \Ass{B1}\,\thetag{i}.
(\textbf{B2}) in \cite{o-t.tail} follows from \Ass{B1} and \Ass{B2}. 

In \cite[Prop.\,11.1]{o-t.tail}, it is further assumed that the
coefficient $ \sigma^m $ of the ISDE is constant. This assumption can,
however, be removed. Indeed, it suffices to use
\cite[Lem.\,11.2]{o-t.tail}\thetag{1} in place of
\cite[Lem.\,11.2]{o-t.tail}\thetag{3} in the proof, which can be verified directly. 
\PFEND

The next step is to provide a sufficient condition for \Ass{B1}. 

Let $ a = \sigma {}^t \sigma $, where $ \sigma $ is as in $ \eqref{:IFCa} $. 

\noindent 
\Ass{UB} 
$ a = (a _{kl } (x,\sss ) ) _{k,l=1}^d $ is uniformly elliptic with upper bound $ \Ct \label{;34}$: 
\begin{align} 	\notag &%	\label{:UB}&
\sum _{k,l=1}^d a _{kl } (x,\sss ) \xi _{k} \xi _{l} \le \cref{;34} |\xi |^2 ,\quad \xi \in \Rd , (x,\sss ) \in \SSsde ^{[1]}
.\end{align}

Let $ \la $ be an RPF such that $ \la (\SSsde )=1 $. 
Let  $ \la ^{[m]}$ be the $ m $-reduced Campbell measure of $ \la $ for $ m \in \N $ and $ \la ^{[0]} = \la $. 

Let $ \map{\lab }{\SSsi }{\SN }$ be a label. 
Let $\{ \QQs \}$ be a family of probability measures on $ \OFF $ such that $ \XB $ defined on $ \OFQFs $ 
 is a weak solution of \eqref{:IFCa} starting at $ \s = \lab (\sss )$ for $ \la $-a.s.\,$ \sss $. 

We assume $\{ \QQs \}$ is a measurable family in the following sense.

\smallskip 
\noindent 
\Ass{MF} 
$ \QQs (A)$ is 
$ \overline{\mathscr{B}(\sSS )}^{\la } $-measurable in $ \sss $ for each $ A \in \mathcal{F} $. 
\smallskip 

Assume \Ass{MF} and let $ \QQla = \int _{\sSS } \QQs d\la $. 
Then $ \XB $ under $ \QQla $ is a solution of \eqref{:IFCa} with the initial distribution $ \la \circ \lab ^{-1}$. 
\smallskip 

\noindent 
\Ass{BX} 
$ \sigma [\mathbf{B}_s ; s\le t] \subset \sigma [\X _s ;s\le t ]$ for all $ t $ under $ \QQla $. 

\smallskip

Let $ \QQxsM $ be the distribution of $ \Xm = (\X ^m, \XX ^{m*}) $ under 
$ \qQ _{\ulab (\mathbf{x}) + \sss } $. We define 
$ \QQlaM = \int _{\RdmSS } \QQxsM d\la ^{[m]} $ for $ m \in \N $ and $ \QQla ^{[0]} = \QQla \circ \XX ^{-1}$. 

\smallskip 

\noindent 
\Ass{S$_{\la }$} 
For each $ m \in \zN $, $ \Xm $ under $ \OFQFm $ 
gives a symmetric, Markovian semi-group $ T_t^{[m]} $ on $ L^2( \la ^{[m]})$ defined by 
\begin{align} 	\notag &%	\label{:IFC3c}&
T_t^{[m]} f ( \mathbf{x},\sss ) = \int _{C([0,\infty); \RdmSS ) }f (\w _t^{[m]}) d \QQxsM 
.\end{align}
Furthermore, $ \la ^{[m]}$ is an invariant measure of $ T_t^{[m]} $. 

\ms

We label $ \sss =\sum_i \delta_{\si }$ in such a way that $ |\si | \le |\sii | $ for all $ i $. 
Let $ \ak = \{ \ak (\rr ) \}_{\rr \in\mathbb{N}} $ be the increasing sequences in \eqref{:ak}. 

For $ \qQ \in \N \cup \{ \infty \} $, define 
\begin{align} &\notag %\label{:CUTm}&
\dkQ (\sss ) = 
 \Big\{ \sum_{\rR =1}^{\qQ } \sum_{ \sioLSR , i > \akR }(\rR - |\si |)^2 \Big\}^{1/2} 
. \end{align}
Let $ \theta \in C^{\infty}(\R )$ such that $ 0 \le \theta (t) \le 1 $ for all $ t \in \R $, 
$ \theta (t) = 0 $ for $ t \le \epsilon $, and 
$ \theta (t) = 1 $ for $ t \ge 1 - \epsilon $ for a sufficiently small $ \epsilon > 0 $. 
Assume that $ 0 \le \theta ' (t) \le \sqrt{2} $ for all $ t \in \R $. 
Let 
\begin{align}\label{:IFC3f}&
\chiwtI (\sss ) = \limi{\qQ }\sum_{ q =1}^{\infty} \theta \circ \dkQ (\sss ) 
.\end{align}
The limit in \eqref{:IFC3f} exists since $ \theta \circ \dkQ \ge 0 $ is non-decreasing in $ \qQ $. 

\begin{lemma}[{\cite[Lem.\,5.4]{k-o-t.ifc}}] \label{l:IFC5} 
Let $ \Ki [\ak ] $ be as in \eqref{:CUTw}. 
 Assume  \eqref{:CUTy} and 
$ \sum_{ \qqq =1}^{\infty} \qqq ^2 \la ( \Ki [\ak ]^c ) < \infty $. 
Then 
\begin{align}\label{:IFC3g}&
 \int_{\sSS } \lvert \chiwtI \rvert ^2 d\la < \infty 
.\end{align}
\end{lemma}
The following lemma is a key step in verifying \Ass{B1}. 
\begin{lemma}[{\cite[Prop.\,6.3]{k-o-t.ifc}}]\label{l:IFC)}
Assume $\{ \mathbf{UB}\} , \{ \mathbf{MF}\} , \{ \mathbf{BX}\} , \{ \mathbf{S}_{\la }\} $, and \eqref{:IFC3g}.
Let $ \XX = \upath (\X )$ be the unlabeled dynamics under $ \QQla $. 
Then 
\begin{align*}&
P ( \limi{\qqq } \kappa_{\qqq } = \infty ) = 1
.\end{align*}
Here $\kappa_{\qqq } = \inf \{ t > 0 \,;\, \XX_t \notin \Ki[\ak ] \} $. 
\end{lemma}
\begin{proof}
The proof follows from \cite[Prop.\,6.2]{k-o-t.ifc}, where an additional assumption \Ass{D} was imposed. 
This assumption was not used in \cite[Prop.\,6.3]{k-o-t.ifc}. (It was used in \cite[Prop.\,6.2]{k-o-t.ifc}). 
\PFEND

\end{document}